\documentclass{conm-p-l}

\usepackage{amssymb}

\usepackage[cmtip,all]{xy}

\usepackage{tikz}
\usepgflibrary{arrows}
\usetikzlibrary{matrix}

\newtheorem{thm}{Theorem}[section]

\newtheorem{lem}[thm]{Lemma}
\newtheorem{cor}[thm]{Corollary}

\newtheorem{prop}[thm]{Proposition}
\newtheorem{claim}[thm]{Claim}

\theoremstyle{definition}
\newtheorem{defi}[thm]{Definition}

\newtheorem{example}[thm]{Example}
\newtheorem{pty}[thm]{Property}
\newtheorem{exer}[thm]{Exercise}

\theoremstyle{remark}
\newtheorem{rem}[thm]{Remark}

\numberwithin{equation}{section}

\newcommand{\bfzero}{{\bf 0}}
\newcommand{\bfone}{{\bf 1}}

\newcommand{\ed}{{\bullet\hspace{-0.05cm}-\hspace{-0.05cm}\bullet}}

\newcommand{\br}{{\}\hspace{-0.07cm}.\hspace{-0.03cm}.\hspace{-0.07cm}\} }}

\newcommand{\tri}{{~\blacktriangleleft~}}
\newcommand{\bb}{{\bullet\, \bullet}}
\newcommand{\cc}{{\circ\hspace{-0.05cm}-\hspace{-0.05cm}\circ}}
\newcommand{\ww}{{\circ\,\circ}}

\newcommand{\lp}{{\circlearrowleft}}
\newcommand{\nl}{\not{\circlearrowleft}}

\newcommand{\colim}{{\mathsf{colim}}}

\newcommand{\e}{{\mathsf{e}}}

\newcommand{\Graphs}{{\mathsf{Graphs}}}

\newcommand{\fGraphs}{{\mathsf{fGraphs}}}

\newcommand{\sfGraphs}{{\mathsf{fGraphs}^{\sharp}}}
\newcommand{\sGraphs}{{\mathsf{Graphs}^{\sharp}}}
\newcommand{\snlGraphs}{{\mathsf{Graphs}^{\sharp}_{\nl}}}
\newcommand{\snlgraphs}{{\mathsf{graphs}^{\sharp}_{\nl}}}

\newcommand{\sfgraphs}{{\mathsf{fgraphs}^{\sharp}}}
\newcommand{\sgraphs}{{\mathsf{graphs}^{\sharp}}}

\newcommand{\fgraphs}{{\mathsf{fgraphs}}}
\newcommand{\graphs}{{\mathsf{graphs}}}

\newcommand{\Gra}{{\mathsf{Gra}}}

\newcommand{\GC}{{\mathsf{GC} }}
\newcommand{\fGC}{{\mathsf{fGC}}}

\newcommand{\gra}{{\mathsf{g r a} }}

\newcommand{\Tree}{{\mathsf{T r e e}}}

\newcommand{\Lie}{{\mathsf{Lie}}}
\newcommand{\Com}{{\mathsf{Com}}}

\newcommand{\coLie}{{\mathsf{coLie}}}
\newcommand{\coCom}{{\mathsf{coCom}}}

\newcommand{\Ger}{{\mathsf{Ger}}}

\newcommand{\GRT}{{\mathsf{ GRT}}}

\newcommand{\Cobar}{{\mathrm{ C o b a r}}}

\newcommand{\Conv}{{\mathrm{Conv}}}
\newcommand{\Ch}{{\mathsf{C h}}}
\newcommand{\grVect}{{\mathsf{grVect}}}

\renewcommand{\c}{{\circ}}

\newcommand{\Tw}{{\mathrm{Tw}}}

\newcommand{\MC}{{\mathrm{MC} }}
\newcommand{\CH}{{\mathrm{CH}}}

\newcommand{\End}{{\mathsf {E n d} }}
\newcommand{\Hom}{{\mathrm {Hom}}}
\newcommand{\PV}{{\mathrm {PV}}}
\newcommand{\emb}{{\mathrm {emb}}}
\newcommand{\Cone}{{\mathrm {Cone}}}
\newcommand{\Aut}{{\mathrm {Aut}}}
\newcommand{\Av}{{\mathrm {Av}}}
\newcommand{\oub}{{\mathrm {oub}}}
\newcommand{\cab}{{-\hspace{-0.12cm}-}}

\newcommand{\Gr}{{\mathrm {Gr}}}

\newcommand{\coDer}{{\mathrm {coDer}}}
\newcommand{\id}{{\mathsf{ i d} }}
\newcommand{\ad}{{\rm a d }}
\newcommand{\Sh}{{\mathrm {S h} }}
\newcommand{\conn}{{\mathrm {conn} }}
\newcommand{\Frame}{{\mathsf{Frame}}}

\newcommand{\wt}[1]{{\widetilde{#1}}}
\newcommand{\ti}[1]{{\tilde{#1}}}
\newcommand{\wh}[1]{{\widehat{#1}}}

\newcommand{\und}[1]{{\underline{#1}}}

\newcommand{\dia}{\diamond}
\newcommand{\hrt}{\heartsuit}

\newcommand{\tcL}{\widetilde{\mathcal L}}

\newcommand{\Tp}{{\mathrm{Tp}}}

\newcommand{\Cbu}{C^{\bullet}}

\newcommand{\nod}{{\mathrm{nod}}}

\newcommand{\tp}{{\mathrm{tp}}}

\newcommand{\tG}{\widetilde{\Gamma}}

\newcommand{\al}{{\alpha}}
\newcommand{\Ups}{{\Upsilon}}
\newcommand{\la}{{\lambda}}

\newcommand{\bul}{{\bullet}}

\newcommand{\mR}{{\mathfrak{R}}}

\newcommand{\mI}{{\mathfrak{I}}}

\newcommand{\mb}{{\mathfrak{b}}}
\newcommand{\mc}{{\mathfrak{c}}}
\newcommand{\ml}{{\mathfrak{l}}}

\newcommand{\mj}{{\mathfrak{j}}}

\newcommand{\mL}{{\mathfrak{L}}}

\newcommand{\mC}{{\mathfrak{C}}}

\newcommand{\grt}{{\mathfrak{grt}}}

\newcommand{\vs}{{\varsigma}}
\newcommand{\vr}{{\varrho}}

\newcommand{\si}{{\sigma}}
\newcommand{\ga}{{\gamma}}
\newcommand{\io}{{\iota}}
\newcommand{\vf}{{\varphi}}
\newcommand{\ve}{{\varepsilon}}
\newcommand{\ka}{{\kappa}}
\newcommand{\G}{{\Gamma}}
\newcommand{\Gim}{{\gimel}}
\newcommand{\cF}{{\mathcal F}}
\newcommand{\pa}{{\partial}}

\newcommand{\bsi}{{\bf s}^{-1}\,}

\newcommand{\bs}{{\bf s}}
\newcommand{\bt}{{\bf t}}
\newcommand{\bq}{{\bf q}}
\newcommand{\bu}{{\bf u}}

\newcommand{\cK}{{\mathcal K}}
\newcommand{\cC}{{\mathcal C}}
\newcommand{\cQ}{{\mathcal Q}}
\newcommand{\cO}{{\mathcal O}}

\newcommand{\cN}{{\mathcal N}}
\newcommand{\cL}{{\mathcal L}}

\newcommand{\cD}{{\mathcal D}}

\newcommand{\cG}{{\mathcal G}}
\newcommand{\cH}{{\mathcal H}}

\newcommand{\cV}{{\mathcal V}}

\newcommand{\Op}{{\mathbb{O P}}}
\newcommand{\psop}{\Psi{\mathbb{O P}}}

\newcommand{\bbK}{{\mathbb K}}

\newcommand{\bbR}{{\mathbb R}}
\newcommand{\bbZ}{{\mathbb Z}}
\newcommand{\bbQ}{{\mathbb Q}}

\newcommand{\La}{{\Lambda}}

\newcommand{\te}{\theta}
\newcommand{\Te}{\Theta}

\newcommand{\de}{{\delta}}
\newcommand{\D}{{\Delta}}

\renewcommand{\Im}{{\mathrm{Im}}}

\newcommand{\sgn}{{\mathrm {s g n}}}

\newcommand{\onto}{\twoheadrightarrow}

\renewcommand{\L}{\langle}
\newcommand{\R}{\rangle}

\begin{document}





\title[Notes on Algebraic Operads]{Notes on Algebraic Operads, Graph
  Complexes, and Willwacher's Construction}


\author{Vasily A. Dolgushev}
\address{Department of Mathematics,
Temple University, \\
Wachman Hall Rm. 638\\
1805 N. Broad St.,\\
Philadelphia PA, 19122 USA}
\curraddr{}
\email{vald@temple.edu}
\thanks{V.A.D.\ was partially supported 
by the NSF grant DMS 0856196 and
the grant FASI RF 14.740.11.0347.}

\author{Christopher L. Rogers}
\address{Courant Research Centre\\
  ``Higher Order Structures in Mathematics''\\
  University  of G\"{o}ttingen \\
  Bunsenstra\ss e 3-5 \\
  D-37073 G\"{o}ttingen, Germany} \curraddr{}
\email{crogers@uni-math.gwdg.de} \thanks{C.L.R.\ was partially supported by the
  NSF grant DMS 0856196 and the German Research Foundation (DFG)
  through the Institutional Strategy of the University of
  G\"{o}ttingen.}


\date{}

\begin{abstract}
We give a detailed proof of T. Willwacher's 
theorem \cite{Thomas} which links the cohomology 
of the full graph complex $\fGC$ to the cohomology of the deformation 
complex of the operad $\Ger$, governing Gerstenhaber algebras.  
We also present various prerequisites required for understanding the 
material of \cite{Thomas}. In particular, we review  operads, cooperads, 
and the cobar construction. We give a detailed exposition of 
the convolution Lie algebra and its properties. We prove a useful 
lifting property for maps from a dg operad obtained via 
the cobar construction.  We describe in detail  Willwacher's 
twisting construction, and then use it to work  
with various operads assembled from graphs, in particular, the full graph 
complex and its subcomplexes. 
These notes are loosely based 
on lectures given by the first author at the Graduate and Postdoc 
Summer School at the Center for Mathematics at Notre Dame
(May 31 - June 4, 2011).
\end{abstract}

\maketitle





\begin{flushright}
 ~\\
{\it  To Orit and Rosie}

 ~\\
\end{flushright}

\tableofcontents

\section{Introduction}
In his seminal paper \cite{K}, M. Kontsevich constructed 
an $L_{\infty}$ quasi-iso\-mor\-phism from the graded Lie algebra $\PV_d$ of 
polyvector fields on the affine space $\bbR^d$ to the differential graded (dg) 
Lie algebra of Hochschild cochains 
\begin{equation}
\label{Cbu-A}
\Cbu(A) = \bigoplus_{m = 0}^{\infty} \Hom(A^{\otimes\, m},  A)
\end{equation}
for the polynomial algebra $A = \bbR[x^1,x^2, \dots, x^d]$\,. Among other things, this 
result implies that formal associative deformations of the algebra  $A$
can be described in terms of formal Poisson structures on $\bbR^d$.

According to \cite{K-conjecture}, there exist many homotopy inequivalent 
$L_{\infty}$ quasi-iso\-mor\-phisms 
\begin{equation}
\label{PV-d-Cbu-A}
\PV_d \leadsto \Cbu(A) 
\end{equation}
from $\PV_d$ to $\Cbu(A)$\,.
More precisely, the full graph complex $\fGC$ (see Section \ref{sec:fGC-first}) maps to 
the Chevalley-Eilenberg complex of $\PV_d$ and, using this map, one can
define an action of the Lie algebra $H^0(\fGC)$ on the homotopy classes of 
$L_{\infty}$ quasi-iso\-mor\-phisms \eqref{PV-d-Cbu-A}.

In 1998,  D. Tamarkin \cite{Hinich}, \cite{Dima-Proof} proposed a completely 
different approach to constructing $L_{\infty}$ quasi-iso\-mor\-phisms \eqref{PV-d-Cbu-A}.
His approach works for an arbitrary field $\bbK$ of characteristic zero and it is
based on several deep results such as a proof of Deligne's conjecture on Hochschild
complex \cite{swiss}, \cite{K-Soi}, \cite{M-Smith}, the formality for the operad of 
little discs \cite{Dima-Disk}, and the existence of a Drinfeld associator 
\cite{Drinfeld}.  

The main idea of Tamarkin's approach to Kontsevich's formality 
theorem is to use the existence of a $\Ger_{\infty}$-structure on 
the Hochschild complex $\Cbu(A)$ \eqref{Cbu-A}, 
whose structure maps are expressed in terms of the cup product 
and insertions of cochains into a cochain. Showing the existence 
of such a $\Ger_{\infty}$-structure  is the most difficult and 
the most interesting part of the proof. The construction of this  
$\Ger_{\infty}$-structure involves the choice of a Drinfeld associator. 
Furthermore, it is known \cite{Dima-GT} that different choices of Drinfeld associators 
result in homotopy inequivalent $\Ger_{\infty}$-structures. 

According to \cite{Drinfeld}, the set of Drinfeld associators forms a torsor
(i.e. principle homogeneous space) for an infinite dimensional algebraic 
group $\GRT$, which is called the Grothendieck-Teichmueller 
group\footnote{Following \cite{AT} we denote by $\GRT$ the unipotent 
radical of the group introduced by Drinfeld.}. This group is related to moduli 
of curves, to the absolute Galois group of the field of rationals, and to 
the theory of motives \cite{Furusho}.   

In preprint \cite{Thomas}, T. Willwacher established remarkable 
links\footnote{We believe that the same link between  the group $\GRT$ and 
the deformation complex of the operad $\Ger$ was established via different methods 
in paper \cite{Fresse-grt} by B. Fresse.} 
between three objects: the group $\GRT$, the full graph complex $\fGC$
and the deformation complex of the operad $\Ger$ governing Gerstenhaber algebras. 
Using these links one can connect the above seemingly unrelated stories: 
\begin{itemize}

\item Tamarkin's approach to Kontsevich's formality theorem based on the 
use of Drinfeld associators, and

\item the action of the full graph complex $\fGC$ on  $L_{\infty}$ quasi-isomorphisms 
\eqref{PV-d-Cbu-A}.

\end{itemize}
We refer the reader to \cite{stable1}, \cite{stable11}, and \cite{Br-infinity} for 
more details.

It is already clear that Willwacher's results have important consequences for
deformation quantization, and they will certainly play a very influential
role in future research. The details presented in \cite{Thomas},
however, are technically subtle and difficult to access -- even for
experts. Many intermediate steps in the proofs are either left for the reader, or embedded 
in remarks and comments throughout the text. Moreover, several key statements are 
proved for a particular case, and then used in their full generality.

The goal of these notes is to give a detailed proof of  T. Willwacher's 
theorem (See Theorem \ref{thm:main}) which links the cohomology 
of the full graph complex $\fGC$ to the cohomology of the deformation 
complex of the operad $\Ger$\,. 

In addition, we also present here various prerequisites
required for understanding the material of \cite{Thomas}. 
Thus, in Section \ref{sec:oper-coper}, we review operads, cooperads, 
and the cobar construction. This construction 
assigns to a coaugmented cooperad $\cC$ a free operad 
$\Cobar(\cC)$ with the differential defined in terms of the 
cooperad structure on $\cC$\,. 
In Section \ref{sec:convolution}, we give a detailed exposition of 
the convolution Lie algebra and its properties. 
In Section \ref{sec:to-invert}, we discuss homotopies of maps 
from $\Cobar(\cC)$ and prove a useful lifting property   
for such maps.

In Section \ref{sec:twist} we describe in detail Willwacher's twisting 
construction $\Tw$ which assigns to a dg operad $\cO$ and a 
map\footnote{The dg operad $\La\Lie_{\infty}$ differs from the dg 
operad $\Lie_{\infty}$ governing $L_{\infty}$-algebras by a degree shift.
Namely, $\La\Lie_{\infty}$-structures on a cochain complex $V$ 
are in bijection with $\Lie_{\infty}$-structures on $\bsi V$\,. }
(of dg operads)
\begin{equation}
\label{LaLie-cO}
\La\Lie_{\infty} \to \cO
\end{equation}
another dg operad $\Tw\cO$\,. We refer to $\Tw\cO$ as {\it the twisted version}
of the (dg) operad $\cO$\,.

Algebras over $\Tw\cO$ (satisfying minor technical conditions) can be identified 
with $\cO$-algebras equipped with a chosen Maurer-Cartan element for 
the  $\La\Lie_{\infty}$-structure induced by the map \eqref{LaLie-cO}. 
It is the twisting construction which gives us a convenient framework for working 
with various operads assembled from graphs, in particular, the full graph 
complex and its subcomplexes. 

In Section \ref{sec:Gra}, we introduce the operad $\Gra$ and 
define an embedding from the operad $\Ger$ to $\Gra$\,. 

In Section \ref{sec:fGC-first}, we introduce the full graph complex 
$\fGC$ and its ``connected part'' $\fGC_{\conn} \subset \fGC$\,. 
We also present a link between $\fGC$ and its subcomplex $\fGC_{\conn}$\,.
This link allows us to reduce the question of computing cohomology of $\fGC$
to the question of computing cohomology of $\fGC_{\conn}$\,.

Section \ref{sec:TwGra} is devoted to a thorough analysis of the 
dg operad $\Tw\Gra$ and its various suboperads. Several useful
statements about  suboperads of $\Tw\Gra$ and the operad $\Ger$ are assembled in 
the commutative diagram \eqref{master-TwGra} at the end of Section 
\ref{sec:TwGra}.
  
In Section \ref{sec:fGC-revisited}, we use the results of the previous 
section to deduce deeper statements about the full graph complex $\fGC$\,. 
In particular, we prove the decomposition theorem for the graph cohomology
(see Theorem \ref{thm:fGC-decomp}). 

In Section \ref{sec:Def-Ger}, we introduce the deformation complex \eqref{Def-Ger-new} of 
the operad $\Ger$ and prove a technical statement about this complex.

In Section \ref{sec:rigidity}, we consider the convolution Lie 
algebra  $\Conv(\Ger^{\vee}, \Gra)$ with the differential coming from 
a natural composition $\Cobar(\Ger^{\vee}) \to \Ger \to \Gra $\,.
We prove that the cohomology of $\Conv(\Ger^{\vee}, \Gra)$ is spanned by 
the class of a single given vector. In particular,  $\Conv(\Ger^{\vee}, \Gra)$
does not have non-zero cohomology classes ``coming from  
arities $\ge 3$''. This statement is a version of Tamarkin's rigidity theorem 
for the Gerstenhaber algebra $\PV_d$ of polyvector fields on $\bbK^d$\,, 
which is one of the corner stones of Tamarkin's proof of Kontsevich's 
formality theorem.

Section \ref{sec:DefGer-fGC} is the culmination of our notes.
In this section, we give a proof of Theorem \ref{thm:main} 
which links the cohomology of the ``connected part'' of the 
full graph complex $\fGC$ to the cohomology of the ``connected part'' of 
the deformation complex of the operad $\Ger$\,.  The cohomology of the 
full graph complex and the cohomology of the deformation complex of the operad $\Ger$
can be easily expressed in terms of the cohomology of their ``connected parts''.

The proof of Theorem \ref{thm:main} is assembled from several building blocks.
First, this proof relies on Corollary \ref{cor:Ger-fgraphs} which links the operad 
$\Ger$ to a suboperad of the dg operad $\Tw\Gra$\,. Second, it relies on 
technical Theorem \ref{thm:Xi} which is given in Subsection \ref{sec:thm-Xi}. 
This theorem states that the (extended) deformation complex of the operad 
$\Ger$ is quasi-isomorphic to a certain subcomplex. 
Finally,  the proof of Theorem \ref{thm:main} relies on a version 
of Tamarkin's rigidity (see Corollary \ref{cor:Conv-Ger-Gra-conn}).

We should remark that the proof of Theorem \ref{thm:main} given here is not 
different from the one outlined 
in Willwacher's preprint  \cite{Thomas}. We only make the logic ``more linear'' 
and fill in many omitted details.  
  
Appendices \ref{app:q-iso}, \ref{app:Harr}, \ref{app:GM} contain proofs 
of three useful statements: a lemma on a quasi-isomorphism between filtered complexes, 
the theorem on the Harrison homology of the cofree cocommutative coalgebra, 
and a version of the Goldman-Millson theorem \cite{GM}.  Although all these 
statements are well known, it is hard to find in the literature proofs which are formulated 
in the desired generality. 

Many minor steps in proofs are left as exercises, which are 
formulated in the body of the text. Appendix \ref{app:solutions} at 
the end of the paper contains solutions to some of these exercises.  

Theorem \ref{thm:main} accounts for only $30\%$ of results of
T. Willwacher's preprint  \cite{Thomas}. So we hope to write a separate 
paper, in which we will give a detailed proof of Willwacher's theorem 
which links the full graph complex to the 
Lie algebra $\grt$ of the Grothendieck-Teichmueller group $\GRT$.

In our exposition, we tried to follow (or rather not to follow)
Serre's suggestions from his famous lecture \cite{Serre}.  
We hope that this text will be useful both for specialists working on 
operads and deformation quantization, and for graduate students 
interested in this subject.

~\\
{\bf Acknowledgment.}
We would like to thank Thomas Willwacher for numerous 
illuminating discussions. We are also thankful to 
Thomas for his patience with explaining to us various 
unwritten but implied claims in his paper \cite{Thomas}. 
These notes are loosely based 
on lectures given by V.A.D. at the Graduate and Postdoc 
Summer School at the Center for Mathematics at Notre Dame
(May 31 - June 4, 2011). We would like to thank Samuel Evens 
and Michael Gekhtman for organizing such a wonderful summer school. 
We are thankful to our previous institution, the UC Riverside, 
in which we started discussing topics related to this paper.
V.A.D. would like to thank Ezra Getzler for his kind offer
to use his office during V.A.D.'s visit of Northwestern 
University in May of 2011. V.A.D. is also thankful to Brian Paljug 
for his remarks about early versions of the draft.

\subsection{Notation and Conventions}

The base field $\bbK$ has characteristic zero. 
For a set $X$ we denote by $\bbK\L X\R$ the $\bbK$-vector 
space of finite linear combinations of elements in $X$\,.

The underlying symmetric monoidal category $\mC$  is 
often the category $\Ch_{\bbK}$ of unbounded cochain complexes 
of $\bbK$-vector spaces or the category $\grVect_{\bbK}$ of 
$\bbZ$-graded $\bbK$-vector spaces. We will frequently use the ubiquitous
combination ``dg'' (differential graded) to refer to algebraic objects 
in $\Ch_{\bbK}$\,.  For a homogeneous vector $v$ in 
a cochain complex (or a graded vector space), $|v|$ denotes the degree of $v$\,.
We denote by $\bs$ (resp. $\bs^{-1}$) the operation of suspension 
(resp. desuspension).  Namely, for a cochain complex (or a graded vector space) 
$\cV$, we have 
$$
\big( \bs \, \cV \big)^{\bul} = \cV^{\bul-1}\,, \qquad
\big( \bsi  \cV \big)^{\bul} = \cV^{\bul+1}\,.
$$ 
The notation $\bfone$ is reserved for the unit of the  underlying  
 symmetric monoidal category $\mC$

By a commutative algebra we always mean commutative and 
associative algebra.
The notation $\Lie$ (resp. $\Com$, $\Ger$) is reserved for the 
operad governing Lie algebras (resp. commutative 
algebras without unit, Gerstenhaber algebras without unit).
Dually, the notation $\coLie$ (resp. $\coCom$) is reserved for the 
 cooperad governing Lie coalgebras (resp. cocommutative 
coalgebras without counit). The notation $\CH(x,y)$ is reserved for 
the Campbell-Hausdorff series in $x$ and $y$.

The notation $S_{n}$ is reserved for the symmetric group 
on $n$ letters and  $\Sh_{p_1, \dots, p_k}$ denotes 
the subset of $(p_1, \dots, p_k)$-shuffles 
in $S_n$, i.e.  $\Sh_{p_1, \dots, p_k}$ consists of 
elements $\si \in S_n$, $n= p_1 +p_2 + \dots + p_k$ such that 
$$
\begin{array}{c}
\si(1) <  \si(2) < \dots < \si(p_1),  \\[0.3cm]
\si(p_1+1) <  \si(p_1+2) < \dots < \si(p_1+p_2), \\[0.3cm]
\dots   \\[0.3cm]
\si(n-p_k+1) <  \si(n-p_k+2) < \dots < \si(n)\,.
\end{array}
$$
For $i \le j \le n$ we denote by $\vs_{i,j}$ the following cycle in $S_n$
\begin{equation}
\label{vs-i-j}
\vs_{i,j} = \begin{cases}
(i, i+1, \dots, j-1, j)  \qquad {\rm if} ~~ i < j\,, \\
 \id  \qquad {\rm if} ~~ i = j\,.
 \end{cases}  
\end{equation}
It is clear that 
\begin{equation}
\label{vs-i-j-inv}
\vs^{-1}_{i,j} = \begin{cases}
(j, j-1, \dots, i+1, i)  \qquad {\rm if} ~~ i < j\,, \\
 \id  \qquad {\rm if} ~~ i = j\,.
 \end{cases}  
\end{equation}
For example, the set $\{ \vs_{i,n} \}_{ 1 \le i \le n}$ is 
exactly the set $\Sh_{n-1,1}$ of $(n-1,1)$-shuffles and   
$\{ \vs^{-1}_{1,i} \}_{ 1 \le i \le n}$ is the set $\Sh_{1,n-1}$ of 
$(1, n-1)$-shuffles. 

For a group $G$ and a $G$-module $W$, the notation $W^{G}$ (resp. $W_{G}$)
is reserved for the subspace of $G$-invariants (resp. the quotient space of 
$G$-coinvariants).

For an operad $\cO$ (resp. a cooperad $\cC$) and a cochain complex $V$, the notation 
$\cO(V)$ (resp. $\cC(V)$) is reserved for the free $\cO$-algebra (resp. cofree $\cC$-coalgebra). 
Namely, 
\begin{equation}
\label{cO-V}
\cO(V) : = \bigoplus_{n \ge 0} \Big( \cO(n) \otimes V^{\otimes \, n} \Big)_{S_n}\,,
\end{equation}
\begin{equation}
\label{cC-V}
\cC(V) : = \bigoplus_{n \ge 0} \Big( \cC(n) \otimes V^{\otimes \, n} \Big)^{S_n}\,.
\end{equation}

For an augmented operad $\cO$ (in $\Ch_{\bbK}$) 
we denote by $\cO_{\c}$ the kernel of the augmentation.
Dually, for a coaugmented cooperad $\cC$ (in $\Ch_{\bbK}$) 
we denote by $\cC_{\c}$ the cokernel of the coaugmentation.
(We refer the reader to Subsections \ref{sec:aug-op} and
\ref{sec:coaug-coop} for more details.)

For a groupoid $\cG$ the notation $\pi_0(\cG)$ is reserved for 
the set of its isomorphism classes. 

A directed graph (resp. graph) $\G$ is a pair $(V(\G), E(\G))$, 
where $V(\G)$ is a finite non-empty 
set and $E(\G)$ is a set of ordered (resp. unordered) pairs 
of elements of $V(\G)$. Elements of $V(\G)$ are called vertices 
and elements of $E(\G)$ are called edges. 
We say that a directed graph (resp. graph) $\G$ is labeled 
if it is equipped with a bijection between the set $V(\G)$ and 
the set of numbers $\{1,2, \dots, |V(\G)|\}$\,. We allow a graph 
with the empty set of edges. 

\begin{figure}[htp] 
\centering
\begin{minipage}[t]{0.3\linewidth} 
\centering 
\begin{tikzpicture}[scale=0.5, >=stealth']
\tikzstyle{w}=[circle, draw, minimum size=4, inner sep=1]
\tikzstyle{b}=[circle, draw, fill, minimum size=4, inner sep=1]
\node [b] (b1) at (0,0) {};
\draw (-0.4,0) node[anchor=center] {{\small $v_1$}};
\node [b] (b2) at (1,2) {};
\draw (1,2.6) node[anchor=center] {{\small $v_2$}};
\node [b] (b3) at (1,-2) {};
\draw (1,-2.6) node[anchor=center] {{\small $v_3$}};
\node [b] (b4) at (2,0) {};
\draw (2,0.6) node[anchor=center] {{\small $v_4$}};
\draw (b1) edge (b2);
\draw (b1) edge (b3);
\draw (-1,0) circle (1);
\end{tikzpicture}
\caption{A graph $\G$} \label{fig:G-undir}
\end{minipage}
\hspace{0.03\linewidth}
\begin{minipage}[t]{0.3\linewidth} 
\centering 
\begin{tikzpicture}[scale=0.5, >=stealth']
\tikzstyle{w}=[circle, draw, minimum size=4, inner sep=1]
\tikzstyle{b}=[circle, draw, fill, minimum size=4, inner sep=1]
\node [b] (b1) at (0,0) {};
\draw (-0.4,0) node[anchor=center] {{\small $v_1$}};
\node [b] (b2) at (1,2) {};
\draw (1,2.6) node[anchor=center] {{\small $v_2$}};
\node [b] (b3) at (1,-2) {};
\draw (1,-2.6) node[anchor=center] {{\small $v_3$}};
\node [b] (b4) at (2,0) {};
\draw (2,0.6) node[anchor=center] {{\small $v_4$}};
\draw [->] (b1) edge (b2);
\draw [->] (b1) edge (b3);
\draw (-1,0) circle (1);
\end{tikzpicture}
\caption{A directed graph $\G'$} \label{fig:G-dir}
\end{minipage}
\hspace{0.03\linewidth}
\begin{minipage}[t]{0.3\linewidth} 
\centering 
\begin{tikzpicture}[scale=0.5, >=stealth']
\tikzstyle{w}=[circle, draw, minimum size=4, inner sep=1]
\tikzstyle{b}=[circle, draw, fill, minimum size=4, inner sep=1]
\node [b] (b1) at (0,0) {};
\draw (-0.4,0) node[anchor=center] {{\small $1$}};
\node [b] (b2) at (1,2) {};
\draw (1,2.6) node[anchor=center] {{\small $2$}};
\node [b] (b3) at (1,-2) {};
\draw (1,-2.6) node[anchor=center] {{\small $3$}};
\node [b] (b4) at (2,0) {};
\draw (2,0.6) node[anchor=center] {{\small $4$}};
\draw (b1) edge (b2);
\draw (b1) edge (b3);
\draw (-1,0) circle (1);
\end{tikzpicture}
\caption{A labeled graph $\G''$} \label{fig:G-labeled}
\end{minipage}
\end{figure}
For example, the graph $\G$ on figure \ref{fig:G-undir} has 
$$
V(\G) = \{v_1, v_2, v_3, v_4\} \qquad \textrm{and} \qquad
E(\G) = \{\{v_1, v_1\}, \{v_2, v_1\}, \{v_1, v_3\}\}\,.
$$
For the directed graph $\G'$ on figure \ref{fig:G-dir} we have 
$$
V(\G') = \{v_1, v_2, v_3, v_4\} \qquad \textrm{and} \qquad
E(\G') = \{(v_1, v_1), (v_1, v_2), (v_1, v_3) \}\,.
$$
Finally, figure \ref{fig:G-labeled} gives us an example of 
a labeled graph.

A valency of a vertex $v$ in a (directed) graph $\G$ is the
total number of its appearances in the pairs $E(\G)$. For example, 
vertex $v_1$ in the graph on figure \ref{fig:G-dir} has valency $4$.

\section{Trees} 
\label{sec:trees}
A connected graph without cycles is called a tree. 
In this paper all trees are {\it planted}, i.e. each tree has a marked 
vertex (called the {\it root}) and this marked vertex has valency $1$\,. 
(In particular, each tree has at least one edge.) The edge adjacent 
to the root vertex is called the {\it root edge}.
Non-root vertices of valency $1$ are 
called {\it leaves}.  A vertex is called {\it internal} if it is 
neither a root nor a leaf. We always orient trees in the 
direction towards the root. Thus every internal vertex 
has at least $1$ incoming edge and exactly $1$ outgoing edge. 
An edge adjacent to a leaf is called {\it external}.   
We allow a degenerate tree, that is a tree with exactly two vertices
(the root vertex and a leaf) connected by a single edge. 
A tree $\bt$ is called {\it planar} if, for every internal vertex $v$ of $\bt$, the set 
of edges terminating at $v$ carries a total order.

Let us recall that for every  planar tree $\bt$ the set $V(\bt)$ of 
all its vertices is equipped with a natural total order.
To define this total order on $V(\bt)$  
we introduce the function 
\begin{equation}
\label{cN}
\cN:   V(\bt) \to V(\bt)\,.
\end{equation}
To a non-root vertex $v$ the function 
$\cN$ assigns the next vertex along the (unique) path connecting $v$ to the 
root vertex. Furthermore $\cN$ sends the root vertex to the root vertex. 

Let $v_1, v_2$ be two distinct vertices of $\bt$\,. 
If $v_1$ lies on the path which connects $v_2$ to the root 
vertex then we declare that 
$$
v_1 < v_2\,.
$$
Similarly, if $v_2$ lies on the path which connects $v_1$ to the root 
vertex then we declare that 
$$
v_2 < v_1\,.
$$
If neither of the above options realize then there exist 
numbers $k_1$ and $k_2$ such that
\begin{equation}
\label{same-vertex}
\cN^{k_1}(v_1) = \cN^{k_2}(v_2)
\end{equation}
but 
$$
\cN^{k_1-1}(v_1) \neq \cN^{k_2-1}(v_2)\,. 
$$
Since the tree $\bt$ is planar the set of $\cN^{-1} (\cN^{k_1}(v_1))$
is equipped with a total order. Furthermore, since both 
vertices $\cN^{k_1-1}(v_1) $ and  $\cN^{k_2-1}(v_2)$ belong to 
the set  $\cN^{-1}(\cN^{k_1}(v_1))$, we may compare them with 
respect to this order. 

We declare that, if  $\cN^{k_1-1}(v_1) < \cN^{k_2-1}(v_2)$,  then
$$
v_1 <  v_2\,.
$$  
Otherwise we set $v_2 < v_1$\,. 

It is not hard to see that the resulting relation $<$ on
$V(\bt)$ is indeed a total order. 

The total order on $V(\bt)$ can be defined graphically.
Indeed, draw a planar tree $\bt$ on the plane.
Then choose a small tubular neighborhood of $\bt$ on 
the plane and walk along its boundary 
starting from a vicinity of the root vertex in the
clockwise direction. On our way, we will
meet each vertex of $\bt$ at least once. 
So we declare that $v_1 < v_2$ if the first occurrence of 
$v_1$ precedes the first occurrence of $v_2$. 

For example, consider the planar tree depicted on 
figure \ref{fig:walking}. Following the path drawn around this 
tree we get   
$$
r < v_1 < v_2 < v_3 < v_4 < v_5 < v_6\,.
$$

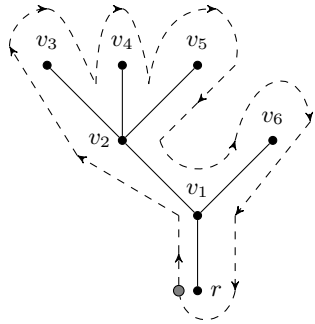
\begin{figure}[htp]
\centering 
\begin{tikzpicture}[scale=0.5, >=stealth']
\tikzstyle{w}=[circle, draw, minimum size=3, inner sep=1]
\tikzstyle{b}=[circle, draw, fill, minimum size=3, inner sep=1]
\tikzstyle{g}=[circle, draw, fill=gray, minimum size=4, inner sep=1]
\node[b] (r) at (0, 0) {};
\draw (0.5,0) node[anchor=center] {{\small $r$}};
\node[b] (v1) at (0, 2) {};
\draw (0,2.67) node[anchor=center] {{\small $v_1$}};
\node[b] (v2) at (-2, 4) {};
\draw (-2.6,4) node[anchor=center] {{\small $v_2$}};

\node[b] (v3) at (-4, 6) {};
\draw (-4,6.6) node[anchor=center] {{\small $v_3$}};
\node[b] (v4) at (-2, 6) {};
\draw (-2,6.6) node[anchor=center] {{\small $v_4$}};
\node[b] (v5) at (0, 6) {};
\draw (0,6.6) node[anchor=center] {{\small $v_5$}};
\node[b] (v6) at (2, 4) {};
\draw (2,4.6) node[anchor=center] {{\small $v_6$}};

\node[g] (beg) at (-0.5, 0) {};

\draw (r) edge (v1);
\draw (v1) edge (v2);
\draw (v2) edge (v3);
\draw (v2) edge (v4);
\draw (v2) edge (v5);
\draw (v1) edge (v6);
\draw [dashed, ->] (beg) -- (-0.5,1);
\draw [dashed] (-0.5,1) -- (-0.5,2);
\draw [dashed, ->] (-0.5,2) -- (-3.2,3.6);
\draw [dashed, ->] (-3.2,3.6) --  (-5,6.5);

\draw [dashed, ->] (-5,6.5)  .. controls (-5,7) and (-4.5,7.5) .. (-4,7.5);

\draw [dashed] (-4, 7.5) ..controls (-3.5,7.5) and (-3,6) .. (-2.7,5.5); 

\draw [dashed, ->]  (-2.7,5.5)   ..controls (-2.7,6) and (-2.6,7.5) .. (-2,7.5); 
\draw [dashed]  (-2,7.5)  ..controls (-1.7,7.5) and (-1.3 , 6) .. (-1.3, 5.5); 
\draw [dashed, ->]  (-1.3, 5.5)   ..controls (-1,7.5) and (0, 7.5) .. (0.5, 7.5); 

\draw [dashed] (0.5, 7.5) ..controls (0.8,7.3) and (1.2, 6.8) .. (1, 6); 
\draw [dashed, ->] (1,6) --  (0,5);
\draw [dashed] (0,5) --  (-1,4);
\draw [dashed, ->] (-1, 4) ..controls (-0.7,3) and (0.7, 3) .. (1, 4);
\draw [dashed, ->] (1, 4) ..controls (1.5,6) and (2.8, 6) .. (3, 4.5); 
\draw [dashed, ->] (3,4.5) --  (1,2);
\draw [dashed, ->] (1,2) --  (1,0);
\draw [dashed] (1,0)  ..controls (0.8,-1) and (-0.3, -1) .. (beg);
\end{tikzpicture}
\caption{We start walking around the planar tree 
from the gray circle} \label{fig:walking}
\end{figure}

Keeping this order in mind, we can say 
things like ``the first vertex'', ``the second  
vertex'', and ``the $i$-th vertex'' of a planar tree $\bt$\,.
In fact, the first vertex of a tree is always its root vertex. 

We have an obvious bijection between the set of 
edges $E(\bt)$ of a tree $\bt$ and the subset of vertices: 
\begin{equation}
\label{no-root}
V(\bt) \setminus \{\textrm{root vertex}\}\,.
\end{equation}
This bijection assigns to a vertex $v$ in (\ref{no-root}) its 
outgoing edge. 

Thus the canonical total order on the set (\ref{no-root}) gives 
us a natural total order on the set of edges $E(\bt)$\,. 

For our purposes we also extend the total orders 
on the sets $V(\bt) \setminus \{{\rm root~vertex}\}$ 
and $E(\bt)$ to the disjoint union
\begin{equation}
\label{vertices-edges}
V(\bt) \setminus \{{\rm root~vertex}\}  \sqcup  E(\bt)
\end{equation}
by declaring that a vertex is bigger than its outgoing edge. 
For example, the root edge is the minimal element 
in the set (\ref{vertices-edges}).

\subsection{Groupoid of labeled planar trees}
Let $n$ be a non-negative integer.
An $n$-labeled planar tree $\bt$ is a planar tree equipped with 
an injective map 
\begin{equation}
\label{labeling}
\ml : \{1,2, \dots, n\} \to L(\bt)
\end{equation}
from the set $\{1,2, \dots, n\}$ to the set $L(\bt)$ of leaves of $\bt$\,.
Although the set $L(\bt)$ has a natural total order we do not require 
that the map \eqref{labeling} is monotonous. 

The set $L(\bt)$ of leaves of an $n$-labeled planar tree $\bt$
splits into the disjoint union of the image $\ml(\{1,2, \dots, n\})$
and its complement. We call leaves in the image of $\ml$
{\it labeled}.

A vertex $x$ of an $n$-labeled planar tree $\bt$ is called 
{\it nodal} if it is neither a root vertex, nor a labeled leaf. 
We denote by $V_{\nod}(\bt)$ the set of all nodal vertices of 
$\bt$. Keeping in mind the canonical total order on 
the set of all vertices of $\bt$ we can say things like
``the first nodal vertex'', ``the second nodal vertex'', and
``the $i$-th nodal vertex''. 

\begin{example}
\label{ex:labeled}
An example of a $4$-labeled planar tree is depicted on 
figure \ref{fig:labeled}.
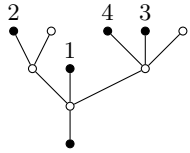
\begin{figure}[htp] 
\centering 
\begin{tikzpicture}[scale=0.5]
\tikzstyle{w}=[circle, draw, minimum size=3, inner sep=1]
\tikzstyle{b}=[circle, draw, fill, minimum size=3, inner sep=1]
\node[b] (r) at (0, 0) {};
\node[w] (v1) at (0, 1) {};
\node[w] (v2) at (-1, 2) {};
\node[b] (v3) at (0, 2) {};
\draw (0,2.5) node[anchor=center] {{\small $1$}};
\node[w] (v4) at (2, 2) {};
\node[b] (v5) at (-1.5, 3) {};
\draw (-1.5,3.5) node[anchor=center] {{\small $2$}};
\node[w] (v6) at (-0.5, 3) {};
\node[b] (v7) at (1, 3) {};
\draw (1,3.5) node[anchor=center] {{\small $4$}};
\node[b] (v8) at (2, 3) {};
\draw (2,3.5) node[anchor=center] {{\small $3$}};
\node[w] (v9) at (3, 3) {};
\draw (r) edge (v1);
\draw (v1) edge (v2);
\draw (v1) edge (v3);
\draw (v1) edge (v4);
\draw (v2) edge (v5);
\draw (v2) edge (v6);
\draw (v4) edge (v7);
\draw (v4) edge (v8);
\draw (v4) edge (v9);
\end{tikzpicture}
\caption{A $4$-labeled planar tree} \label{fig:labeled}
\end{figure}
On figures we use small white circles for 
nodal vertices and small black circles for labeled 
leaves and the root vertex.  

\end{example}

For our purposes we need to upgrade the set of $n$-labeled 
planar trees to 
the groupoid  $\Tree(n)$\,. 
Objects of $\Tree(n)$ are $n$-labeled planar trees and 
morphisms are \und{non-planar} isomorphisms of the corresponding 
(non-planar) trees compatible with labeling. 
The groupoid   $\Tree(n)$ is equipped with an obvious left action 
of the symmetric group $S_n$\,. 

As far as we know the groupoid  $\Tree(n)$ was originally 
introduced by E. Getzler and M. Kapranov in \cite{GetKap}. 
However, here we do not exactly follow the notation 
from  \cite{GetKap}.

The notation $\Tree_2(n)$ is reserved 
for the full sub-category
of $\Tree(n)$ whose objects are $n$-labeled planar 
trees with exactly $2$ nodal vertices. It is not hard to see that 
every object in  $\Tree_2(n)$ has at most $n+1$ leaves. 
Due to Exercise \ref{exer:shuffles}, isomorphism classes of $\Tree_2(n)$ are 
in bijection with the union 
\begin{equation}
\label{shuffles}
\bigsqcup_{0 \le p \le n} \Sh_{p, n-p}\,.
\end{equation}

\begin{exer}
\label{exer:shuffles}
Let us assign to a shuffle $\tau\in \Sh_{p, n-p}$
the $n$-labeled planar tree depicted on figure \ref{fig:shuffle}. 
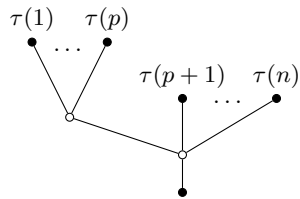
\begin{figure}[htp]
\centering
\begin{tikzpicture}[scale=0.5]
\tikzstyle{w}=[circle, draw, minimum size=3, inner sep=1]
\tikzstyle{b}=[circle, draw, fill, minimum size=3, inner sep=1]
\node[b] (l1) at (-1, 4) {};
\draw (-1,4.6) node[anchor=center] {{\small $\tau(1)$}};
\draw (0,3.8) node[anchor=center] {{\small $\dots$}};
\node[b] (lp) at (1, 4) {};
\draw (1,4.6) node[anchor=center] {{\small $\tau(p)$}};
\node[w] (vv) at (0, 2) {};
\node[b] (lp1) at (3, 2.5) {};
\draw (3,3.1) node[anchor=center] {{\small $\tau(p+1)$}};
\draw (4.25,2.4) node[anchor=center] {{\small $\dots$}};
\node[b] (ln) at (5.5, 2.5) {};
\draw (5.5,3.1) node[anchor=center] {{\small $\tau(n)$}};
\node[w] (v) at (3, 1) {};
\node[b] (r) at (3, 0) {};
\draw (vv) edge (l1);
\draw (vv) edge (lp);
\draw (v) edge (vv);
\draw (v) edge (lp1);
\draw (v) edge (ln);
\draw (r) edge (v);
\end{tikzpicture}
\caption{\label{fig:shuffle} Here $\tau$ is a $(p, n-p)$-shuffle}
\end{figure} 
Prove that this assignment gives us a bijection between the 
set \eqref{shuffles} and the set of isomorphism 
classes in $\Tree_2(n)$\,.
\end{exer}

\begin{rem}
\label{rem:degen-Tree2}
The groupoid $\Tree_2(0)$ has exactly one 
object (see figure \ref{fig:ww}) and hence exactly one isomorphism class.
The groupoid $\Tree_2(1)$ has three objects and two isomorphisms classes.  
Representatives of isomorphism classes in  $\Tree_2(1)$ are 
depicted on figures \ref{fig:ww1}  and \ref{fig:w1w}. 
\begin{figure}[htp]  
\begin{minipage}[t]{0.4\linewidth} 
\centering 
\begin{tikzpicture}[scale=0.5]
\tikzstyle{w}=[circle, draw, minimum size=3, inner sep=1]
\tikzstyle{b}=[circle, draw, fill, minimum size=3, inner sep=1]
\node[b] (r) at (0, 0) {};
\node[w] (v1) at (0, 1) {};
\node[w] (v2) at (0, 2) {};
\draw (r) edge (v1);
\draw (v1) edge (v2);
\end{tikzpicture}
\caption{The unique object in $\Tree_2(0)$} \label{fig:ww}
\end{minipage}
\begin{minipage}[t]{0.28\linewidth} 
\centering 
\begin{tikzpicture}[scale=0.5]
\tikzstyle{w}=[circle, draw, minimum size=3, inner sep=1]
\tikzstyle{b}=[circle, draw, fill, minimum size=3, inner sep=1]
\node[b] (r) at (0, 0) {};
\node[w] (v1) at (0, 1) {};
\node[w] (v2) at (-1, 1.5) {};
\node[b] (l1) at (-1, 2.5) {};
\draw (-1,3) node[anchor=center] {{\small $1$}};
\draw (r) edge (v1);
\draw (v1) edge (v2);
\draw (v2) edge (l1);
\end{tikzpicture}
\caption{A tree in $\Tree_2(1)$} \label{fig:ww1}
\end{minipage}
\begin{minipage}[t]{0.28\linewidth} 
\centering 
\begin{tikzpicture}[scale=0.5]
\tikzstyle{w}=[circle, draw, minimum size=3, inner sep=1]
\tikzstyle{b}=[circle, draw, fill, minimum size=3, inner sep=1]
\node[b] (r) at (0, 0) {};
\node[w] (v1) at (0, 1) {};
\node[w] (v2) at (-0.5, 1.8) {};
\node[b] (l1) at (0.5, 1.8) {};
\draw (0.5,2.3) node[anchor=center] {{\small $1$}};
\draw (r) edge (v1);
\draw (v1) edge (v2);
\draw (v1) edge (l1);
\end{tikzpicture}
\caption{A tree in $\Tree_2(1)$} \label{fig:w1w}
\end{minipage}
\end{figure}  
\end{rem}

\subsection{Insertions of trees}
Let $\wt{\bt}$ be an $n$-labeled planar tree with a non-empty 
set of nodal vertices. 
If the $i$-th nodal vertex of $\wt{\bt}$
has $m_i$ incoming edges
then for every $m_i$-labeled planar tree 
$\bt$ we can define the insertion $\bul_i$ of the tree $\bt$ into 
the $i$-th nodal vertex of $\wt{\bt}$\,. The resulting planar 
tree $\wt{\bt} \bul_i \bt$ is also $n$-labeled. 

If $m_i=0$ then  $\wt{\bt} \bul_i \bt$ is obtained via identifying 
the root edge of $\bt$ with edge originating at the 
$i$-th nodal vertex.  

If $m_i > 0$ then the tree $\wt{\bt} \bul_i \bt$ is 
built following these steps: 
\begin{itemize}

\item First, we denote by $E_i(\wt{\bt})$ the set of edges 
terminating at the $i$-th nodal vertex of $\wt{\bt}$\,. 
Since $\wt{\bt}$ is planar, the set $E_i(\wt{\bt})$ comes 
with a total order;

\item second, we erase the $i$-th nodal vertex of $\wt{\bt}$;

\item third, we identify the root edge of $\bt$ with the edge
of $\wt{\bt}$ which originates at the $i$-th nodal vertex;

\item finally, we identify the edges of $\bt$ adjacent to labeled 
leaves with the edges in the set $E_i(\wt{\bt})$ following 
this rule: the external edge with label $j$ gets identified 
with the $j$-th edge in the set  $E_i(\wt{\bt})$\,. In doing this, 
we keep the same planar structure on $\bt$, so, in general,  
branches of $\wt{\bt}$ move around.

\end{itemize}

\begin{example}
\label{ex:insertion}
Let $\wt{\bt}$ be the $4$-labeled planar tree depicted on 
figure \ref{fig:labeled} and $\bt$ be the $3$-labeled planar 
tree depicted on figure \ref{fig:bt}. Then the insertion $\wt{\bt}\, \bul_1\, \bt$
of $\bt$ into the first nodal vertex of $\wt{\bt}$ is shown on figure 
\ref{fig:ins}. Figure \ref{fig:algo} illustrates the construction algorithm of $\wt{\bt} \,\bul_1\, \bt$
step by step.

\begin{figure}[htp]
\begin{minipage}[t]{0.4\linewidth} 
\centering 
\begin{tikzpicture}[scale=0.5]
\tikzstyle{w}=[circle, draw, minimum size=3, inner sep=1]
\tikzstyle{b}=[circle, draw, fill, minimum size=3, inner sep=1]
\node[b] (r) at (0, 0) {};
\node[w] (v1) at (0, 1) {};
\node[w] (v2) at (-1, 2) {};
\node[b] (v3) at (1, 2) {};
\draw (1,2.5) node[anchor=center] {{\small $1$}};
\node[b] (v4) at (-2, 3) {};
\draw (-2,3.5) node[anchor=center] {{\small $2$}};
\node[w] (v5) at (-1, 3) {};
\node[b] (v6) at (0, 3) {};
\draw (0,3.5) node[anchor=center] {{\small $3$}};
\draw (r) edge (v1);
\draw (v1) edge (v2);
\draw (v1) edge (v3);
\draw (v2) edge (v4);
\draw (v2) edge (v5);
\draw (v2) edge (v6);
\end{tikzpicture}
\caption{A $3$-labeled planar tree $\bt$} \label{fig:bt}
\end{minipage}
~
\begin{minipage}[t]{0.45\linewidth} 
\centering 
\begin{tikzpicture}[scale=0.5]
\tikzstyle{w}=[circle, draw, minimum size=3, inner sep=1]
\tikzstyle{b}=[circle, draw, fill, minimum size=3, inner sep=1]
\node[b] (r) at (0, 0) {};
\node[w] (v1) at (0, 1) {};
\node[w] (v2) at (-1, 2) {};
\node[w] (v3) at (2, 2) {};
\node[b] (vv1) at (2, 3) {};
\draw (2,3.5) node[anchor=center] {{\small $2$}};
\node[w] (vv2) at (3, 3) {};
\node[b] (v4) at (-2, 3) {};
\draw (-2,3.5) node[anchor=center] {{\small $1$}};
\node[w] (v5) at (-1, 3) {};
\node[w] (v6) at (0, 3) {};
\node[b] (v61) at (-1, 4) {};
\draw (-1,4.5) node[anchor=center] {{\small $4$}};
\node[b] (v62) at (0, 4) {};
\draw (0,4.5) node[anchor=center] {{\small $3$}};
\node[w] (v63) at (1, 4) {};
\draw (r) edge (v1);
\draw (v1) edge (v2);
\draw (v1) edge (v3);
\draw (v2) edge (v4);
\draw (v2) edge (v5);
\draw (v2) edge (v6);
\draw (v3) edge (vv1);
\draw (v3) edge (vv2);
\draw (v6) edge (v61);
\draw (v6) edge (v62);
\draw (v6) edge (v63);
\end{tikzpicture}
\caption{The $4$-labeled planar tree $\wt{\bt} \,\bul_1\, \bt$} \label{fig:ins}
\end{minipage}
\end{figure}

\begin{figure}[htp]
\centering
\begin{minipage}[t]{0.3\linewidth}
\centering 
\begin{tikzpicture}[scale=0.5]
\tikzstyle{w}=[circle, draw, minimum size=3, inner sep=1]
\tikzstyle{b}=[circle, draw, fill, minimum size=3, inner sep=1]
\node[b] (r) at (0, 0) {};
\node[w] (v1) at (0, 1) {};
\node[w] (v2) at (-1, 2) {};
\node[b] (v3) at (0, 2) {};
\draw (0,2.5) node[anchor=center] {{\small $1$}};
\node[w] (v4) at (2, 2) {};
\node[b] (v5) at (-1.5, 3) {};
\draw (-1.5,3.5) node[anchor=center] {{\small $2$}};
\node[w] (v6) at (-0.5, 3) {};
\node[b] (v7) at (1, 3) {};
\draw (1,3.5) node[anchor=center] {{\small $4$}};
\node[b] (v8) at (2, 3) {};
\draw (2,3.5) node[anchor=center] {{\small $3$}};
\node[w] (v9) at (3, 3) {};

\draw (4,2) node[anchor=center] {{$\longrightarrow$}};
\draw (r) edge (v1);
\draw (v1) edge (v2);
\draw (v1) edge (v3);
\draw (v1) edge (v4);
\draw (v2) edge (v5);
\draw (v2) edge (v6);
\draw (v4) edge (v7);
\draw (v4) edge (v8);
\draw (v4) edge (v9);
\end{tikzpicture}
\end{minipage}
\begin{minipage}[t]{0.3\linewidth}
\centering 
\begin{tikzpicture}[scale=0.5]
\tikzstyle{w}=[circle, draw, minimum size=3, inner sep=1]
\tikzstyle{b}=[circle, draw, fill, minimum size=3, inner sep=1]
\tikzstyle{m}=[circle, minimum size=3, inner sep=1]

\node[b] (r) at (0, 0) {};
\node[w] (v1) at (0, 1) {};
\node[m] (m1) at (-1, 2) {{\small $1$}};
\node[w] (v2) at (-2, 3) {};
\node[b] (v3) at (-3, 4) {};
\draw (-3,4.5) node[anchor=center] {{\small $2$}};
\node[w] (v4) at (-1, 4) {};

\node[m] (m2) at (0, 2) {{\small $2$}};

\node[b] (v5) at (0, 3) {};
\draw (0,3.5) node[anchor=center] {{\small $1$}};

\node[m] (m3) at (1, 2) {{\small $3$}};
\node[w] (v6) at (2, 3) {};

\node[b] (v7) at (1, 4) {};
\draw (1,4.5) node[anchor=center] {{\small $4$}};
\node[b] (v8) at (2, 4) {};
\draw (2,4.5) node[anchor=center] {{\small $3$}};
\node[w] (v9) at (3, 4) {};

\draw (4,2) node[anchor=center] {{$\longrightarrow$}};
\draw (r) edge (v1);
\draw (v1) edge (m1); \draw (m1) edge (v2);
\draw (v1) edge (m2); \draw (m2) edge (v5);
\draw (v1) edge (m3); \draw (m3) edge (v6);
\draw (v2) edge (v3);
\draw (v2) edge (v4);
\draw (v6) edge (v7);
\draw (v6) edge (v8);
\draw (v6) edge (v9);
\end{tikzpicture}
\end{minipage}
\begin{minipage}[t]{0.3\linewidth}
\centering 
\begin{tikzpicture}[scale=0.5]
\tikzstyle{w}=[circle, draw, minimum size=3, inner sep=1]
\tikzstyle{b}=[circle, draw, fill, minimum size=3, inner sep=1]
\tikzstyle{m}=[circle, minimum size=3, inner sep=1]

\node[b] (r) at (0, 0) {};

\node[m] (m1) at (-1, 2) {{\small $1$}};
\node[w] (v2) at (-2, 3) {};
\node[b] (v3) at (-3, 4) {};
\draw (-3,4.5) node[anchor=center] {{\small $2$}};
\node[w] (v4) at (-1, 4) {};

\node[m] (m2) at (0, 2) {{\small $2$}};

\node[b] (v5) at (0, 3) {};
\draw (0,3.5) node[anchor=center] {{\small $1$}};

\node[m] (m3) at (1, 2) {{\small $3$}};
\node[w] (v6) at (2, 3) {};

\node[b] (v7) at (1, 4) {};
\draw (1,4.5) node[anchor=center] {{\small $4$}};
\node[b] (v8) at (2, 4) {};
\draw (2,4.5) node[anchor=center] {{\small $3$}};
\node[w] (v9) at (3, 4) {};

\draw (4,2) node[anchor=center] {{$\longrightarrow$}};
\draw (r) edge (0, 0.5);
\draw (m1) edge (v2);
\draw (m2) edge (v5);
\draw (m3) edge (v6);
\draw (v2) edge (v3);
\draw (v2) edge (v4);
\draw (v6) edge (v7);
\draw (v6) edge (v8);
\draw (v6) edge (v9);
\end{tikzpicture}
\end{minipage}
\vspace{0.26cm}
\begin{minipage}[t]{0.45\linewidth}
\centering 
\begin{tikzpicture}[scale=0.5]
\tikzstyle{w}=[circle, draw, minimum size=3, inner sep=1]
\tikzstyle{b}=[circle, draw, fill, minimum size=3, inner sep=1]
\tikzstyle{m}=[circle, minimum size=3, inner sep=1]
\node[b] (r) at (0, 0) {};
\node[w] (vv1) at (0, 1) {};
\node[w] (vv2) at (-1, 2) {};
\node[b] (vv3) at (1, 2) {};
\draw (1,2.5) node[anchor=center] {{\small $1$}};
\node[b] (vv4) at (-2, 3) {};
\draw (-2,3.5) node[anchor=center] {{\small $2$}};
\node[w] (vv5) at (-1, 3) {};
\node[b] (vv6) at (0, 3) {};
\draw (0,3.5) node[anchor=center] {{\small $3$}};
\draw (r) edge (vv1);
\draw (vv1) edge (vv2);
\draw (vv1) edge (vv3);
\draw (vv2) edge (vv4);
\draw (vv2) edge (vv5);
\draw (vv2) edge (vv6);


\node[m] (m1) at (-2, 5) {{\small $1$}};
\node[w] (v2) at (-2, 6) {};
\node[b] (v3) at (-3, 7) {};
\draw (-3,7.5) node[anchor=center] {{\small $2$}};
\node[w] (v4) at (-1, 7) {};

\node[m] (m2) at (0, 5) {{\small $2$}};

\node[b] (v5) at (0, 6) {};
\draw (0,6.5) node[anchor=center] {{\small $1$}};

\node[m] (m3) at (2, 5) {{\small $3$}};
\node[w] (v6) at (2, 6) {};

\node[b] (v7) at (1, 7) {};
\draw (1,7.5) node[anchor=center] {{\small $4$}};
\node[b] (v8) at (2, 7) {};
\draw (2,7.5) node[anchor=center] {{\small $3$}};
\node[w] (v9) at (3, 7) {};

\draw (6,2) node[anchor=center] {{$\longrightarrow$}};
\draw (r) edge (0, 0.5);
\draw (m1) edge (v2);
\draw (m2) edge (v5);
\draw (m3) edge (v6);
\draw (v2) edge (v3);
\draw (v2) edge (v4);
\draw (v6) edge (v7);
\draw (v6) edge (v8);
\draw (v6) edge (v9);
\end{tikzpicture}
\end{minipage}
\begin{minipage}[t]{0.3\linewidth} 
\centering 
\begin{tikzpicture}[scale=0.5]
\tikzstyle{w}=[circle, draw, minimum size=3, inner sep=1]
\tikzstyle{b}=[circle, draw, fill, minimum size=3, inner sep=1]
\node[b] (r) at (0, 0) {};
\node[w] (v1) at (0, 1) {};
\node[w] (v2) at (-1, 2) {};
\node[w] (v3) at (2, 2) {};
\node[b] (vv1) at (2, 3) {};
\draw (2,3.5) node[anchor=center] {{\small $2$}};
\node[w] (vv2) at (3, 3) {};
\node[b] (v4) at (-2, 3) {};
\draw (-2,3.5) node[anchor=center] {{\small $1$}};
\node[w] (v5) at (-1, 3) {};
\node[w] (v6) at (0, 3) {};
\node[b] (v61) at (-1, 4) {};
\draw (-1,4.5) node[anchor=center] {{\small $4$}};
\node[b] (v62) at (0, 4) {};
\draw (0,4.5) node[anchor=center] {{\small $3$}};
\node[w] (v63) at (1, 4) {};
\draw (r) edge (v1);
\draw (v1) edge (v2);
\draw (v1) edge (v3);
\draw (v2) edge (v4);
\draw (v2) edge (v5);
\draw (v2) edge (v6);
\draw (v3) edge (vv1);
\draw (v3) edge (vv2);
\draw (v6) edge (v61);
\draw (v6) edge (v62);
\draw (v6) edge (v63);
\end{tikzpicture}
\end{minipage}
\caption{Algorithm for constructing $\wt{\bt} \,\bul_1\, \bt$} \label{fig:algo}
\end{figure}
\end{example}


\section{Operads, pseudo-operads, and their dual versions}
\label{sec:oper-coper}

\subsection{Collections} 
\label{sec:coll}

By a {\it collection} we mean the sequence $\{ P(n) \}_{n \ge 0}$ 
of objects of the underlying symmetric monoidal category $\mC$
such that for each $n$, the object $P(n)$ is equipped
with a left action of the symmetric group $S_n$\,.

Given a collection $P$ we form covariant functors for 
$n \ge 0$
$$
\und{P}_n :  \Tree(n) \to \mC\,.
$$
To an $n$-labelled planar tree $\bt$ the functor $\und{P}_n$ assigns 
the object
\begin{equation}
\label{und-P-n-def}
\und{P}_n  (\bt) = \bigotimes_{x\in V_{\nod}(\bt)} P(m(x))\,, 
\end{equation}
where $V_{\nod}(\bt)$ is the set of all nodal vertices
of $\bt$, the notation $m(x)$ is reserved for the number 
of edges terminating at the vertex $x$, and
the order of the factors in the right hand side 
of the equation agrees with the natural order on 
the set  $V_{\nod}(\bt)$.  

To define the functor $\und{P}_n$ on the level of morphisms 
we use the actions of the symmetric 
groups and the structure of the 
symmetric monoidal category $\mC$ in the obvious way. 
\begin{example}
\label{exam:und-P}
Let $\bt_1$ (resp. $\bt_2$) be a $2$-labeled planar tree depicted 
on figure \ref{fig:bt1} (resp. figure \ref{fig:bt2}). There is a unique 
morphism $\la$ from $\bt_1$ to $\bt_2$ in $\Tree(2)$\,.
\begin{figure}[htp]
\begin{minipage}[t]{0.45\linewidth} 
\centering 
\begin{tikzpicture}[scale=0.5]
\tikzstyle{w}=[circle, draw, minimum size=3, inner sep=1]
\tikzstyle{b}=[circle, draw, fill, minimum size=3, inner sep=1]
\node[b] (r) at (0, 0) {};
\node[w] (v1) at (0, 1) {};
\node[w] (v2) at (-1, 2) {};
\node[w] (v3) at (1, 2) {};
\node[b] (v4) at (-2, 3) {};
\draw (-2,3.5) node[anchor=center] {{\small $1$}};
\node[w] (v5) at (-1, 3) {};
\node[b] (v6) at (0, 3) {};
\draw (0,3.5) node[anchor=center] {{\small $2$}};
\draw (r) edge (v1);
\draw (v1) edge (v2);
\draw (v1) edge (v3);
\draw (v2) edge (v4);
\draw (v2) edge (v5);
\draw (v2) edge (v6);
\end{tikzpicture}
\caption{A $2$-labeled planar tree $\bt_1$} \label{fig:bt1}
\end{minipage}
~
\begin{minipage}[t]{0.45\linewidth} 
\centering 
\begin{tikzpicture}[scale=0.5]
\tikzstyle{w}=[circle, draw, minimum size=3, inner sep=1]
\tikzstyle{b}=[circle, draw, fill, minimum size=3, inner sep=1]
\node[b] (r) at (0, 0) {};
\node[w] (v1) at (0, 1) {};
\node[w] (v2) at (1, 2) {};
\node[w] (v3) at (-1, 2) {};
\node[b] (v4) at (0, 3) {};
\draw (0,3.5) node[anchor=center] {{\small $2$}};
\node[w] (v5) at (1, 3) {};
\node[b] (v6) at (2, 3) {};
\draw (2,3.5) node[anchor=center] {{\small $1$}};
\draw (r) edge (v1);
\draw (v1) edge (v2);
\draw (v1) edge (v3);
\draw (v2) edge (v4);
\draw (v2) edge (v5);
\draw (v2) edge (v6);
\end{tikzpicture}
\caption{A $2$-labeled planar tree $\bt_2$} \label{fig:bt2}
\end{minipage}
\end{figure}
For these trees we have 
$$
\und{P}_2(\bt_1) = P(2) \otimes P(3) \otimes P(0) \otimes P(0)\,,
$$
$$
\und{P}_2(\bt_2) = P(2) \otimes P(0) \otimes P(3) \otimes P(0)\,,
$$
and the morphism 
$$
\und{P}_2(\la) :  P(2) \otimes P(3) \otimes P(0) \otimes P(0) 
\to  P(2)  \otimes P(0) \otimes P(3) \otimes P(0)
$$
is the composition 
$$
\und{P}_2 (\la) =   (1 \otimes \beta) \circ  (\si_{12} \otimes \si_{13} \otimes 1 \otimes 1),  
$$
where $\si_{12}$ (resp. $\si_{13}$) is the corresponding transposition in $S_2$
(resp. in $S_3$) and $\beta$ is the braiding 
$$
\beta:  \big(P(3) \otimes P(0) \big)\otimes P(0)  \to 
P(0) \otimes \big(P(3) \otimes P(0) \big)\,.
$$ 
\end{example}

\subsection{Pseudo-operads}

\label{sec:ps-op}

We now recall that a {\it pseudo-operad} is a collection $\{P(n)\}_{n\ge 0}$
equipped with multiplication maps 
\begin{equation}
\label{mu-bt}
\mu_{\bt} : \und{P}_n (\bt) \to P(n) 
\end{equation}
for all $n$-labeled trees $\bt$ and for all $n\ge 0$\,.
These multiplications should satisfy the axioms which 
we list below. 

First, we require that for the standard 
corolla $\bq_n$ (see figures \ref{fig:corolla0}, \ref{fig:corolla12n})
\begin{figure}[htp]  
\begin{minipage}[t]{0.45\linewidth} 
\centering 
\begin{tikzpicture}[scale=0.5]
\tikzstyle{w}=[circle, draw, minimum size=3, inner sep=1]
\tikzstyle{b}=[circle, draw, fill, minimum size=3, inner sep=1]
\node[b] (r) at (0, 0) {};
\node[w] (v) at (0, 1) {};
\draw (r) edge (v);
\end{tikzpicture}
\caption{The corolla $\bq_0$} \label{fig:corolla0}
\end{minipage}
\begin{minipage}[t]{0.45\linewidth} 
\centering 
\begin{tikzpicture}[scale=0.5]
\tikzstyle{w}=[circle, draw, minimum size=3, inner sep=1]
\tikzstyle{b}=[circle, draw, fill, minimum size=3, inner sep=1]
\node[b] (r) at (0, 0) {};
\node[w] (v) at (0, 1) {};
\node[b] (v1) at (-1.5, 2) {};
\draw (-1.5,2.5) node[anchor=center] {{\small $1$}};
\node[b] (v2) at (-0.5, 2) {};
\draw (-0.5,2.5) node[anchor=center] {{\small $2$}};
\draw (0.3,1.8) node[anchor=center] {{\small $\dots$}};
\node[b] (vn) at (1, 2) {};
\draw (1,2.5) node[anchor=center] {{\small $n$}};
\draw (r) edge (v);
\draw (v) edge (v1);
\draw (v) edge (v2);
\draw (v) edge (vn);
\end{tikzpicture}
\caption{The corolla $\bq_n$ for $n \ge 1$} \label{fig:corolla12n}
\end{minipage}
\end{figure}  
the multiplication map $\mu_{\bq_n}$ is the identity morphism 
on $P(n)$
\begin{equation}
\label{corolla-iden}
\mu_{\bq_n} = \id_{P(n)}\,.
\end{equation}

Second, we require that the operations \eqref{mu-bt}
are $S_n$-equivariant  
\begin{equation}
\label{mu-S-n-equiv}
 \mu_{\si(\bt)} = \si \circ \mu_{\bt}\,, \qquad \forall~~ \si \in S_n, ~~ \bt \in \Tree(n)\,.
\end{equation}

Third, for every morphism $\la: \bt \to \bt'$ in $\Tree(n)$ we have
\begin{equation}
\label{la-equiv}
 \mu_{\bt'} \circ   \und{P}_n(\la) =  \mu_{\bt}\,.
\end{equation}

Finally, we need to formulate the associativity axiom 
for multiplications \eqref{mu-bt}.
For this purpose we consider the following 
quadruple $(\wt{\bt}, i, m_i, \bt)$ where $\wt{\bt}$ is an 
$n$-labeled planar tree with $k$ nodal vertices, 
$1 \le i \le k$, $m_i$ is the number of edges terminating 
at the $i$-th nodal vertex of $\wt{\bt}$, and $\bt$ is 
an $m_i$-labeled planar tree. 

The associativity axioms states that for each such 
quadruple  $(\wt{\bt}, i, m_i, \bt)$ we have
\begin{equation}
\label{assoc}
\mu_{\wt{\bt}} \circ  (\id  \otimes \dots \otimes \id \otimes 
\underbrace{\mu_{\bt}}_{i\textrm{-th spot}}
  \otimes \id
\otimes \dots \otimes \id)  \circ \beta_{\,\wt{\bt}, i, m_i, \bt}  = \mu_{\,\wt{\bt} \bullet_i \bt} 
\end{equation}
where $\wt{\bt} \bullet_i \bt$ is the $n$-labeled planar tree obtained 
by inserting $\bt$ into the $i$-th nodal vertex of $\wt{\bt}$ and
$\beta_{\,\wt{\bt}, i, m_i, \bt}$ is the isomorphism in $\mC$ which 
is responsible for ``putting tensor factors in the correct order''. 

To define the isomorphism  $\beta_{\,\wt{\bt}, i, m_i, \bt}$ we 
observe that the source of the map $\mu_{\,\wt{\bt} \bullet_i \bt} $ is 
\begin{equation}
\label{source}
\bigotimes_{x \in V_{\nod} (\wt{\bt} \bullet_i \bt)} P(m(x)) 
\end{equation}
where $m(x)$ denotes the number of edges 
of $\wt{\bt} \bullet_i \bt$ terminating at the nodal vertex $x$
and the order of factors agrees with the total order on 
the set  $V_{\nod} (\wt{\bt} \bullet_i \bt)$\,. The source of 
the map 
\begin{equation}
\label{mu-circ-mu}
\mu_{\wt{\bt}} \circ  (\id  \otimes \dots \otimes \id \otimes 
\underbrace{\mu_{\bt}}_{i-th~spot}
  \otimes \id
\otimes \dots \otimes \id) 
\end{equation}
is also the product (\ref{source})
with a possibly different order of tensor factors. 
The map $\beta_{\,\wt{\bt}, i, m_i, \bt}$ in (\ref{assoc}) is the isomorphism 
in $\mC$ which connects the source of $\mu_{\,\wt{\bt} \bullet_i \bt}$ 
to the source of (\ref{mu-circ-mu}).

Given integers $n \ge 1$, $k \ge 0$, $1 \le i \le n$ and a permutation
$\si \in S_{n+k-1}$ we can form the $(n+k-1)$-labeled planar tree
$\bt^{n,k,i}_{\si}$ shown on figure  \ref{fig:bt-si-and}.
\begin{figure}[htp]
\centering 
\begin{tikzpicture}[scale=0.5]
\tikzstyle{w}=[circle, draw, minimum size=3, inner sep=1]
\tikzstyle{b}=[circle, draw, fill, minimum size=3, inner sep=1]
\node[b] (v1) at (0, 2) {};
\draw (0,2.6) node[anchor=center] {{\small $\si(1)$}};
\node[b] (v2) at (1.5, 2) {};
\draw (1.5,2.6) node[anchor=center] {{\small $\si(2)$}};
\draw (2.8,2) node[anchor=center] {{\small $\dots$}};
\node[b] (v1i) at (4, 2) {};
\draw (4,2.6) node[anchor=center] {{\small $\si(i-1)$}};
\node[w] (vv) at (6, 2) {};
\node[b] (vi) at (4.5, 4) {};
\draw (4.5,4.6) node[anchor=center] {{\small $\si(i)$}};
\draw (5.8,4) node[anchor=center] {{\small $\dots$}};
\node[b] (v1ik) at (7, 4) {};
\draw (7.2,4.6) node[anchor=center] {{\small $\si(i+k-1)$}};
\node[b] (vik) at (8, 2) {};
\draw (8,2.6) node[anchor=center] {{\small $\si(i+k)$}};
\draw (10,2) node[anchor=center] {{\small $\dots$}};
\node[b] (v1nk) at (12, 2) {};
\draw (12.2,2.6) node[anchor=center] {{\small $\si(n+k-1)$}};
\node[w] (v) at (6, 0.5) {};
\node[w] (r) at (6, -0.5) {};
\draw (vv) edge (vi);
\draw (vv) edge (v1ik);
\draw (v) edge (v1);
\draw (v) edge (v2);
\draw (v) edge (v1i);
\draw (v) edge (vv);
\draw (v) edge (vik);
\draw (v) edge (v1nk);
\draw (r) edge (v);
\end{tikzpicture}
\caption{The $(n+k-1)$-labeled planar tree $\bt^{n,k,i}_{\si}$} \label{fig:bt-si-and}
\end{figure}

In the case when $\mC = \Ch_{\bbK}$ or $\mC= \grVect_{\bbK}$, 
it is convenient to use a slightly different notation for values of
the multiplication map $\mu_{\bt^{n,k,i}_{\si}}$ corresponding to
the tree $\bt^{n,k,i}_{\si}$\,.
More precisely, for a vector $v \in P(n)$ and $w \in P(k)$ of 
a pseudo-operad $P$ we set  
\begin{equation}
\label{new-notation}
v\big(\si(1), \dots \si(i-1), w\big(\si(i), \dots, \si(i+k-1) \big), 
\si(i+k), \dots, \si(n+k-1)\big) : = \mu_{\bt^{n,k,i}_{\si}} (v,w)\,. 
\end{equation}

Recall that, for $\si = \id \in S_{n+k-1}$, the multiplication 
$$
\mu_{\bt^{n,k,i}_{\id}} : P(n) \otimes P(k) \to P(n+k-1)
$$
is called  {\it the elementary insertion} and often 
denoted by $\circ_i$\,. Namely,
for $v \in P(n)$ and $w \in P(k)$ we 
have\footnote{Numbers $n$ and $k$ are suppressed from 
the notation.} 
\begin{equation}
\label{circ-i}
v \circ_i w  :=  \mu_{\bt^{n,k,i}_{\id}} (v, w)\,.
\end{equation}

It is not hard to see that a pseudo-operad structure on 
a collection $P$ (in $\Ch_{\bbK}$ or $\grVect_{\bbK}$) 
is uniquely determined by elementary 
insertions \eqref{circ-i}.
All the remaining multiplications \eqref{mu-bt} can be 
expressed in terms of \eqref{circ-i} using the axioms 
of a pseudo-operad. 

Thus, it is not hard to see that, the following definition of 
a pseudo-operad is equivalent to ours. 
\begin{defi}[Definition 17, \cite{Markl}]
\label{dfn:ps-op}
A pseudo-operad in $\Ch_{\bbK}$ (resp. $\grVect_{\bbK}$)
is a collection $P$ in $\Ch_{\bbK}$ (resp. $\grVect_{\bbK}$) equipped with maps
\begin{equation}
\label{circ-i-dfn}
\circ_i : P(n) \otimes P(k) \to P(n+k-1)\,, 
\qquad 1\le i \le n
\end{equation}
satisfying the associativity axiom and equivariance axioms. 
The associativity axiom states that for all 
homogeneous vectors $a,b,c$ in $P(n_a), P(n_b)$, and $P(n_c)$, 
respectively and for all $1\le i\le n_a$ and $1 \le j \le n_a + n_b -1$
\begin{equation}
\label{circ-i-assoc}
(a \circ_i b) \circ_j c = 
\begin{cases}
(-1)^{|b| |c|}(a \circ_j c) \circ_{i + n_c -1} b  \qquad {\rm if}~~  j < i\,,  \\[0.3cm]
 a \circ_i (b \circ_{j-i+1} c) \qquad {\rm if}~~ i \le j \le i + n_b - 1\,,  \\[0.3cm]
(-1)^{|b| |c|} (a\circ_{j-n_b+1} c) \circ_i b   \qquad {\rm if}~~ j \ge i+n_b \,.
\end{cases}
\end{equation}
The equivariance axioms state that for all $1 \le p \le n_b-1$ and  
$1 \le k \le n_a-1$ we have
\begin{equation}
\label{equiv-1}
a \circ_i (\si_{p\, (p+1)} b) = \si_{(p+i-1)\, (p+i)}(a \circ_i b)\,,
\end{equation}
\begin{equation}
\label{equiv-11}
(\si_{k \, (k+1)} a) \circ_i b = 
\begin{cases}
 \si_{k \, (k+1)} (a\circ_i b) \qquad {\rm if} ~~k+1 < i  \\[0.3cm]
\vs_{i-1,\, i+n_b-1} (a \circ_{i-1} b) \qquad {\rm if} ~~ k+1 = i \\[0.3cm]
\vs^{-1}_{i,\, i+n_b} (a \circ_{i+1} b) \qquad {\rm if} ~~ k+1 = i \\[0.3cm]
\si_{(k+n_b-1) \, (k+n_b)} (a\circ_i b) \qquad {\rm if} ~~ k > i\,.
\end{cases}
\end{equation}
Here $\si_{i_1\, i_2}$ denotes the transposition $(i_1, i_2)$ 
and $\vs_{i_1, i_2}$ is the cycle defined in \eqref{vs-i-j}.
\end{defi}
In \cite{Markl} a pseudo-operad is called non-unital Markl's operad.

\subsection{Operads}
\label{sec:operads}  

An operad is a pseudo-operad $P$ with unit, that is 
a map 
\begin{equation}
\label{unit}
\bu :  \bfone \to P(1) 
\end{equation} 
for which the compositions
\begin{equation}
\label{unit-axiom}
\begin{array}{c}
P(n) \cong P(n) \otimes \bfone  \,
\stackrel{\id \otimes \bu ~~}{\,\longrightarrow\,}\,
 P(n) \otimes P(1)\,
\stackrel{\c_i~}{\,\longrightarrow\,}\, P(n)\\[0.3cm]
P(n) \cong  \bfone \otimes P(n)  \,
\stackrel{\bu \otimes \id ~~}{\,\longrightarrow\,}\,
 P(1) \otimes P(n)\,
\stackrel{\c_1 ~}{\,\longrightarrow\,}\,
 P(n)
\end{array}
\end{equation} 
coincide with the identity map on $P(n)$\,.

Morphisms of pseudo-operads and operads are 
defined in the obvious way. 

\begin{example}
\label{ex:End}
For an object $\cV$ of $\mC$ we denote by 
$\End_{\cV}$ the following collection\footnote{We tacitly assume that the 
symmetric monoidal category $\mC$ has inner $\Hom$.} 
\begin{equation}
\label{End-cV}
\End_{\cV} (n) = \Hom(\cV^{\otimes\, n}, \cV)\,.
\end{equation}
This collection is equipped with the obvious structure of
an operad. Namely, the elementary insertions 
$$
\circ_i \colon \End_{\cV} (n) \otimes \End_{\cV} (m) \to \End_{\cV} (n+m-1)
$$
are defined by the equation 
$$
f \circ_i g := f \circ (\id^{\otimes\, (i-1)} \otimes g \otimes \id^{\otimes\, (n-i)} )
$$
and the unit 
$$
\bu :  \bfone \to \Hom(\cV, \cV)
$$
corresponds to the isomorphism $\bfone \otimes \cV \cong \cV$\,.
We call $\End_{\cV}$ the {\it endomorphism 
operad} of $\cV$\,.
\end{example}

This example plays an important role because it is 
used in the definition of an algebra over an operad. 
Namely, an {\it algebra over an operad} $P$ (in $\mC$)
is an object $\cV$ of $\mC$ together with an operad 
map 
$$
P \to \End_{\cV}\,.
$$

It is not hard to see that an object $\cV$ in $\mC$ is 
an algebra over an operad $P$ if and only if $\cV$ is equipped 
with a collection of multiplication maps
\begin{equation}
\label{mu-cV}
\mu_{\cV} : P(n) \otimes \cV^{\otimes \, n} \to \cV  \qquad n \ge 0\,,
\end{equation}
satisfying the associativity axiom, the equivariance axiom 
and the unitality axiom formulated, for instance, in \cite[Proposition 24]{Markl}\,.

\begin{exer}
\label{exer:Com}
Let $\mC = \grVect_{\bbK}$\,.
Consider the collections
\begin{equation}
\label{Com-u}
\Com_u(n) = \bbK \,,
\end{equation}
and
\begin{equation}
\label{Com}
\Com(n) = \begin{cases}
  \bbK \qquad {\rm if} ~~ n \ge 1\,,  \\
  \bfzero \qquad {\rm if} ~~ n = 0\,.
\end{cases}
\end{equation} 
with the trivial $S_n$-action on $\Com(n)$ 
(resp. $\Com_u(n)$)\,. 
The collections $\Com$ and $\Com_u$ are equipped with 
the obvious operad structures. For $\Com_{u}$ we have 
$$
\c_i  = \id ~:~ \Com_u(n) \otimes \Com_u (k) \cong 
\bbK \otimes \bbK \to \Com_{u}(n+k-1) \cong \bbK\,,
$$
$$
 \bu  = \id ~:~ \bbK ~ \to ~ \Com_u(1) \cong \bbK\,, 
$$
and for $\Com$ we have
$$
\c_i  = \id ~:~ \Com(n) \otimes \Com (k) \cong 
\bbK \otimes \bbK \to \Com(n+k-1) \cong \bbK\,,
$$
if $k \neq 0$, and
$$
 \bu  = \id ~:~ \bbK ~ \to ~ \Com(1) \cong \bbK\,. 
$$
Show that $\Com_{u}$-algebras (resp. $\Com$-algebras) are 
exactly unital (resp. non-unital) commutative algebras. 
\end{exer}

\begin{exer}
\label{exer:tensor}
Let $P$ and $\cO$ be operads (resp. pseudo-operad) 
in $\mC$\,. Show that the collection $P\otimes \cO$ 
with
$$
P \otimes \cO(n) = P(n) \otimes \cO(n)
$$
is naturally an operad (resp. pseudo-operad)\,. 
{\it For this exercise it may be more convenient to 
use Markl's definition \cite[Definition 17]{Markl}.}
\end{exer}

\begin{exer}[The operad $\La$]
\label{exer:La}
Let $\mC = \grVect_{\bbK}$ or 
$\mC =  \Ch_{\bbK}$\,.
Consider the collection $\La$
\begin{equation}
\label{La}
\La(n) = \begin{cases}
\bs^{1-n} \sgn_{n}  \qquad {\rm if} ~~ n \ge 1  \\
 \bfzero  \qquad {\rm if} ~~ n = 0
\end{cases}
\end{equation}
with $\sgn_n$ being the sign representation of $S_n$\,.
Let 
$$
\c_i ~:~ \La(n) \otimes \La(k) \to \La(n+k-1)
$$ 
be the operations defined by 
\begin{equation}
\label{circ-i-La}
1_n \,\c_i\, 1_k  = (-1)^{(1-k) (n-i)} 1_{n+k-1}\,,
\end{equation}
where $1_m$ denotes the generator 
$\bs^{1-m} 1 \in \bs^{1-m} \sgn_{m}$\,.
Prove that \eqref{circ-i-La} together with the 
obvious unit map $\bu = \id :  \bbK \to \La(1) \cong \bbK$
equip the collection $\La$ with a structure of an operad. 
Show that $\La$-algebra structures on $\cV$ are in bijection 
with $\Com$-algebra structure on $\bs^{-1} \cV$\,.
\end{exer}

\begin{exer}
\label{exer:La-oper}
For an operad $\cO$ in the category $\Ch_{\bbK}$ (resp. $\grVect_{\bbK}$)
we denote by $\La\cO$ the operad 
\begin{equation}
\label{La-cO}
\La \cO := \La \otimes \cO\,. 
\end{equation}
Show that $\La\cO$-algebra structures on a cochain 
complex (resp. graded vector space) $\cV$ are in bijection 
with $\cO$-algebra structures on $\bsi \cV$\,.
\end{exer}

\begin{example}
\label{exam:LaLie}
Let $\Lie$ be the operad which governs Lie algebras.
An algebra over $\La\Lie$ in $\grVect_{\bbK}$ is
a graded vector space $\cV$ equipped with the binary operation: 
$$
\{\,,\, \} : \cV \otimes  \cV \to \cV 
$$
of degree $-1$ satisfying the identities:  
$$
\{v_1,v_2\} = (-1)^{|v_1| |v_2|} \{v_2, v_1\}\,,
$$
$$
\{\{v_1, v_2\} , v_3\} +  
(-1)^{|v_1|(|v_2|+|v_3|)} \{\{v_2, v_3\} , v_1\} +
(-1)^{|v_3|(|v_1|+|v_2|)} \{\{v_3, v_1\} , v_2\}  = 0\,,
$$
where $v_1, v_2, v_3 $ are homogeneous vectors in $\cV$\,.
\end{example}

\begin{exer}[Free algebra over an operad $\cO$]
\label{exer:free-alg}
Let $\cO$ be an operad in the category $\Ch_{\bbK}$ 
(resp. $\grVect_{\bbK}$). Show that for every cochain 
complex (resp. graded vector space) $\cV$ the direct sum 
\begin{equation}
\label{cO-cV}
\cO(\cV) : = \bigoplus_{n=0}^{\infty} \Big( \cO(n) \otimes \cV^{\otimes \, n} \Big)_{S_n}   
\end{equation}
carries a natural structure of an algebra over $\cO$\,.
Prove that the $\cO$-algebra $\cO(\cV)$ is free. In other words, the assignment 
$$
\cV \to \cO(\cV)
$$
upgrades to a functor which is left adjoint to 
the forgetful functor from the category of $\cO$-algebras 
to the category $\Ch_{\bbK}$ (resp. $\grVect_{\bbK}$).
\end{exer}

\subsubsection{Augmented operads}
\label{sec:aug-op}
In this subsection $\mC$ is either $\Ch_{\bbK}$ or 
$\grVect_{\bbK}$\,.

Let us observe that the collection 
\begin{equation}
\label{ast}
\ast(n) = \begin{cases}
 \bbK \qquad {\rm if} ~~ n = 1 \\
 \bfzero \qquad {\rm otherwise}
\end{cases}
\end{equation}
is equipped with the unique structure of an operad. 
In fact, $\ast$ is the initial object in the category 
of operads (in $\mC$)\,.

An augmentation of an operad $\cO$ is an 
operad morphism
$$
\ve : \cO \to \ast\,.
$$ 
 
Given a pseudo-operad $P$ in $\mC$  we can always 
form an operad by formally adjoining a unit. 
The resulting operad is naturally augmented. 

Furthermore, the kernel of the augmentation 
for any augmented operad is naturally a pseudo-operad. 
According to \cite[Proposition 21]{Markl}, 
these two constructions give us an equivalence between
the category of augmented operads and the category 
pseudo-operads.  

For an augmented operad $\cO$ we denote by $\cO_{\c}$ the kernel 
of its augmentation. 

\begin{exer}
\label{ex:aug-nonaug}
Show that the operads $\Com$ and $\Lie$ have natural augmentations. 
Prove that the operad $\Com_{u}$  (from Exercise \ref{exer:Com}) 
does not admit an augmentation.
\end{exer}

\subsubsection{Example: the operad $\Ger$}
\label{sec:Ger}

Let us recall that a {\it Gerstenhaber algebra} is a graded vector space $V$
equipped with a commutative (and associative) product 
(without identity) and a \und{degree $-1$} binary operation $\{\,,\,\}$ which 
satisfies the following relations: 
\begin{equation}
\label{Ger-axiom}
\{v_1 , v_2 \}  = (-1)^{|v_1||v_2|} \{ v_2, v_1\}\,,
\end{equation}
\begin{equation}
\label{Ger-axiom1}
\{v , v_1 v_2 \}  = \{v, v_1\} v_2 + (-1)^{|v_1||v| + |v_1|} v_1 \{v, v_2\} \,,
\end{equation}
\begin{equation}
\label{Ger-axiom-Jac}
\{\{v_1, v_2\} , v_3\} +  
(-1)^{|v_1|(|v_2|+|v_3|)} \{\{v_2, v_3\} , v_1\} +
(-1)^{|v_3|(|v_1|+|v_2|)} \{\{v_3, v_1\} , v_2\}  = 0\,.
\end{equation}
In particular, $(V, \{\,,\,\})$ is a $\La\Lie$-algebra.

To define spaces of the operad $\Ger$ governing Gerstenhaber algebras 
we introduce the free Gerstenhaber algebra $\Ger_n$ in $n$ dummy 
variables  $a_1, a_2, \dots, a_n$ of degree $0$\,.  Next we set
$\Ger(0) =  \bfzero$ and $\Ger(1) = \bbK$\,. Then we declare that, 
for $n\ge 2$, $\Ger(n)$ is spanned by monomials of $\Ger_n$ in which 
each dummy variable $a_i$ appears exactly once. 

The symmetric group $S_n$ acts on $\Ger(n)$ by permuting 
the dummy variables and the elementary insertions are 
defined in the obvious way. 
\begin{example}
\label{exam:Ger-insertions}
Let us consider the monomials $u = \{a_2, a_3\} a_1 \{a_4,a_5\} \in \Ger(5)$ and
$w =\{a_1,a_2\}\in \Ger(2)$ and compute the insertions 
$u \circ_2 w$, $u \circ_4 w$ and $w \circ_1 u$\,. We get 
$$
u \circ_2 w =  - \{\{a_2, a_3\}, a_4\} a_1 \{a_5,a_6\}\,, 
\qquad 
u \circ_4 w =  \{a_2, a_3\} a_1 \{\{a_4, a_5 \},a_6\}\,,
$$
$$
w \circ_1 u =  \{ \{a_2, a_3\} a_1 \{a_4,a_5\} , a_6\} =
$$
$$
=  \{a_6,  \{a_2, a_3\} a_1 \{a_4,a_5\} \} =  \{a_6,  \{a_2, a_3\} \} a_1 \{a_4,a_5\} 
$$
$$
-  \{a_2, a_3\} \{a_6, a_1\} \{a_4,a_5\} - \{a_2, a_3\} a_1 \{a_6, \{a_4,a_5\}\}\,.  
$$
(Note that the insertions obey the usual Koszul rule for signs.)
\end{example}

It is easy to see that the operad $\Ger$ is generated by the monomials
$a_1 a_2, \{a_1,a_2\} \in \Ger(2)$ and algebras over the operad $\Ger$ 
are Gerstenhaber algebras. It is also easy to see that the monomial 
$\{a_1,a_2\}$ generates a suboperad of $\Ger$ isomorphic to $\La\Lie$\,.
The operad $\Ger$ carries the obvious augmentation.

We would like to remark that the space $\Ger(n)$ is spanned 
by monomials $v \in \Ger(n)$ of the form 
\begin{equation}
\label{monomial-Ger}
v = v_1 (a_{i_{11}}, a_{i_{12}}, \dots, a_{i_{1p_1}}) \, v_2 (a_{i_{21}}, a_{i_{22}}, \dots, a_{i_{2p_2}})
\, \dots \, v_t (a_{i_{t1}}, a_{i_{t2}}, \dots, a_{i_{t p_t}})\,,
\end{equation}
where $v_1, v_2, \dots, v_t$ are $\La\Lie$-words in $p_1$, $p_2$, \dots, $p_t$ variables, 
respectively, without repetitions and 
$$
\{i_{11}, i_{12}, \dots, i_{1 p_1}\} \sqcup 
\{i_{21}, i_{22}, \dots, i_{2 p_2}\} \sqcup \dots \sqcup \{i_{t1}, i_{t 2}, \dots, i_{t p_t}\}
$$
is a partition of the set of indices $\{1, 2, \dots, n\}$\,. So, from now on, by a monomial 
in $\Ger(n)$ we mean a monomial of the form \eqref{monomial-Ger}\,.

\begin{exer}
\label{exer:Ger-n-basis}
Consider the ordered partitions of the set $\{1, 2, \dots, n\}$ 
\begin{equation}
\label{sp-partition}
\{i_{11}, i_{12}, \dots, i_{1 p_1}\} \sqcup 
\{i_{21}, i_{22}, \dots, i_{2 p_2}\} \sqcup \dots \sqcup \{i_{t1}, i_{t 2}, \dots, i_{t p_t}\}
\end{equation}
satisfying the following properties: 
\begin{itemize}

\item for each $1 \le \beta \le t$ the index $i_{\beta p_{\beta}}$ is 
the biggest among $i_{\beta 1}, \dots, i_{\beta (p_{\beta}-1)}$

\item $i_{1 p_1}  <  i_{2 p_2} < \dots <  i_{t p_t}$ (in particular, $i_{t p_t} = n$).

\end{itemize}

Prove that the monomials 
\begin{equation}
\label{Ger-n-basis}
\{ a_{i_{11}},  \dots, \{ a_{i_{1 (p_1-1)}}, a_{i_{1 p_1}} \br \dots
\{ a_{i_{t1}},  \dots, \{ a_{i_{t (p_t-1)}}, a_{i_{t p_t}} \br
\end{equation}
corresponding to all ordered partitions \eqref{sp-partition} satisfying the above properties
form a basis of $\Ger(n)$\,. Use this fact to show that
$$
\dim(\Ger(n)) = n!\,.
$$

\end{exer}

\subsection{Pseudo-cooperads}

Reversing the arrows in the definition of a pseudo-operad 
we get the definition of a pseudo-cooperad.
More precisely, a pseudo-cooperad is a collection $Q$ 
equipped with comultiplication maps
\begin{equation}
\label{D-bt}
\D_{\bt} :  Q(n) \to \und{Q}_n (\bt)\,, 
\end{equation}
which satisfy a similar list of axioms. 

Just as for pseudo-operads, we have 
\begin{equation}
\label{corolla-iden-co}
\D_{\bq_n} = \id_{Q(n)}\,, 
\end{equation}
where $\bq_n$ is the standard corolla 
(see figures \ref{fig:corolla0}, \ref{fig:corolla12n}).

We also require that the operations \eqref{D-bt}
are $S_n$-equivariant  
\begin{equation}
\label{D-S-n-equiv}
 \D_{\si(\bt)} \circ  \si =  \D_{\bt}\,,  \qquad \forall~~ \si \in S_n, ~~ \bt \in \Tree(n)\,. 
\end{equation}

For every morphism $\la: \bt \to \bt'$ in $\Tree(n)$ we have
\begin{equation}
\label{la-equiv-co}
 \D_{\bt'} = \und{Q}_n(\la) \circ \D_{\bt}\,.
\end{equation}

Finally,  to formulate the coassociativity axiom 
for \eqref{D-bt}, we consider the following 
quadruple $(\wt{\bt}, i, m_i, \bt)$ where $\wt{\bt}$ is an 
$n$-labeled planar tree with $k$ nodal vertices, 
$1 \le i \le k$, $m_i$ is the number of edges terminating 
at the $i$-th nodal vertex of $\wt{\bt}$, and $\bt$ is 
an $m_i$-labeled planar tree. 

The coassociativity axioms states that for each such 
quadruple  $(\wt{\bt}, i, m_i, \bt)$ we have
\begin{equation}
\label{coassoc} 
(\id  \otimes \dots \otimes \id \otimes 
\underbrace{\D_{\bt}}_{i\textrm{-th spot}}
  \otimes \id
\otimes \dots \otimes \id)  \circ
\D_{\wt{\bt}} 
= \beta_{\,\wt{\bt}, i, m_i, \bt} \circ \D_{\,\wt{\bt} \bullet_i \bt}\,, 
\end{equation}
where $\wt{\bt} \bullet_i \bt$ is the $n$-labeled planar tree obtained 
by inserting $\bt$ into the $i$-th nodal vertex of $\wt{\bt}$ and
$\beta_{\,\wt{\bt}, i, m_i, \bt}$ is the isomorphism in $\mC$ which 
is responsible for ``putting tensor factors in the correct order''. 

Just as for pseudo-operads, a pseudo-cooperad structure on 
a collection $Q$  is 
uniquely determined by the comultiplications: 
\begin{equation}
\label{D-i}
\D_i  := D_{\bt^{n,k,i}_{\id}} : Q(n+k-1) \to Q(n) \otimes Q(k)\,,
\end{equation}
where  $\{\bt^{n,k,i}_{\si}\}_{\si \in S_{n+k-1}}$ is
the family of labeled planar trees depicted on
figure \ref{fig:bt-si-and}. 

The comultiplications \eqref{D-i} are called {\it elementary co-insertions}.

\subsection{Cooperads}

We recall that a cooperad is a pseudo-cooperad $Q$ with counit, that is 
a map 
\begin{equation}
\label{counit}
\bu^* :  Q(1) \to \bfone
\end{equation} 
for which the compositions
\begin{equation}
\label{counit-axiom}
\begin{array}{c}
Q(n) \stackrel{\D_i}{\,\longrightarrow\,}\,  Q(n)\otimes Q(1)
\stackrel{\id \otimes \bu^* ~~}{\,\longrightarrow\,}\, Q(n)\otimes \bfone 
\cong Q(n)\\[0.3cm]
Q(n) \stackrel{\D_1}{\,\longrightarrow\,}\,  Q(1)\otimes Q(n)
\stackrel{\bu^* \otimes \id ~~~}{\,\longrightarrow\,}\, \bfone\otimes Q(n)
\cong Q(n)
\end{array}
\end{equation} 
coincide with the identity map on $Q(n)$\,.

Morphisms of pseudo-cooperads and cooperads are 
defined in the obvious way.

Unfortunately there is no natural notion of ``endomorphism cooperad''. 
So a coalgebra over a cooperad $Q$ is defined as an object $\cV$ in $\mC$
equipped with a collection of comultiplication maps
\begin{equation}
\label{D-cV}
\D_{\cV} : \cV \to Q(n) \otimes \cV^{\otimes \, n}  \qquad n \ge 0\,,
\end{equation}
satisfying axioms which are dual to the associativity axiom, 
the equivariance axiom and the unitality axiom from 
\cite[Proposition 24]{Markl}.

\subsubsection{Coaugmented cooperads} 
\label{sec:coaug-coop}

In this subsection $\mC$ is either $\Ch_{\bbK}$ or 
$\grVect_{\bbK}$\,.

It is not hard to see that 
the collection $\ast$ \eqref{ast} is equipped with the unique 
cooperad structure. Furthermore,   $\ast$ is the terminal object 
in the category of cooperads. 

We say that a cooperad $\cC$ is coaugmented if we have 
a cooperad morphism 
\begin{equation}
\label{co-aug}
\ve' : \ast \to \cC\,.
\end{equation}

Given a pseudo-cooperad $\cC$ we can always form 
a cooperad by formally adjoining a counit. The resulting cooperad 
is naturally coaugmented. 

Furthermore, the cokernel of the coaugmentation for any coaugmented 
cooperad is naturally a pseudo-cooperad. Dualizing the line of 
arguments in \cite[Proposition 21]{Markl} we see that these two 
constructions give an equivalence between the category of 
coaugmented cooperads and the category of pseudo-cooperads. 

For a coaugmented cooperad $\cC$ we will denote by $\cC_{\c}$ 
the cokernel of the coaugmentation.

Just as for operads (see Exercise \ref{exer:tensor}), the tensor 
product of two cooperads is naturally a cooperad. 
Furthermore, the collection $\La$ \eqref{La} introduced 
in Exercise  \ref{exer:La} carries a cooperad structure with 
the following elementary co-insertions:
\begin{equation}
\label{La-co-ins}
\D_i (1_{n+k-1}) =  (-1)^{(1-k) (n-i)} 1_n \otimes 1_k \,,
\end{equation}
where  $1_m$ denotes the generator $\bs^{1-m} 1 \in \bs^{1-m} \sgn_{m}$\,.

For a cooperad $\cC$ in the category $\Ch_{\bbK}$ or $\grVect_{\bbK}$ 
we denote by $\La \cC$ the cooperad 
\begin{equation}
\label{La-cC}
\La \cC : = \La \otimes \cC\,.
\end{equation}
Just as for operads (see Exercise \ref{exer:La-oper}), it is easy 
to see that $\La\cC$-coalgebra structures on a cochain complex 
(or a graded vector space) $\cV$ are in bijection with $\cC$-coalgebra 
structures on $\bs^{-1}\cV$\,.

\begin{exer}[Cofree coalgebra over a cooperad $\cC$]
\label{exer:cofree-coalg}
Let $\cC$ be a cooperad in the category $\Ch_{\bbK}$ 
(resp. $\grVect_{\bbK}$). Show that for every cochain 
complex (resp. graded vector space) $\cV$ the direct sum 
\begin{equation}
\label{cC-cV}
\cC(\cV) : = \bigoplus_{n=0}^{\infty} \Big( \cC(n) \otimes \cV^{\otimes \, n} \Big)^{S_n}   
\end{equation}
carries a natural structure of a coalgebra over $\cC$\,.
Prove that the $\cC$-coalgebra $\cC(\cV)$ is cofree\footnote{All $\cC$-coalgebras are
assumed to be nilpotent in the sense of \cite[Section 2.4.1]{Hinich}.}. 
In other words, the assignment 
$$
\cV \to \cC(\cV)
$$
upgrades to a functor which is right adjoint to 
the forgetful functor from the category of $\cC$-coalgebras 
to the category $\Ch_{\bbK}$ (resp. $\grVect_{\bbK}$).
\end{exer}

\subsection{Free operad}
In this section $\mC= \Ch_{\bbK}$ or $\grVect_{\bbK}$\,.

Let $Q$ be a collection. Following \cite[Section 5.8]{BM} the 
spaces $\big\{\psop(Q)(n)\big\}_{n\ge 0}$ of the free pseudo-operad generated by 
the collection $Q$ are 
\begin{equation}
\label{psop-Q-n}
\psop(Q)(n) = \colim \,\und{Q}_n
\end{equation}
where $\und{Q}_n$ is the functor from the groupoid $\Tree(n)$
to $\mC$ defined in Subsection \ref{sec:coll}. 

The pseudo-operad structure on $\psop(Q)$ is defined in 
the obvious way using grafting of trees and the free operad 
$\Op(Q)$ generated by $Q$ is obtained from $\psop(Q)$ by 
formally adjoining the unit. 

Unfolding \eqref{psop-Q-n} we see that $\psop(Q)(n)$ is 
the quotient of the direct sum 
\begin{equation}
\label{oplus-undQ}
\bigoplus_{\bt \in  \Tree(n)} \und{Q}_n(\bt)  
\end{equation}
by the subspace spanned by vectors of the form 
$$
(\bt, X) - (\bt', \und{Q}_n(\la)(X))
$$
where $\la : \bt \to \bt'$ is a morphism in $\Tree(n)$
and $X \in \und{Q}_n(\bt)$\,.

Thus it is convenient to represent vectors in  $\psop(Q)$ and in $\Op(Q)$
by labeled  planar trees with nodal vertices decorated 
by vectors in $Q$\,. The decoration is subject to this rule: if $m(x)$
is the number of edges which terminate at a nodal vertex $x$ 
then $x$ is decorated by a vector $v_x \in Q(m(x))$\,.

If a decorated tree $\bt'$ is obtained from a decorated tree $\bt$
by applying an element $\si \in S_{m(x)}$ to incoming edges 
of a vertex $x$ and replacing the vector $v_x$ by $\si^{-1} (v_x)$
then $\bt'$ and $\bt$ represent the same vectors  in $\psop(Q)$ 
(and in $\Op(Q)$).

\begin{example}
\label{ex:decor-trees}
Let $Q$ be a collection.
Figure \ref{fig:bt-decor} shows a $4$-labeled planar
tree $\bt$ decorated by vectors $v_1 \in Q(3)$, 
$v_2 \in Q(2)$ and $v_3 \in Q(1)$\,. 
Figure \ref{fig:wtbt-decor} shows another decorated 
tree with  $v'_1 = \si_{23}(v_1)$ and 
$v'_2 = \si_{12} (v_2)$, where $\si_{23}$ and 
$\si_{12}$ are the corresponding transpositions 
in $S_3$ and $S_2$, respectively. 
According to our discussion, these decorated trees represent the 
same vector in $\Op(Q)(4)$\,.
\begin{figure}[htp] 
\begin{minipage}[t]{0.45\linewidth}
\centering 
\begin{tikzpicture}[scale=0.5]
\tikzstyle{lab}=[circle, draw, minimum size=4, inner sep=1]
\tikzstyle{n}=[circle, draw, fill, minimum size=4]
\tikzstyle{vt}=[circle, draw, fill, minimum size=0, inner sep=1]
\node[vt] (l1) at (1, 4) {};
\draw (1,4.5) node[anchor=center] {{\small $2$}};
\node[vt] (l2) at (2, 4) {};
\draw (2,4.5) node[anchor=center] {{\small $1$}};
\node[vt] (l3) at (3, 4) {};
\draw (3,4.5) node[anchor=center] {{\small $4$}};
\node[vt] (l4) at (5, 4) {};
\draw (5,4.5) node[anchor=center] {{\small $3$}};
\node[lab] (v2) at (1.5, 3) {{\small $v_2$}};
\node[lab] (v1) at (3, 2) {{\small $v_1$}};
\node[lab] (v3) at (4.5, 3) {{\small $v_3$}};
\node[vt] (r) at (3, 0.5) {};
\draw  (r) edge (v1);
\draw (v1) edge (v2);
\draw (v2) edge (l1);
\draw (v2) edge (l2);
\draw (v1) edge (l3);
\draw (v1) edge (v3);
\draw (v3) edge (l4);
\end{tikzpicture}
\caption{A $4$-labeled decorated tree $\bt$} \label{fig:bt-decor}
\end{minipage}
\begin{minipage}[t]{0.45\linewidth}
\centering 
\begin{tikzpicture}[scale=0.5]
\tikzstyle{lab}=[circle, draw, minimum size=4, inner sep=1]
\tikzstyle{n}=[circle, draw, fill, minimum size=4]
\tikzstyle{vt}=[circle, draw, fill, minimum size=0, inner sep=1]
\node[vt] (l1) at (1, 5) {};
\draw (1,5.5) node[anchor=center] {{\small $1$}};
\node[vt] (l2) at (2, 5) {};
\draw (2,5.5) node[anchor=center] {{\small $2$}};
\node[vt] (l3) at (5, 3.5) {};
\draw (5,4) node[anchor=center] {{\small $4$}};
\node[vt] (l4) at (3, 5) {};
\draw (3,5.5) node[anchor=center] {{\small $3$}};
\node[lab] (v2) at (1.5, 3.5) {{\small $v'_2$}};
\node[lab] (v1) at (3, 2) {{\small $v'_1$}};
\node[lab] (v3) at (3, 3.8) {{\small $v_3$}};
\node[vt] (r) at (3, 0.5) {};
\draw (r) edge (v1);
\draw (v1) edge (v2);
\draw (v2) edge (l1);
\draw (v2) edge (l2);
\draw (v1) edge (l3);
\draw (v1) edge (v3);
\draw (v3) edge (l4);
\end{tikzpicture}
\caption{A $4$-labeled decorated tree $\wt{\bt}$. 
Here $v'_1 = \si_{23}(v_1)$ and $v'_2 = \si_{12}(v_2)$} \label{fig:wtbt-decor}
\end{minipage}
\end{figure}
\end{example}

\begin{rem}
\label{rem:gener-of-free}
In view of the above description, generators  $X \in Q(n)$  of the free operad $\Op(Q)$ can be 
also written in the form  
$$
(\bq_n, X)\,,
$$ 
where $\bq_n$ is the standard $n$-corolla (see figures \ref{fig:corolla0}, \ref{fig:corolla12n}).
\end{rem}

\subsection{Cobar construction}
The underlying symmetric monoidal category $\mC$ is 
the category $\Ch_{\bbK}$ of unbounded cochain complexes
of $\bbK$-vector spaces. 

The cobar construction $\Cobar$ \cite{Fresse}, \cite{GJ}, \cite{GK}, \cite[Section 6.5]{LodayVallette} 
is a functor from the category of coaugmented cooperads in $\Ch_{\bbK}$ to the category 
of augmented operads  in $\Ch_{\bbK}$. It is used to construct free resolutions 
for operads.  

Let $\cC$ be a coaugmented cooperad in $\Ch_{\bbK}$\,.
As an operad in the category $\grVect_{\bbK}$, 
$\Cobar(\cC)$ is freely generated by the collection $\bs\, \cC_{\c}$ 
\begin{equation}
\label{Cobar-cC}
\Cobar(\cC) = \Op(\bs \, \cC_{\c})\,, 
\end{equation}
where $\cC_{\c}$ denotes the cokernel of the coaugmentation.

To define the differential $\pa^{\Cobar}$ on $\Cobar(\cC)$, we recall that 
 $\Tree_2(n)$ is the full subcategory of
$\Tree(n)$ which consists of $n$-labeled planar trees 
with exactly $2$ nodal vertices and $\pi_0(\Tree_2(n))$
is the set of isomorphism
classes in the groupoid $\Tree_2(n)$. 
Due to Exercise 
\ref{exer:shuffles}, the set  $\pi_0(\Tree_2(n))$ is in bijection with 
$(p,n-p)$-shuffles for all $0 \le p \le n$\,.

Since the operad $\Cobar(\cC)$ is freely generated by 
the collection $\bs\,\cC_{\c}$, it suffices to define the differential 
$\pa^{\Cobar}$ on generators. 

We have 
$$
\pa^{\Cobar} = \pa' + \pa''\,, 
$$
with 
\begin{equation}
\label{pa-pr}
\pa' (X) = - \bs\, \pa_{\cC} \bsi X\,,
\end{equation}
and 
\begin{equation}
\label{pa-prpr}
\pa'' (X) = - \sum_{z \in \pi_0(\Tree_2(n)) } (\bs \otimes \bs)
\big( \bt_z ; \D_{\bt_z} (\bsi X) \big)
\,,
\end{equation}
where $X \in \bs\, \cC_{\c}(n)$, $\bt_z$ is any representative of 
the isomorphism class $z\in \pi_0(\Tree_2(n))$, and $\pa_{\cC}$ is 
the differential on $\cC$\,.
The axioms of a pseudo-cooperad imply that the right hand side of
\eqref{pa-prpr} does not depend on the choice of representatives 
$\bt_z$\,.

\begin{exer}
\label{ex:Cobar-diff}
Identity 
$$
\pa' \circ \pa' =0
$$
readily follows from $(\pa_{\cC})^2 = 0$\,.
Use the compatibility of the differential $\pa_{\cC}$
with the cooperad structure and the coassociativity 
axiom (\ref{coassoc}) to deduce the identities
\begin{equation}
\label{papr-paprpr}
\pa' \circ \pa'' + \pa'' \circ \pa' = 0 
\end{equation}
and
\begin{equation}
\label{pa-prpr-square}
\pa'' \circ  \pa'' = 0\,. 
\end{equation}
\end{exer}

\section{Convolution Lie algebra}
\label{sec:convolution}

Let $P$ (resp.  $Q$) be a dg pseudo-operad  (resp. a dg pseudo-cooperad).

We consider the following cochain complex
\begin{equation}
\label{Conv}
\Conv(Q, P) =  \prod_{n \ge 0}
\Hom_{S_n} (Q(n), P(n))\,. 
\end{equation}
with the binary operation $\bul$ defined 
by the formula
\begin{equation}
\label{Conv-bullet}
f \bullet g (X) =  
\sum_{z\in \pi_0(\Tree_2(n))} 
\mu_{\bt_z} (f \otimes g  \,\, \D_{\bt_z} (X))\,, 
\end{equation}
$$
f, g \in  \Conv(Q, P), \qquad X \in Q(n)\,,  
$$
where $\bt_z$ is any representative of the 
isomorphism class $z\in \pi_0(\Tree_2(n))$\,.
The axioms of pseudo-operad (resp. pseudo-cooperad)
imply that the right hand side of \eqref{Conv-bullet} does not 
depend on the choice of representatives $\bt_z$\,.

It follows directly from the definition that 
the operation $\bullet$ is compatible with the 
differential on $\Conv(Q, P)$ coming from 
$Q$ and $P$\,.
Furthermore, 
we claim that 
\begin{prop}
\label{prop:ConvLie}
The bracket 
$$
[f,g] = \big( f \bullet g - (-1)^{|f|\, |g|} g \bullet f \big)
$$
satisfies the Jacobi identity. 
\end{prop}
\begin{proof}
We will prove the proposition by showing that 
the operation (\ref{Conv-bullet}) satisfies the axiom of the 
pre-Lie algebra
\begin{equation}
\label{pre-Lie}
(f \bullet g ) \bullet h - f \bullet (g \bullet h) 
= (-1)^{|g| |h|} (f \bullet h) \bullet g  -  (-1)^{|g| |h|} f \bullet (h \bullet g)\,, 
\end{equation}
where $f,g,h$ are homogeneous vectors in $\Conv(Q, P)$\,.

The expression $\big( (f \bullet g ) \bullet h  - f \bullet (g \bullet h) \big) \, (X) $
can be rewritten as
$$
\big( (f \bullet g ) \bullet h  - f \bullet (g \bullet h) \big) \, (X) = 
$$
$$
\sum_{z \in  \pi_0(\Tree_2(n))} ~
\sum_{z' \in  \pi_0(\Tree_2(m_1(z))) }~
\mu_{\bt_z} \circ (\mu_{\bt_{z'}} \otimes \id) \circ 
(f \otimes  g \otimes  h) \circ  (\D_{\bt_{z'}} \otimes \id) 
\circ \D_{\bt_z} (X) -
$$
$$
\sum_{z \in  \pi_0(\Tree_2(n))}~ 
\sum_{z' \in \pi_0(\Tree_2(m_2(z))) }~
\mu_{\bt_z} \circ (\id \otimes \mu_{\bt_{z'}}) \circ 
(f \otimes  g \otimes  h) \circ  (\id \otimes \D_{\bt_{z'}}) 
\circ \D_{\bt_z} (X)\,,
$$
where $m_1(z)$ (resp. $m_2(z)$) is the number of edges
terminating at the first (resp. the second) nodal vertex 
of the planar tree $\bt_z$\,.

Due to the axioms for the maps $\mu_{\bt}$ and $\D_{\bt}$, we get 
$$
\big( (f \bullet g ) \bullet h  - f \bullet (g \bullet h) \big) \, (X)
$$
$$
\sum_{p, q \ge 0}
\sum_{\tau\in \Sh_{p, q, n-p-q}} 
\mu_{\bt_{\tau}}
\big( (f \otimes  g \otimes  h) \, \D_{\bt_{\tau}} (X) \big)
$$
where $\bt_{\tau}$ is the $n$-labeled planar 
tree depicted on figure  \ref{fig:razdvo}.
\begin{figure}[htp]
\centering
\begin{tikzpicture}[scale=0.5]
\tikzstyle{w}=[circle, draw, minimum size=3, inner sep=1]
\tikzstyle{b}=[circle, draw, fill, minimum size=3, inner sep=1]
\node[b] (b1) at (0, 4) {};
\draw (0,4.6) node[anchor=center] {{\small $\tau(1)$}};
\draw (1,4) node[anchor=center] {{\small $\dots$}};
\node[b] (bp) at (2, 4) {};
\draw (2,4.6) node[anchor=center] {{\small $\tau(p)$}};
\node[b] (bp1) at (4, 4) {};
\draw (4,4.6) node[anchor=center] {{\small $\tau(p+1)$}};
\draw (5.5,4) node[anchor=center] {{\small $\dots$}};
\node[b] (bpq) at (7, 4) {};
\draw (7,4.6) node[anchor=center] {{\small $\tau(p+q)$}};
\node[w] (w2) at (1, 2) {};
\node[w] (w3) at (5.5, 2) {};
\node[b] (bpq1) at (8.5, 2) {};
\draw (8.5,2.6) node[anchor=center] {{\small $\tau(p+q+1)$}};
\draw (9.5,1.8) node[anchor=center] {{\small $\dots$}};
\node[b] (bn) at (11.5, 2) {};
\draw (11.5,2.6) node[anchor=center] {{\small $\tau(n)$}};
\node[w] (w1) at (6, 0) {};
\node[b] (r) at (6, -1) {};
\draw (w2) edge (b1);
\draw (w2) edge (bp);
\draw (w3) edge (bp1);
\draw (w3) edge (bpq);
\draw (w1) edge (w2);
\draw (w1) edge (w3);
\draw (w1) edge (bpq1);
\draw (w1) edge (bn);
\draw (r) edge (w1);
\end{tikzpicture}
\caption{\label{fig:razdvo} $\tau$ is a $(p,q, n-p-q)$-shuffle}
\end{figure}
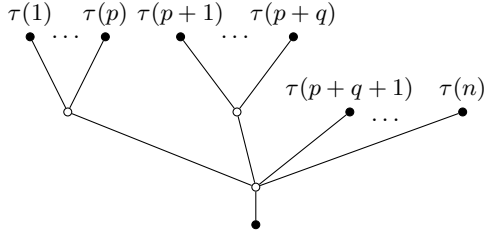 

The set $\{\bt_{\tau} ~|~ \tau \in \Sh_{p, q, n-p-q}\}$ is stable under the 
obvious isomorphism $\la$ which switches the second nodal vertex with 
the third one. Hence, we have 
$$
\big( (f \bullet g ) \bullet h  - f \bullet (g \bullet h) \big) \, (X) =
$$
$$
\sum_{p, q \ge 1}
\sum_{\tau\in \Sh_{p, q, n-p-q}} 
\mu_{\la(\bt_{\tau})} \big(
(f \otimes  g \otimes  h) \, \D_{\la(\bt_{\tau})} (X) \big)\,.
$$

Using axioms (\ref{la-equiv}) and (\ref{la-equiv-co})
and the fact that $f$ is equivariant with respect to the action of 
the symmetric group,  we 
can rewrite the latter expression as follows
$$
\big( (f \bullet g ) \bullet h  - f \bullet (g \bullet h) \big) \, (X) =
$$
$$
\sum_{p, q \ge 1}
\sum_{\tau\in \Sh_{p, q, n-p-q}} 
\mu_{\bt_{\tau}} \circ \und{P}_n(\la) \circ 
(f \otimes  g \otimes  h) \circ  \und{Q}_n(\la) \,  \D_{\bt_{\tau}} (X)= 
$$
$$
\sum_{p, q \ge 1}
\sum_{\tau, \al}
(-1)^{| X^{\tau, \al}_2| \, | X^{\tau, \al}_3 |} \,
\mu_{\bt_{\tau}} \circ \und{P}_n(\la) \circ 
(f \otimes  g \otimes  h) \,
(\si_{12}(X^{\tau, \al}_1), X^{\tau, \al}_3,  X^{\tau, \al}_2) =
$$  
\begin{equation}
\label{associator}
\sum_{p, q \ge 1}
\sum_{\tau, \al}
(-1)^{\ve(\tau, \al, g,h)} \,
\mu_{\bt_{\tau}} \circ \und{P}_n(\la) \,
(\si_{12} f(X^{\tau, \al}_1), g(X^{\tau, \al}_3),  h(X^{\tau, \al}_2))
\end{equation}
where $\si_{12}$ is the transposition $(1,2)$, 
\begin{equation}
\label{Del-bt-tau}
\D_{\bt_{\tau}} (X) = \sum_{\al} (X^{\tau, \al}_1, X^{\tau, \al}_2,  X^{\tau, \al}_3),
\end{equation}
and
\begin{equation}
\label{ve-tau-al-g-h}
\ve(\tau, \al, g,h) = | X^{\tau, \al}_2| \, | X^{\tau, \al}_3 | + 
|h| ( | X^{\tau, \al}_1 |+ | X^{\tau, \al}_3 |) + 
|g|  | X^{\tau, \al}_1 |\,.
\end{equation}

Applying $\und{P}_n(\la)$ to 
$(\si_{12} f(X^{\tau, \al}_1), g(X^{\tau, \al}_3),  h(X^{\tau, \al}_2))$
in (\ref{associator}) we get 
$$
\big( (f \bullet g ) \bullet h  - f \bullet (g \bullet h) \big) \, (X) =
$$
$$
\sum_{p, q \ge 1}
\sum_{\tau, \al}
(-1)^{\ve(\tau, \al, g,h)} (-1)^{(|g|+| X^{\tau, \al}_3|)(|h|+| X^{\tau, \al}_2|)} \,
\mu_{\bt_{\tau}}  \,
(f(X^{\tau, \al}_1),  h(X^{\tau, \al}_2), g(X^{\tau, \al}_3)) =
$$
$$
\sum_{p, q \ge 1}
\sum_{\tau, \al}
(-1)^{\wt{\ve}(\tau, \al, g,h)}\,
\mu_{\bt_{\tau}}  \,
\big( (f\otimes h \otimes g)\, 
\D_{\bt_{\tau}} (X) \big)\,,
$$
where 
\begin{equation}
\label{wt-ve-tau-al-g-h}
\wt{\ve}(\tau, \al, g,h) = \ve(\tau, \al, g,h) +
(|g|+| X^{\tau, \al}_3|)(|h|+| X^{\tau, \al}_2|) + 
|h| |X^{\tau, \al}_1| + |g| (  |X^{\tau, \al}_1| + |X^{\tau, \al}_2|)\,.
\end{equation}
The direct computation shows that 
$$
\wt{\ve}(\tau, \al, g,h) = |g| |h|\quad  \text{mod} \quad  2
$$
and the proposition follows. 
\end{proof}

\subsection{A useful modification $\Conv^{\oplus}(Q,P)$}
\label{sec:Conv-oplus}

Let us observe that for every  dg pseudo-operad $P$  and 
every  dg pseudo-cooperad $Q$ the subcomplex
\begin{equation}
\label{Conv-oplus}
\Conv^{\oplus}(Q,P) :=
  \bigoplus_{n \ge 0}
\Hom_{S_n} (Q(n), P(n)) \subset \Conv(Q,P) 
\end{equation}
is closed with respect to the pre-Lie operation \eqref{Conv-bullet}. 
Thus $\Conv^{\oplus}(Q,P)$ is a dg Lie subalgebra of $\Conv(Q,P)$\,.

We often use this subalgebra in our notes to prove facts 
about its completion  $\Conv(Q,P)$\,.

\subsection{Example: the dg Lie algebra $\Conv(\cC_{\c}, \End_V)$} 
Let $V$ be a cochain complex, $\cC$ be a coaugmented dg cooperad, 
and $\cC_{\c}$ be the cokernel of the coaugmentation. 
We denote by $\cC(V)$ the cofree $\cC$-coalgebra cogenerated 
by $V$\,. Furthermore, we denote by  $p_V$ the natural projection
\begin{equation}
\label{proj-n}
p_V : \cC(V) \to V\,.
\end{equation}
In this subsection we interpret  $\Conv(\cC_{\c}, \End_V)$
as a subalgebra in the dg Lie algebra $\coDer(\cC(V))$ of
coderivations of $\cC(V)$\,.

Let us recall \cite[Proposition 2.14]{GJ} that the map
\begin{equation}
\label{Der-Hom}
\cD \mapsto p_V \circ \cD
\end{equation}
defines an isomorphism of cochain complexes
$$
\coDer(\cC(V)) \cong \Hom(\cC(V), V)\,.
$$

Then we observe that coderivations $\cD \in \coDer(\cC(V))$
satisfying the property 
\begin{equation}
\label{coder-coaug}
\cD \Big|_{V} =  0
\end{equation}
form a dg Lie subalgebra of $\coDer(\cC(V))$\,.
We denote this dg Lie subalgebra by $\coDer'(\cC(V))$\,.

Next, we remark that the formula
\begin{equation}
\label{Conv-Der}
p \circ \cD_{f} (\ga; v_1, v_2,  \dots, v_n) = f(\ga) (v_1, v_2, \dots, v_n)  
\end{equation}
$$
f\in \Conv(\cC_{\c}, \End_V), \qquad  \ga \in \cC_{\c}(n)\,, 
\qquad v_1, v_2, \dots, v_n \in V
$$
defines a map (of graded vector spaces)
\begin{equation}
\label{Conv-Der-map}
f \mapsto \cD_f : \Conv(\cC_{\c}, \End_V) \to \coDer'(\cC(V))\,.
\end{equation}

Finally, we claim that\footnote{Proposition \ref{prop:Conv-Der} is a version 
of \cite[Proposition 2.15]{GJ}.} 
\begin{prop}
\label{prop:Conv-Der}
For every cochain complex $V$ and for 
every  coaugmented dg cooperad $\cC$ the 
map \eqref{Conv-Der-map} is an isomorphism 
of dg Lie algebras
$$
 \Conv(\cC_{\c}, \End_V) \cong \coDer'(\cC(V))\,.
$$
\end{prop}
A proof of this proposition is straightforward so we leave 
it as an exercise.
\begin{exer}
\label{exer:Conv-Der}
Prove Proposition \ref{prop:Conv-Der}.
\end{exer}

\subsection{What if $Q(n)$ is finite dimensional for all $n$?}
Let us assume that the pseudo-cooperad $Q$
satisfies the property 
\begin{pty}
\label{P:fin-dim}
For each $n$ the graded vector space $Q(n)$ is 
finite dimensional. 
\end{pty}

Due to this property we have 
\begin{equation}
\label{Conv-fin}
\Conv(Q, P)  \cong \prod_{n \ge 0} 
\big( P(n) \otimes Q^*(n) \big)^{S_n}\,.
\end{equation}
where $Q^*(n)$ denotes the linear dual of 
the vector space $Q(n)$\,.

The collection $Q^* := \{Q^*(n)\}_{n \ge 0}$ is naturally 
a pseudo-operad and we can express the pre-Lie structure  
\eqref{Conv-bullet} in terms of elementary insertions on 
$P$ and $Q^*$. Namely, given two vectors 
$$
X = \sum_{n \ge 0} v_n \otimes w_n, \qquad 
X' = \sum_{n \ge 0} v'_n \otimes w'_n
$$
in 
$$
\prod_{n \ge 0} \big( P(n) \otimes Q^*(n) \big)^{S_n}\,,
$$
we have 
\begin{equation}
\label{bullet-fin}
X \bullet X' = \sum_{n \ge 1, m \ge 0} (-1)^{|v'_m| |w_n|}
\sum_{\si \in \Sh_{m, n-1}}
\si(v_n \circ_1 v'_m) \otimes \si(w_n \circ_1 w'_m)\,.
\end{equation}

\subsection{The functors $\Conv(Q, ?)$ and $\Conv(?, P)$ preserve quasi-isomorphisms}
\label{sec:Conv-exact} 

It often happens that a pseudo-cooperad 
$Q$ is equipped with a cocomplete ascending filtration
\begin{equation}
\label{filtr-Q}
\bfzero =  \cF^0 Q \subset \cF^1 Q \subset \cF^2 Q \subset  \dots  
\end{equation}
$$
{\rm colim}_m ~ \cF^m Q(n) = Q(n) \qquad \forall~ n
$$
which is compatible with comultiplications $\D_{\bt}$
in the following sense: 
\begin{equation}
\label{filtr-compat}
\D_{\bt} \big( \cF^m Q(n) \big) \subset
\bigoplus_{q_1 + q_2 + \dots + q_k = m}
 \cF^{q_1} Q(r_1) \otimes \cF^{q_2}  Q(r_2) 
 \otimes \dots \otimes  \cF^{q_k} Q(r_k)\,, 
\end{equation}
where $\bt$ is an $n$-labeled planar tree with $k$ nodal 
vertices and $r_i$ is the number of edges terminating at 
the $i$-th nodal vertex of $\bt$\,.

\begin{defi}
\label{dfn:cofiltr}
If a pseudo-cooperad $Q$ is equipped with such 
a filtration then we say that $Q$ is \und{cofiltered}. 
\end{defi}

\begin{exer}
\label{ex:arity-filtr}
Let us recall that a coaugmented (dg) cooperad $\cC$ is 
called \und{reduced} if 
$$
\cC(0) = \bfzero\,, \qquad  \textrm{and} \qquad \cC(1) = \bbK\,.
$$
For every reduced coaugmented cooperad $\cC$, 
the cokernel of the coaugmentation $\cC_{\c}$ carries the 
ascending filtration ``by arity'': 
\begin{equation}
\label{filtr-arity}
\cF^m \cC_{\c}(n) = \begin{cases}
 \cC_{\c}(n) \qquad {\rm if} \quad n \le m + 1   \\
 \bfzero \qquad \qquad {\rm otherwise}\,.
\end{cases}
\end{equation}
Show that this filtration is cocomplete 
and compatible with comultiplications $\D_{\bt}$
in the sense of \eqref{filtr-compat}.
\end{exer}

For any dg operad $P$ and for any cofiltered 
dg pseudo-cooperad $Q$ the dg Lie algebra 
$\Conv(Q, P)$ is equipped with the descending
filtration  
\begin{equation}
\label{filtr-Conv}
\Conv(Q, P) = \cF_1 \Conv(Q, P) \supset \cF_2 \Conv(Q, P) \supset 
\cF_3 \Conv(Q, P) \supset  \dots\,,  
\end{equation}
$$
\cF_m \Conv(Q, P) = 
\{f \in   \Conv(Q, P) ~|~ f(X) =0\quad \forall~ X \in \cF^{m-1}(Q)\}\,.  
$$

Inclusion  \eqref{filtr-compat} implies that the filtration 
on $\Conv(Q, P)$ is compatible with the Lie bracket.  
Furthermore, since the filtration on $Q$ is cocomplete,   
the filtration \eqref{filtr-Conv} is complete
\begin{equation}
\label{Conv-complete}
\lim_m ~ \Conv(Q, P) \big /
\cF_m \Conv(Q, P) = \Conv(Q, P)\,. 
\end{equation}

Any morphism of dg pseudo-operads
\begin{equation}
\label{f}
f : P \to P'\,. 
\end{equation}
induces the obvious map of dg Lie algebras
\begin{equation}
\label{f-star}
f_* : \Conv(Q, P) \to \Conv(Q, P')\,. 
\end{equation}

We claim that 
\begin{thm}
\label{thm:Conv}
If the map (\ref{f}) is a quasi-isomorphism
then so is the map (\ref{f-star}). In addition, if
$Q$ is cofiltered, then the restriction of $f_*$ onto 
$\cF_m \Conv(Q, P)$
$$
f_* ~\Big|_{\cF_m \Conv(Q, P)} ~:~
\cF_m \Conv(Q, P) \to \cF_m\Conv(Q, P')
$$
is a quasi-isomorphism for all $m$\,.
\end{thm}
\begin{proof}
According to \cite[Section 1.4]{Weibel}, every cochain complex 
of $\bbK$-vector spaces is chain homotopy equivalent to 
its cohomology.

Therefore, there exist collections of maps
\begin{equation}
\label{g-n}
g_n : P'(n) \to P(n)
\end{equation}
\begin{equation}
\label{chi-n}
\chi_n : P(n) \to P(n)
\end{equation}
\begin{equation}
\label{chi-prime-n}
\chi'_n : P'(n) \to P'(n)
\end{equation}
such that 
\begin{equation}
\label{f-g-homot}
f_n  \circ g_n - \id_{P'(n)} = \pa \chi'_n + \chi'_n \pa\,,  
\end{equation}
and
\begin{equation}
\label{g-f-homot}
g_n \circ f_n - \id_{P(n)} =  \pa \chi_n + \chi_n \pa\,.  
\end{equation}
In other words, $f_n  \circ g_n $ (resp. $g_n \circ f_n$) is 
homotopic to $ \id_{P'(n)}$ (resp.  $ \id_{P(n)}$).

In general, the set of maps $\{g_n\}_{n\ge 0}$ gives 
us neither a map of operads nor a map of 
the underlying collections. Similarly,
the maps \eqref{chi-n} and \eqref{chi-prime-n} may not be 
$S_n$-equivariant. 

For this reason we switch  
 from the set $\{g_n\}_{n \ge 0}$ to 
the set
\begin{equation}
\label{wt-g}
\wt{g}_n = \frac{1}{(n!)^2}\sum_{\si, \tau \in S_n}
\si \circ g_n \circ \tau\,.  
\end{equation}

It  is easy to see these new maps  $\wt{g}_n$
give us a morphism of the underlying 
collections. Moreover, equations \eqref{f-g-homot}
and  \eqref{g-f-homot} imply the identities
\begin{equation}
\label{f-wtg-homot}
f_n  \circ \wt{g}_n - \id_{P'(n)} = \pa \wt{\chi'}_n + \wt{\chi'}_n \pa\,,  
\end{equation}
and
\begin{equation}
\label{wtg-f-homot}
\wt{g}_n \circ f_n - \id_{P(n)} =  \pa \wt{\chi}_n + \wt{\chi}_n \pa
\end{equation}
with $S_n$-equivariant homotopy operators
\begin{equation}
\label{wtchi-n}
\wt{\chi}_n  =  \frac{1}{(n!)^2}\sum_{\si, \tau \in S_n}
\si \circ \chi_n \circ \tau\,,
\end{equation}
\begin{equation}
\label{wtchi-prime-n}
\wt{\chi'}_n = \frac{1}{(n!)^2}\sum_{\si, \tau \in S_n}
\si \circ \chi'_n \circ \tau\,.   
\end{equation}

Let us now consider the map
\begin{equation}
\label{wt-g-star}
\wt{g}_* :   \Conv(Q, P') \to  \Conv(Q, P) 
\end{equation}

In general $\wt{g}_*$ is not compatible with the Lie 
brackets. Regardless, using equations  \eqref{f-wtg-homot}, 
\eqref{wtg-f-homot}  and $S_n$-equivariance of the homotopy 
operators \eqref{wtchi-n}, \eqref{wtchi-prime-n}, it is not hard 
to see that the compositions  $f_* \circ \wt{g}_*$ and 
$\wt{g}_* \circ f_* $ are homotopic to $\id_{  \Conv(Q, P')}$
and $\id_{\Conv(Q, P)}$, respectively.  

Thus $f_*$ is indeed a quasi-isomorphism.     

To prove the second statement we denote by $f^m_*$ and $\wt{g}^m_* $ the 
restriction of $f_*$ and $\wt{g}_*$ onto 
$$
\cF_m \Conv(Q, P) \qquad \textrm{and} \qquad 
\cF_m \Conv(Q, P') 
$$
respectively. 
 
Using the same homotopy 
operators \eqref{wtchi-n}, \eqref{wtchi-prime-n}, it is not hard 
to see that the compositions  $f^m_* \circ \wt{g}^m_*$ and 
$\wt{g}^m_* \circ f^m_* $ are homotopic to $\id_{\cF_m \Conv(Q, P')}$
and $\id_{\cF_m \Conv(Q, P)}$, respectively.

Theorem \ref{thm:Conv} is proved. 
\end{proof}

\begin{exer}
\label{ex:Conv}
Using the ideas of the above proof, show 
that the (contravariant) functor $\Conv(? , P)$
also preserves quasi-isomorphisms.
\end{exer}

\section{To invert, or not to invert: that is the question}
\label{sec:to-invert}
Let $\cC$ be a coaugmented dg cooperad and $\cC_{\c}$
be the cokernel of the coaugmentation. 
This section is devoted to the lifting property for 
maps from the dg operad $\Cobar(\cC)$\,.
The material contained in this section is 
an adaptation of constructions from \cite{MV-nado}
to the setting of dg operads. 

First, we observe that, since $\Cobar(\cC)$ is 
freely generated by $\bs\,\cC_{\c}$, 
any map of dg operads
\begin{equation}
\label{CobarC-O}
F : \Cobar(\cC) \to \cO
\end{equation}
is uniquely determined by its restriction to 
generators: 
$$
F \Big|_{\bs\, \cC_{\c}} ~:~ \bs\, \cC_{\c} \to \cO\,.
$$

Hence, composing the latter map with the suspension 
operator $\bs$, we get a degree one 
element 
\begin{equation}
\label{al-F}
\al_F  \in \Conv(\cC_{\c}, \cO)
\end{equation}
in the dg Lie algebra $\Conv(\cC_{\c}, \cO)$\,.

\begin{exer}
\label{exer:cobar-conv}
Prove that the compatibility of $F$ with 
the differentials on $\Cobar(\cC)$ and $\cO$
is equivalent to the Maurer-Cartan equation for  
the element $\al_F$ \eqref{al-F} 
in $\Conv(\cC_{\c},\cO)$\,.
\end{exer}

Thus we arrive at the following proposition
\begin{prop}
\label{prop:from-Cobar}
For an arbitrary coaugmented dg cooperad $\cC$
and for an arbitrary dg operad $\cO$, the correspondence
$$
F \mapsto \al_F
$$ 
is a bijection between the set of maps (of dg operads)
\eqref{CobarC-O} and the set of Maurer-Cartan elements in 
 $\Conv(\cC_{\c},\cO)$\,.    $~~~\Box$
\end{prop}

Combining Proposition \ref{prop:Conv-Der} with \ref{prop:from-Cobar}
we deduce the following Corollary
\begin{cor}
\label{cor:Cobar-alg}
For every coaugmented dg cooperad $\cC$ and 
for every cochain complex $V$ the set of $\Cobar(\cC)$-algebra 
structures on $V$ is in bijection with the set of Maurer-Cartan elements 
in the dg Lie algebra
$$
\coDer'(\cC(V)) := \Big\{ \cD \in \coDer(\cC(V)) ~\Big|~  \cD \Big|_{V} = 0  \Big\} \,. ~~~~~~~~ \Box
$$
\end{cor}

\subsection{Homotopies of maps from $\Cobar(\cC)$}

Let $\Omega^{\bul}(\bbK) = \bbK[t] \oplus \bbK[t] dt$ be the polynomial
de Rham algebra on the affine line with $d t$ sitting in degree 1.  
We denote by
\[
\cO^{I} := \cO \otimes \Omega^{\bul}(\bbK) 
\]
the dg operad with underlying collection 
\[
\{ \cO(n) \otimes \Omega^{\bul}(\bbK)\}_{n \geq 0}.
\]
We also denote by $p_0$, $p_1$ the obvious 
maps of dg operads
\begin{equation}
\label{p0-p1}
\begin{array}{c}
p_0 : \cO^I \to \cO, \quad p_{0}(X) = X
\Big|_{t=0,~dt=0}\,,\\[0.3cm]
p_1 : \cO^I \to \cO, \quad p_{1}(X) = X
\Big|_{t=1,~dt=0}\,.
\end{array}
\end{equation}

For our purposes we will use the following 
``pedestrian'' definition of homotopy between 
maps $F, \wt{F} : \Cobar(\cC) \to \cO$\,.
\begin{defi}
\label{dfn:from-Cobar}
We say that maps of dg operads 
$$
F,  \wt{F} : \Cobar(\cC) \to \cO
$$
are \und{homotopic} if there exists a map of dg operads
$$
H :   \Cobar(\cC) \to \cO^I
$$
such that  
$$
F = p_0 \circ H  \qquad \textrm{and} \qquad
\wt{F} = p_1 \circ H\,.
$$
\end{defi}
\begin{rem}
\label{rem:dfn-homot}
Definition \ref{dfn:from-Cobar} leaves out many questions and 
some of these questions may be answered by constructing 
a closed model structure on a subcategory of dg operads 
satisfying certain technical conditions. Unfortunately, many 
dg operads which show up in applications do not satisfy required 
technical conditions. We hope that all such issues will be 
resolved in yet another ``infinity'' treatise \cite{HA} of J. Lurie.
\end{rem}

We now state a theorem which characterizes 
homotopic maps from $\Cobar(\cC)$ in terms 
of the corresponding Maurer-Cartan elements in $\Conv(\cC_{\c}, \cO)$\,.
\begin{thm}
\label{thm:from-Cobar}
Let $\cO$ be an arbitrary dg operad and $\cC$ be a 
coaugmented dg cooperad for which the pseudo-operad $\cC_{\c}$ 
is cofiltered (see Definition \ref{dfn:cofiltr}) and the vector 
space 
\begin{equation}
\label{finite}
\bigoplus_{n \ge 0} \cF^m \cC_{\c}(n)
\end{equation}
is finite dimensional for all $m$\,.
Then two maps of dg operads 
$$
F,  \wt{F} : \Cobar(\cC) \to \cO
$$
are homotopic if and only if the corresponding Maurer-Cartan elements 
$$
\al_{F}, \al_{\wt{F}}  \in \Conv(\cC_{\c}, \cO)
$$
are isomorphic\footnote{We view Maurer-Cartan elements as objects of the Deligne groupoid.
See Appendix \ref{app:GM} for details.}. 
\end{thm}
\begin{proof}
Let  
\begin{equation}
\label{Conv-t}
\Conv(\cC_{\c}, \cO)\{t\}
\end{equation} 
be the  dg Lie subalgebra of $\Conv(\cC_{\c}, \cO)[[t]]$
which consists of infinite series 
\begin{equation}
\label{f-sum}
f =\sum_{k = 0}^{\infty} f_k t^k\,, \qquad f_k \in \cF_{m_k} \Conv(\cC_{\c}, \cO)
\end{equation}
satisfying the condition
\begin{equation}
\label{cond-m-s}
m_1 \le m_2 \le m_3 \le \dots \qquad  \lim_{k\to \infty} m_k = \infty\,. 
\end{equation}

Combining condition \eqref{cond-m-s} together with 
the fact that the filtration on $\cC_{\c}$ is cocomplete,
we conclude that, for every $X\in \cC_{\c}(n)$ and 
for every $f$ in \eqref{Conv-t}, the sum 
$$
 \sum_{k = 0}^{\infty} f_k(X) t^k
$$
has only finitely many non-zero terms. 
 
Therefore, the formula 
\begin{equation}
\label{Psi-dfn}
\Psi(f)(X) = \sum_{k = 0}^{\infty} f_k(X) t^k\,, \qquad 
X \in \cC_{\c}(n) 
\end{equation}
defines a map 
$$
\Psi :  \Conv(\cC_{\c}, \cO)\{t\} \to 
\Conv(\cC_{\c}, \cO[t])\,.  
$$

Let us now consider a vector $g \in \Conv(\cC_{\c}, \cO[t])$\,. 

Since the vector spaces
$$
\bigoplus_{n \ge 0} \cF^{m-1} \cC_{\c}(n)
$$
are finite dimensional, for each $m$ there exists
a positive integer $N_m$ such that the polynomials
$$
g(X) = \sum_{k \ge 0} g_k (X) t^k
$$ 
have degrees $\le N_m$ for all $X \in \cF^{m-1} \cC_{\c}(n)$ and for all $n$\,.
 Moreover, the integers $\{N_m\}_{m \ge 1}$
can be chosen in such a way that 
$$
N_1 \le N_2 \le N_3 \le \dots\,.
$$

Therefore, the formula: 
\begin{equation}
\label{Psi-inv}
\Psi'(g) = \sum_{k=0}^{\infty} \Psi'_k(g) t^k
\end{equation}
$$
\Psi'_k(g) (X) : = g_k (X)\,, \qquad X \in \cC_{\c}
$$
defines a map 
$$
\Psi' : \Conv(\cC_{\c}, \cO[t]) \to \Conv(\cC_{\c}, \cO)\{t\}\,.
$$
Furthermore, it is easy to see that $\Psi'$ is the inverse of $\Psi$\,.

Thus the dg Lie algebras $\Conv(\cC_{\c}, \cO[t])$ 
and $\Conv(\cC_{\c}, \cO)\{t\}$ are naturally isomorphic.

To prove the ``only if'' part we start with 
a map of dg operads 
$$
H : \Cobar(\cC) \to \cO^I
$$
which establishes a homotopy between $F$ and $\wt{F}$
and let 
\begin{equation}
\label{al-H}
\al_{H} = \al^{(1)}_H + \al^{(0)}_H\, d t \in \Conv(\cC_{\c}, \cO^{I})
\end{equation}
be the Maurer-Cartan element corresponding to $H$\,.
Here $\al^{(1)}_H$ (resp. $\al^{(0)}_H$) is a degree $1$ (resp. degree $0$) vector in 
$\Conv(\cC_{\c}, \cO[t]) \cong \Conv(\cC_{\c}, \cO)\{t\}$\,.

The Maurer-Cartan equation for $\al_H$
$$
dt \frac{d}{d t} \al_H + \pa \al_H  + \frac{1}{2}[\al_H, \al_H] =0
$$
is equivalent to the pair of equations
\begin{equation}
\label{MC-al-H1}
\pa \al^{(1)}_H  + \frac{1}{2}[\al^{(1)}_H, \al^{(1)}_H] = 0
\end{equation}
and
\begin{equation}
\label{MC-al-H0}
\frac{d}{d t} \al^{(1)}_H =  \pa \al^{(0)}_H  - [\al^{(0)}_H,  \al^{(1)}_H]\,. 
\end{equation}

Using equations \eqref{MC-al-H1} and \eqref{MC-al-H0} we 
deduce from \cite[Theorem C.1, App. C]{stable1} that
the Maurer-Cartan elements 
$$
\al^{(1)}_H \Big|_{t=0} \qquad \textrm{and} \qquad 
\al^{(1)}_H \Big|_{t=1}
$$
in $\Conv(\cC_{\c}, \cO)$ are connected by the action of the 
group 
$$
\exp \Big(\Conv(\cC_{\c}, \cO) \Big)\,.
$$

Since 
$$
\al^{(1)}_H \Big|_{t=0} = \al_F \qquad \textrm{and} \qquad 
\al^{(1)}_H \Big|_{t=1} = \al_{\wt{F}}
$$
we conclude that the ``only if'' part is proved. 

We leave the easier ``if'' part as an exercise.
(See Exercise \ref{exer:if} below.)  
\end{proof}
\begin{exer}
\label{exer:if}
Prove the ``if'' part of Theorem \ref{thm:from-Cobar}.
\end{exer}

We now deduce the following corollary.
\begin{cor}
\label{cor:lifting}
Let $\cC$ be a coaugmented dg cooperad which
satisfies the conditions of Theorem  \ref{thm:from-Cobar}.
If $U: \cO \to \cO'$ is a quasi-isomorphism of 
dg operads then for every operad morphism 
$F' : \Cobar(\cC) \to \cO'$ there exists 
a morphism $F : \Cobar(\cC) \to \cO$
such that the diagram 
\begin{equation}
\label{lifting}
\xymatrix@M=0.5pc{
~~~ & \Cobar(\cC)\ar@{.>}[ld]_{F} \ar[d]^{F'} \\
\cO\ar[r]^{U} & \cO'   
}
\end{equation}
commutes up to homotopy. 
Moreover the morphism $F$ is determined 
uniquely up to homotopy.   
\end{cor}
\begin{proof}
The map $U$ induces the homomorphism of dg Lie algebras 
$$
U_* : \Conv(\cC_{\c}, \cO) \to \Conv(\cC_{\c}, \cO')\,.
$$ 

Due to Theorem \ref{thm:Conv} $U_*$ is a quasi-isomorphism 
of dg Lie algebras. Moreover, the restriction of $U_*$
$$
U_*  \Big|_{ \cF_m \Conv(\cC_{\c}, \cO)} :
\cF_m \Conv(\cC_{\c}, \cO) \to  \cF_m \Conv(\cC_{\c}, \cO')
$$
is also a quasi-isomorphism of dg Lie algebras for all $m$\,.  

Hence, Theorem \ref{thm:GM} from Appendix \ref{app:GM} implies that 
$U_*$ induces a bijection between the isomorphism  
classes of Maurer-Cartan elements in  $\Conv(\cC_{\c}, \cO)$ 
and in $ \Conv(\cC_{\c}, \cO')$\,.

Thus, the statements of the corollary follow immediately 
from Theorem \ref{thm:from-Cobar}. 
\end{proof}

\subsection{Models for homotopy algebras}
\label{sec:HA}
Developing the machinery of algebraic operads is 
partially motivated by the desire to blend 
together concepts of abstract algebra and concepts of 
homotopy theory \cite{Leinster}, \cite{HA}, \cite{Markl-HA}. 

Thus, in homotopy theory, the notions of Lie algebra, commutative
algebra, and Gerstenhaber algebra are replaced by 
their $\infty$-versions (a.k.a homotopy versions): $L_{\infty}$-algebras, $\Com_{\infty}$-algebras
and $\Ger_{\infty}$-algebras, respectively. These are examples of  
{\it homotopy algebras}.

In this paper we will go into a general philosophy for homotopy 
algebras and instead limit ourselves to conventional definitions.

\begin{defi}
\label{dfn:Lie-infty}
An \und{$\Lie_{\infty}$-algebra}  (a.k.a. \und{$L_{\infty}$-algebra}) is an 
algebra (in $\Ch_{\bbK}$) over the operad 
\begin{equation}
\label{Lie-infty}
\Lie_{\infty} = \Cobar(\La\coCom)\,.
\end{equation}
\end{defi}
\begin{defi}
\label{dfn:Com-infty}
A \und{$\Com_{\infty}$-algebra} is an 
algebra (in $\Ch_{\bbK}$) over the operad 
\begin{equation}
\label{Com-infty}
\Com_{\infty} = \Cobar(\La\coLie)\,.
\end{equation}
\end{defi}
Finally, 
\begin{defi}
\label{dfn:Ger-infty}
A \und{$\Ger_{\infty}$-algebra} is an 
algebra (in $\Ch_{\bbK}$) over the operad 
\begin{equation}
\label{Ger-infty}
\Ger_{\infty} = \Cobar(\Ger^{\vee})\,,
\end{equation}
where 
\begin{equation}
\label{Ger-vee}
\Ger^{\vee} = \big( \La^{-2} \Ger \big)^{*}\,,
\end{equation}
and $~^{*}$ is the operation of taking linear dual. 
\end{defi}

The above definitions are partially motivated by the observation that 
the operads $\Lie_{\infty}$, $\Com_{\infty}$ and $\Ger_{\infty}$ are 
free resolutions of the operads $\Lie$, $\Com$ and $\Ger$, respectively. 

Thus the canonical quasi-isomorphism of dg operads
\begin{equation}
\label{Lie-infty-Lie}
U_{\Lie} : \Lie_{\infty} = \Cobar(\La\coCom) ~\to~ \Lie
\end{equation}
corresponds to the Maurer-Cartan element\footnote{Recall that, due to Proposition \ref{prop:from-Cobar}, 
operads maps from $\Cobar(\cC)$ to $\cO$ are identified with 
Maurer-Cartan element of $\Conv(\cC_{\c}, \cO)$.} 
$$
\al_{\Lie} = [a_1, a_2] \otimes b_1 b_2  \in \Conv(\La\coCom_{\c}, \Lie) 
\cong  \prod_{n \ge 2} \Big( \Lie(n) \otimes \La^{-1}\Com(n) \Big)^{S_n}\,,
$$
where $[a_1, a_2]$ (resp.  $b_1b_2$) denotes the 
canonical generator of $\Lie(2)$ (resp. $\La^{-1}\Com(2)$)\,.

Similarly, the canonical quasi-isomorphism of dg operads
\begin{equation}
\label{Com-infty-Com}
U_{\Com} : \Com_{\infty} = \Cobar(\La\coLie) ~\to~ \Com
\end{equation}
corresponds to the Maurer-Cartan element 
$$
\al_{\Com} = a_1 a_2 \otimes \{b_1, b_2\} \in \Conv(\La\coLie_{\c}, \Com) 
\cong  \prod_{n \ge 2} \Big( \Com(n) \otimes \La^{-1}\Lie(n) \Big)^{S_n}\,,
$$
where $a_1 a_2$ (resp.  $\{b_1, b_2 \}$) denotes the 
canonical generator of $\Com(2)$ (resp. $\La^{-1}\Lie(2)$)\,.

Finally the  canonical quasi-isomorphism of dg operads
\begin{equation}
\label{Ger-infty-Ger}
U_{\Ger} : \Ger_{\infty} = \Cobar(\Ger^{\vee}) ~\to~ \Ger
\end{equation}
corresponds to the Maurer-Cartan element 
$$
\al_{\Ger} = a_1 a_2 \otimes \{b_1, b_2\} + \{a_1, a_2\} \otimes b_1 b_2 
\in \Conv(\Ger^{\vee}_{\c}, \Ger) = 
\prod_{n\ge 2}  \Big( \Ger(n) \otimes \La^{-2}\Ger(n) \Big)^{S_n}\,,
$$
where  $a_1 a_2$,  $\{a_1, a_2 \}$ are the canonical 
generators of $\Ger(2)$ and  $b_1 b_2$,  $\{b_1, b_2 \}$ are the canonical 
generators of $\La^{-2}\Ger(2)$\,. 

We should remark that here, instead of Lie algebras and $L_{\infty}$-algebras 
we often deal with $\La\Lie_{\infty}$-algebras. It is not hard to see that 
$\La\Lie_{\infty}$-algebras are algebras in $\Ch_{\bbK}$ over the 
operad 
\begin{equation}
\label{LaLie-infty}
\La\Lie_{\infty} = \Cobar(\La^2 \coCom)\,.
\end{equation} 
Furthermore, the canonical quasi-isomorphism 
 \begin{equation}
\label{LaLie-infty-LaLie}
U_{\La\Lie} : \La\Lie_{\infty} = \Cobar(\La^2\coCom) ~\to~ \La\Lie
\end{equation}
corresponds to the Maurer-Cartan element 
$$
\al_{\La\Lie} = \{a_1, a_2\} \otimes b_1 b_2  \in \Conv(\La^2\coCom_{\c}, \La\Lie) 
\cong  \prod_{n \ge 2} \Big( \La\Lie(n) \otimes \La^{-2}\Com(n) \Big)^{S_n}\,,
$$
where $\{a_1, a_2\}$ (resp.  $b_1b_2$) denotes the 
canonical generator of $\La\Lie(2)$ (resp. $\La^{-2}\Com(2)$)\,.

\subsubsection{$\La\Lie_{\infty}$-algebras}
\label{sec:La-Lie-infty}

Let $V$ be a cochain complex.
 
Since $\La\Lie_{\infty} = \Cobar(\La^2\coCom)$, 
Corollary \ref{cor:Cobar-alg} implies that $\La\Lie_{\infty}$-algebra 
structure on $V$ is a choice of degree $1$ coderivation 
$$
\cD \in \coDer(\La^2\coCom(V)) 
$$
satisfying the Maurer-Cartan equation 
\begin{equation}
\label{MC-cD}
\pa \cD + \frac{1}{2} [\cD, \cD] = 0
\end{equation}
together with the condition 
$$
\cD \Big|_{V} = 0\,.
$$

On the other hand, according to \cite[Proposition 2.14]{GJ}, 
any coderivation of $\La^2\coCom(V)$ is uniquely determined  by 
its composition $p_V \circ \cD$ with the projection 
$$
p_V : \La^2 \coCom(V) \to V\,.
$$

Thus, since
$$
\La^2 \coCom(V) = \bs^2 S(\bs^{-2} V),
$$
an $\La\Lie_{\infty}$-structure on $V$ is determined by the 
infinite sequence of multi-ary operations 
\begin{equation}
\label{p-cD-n}
\{ , , \dots, \}_n = p_V \circ \cD\, \bs^{2n-2} : S^n(V) \to V\,, \qquad n \ge 2 
\end{equation}
where the $n$-th operation $\{ , , \dots, \}_n$ carries degree $3-2n$\,.

The Maurer-Cartan equation \eqref{MC-cD} is equivalent to the sequence 
of the following quadratic relations on operations \eqref{p-cD-n}: 
\begin{equation}
\label{n-th-relation}
\pa \{v_1,v_2, \dots, v_n\}_n  + 
\sum_{i=1}^{n} (-1)^{|v_1|+ \dots + |v_{i-1}|} \{v_1, \dots, v_{i-1}, \pa v_i, v_{i+1}, \dots, v_n\}_n  
+
\end{equation}
$$
\sum_{p = 2}^{n-1} \sum_{\si \in \Sh_{p,n-p}} (-1)^{\ve(\si, v_1, \dots, v_n)}
\{\{v_{\si(1)}, \dots, v_{\si(p)} \}_p, v_{\si(p+1)}, \dots,  v_{\si(n)}\}_{n-p+1} = 0\,,
$$
where $\pa$ is the differential on $V$ and  $(-1)^{\ve(\si, v_1, \dots, v_n)}$
is the sign factor determined by 
the usual Koszul rule.

\begin{rem}
\label{rem:LaLie-Lie}
Even though there is an obvious bijection between 
$\La\Lie_{\infty}$-structures on $V$ and $L_{\infty}$-structures 
on $\bsi V$,  it is often easier to deal with signs in formulas 
for $\La\Lie_{\infty}$-structures.  
\end{rem}

\section{Twisting of operads}
\label{sec:twist}
Let $\cO$ be a dg operad equipped with a map
\begin{equation}
\label{from-hoLie}
\wh{\vf} : \La \Lie_{\infty} \to \cO\,.
\end{equation}

Let $V$ be an algebra over $\cO$\,.
Using the map $\wh{\vf}$, we equip $V$
with an $\La\Lie_{\infty}$-structure. 

If we assume, in addition, that $V$ is equipped with 
a complete descending filtration 
\begin{equation}
\label{filtr-V}
V \supset F_1 V \supset F_2 V \supset F_3 V \supset \dots \,, 
\qquad  V = \lim_{k} V  \Big / F_k V 
\end{equation}
and the $\cO$-algebra structure on $V$ is compatible 
with this filtration then we may define Maurer-Cartan elements of $V$ as 
degree $2$ elements $\al \in  F_1 V $ satisfying the equation
\begin{equation}
\label{MC-eq}
\pa (\al) + \sum_{n \ge 2} \frac{1}{n!} \{\al, \al, \dots, \al\}_n = 0\,,
\end{equation}
where $\pa$ is the differential on $V$ and $\{\cdot, \cdot, \dots, \cdot\}_n$
is the $n$-th operation of the $\La\Lie_{\infty}$-structure on $V$\,. 

Given such a Maurer-Cartan element $\al$ we can twist the differential 
on $V$ and insert $\al$ into various $\cO$-operations on $V$\,. 
This way we get a new algebra structure on $V$\,. 
 
It turns out that this new algebra structure is governed by 
an operad $\Tw \cO$ which is built from the pair 
$(\cO, \wh{\vf})$\,. 

This section is devoted to the construction of $\Tw \cO$\,.

\subsection{Intrinsic derivations of an operad}
Let $\cO$ be an dg operad. We recall that a $\bbK$-linear map
$$
\de : \bigoplus_{n \ge 0} \cO(n)  \to  \bigoplus_{n \ge 0} \cO(n) 
$$ 
is an {\it operadic derivation} if for every $a \in \cO(n)$, $\de(a) \in \cO(n)$
and for all homogeneous vectors $a_1\in \cO(n)$, $a_2\in \cO(k)$
$$
\de (a_1 \circ_i a_2) = \de(a_1) \circ_i a_2  + 
(-1)^{|\de| |a_1|} a_1 \circ_i \de (a_2)\,, 
\qquad \forall~~ 1 \le i \le n\,.
$$

Let us now observe that the 
operation $\circ_1$ equips $\cO(1)$ with a structure of a dg associative 
algebra.  We consider $\cO(1)$
as a dg Lie algebra with the Lie bracket being the commutator. 

We claim that 
\begin{prop}
\label{prop:cO-1-deriv}
The formula 
\begin{equation}
\label{cO-1-deriv}
\de_b (a) = b \circ_1 a - (-1)^{|a| |b|} \sum_{i=1}^n a \circ_i b   
\end{equation}
with 
$$
b \in \cO(1), \qquad \textrm{and} \qquad  a\in \cO(n)
$$ 
defines an operadic derivation of $\cO$ for every $b \in \cO(1)$\,.
\end{prop}
Operadic derivations of the form \eqref{cO-1-deriv} are called {\it intrinsic}.\\

\begin{proof}
Let $a_1\in \cO(n_1)$ and  $a_2\in \cO(n_2)$\,. 
Then for every $b \in \cO(1)$ and $1 \le j \le n_1$ we have
$$
\de_b (a_1  \circ _j a_2) = 
b \circ_1 (a_1  \circ _j a_2) - 
(-1)^{(|a_1|+|a_2|) |b|} \sum_{i=1}^{n_1+n_2-1} (a_1  \circ _j a_2) \circ_i b  = 
$$
$$
(b \circ_1 a_1)  \circ _j a_2 - 
(-1)^{|a_1| |b|} \sum_{i \neq j}^{1 \le i \le n_1} (a_1 \circ_i b) \circ _j a_2 - 
(-1)^{ (|a_1|+|a_2|) |b|} \sum_{i=1}^{n_2} a_1  \circ _j (a_2 \circ_i b)
$$
$$
= 
(b \circ_1 a_1)  \circ _j a_2 - 
(-1)^{|a_1| |b|} \sum_{i =1}^{n_1} (a_1 \circ_i b) \circ _j a_2 
$$
$$ 
+  (-1)^{|a_1| |b|} (a_1 \circ_j b) \circ _j a_2 - 
(-1)^{ (|a_1|+|a_2|) |b|} \sum_{i=1}^{n_2} a_1  \circ _j (a_2 \circ_i b)
$$
$$
= (b \circ_1 a_1)  \circ _j a_2 - 
(-1)^{|a_1| |b|} \sum_{i =1}^{n_1} (a_1 \circ_i b) \circ _j a_2 
$$
$$ 
+  (-1)^{|a_1| |b|} a_1 \circ_j (b \circ _1 a_2) - 
(-1)^{ (|a_1|+|a_2|) |b|} \sum_{i=1}^{n_2} a_1  \circ _j (a_2 \circ_i b)
$$
$$
= \de_{b}(a_1) \circ_j a_2 + (-1)^{|a_1||b|} a_1 \circ_j \de_b(a_2) \,.
$$
Hence $\de_b$ is indeed an operadic derivation of $\cO$\,.

It remains to verify the identity 
\begin{equation}
\label{de-Lie}
[\de_{b_1}, \de_{b_2}] = \de_{[b_1,b_2]}
\end{equation}
and we leave this step as an exercise.
\end{proof}

\begin{exer}
\label{ex:de-Lie}
Verify identity  \eqref{de-Lie}.
\end{exer}

\subsection{Construction of the operad $\wt{\Tw} \cO$}
Let us recall that, since $\La\Lie_{\infty}= \Cobar(\La^2 \coCom)$, 
the morphism \eqref{from-hoLie}
is determined by a Maurer-Cartan element 
\begin{equation}
\label{vf}
\vf \in \Conv(\La^2 \coCom_{\c}, \cO)\,. 
\end{equation}

The $n$-th space of $\La^2 \coCom_{\c}$ is
the trivial $S_{n}$-module placed in degree $2-2n$: 
$$
\La^2 \coCom (n) = \bs^{2-2n} \bbK\,.
$$
So we have
$$
\Conv(\La^2 \coCom_{\c}, \cO) = 
 \prod_{n \ge 2}  \Hom_{S_n}( \bs^{2-2n}\bbK, \cO(n)) 
 =  \prod_{n \ge 2}   \bs^{2n-2} \big( \cO(n) \big)^{S_n}\,.
$$

For our purposes we will need to extend the dg Lie algebra 
$\Conv(\La^2 \coCom_{\c}, \cO)$ to 
\begin{equation}
\label{cL-cO}
\cL_{\cO} = \Conv(\La^2 \coCom, \cO) =
 \prod_{n \ge 1}  \Hom_{S_n}( \bs^{2-2n}\bbK, \cO(n))\,. 
\end{equation}
It is clear that 
$$
\cL_{\cO}  =   \prod_{n \ge 1}   \bs^{2n-2} \big( \cO(n) \big)^{S_n}\,.
$$

For $n, r \ge 1$ we realize the group $S_r$ 
as the following subgroup of $S_{r+n}$
\begin{equation}
\label{S-r-realize}
S_r \cong \big\{ \si\in S_{r+n} ~|~ \si(i) = i\,, \quad \forall ~ i > r \big\}\,.
\end{equation}
In other words, for every $n \ge 1$, the group $S_r$ may be 
viewed as subgroup of $S_{r+n}$ permuting the 
first $r$ letters. We set $S_0$ to be the trivial group.

Using this embedding of $S_{r}$ into $S_{n+r}$
we introduce the following collection ($n \ge 0$)
\begin{equation}
\label{pre-Tw-cO}
\wt{\Tw} \cO(n) =  \prod_{r \ge 0}  \Hom_{S_r} ( \bs^{-2r} \bbK, \cO(r+n))\,.
\end{equation}
It is clear that 
$$
\wt{\Tw} \cO(n) = \prod_{r \ge 0}  \bs^{2r} \big(\cO(r+n)\big)^{S_r}\,. 
$$

To define an operad structure on \eqref{pre-Tw-cO} we denote by 
$1_{r}$ the generator $ \bs^{-2r} \,1 \in  \bs^{-2r} \bbK$\,.  Then 
the identity element $\bu$ in $\wt{\Tw} \cO(1)$ is given by 
\begin{equation}
\label{bu}
\bu(1_r) = \begin{cases}
\bu_{\cO}
  \qquad {\rm if} ~~ r = 0\,, \\
 0 \qquad {\rm otherwise}\,,
\end{cases}
\end{equation}
where $\bu_{\cO}\in \cO(1)$ is the identity element 
for the operad $\cO$\,.

Next, for $f \in \wt{\Tw} \cO(n)$ and $g \in \wt{\Tw} \cO(m)$,
we define the $i$-th elementary insertion $\c_{i}$ 
$1 \le i \le n$ by the formula 
\begin{equation}
\label{circ-i-for-Tw-cO}
f \,\c_i\, g (1_r) = \sum_{p = 0}^{r} 
\sum_{\si \in \Sh_{p, r-p}} \mu_{\bt_{\si,i}} \big( f(1_p) \otimes g (1_{r-p})  \big)\,. 
\end{equation}
where the tree $\bt_{\si,i}$ is depicted on figure \ref{fig:bt-si-i}.
\begin{figure}[htp]
\centering
\begin{tikzpicture}[scale=0.6]
\tikzstyle{w}=[circle, draw, minimum size=3, inner sep=1]
\tikzstyle{b}=[circle, draw, fill, minimum size=3, inner sep=1]
\node[b] (si1) at (0, 2) {};
\draw (0,2.5) node[anchor=center] {{\small $\si(1)$}};
\draw (1,2) node[anchor=center] {{\small $\dots$}};
\node[b] (sip) at (2, 2) {};
\draw (2,2.5) node[anchor=center] {{\small $\si(p)$}};
\node[b] (r1) at (3.5, 2) {};
\draw (3.7,2.5) node[anchor=center] {{\small $r+1$}};
\draw (5,2) node[anchor=center] {{\small $\dots$}};
\node[b] (r1i) at (6.5, 2) {};
\draw (6.4,2.5) node[anchor=center] {{\small $r+i-1$}};
\node[w] (w2) at (8.5, 3) {};
\node[b] (sip1) at (6, 5) {};
\draw (5.8,5.5) node[anchor=center] {{\small $\si(p+1)$}};
\draw (7.3,5) node[anchor=center] {{\small $\dots$}};
\node[b] (sir) at (8, 5) {};
\draw (8,5.5) node[anchor=center] {{\small $\si(r)$}};
\node[b] (ri) at (9.5, 5) {};
\draw (9.5,5.5) node[anchor=center] {{\small $r+i$}};
\draw (10.5,5) node[anchor=center] {{\small $\dots$}};
\node[b] (ri1m) at (12, 5) {};
\draw (12.7,5.5) node[anchor=center] {{\small $r+i+m-1$}};
\node[b] (rim) at (10.5, 2) {};
\draw (10.5,2.5) node[anchor=center] {{\small $r+i+m$}};
\draw (11.7,2) node[anchor=center] {{\small $\dots$}};
\node[b] (rn1m) at (13, 2) {};
\draw (13.8,2.5) node[anchor=center] {{\small $r+n+m-1$}};
\node[w] (w1) at (7.5, 0) {};
\node[b] (r) at (7.5, -1) {};
\draw (w2) edge (sip1);
\draw (w2) edge (sir);
\draw (w2) edge (ri);
\draw (w2) edge (ri1m);
\draw (w1) edge (si1);
\draw (w1) edge (sip);
\draw (w1) edge (r1);
\draw (w1) edge (r1i);
\draw (w1) edge (w2);
\draw (w1) edge (rim);
\draw (w1) edge (rn1m);
\draw (r) edge (w1);
\end{tikzpicture}
\caption{\label{fig:bt-si-i} Here $\si$ is a $(p, r-p)$-shuffle}
\end{figure} 

To see that the element $f \,\c_i\, g (1_r)\in \cO(r+n+m -1)$ is 
$S_r$-invariant one simply needs to use the fact that 
every element $\tau\in S_r$ can 
be uniquely presented as the composition 
$\tau_{sh} \c \tau_{p,r-p} $, where $\tau_{sh}$ is 
a $(p, r-p)$-shuffle and $\tau_{p,r-p} \in S_p \times S_{r-p}$\,.

Let  $f \in \wt{\Tw} \cO(n)$, $g \in \wt{\Tw} \cO(m)$, 
$h \in \wt{\Tw} \cO(k)$, $1\le i \le n$, and $1 \le j \le m$\,.  
To check the identity 
\begin{equation}
\label{fgh-assoc}
f\,\c_i \, (g \,\c_j\, h) =
(f\,\c_i \, g ) \,\c_{j+i-1}\, h 
\end{equation}
we observe that
$$
f\,\c_i \, (g \,\c_j\, h) (1_r) =   \sum_{p = 0}^{r} 
\sum_{\si \in \Sh_{p, r-p}} 
\mu_{\bt_{\si,i}} \big( f(1_p) \otimes   (g \,\c_j\, h) (1_{r-p})  \big) 
$$ 
$$
=\sum_{p_1+ p_2+ p_3 =r}\,
\sum_{\si \in \Sh_{p_1, p_2+ p_3} }\,
\sum_{\si' \in  \Sh_{p_2, p_3} }
\mu_{\bt_{\si,i}} \circ
(1 \otimes \mu_{\bt_{\si',j}} )
\big( f(1_{p_1}) \otimes   g(1_{p_2}) \otimes h(1_{p_3}) \big) 
$$
$$
= \sum_{p_1+ p_2+ p_3 =r}\,
\sum_{\tau \in \Sh_{p_1, p_2, p_3} }
\mu_{\bt_{\tau,i,j}} \big( f(1_{p_1}) \otimes   g(1_{p_2}) \otimes h(1_{p_3}) \big)\,, 
$$
where the tree $\bt_{\tau,i,j}$ is depicted on figure \ref{fig:bt-tau-ij}.  
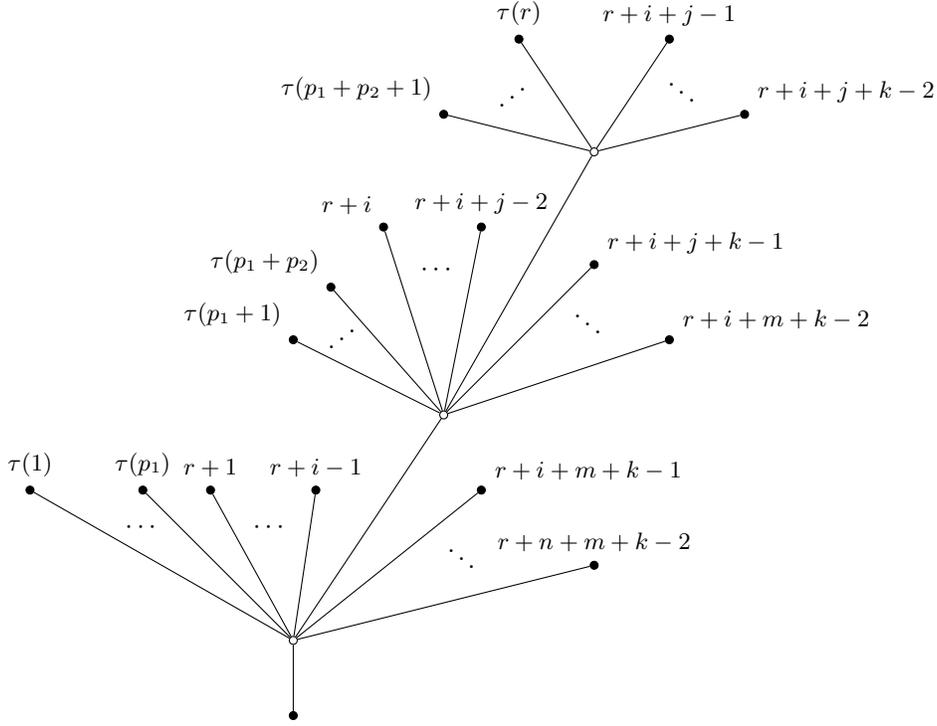
\begin{figure}[htp]
\centering
\makeatletter
\def\Ddots{\mathinner{\mkern1mu\raise\p@
\vbox{\kern7\p@\hbox{.}}\mkern2mu
\raise4\p@\hbox{.}\mkern2mu\raise7\p@\hbox{.}\mkern1mu}}
\makeatother
\begin{tikzpicture}[
xst/.style={draw, cross out,  minimum size=3, inner sep=1 }, arr/.style={-triangle 60},]
\tikzstyle{w}=[circle, draw, minimum size=3, inner sep=1]
\tikzstyle{b}=[circle, draw, fill, minimum size=3, inner sep=1]
\node[b] (v1) at (4,0) {};
\node[w] (v2) at (4,1) {};
\node[b, label=90:{\small $\tau(1)$}] (v3) at (0.5,3) {};
\node[b, label=90:{\small $\tau(p_1)$}] (v4) at (2,3) {};
\node[b, label=90:{\small $r+1$}] (v5) at (2.9,3) {};
\node[b, label=90:{\small $r+i-1$}] (v6) at (4.3,3) {};
\node[w] (v7) at (6,4) {};
\node[b, label=20:{\small $r + i + m  + k -1$}] (v8) at (6.5,3) {};
\node[b, label=90:{\small $r + n + m + k -2$}] (v9) at (8,2) {};
\node[b, label=120:{\small $\tau(p_1+1)$}] (v10) at (4,5) {};
\node[b, label=120:{\small $\tau(p_1+p_2)$}] (v11) at (4.5,5.7) {};
\node[b, label=120:{\small $r+i$}] (v12) at (5.2,6.5) {};
\node[b, label=90:{\small $r + i +j -2$}] (v13) at (6.5,6.5) {};
\node[w] (v16) at (8,7.5) {};
\node[b, label=30:{\small $r + i + j + k-1 $}] (v14) at (8,6) {};
\node[b, label=30:{\small $r + i + m + k - 2$}] (v15) at (9,5) {};
\node[b, label=120:{\small $\tau(p_1+p_2+1)$}] (v17) at (6,8) {};
\node[b, label=90:{\small $\tau(r)$}] (v18) at (7,9) {};
\node[b, label=90:{\small $r + i + j -1$}] (v19) at (9,9) {};
\node[b, label=30:{\small $r + i + j + k -2$}] (v20) at (10,8) {};

\node at (4.6451,5.014) {$\Ddots$};
\node at (6.9122,8.2302) {$\Ddots$};
\node at (7.914,5.3128) {$\ddots$};
\node at (9.1618,8.4059) {$\ddots$};
\node at (6.235,2.1974) {$\ddots$};
\node at (5.928,5.9279) {$\cdots$};
\node at (3.7,2.5) {$\cdots$};
\node at (2,2.5) {$\cdots$};

\draw (v1) edge (v2);
\draw (v2) edge (v3);
\draw (v2) edge (v4);
\draw (v2) edge (v5);
\draw (v2) edge (v6);
\draw (v2) edge (v7);
\draw (v2) edge (v8);
\draw (v2) edge (v9);
\draw (v7) edge (v10);
\draw (v7) edge (v11);
\draw (v7) edge (v12);
\draw (v7) edge (v13);
\draw (v7) edge (v14);
\draw (v7) edge (v15);
\draw (v7) edge (v16);
\draw (v16) edge (v17);
\draw (v16) edge (v18);
\draw (v16) edge (v19);
\draw (v16) edge (v20);
\end{tikzpicture}
\caption{\label{fig:bt-tau-ij} Here $\tau$ is a $(p_1, p_2, p_3)$-shuffle and 
$r= p_1 + p_2 + p_3$}
\end{figure} 
Similar calculations show that
$$
(f\,\c_i \, g ) \,\c_{j+i-1}\, h  =  \sum_{p_1+ p_2+ p_3 =r}\,
\sum_{\tau \in \Sh_{p_1, p_2, p_3} }
\mu_{\bt_{\tau,i,j}} \big( f(1_{p_1}) \otimes   g(1_{p_2}) \otimes h(1_{p_3}) \big)\,, 
$$
with $\bt_{\tau,i,j}$ being the tree depicted on figure \ref{fig:bt-tau-ij}.

We leave the verification of the remaining axioms of the operad 
structure for the reader.

Our next goal is to define an auxiliary action of 
$\cL_{\cO}$ on the operad $\wt{\Tw}\cO$\,. 
For a vector $f \in \wt{\Tw}\cO(n)$ the 
action of $v \in \cL_{\cO}$ (\ref{cL-cO}) on $f$ 
is defined by the formula
\begin{equation}
\label{eq:cL-action1}
v \cdot f(1_r) = -(-1)^{|v||f|} \sum_{p=1}^r \, 
\sum_{\si \in \Sh_{p,r-p}} 
\mu_{\bt'_{\si, p, r-p}} ( f(1_{r-p+1}) \otimes v(1^{\mc}_p) )\,,
\end{equation}
where $1^{\mc}_p$ is the generator $\bs^{2-2p}\, 1 \in  \La^2 \coCom(p) \cong \bs^{2-2p} \, \bbK$
and the tree $\bt'_{\si, p, r-p}$ is depicted on figure 
\ref{fig:bt-si-p-r}.
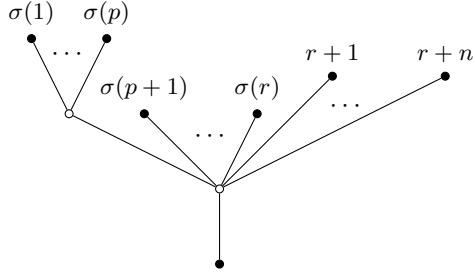
\begin{figure}[htp]
\centering
\begin{tikzpicture}[
xst/.style={draw, cross out,  minimum size=3, inner sep=1 }, 
arr/.style={-triangle 60},]
\tikzstyle{w}=[circle, draw, minimum size=3, inner sep=1]
\tikzstyle{b}=[circle, draw, fill, minimum size=3, inner sep=1]
\node[b] (v1) at (3,0) {};
\node[w] (v2) at (3,1) {};
\node[w] (v3) at (1,2) {};
\node[b, label=90:{\small $\si(p+1)$}] (v4) at (2,2) {};
\node[b, label=90:{\small $\si(r)$}] (v5) at (3.5,2) {};
\node[b, label=90:{\small $r+1$}] (v8) at (4.5,2.5) {};
\node[b, label=90:{\small $r+n$}] (v9) at (6,2.5) {};
\node[b, label=90:{\small $\si(1)$}] (v6) at (0.5,3) {};
\node[b, label=90:{\small $\si(p)$}] (v7) at (1.5,3) {};
\draw (v1) edge (v2);
\draw (v2) edge (v3);
\draw (v2) edge (v4);
\draw (v2) edge (v5);
\draw (v3) edge (v6);
\draw (v3) edge (v7);
\draw (v2) edge (v8);
\draw (v2) edge (v9);
\node at (4.7,2.1) {$\cdots$};
\node at (2.9,1.7) {$\cdots$};
\node at (1,2.8) {$\cdots$};
\end{tikzpicture}
\caption{\label{fig:bt-si-p-r} Here $\si$ is a $(p, r-p)$-shuffle}
\end{figure} 

We claim that 
\begin{prop}
\label{prop:cL-action1}
Formula \eqref{eq:cL-action1} defines an action of 
$\cL_{\cO}$ (\ref{cL-cO}) on the operad $\wt{\Tw}\cO$\,. 
\end{prop}  
\begin{proof}
A simple degree bookkeeping shows that the degree 
of $v\cdot f$ is $|v|+ |f|$\,.

Then we need to check that for two homogeneous vectors  
$v,w\in \cL_{\cO}$ we have  
\begin{equation}
\label{need-action}
[v,w] \cdot f (1_r)  = (v\cdot  (w \cdot f)) (1_r) - 
(-1)^{|v||w|} (w \cdot (v \cdot f))  (1_r) 
\end{equation}

Using the definition of the operation $\cdot$ and 
the associativity axiom for the operad structure on $\cO$
we get
\begin{equation}
\label{rhs-equals}
 (v\cdot  (w \cdot f)) (1_r) - 
(-1)^{|v||w|} (w \cdot (v \cdot f))  (1_r) =  
\end{equation}
$$
(-1)^{|f| (|v|+|w|) + |v||w|}
\sum_{p \ge 1 \, q \ge 0} ~ 
\sum_{\tau \in \Sh_{p, q, r- p- q}} ~
\mu_{\bt^{p, q}_{\tau}} (f(1_{r-p-q+1}) \otimes w(1^{\mc}_{q+1}) 
\otimes v(1^{\mc}_{p}))
$$
$$
+(-1)^{|f| (|v|+|w|) + |v||w|}
\sum_{p, q \ge 1} ~ 
\sum_{\tau \in \Sh_{p, q, r- p- q}} ~
\mu_{\wt{\bt}^{p, q}_{\tau}} (f(1_{r-p-q+2}) \otimes w(1^{\mc}_{q}) 
\otimes v(1^{\mc}_{p})) 
$$
$$
-(-1)^{|v| |w|} (v \leftrightarrow w)\,,
$$ 
where the trees $\bt^{p, q}_{\tau}$ and $\wt{\bt}^{p, q}_{\tau}$ 
are depicted on figures \ref{fig:bt-tau-p-q}  and \ref{fig:wtbt-tau-p-q} , respectively.
\begin{figure}[htp]
\centering 
\begin{tikzpicture}[arr/.style={-triangle 60}]
\tikzstyle{w}=[circle, draw, minimum size=3, inner sep=1]
\tikzstyle{b}=[circle, draw, fill, minimum size=3, inner sep=1]
\node[b] (v1) at (4,0) {};
\node[w] (v2) at (4,1) {};
\node[w] (v3) at (2,3) {};
\node[b, label=90:{\small $\tau(p+q+1)$}] (v4) at (3.5,3) {};
\node[b, label=90:{\small $\tau(r)$}] (v5) at (5.5,3) {};
\node[b, label=90:{\small $r+1$}] (v6) at (7,4) {};
\node[b, label=90:{\small $r+n$}] (v7) at (9,4) {};

\node[w] (v8) at (1,4.5) {};
\node[b, label=90:{\small $\tau(p+1)$}] (v11) at (2,4.5) {};
\node[b, label=90:{\small $\tau(p+q)$}] (v12) at (3.5,4.5) {};
\node[b, label=90:{\small $\tau(1)$}] (v9) at (0.3,5.5) {};
\node[b, label=90:{\small $\tau(p)$}] (v10) at (1.7,5.5) {};

\node at (7.5,3.6) {$\cdots$};
\node at (4.4,2.7) {$\cdots$};
\node at (2.6,4.2) {$\cdots$};
\node at (1.1,5.3) {$\cdots$};

\draw (v1) edge (v2);
\draw (v2) edge (v3);
\draw (v2) edge (v4);
\draw (v2) edge (v5);
\draw (v2) edge (v6);
\draw (v2) edge (v7);
\draw (v8) edge (v9);
\draw (v8) edge (v10);
\draw (v3) edge (v8);
\draw (v3) edge (v11);
\draw (v3) edge (v12);
\end{tikzpicture}
\caption{\label{fig:bt-tau-p-q} The tree  $\bt^{p, q}_{\tau}$}
\end{figure} 
\begin{figure}[htp]
\centering 
\begin{tikzpicture}[
xst/.style={draw, cross out, minimum size=5, }, 
arr/.style={-triangle 60}]
\tikzstyle{w}=[circle, draw, minimum size=3, inner sep=1]
\tikzstyle{b}=[circle, draw, fill, minimum size=3, inner sep=1]
\node[b] (v1) at (4,0) {};
\node[w] (v2) at (4,1) {};
\node[w] (v3) at (1.5,2) {};
\node[w] (v6) at (3,3) {};
\node[b , label=90:{\small $\tau(p+q+1)$}] (v9) at (5,3) {};
\node[b , label=90:{\small $\tau(r)$}] (v10) at (6.5,3) {};
\node[b , label=90:{\small $r+1$}] (v11) at (9,4) {};
\node[b, label=90:{\small $r+n$}] (v12) at (11,4) {};
\node[b , label=90:{\small $\tau(1)$}] (v4) at (0.5,3) {};
\node[b , label=90:{\small $\tau(p)$}] (v5) at (2,3) {};
\node[b , label=90:{\small $\tau(p+1)$}] (v7) at (2.2,4) {};
\node[b , label=90:{\small $\tau(p+q)$}] (v8) at (3.8,4) {};
\node at (9.2,3.6) {$\cdots$};
\node at (5.4,2.6) {$\cdots$};
\node at (3.1,3.8) {$\cdots$};
\node at (1.4,2.8) {$\cdots$};
\draw (v1) edge (v2);
\draw (v2) edge (v3);
\draw (v3) edge (v4);
\draw (v3) edge (v5);
\draw (v2) edge (v6);
\draw (v6) edge (v7);
\draw (v6) edge (v8);
\draw (v2) edge (v9);
\draw (v2) edge (v10);
\draw (v2) edge (v11);
\draw (v2) edge (v12);
\end{tikzpicture}
\caption{\label{fig:wtbt-tau-p-q} The tree $\wt{\bt}^{p, q}_{\tau}$}
\end{figure}

Since $f(1_{r-p-q+2})$ is invariant with respect to the action 
of $S_{r-p-q+2}$ the sums involving $\mu_{\wt{\bt}^{p, q}_{\tau}}$
cancel each other.  
  
Furthermore, it is not hard to see that the sums involving 
$\mu_{\bt^{p, q}_{\tau}}$ form the expression 
$$
[v,w] \cdot f (1_r)\,.
$$
Thus equation \eqref{need-action} follows. 

Due to Exercise \ref{exer:oper-deriv} below, the operation 
$f \mapsto v \cdot f$ is an operadic derivation. 
Proposition \ref{prop:cL-action1} is proved. 
\end{proof}
\begin{exer}
\label{exer:oper-deriv}
Prove that for every triple of homogeneous vectors 
$f \in \wt{\Tw}\cO(n)$,  $g \in \wt{\Tw}\cO(k)$, and 
$v \in \cL_{\cO}$ we have 
$$
v \cdot (f \circ_i g) = (v \cdot f) \circ_i g + 
(-1)^{|v||f|} f \circ_i (v \cdot g) \qquad 
\forall~~  1 \le i \le n\,.
$$
\end{exer}

\subsection{The action of $\cL_{\cO}$ on $\wt{\Tw}\cO$}   
   
Let us view   $\wt{\Tw}\cO(1)$ as the dg Lie algebra with 
the bracket being commutator. 
 
We have an obvious degree zero map 
$$
\ka : \cL_{\cO} \to \wt{\Tw}\cO(1)
$$
defined by the formula
\begin{equation}
\label{kappa}
\ka(v)(1_{r}) = v(1^{\mc}_{r+1})\,,
\end{equation}
where, as above, $1_r$ is the generator $\bs^{-2r}\, 1 \in \bs^{-2r}\, \bbK$ 
and $1^{\mc}_r$ is the generator $\bs^{2-2r}\, 1 \in \La^2 \coCom(r) \cong \bs^{2-2r}\, \bbK$\,.
 
We have the following proposition. 
\begin{prop}
\label{prop:Theta}
Let us form the semi-direct product $\cL_{\cO} \ltimes \wt{\Tw}\cO(1)$
of the dg Lie algebras $\cL_{\cO}$ and $\wt{\Tw}\cO(1)$ 
using the action of $\cL_{\cO}$ on $\wt{\Tw}\cO$ defined 
in Proposition \ref{prop:cL-action1}. Then the formula 
\begin{equation}
\label{Theta}
\Te(v) = v + \ka(v)
\end{equation}
defines a Lie algebra homomorphism 
$$
\Te: \cL_{\cO} \to  \cL_{\cO} \ltimes \wt{\Tw}\cO(1)\,.
$$
\end{prop}
\begin{proof}
First, let us prove that 
for every pair of homogeneous vectors 
$v, w \in \cL_{\cO}$ we have 
\begin{equation}
\label{ka-property}
\ka([v,w])
= [\ka(v), \ka(w)] + v \cdot \ka(w) 
- (-1)^{|v||w|} w \cdot \ka(v)\,.
\end{equation}

Indeed, unfolding the definition of $\ka$ we 
get\footnote{Here we use the notation \eqref{new-notation} introduced in Subsection \ref{sec:ps-op}.}  
\begin{equation}
\label{ka-bracket} 
\ka([v,w])(1_r) =   \sum_{p = 1}^{r}
\sum_{\tau \in \Sh_{p,r-p}}  v_{r-p+2} \big( w_p(\tau(1), \dots,  \tau(p)), \tau(p+1), 
\dots, \tau(r), r+1 \big) 
\end{equation}
$$
 + \sum_{p = 0}^{r}
\sum_{\tau \in \Sh_{p,r-p}}  v_{r-p+1} \big( w_{p+1}(\tau(1), \dots,  \tau(p), r+1), \tau(p+1), 
\dots, \tau(r) \big) 
$$
$$
- (-1)^{|v||w|} (v \leftrightarrow w)\,,
$$
where $v_{t} = v(1^{\mc}_{t})$ and $w_t = w(1^{\mc}_t)$\,. 

The first sum in \eqref{ka-bracket} equals 
$$
- (-)^{|v||w|}( w \cdot \ka(v)) (1_r)\,.
$$
Furthermore, since $v(1^{\mc}_{t})$ is invariant under 
the action of $S_t$, we see that the second sum 
in \eqref{ka-bracket} equals
$$
\big(\ka(v) \circ_1  \ka(w) \big)\, (1_r)\,.
$$
Thus equation \eqref{ka-property} holds.

Now, using  \eqref{ka-property}, it is easy to see that 
$$
[ v + \ka(v),  w + \ka(w)] = [v,w] + v \cdot \ka(w) - 
(-1)^{|v||w|} w \cdot \ka(v) + [\ka(v), \ka(w)]= 
$$
$$
= [v,w] + \ka([v,w])
$$
and the statement of proposition follows.
\end{proof}

The following corollaries are immediate consequences 
of Proposition  \ref{prop:Theta}
\begin{cor}
\label{cor:the-action}
For  $v \in \cL_{\cO}$ and $f \in \wt{\Tw}\cO(n)$ 
the formula
\begin{equation}
\label{the-action}
f \to v \cdot f + \de_{\ka(v)} (f)   
\end{equation}
defines an action of  the Lie algebra $\cL_{\cO}$
on the operad $\wt{\Tw}\cO$\,.
\end{cor}
\begin{cor}
\label{cor:MC-MC}
For every Maurer-Cartan element $\vf \in \cL_{\cO}$, the sum
$$
\vf + \ka(\vf)
$$
is a Maurer-Cartan element of the Lie algebra $ \cL_{\cO} \ltimes \wt{\Tw}\cO(1)$\,.
\end{cor}

We finally give the definition of the operad $\Tw \cO$. 
\begin{defi}
\label{dfn:Tw-cO}
Let $\cO$ be an operad in $\Ch_{\bbK}$ and 
$\vf$ be a Maurer-Cartan element in $\cL_{\cO}$ \eqref{cL-cO} corresponding 
to an operad morphism  $\wh{\vf}$ \eqref{from-hoLie}. 
Let us also denote by $\pa^{\cO}$ the differential on 
$\wt{\Tw}\cO$ coming from the one on $\cO$\,. 
We define the operad $\Tw \cO$ in $\Ch_{\bbK}$ 
by declaring that $\Tw\cO = \wt{\Tw}\cO$ as operads 
in $\grVect_{\bbK}$ and letting 
\begin{equation}
\label{pa-Tw}
\pa^{\Tw} = \pa^{\cO} + \vf \cdot  \,+ \, \de_{\ka(\vf)}
\end{equation}
be the differential on $\Tw\cO$\,.
\end{defi}
Corollaries \ref{cor:the-action} and \ref{cor:MC-MC}
imply that  $\pa^{\Tw}$ is indeed a differential on $\Tw\cO$\,.

\begin{rem}
\label{rem:Tw-cO-0}
It is easy to see that, if $\cO(0)= \bfzero$ then 
the cochain complexes $\bs^{-2}\Tw\cO(0)$ and 
$\cL_{\cO}$ \eqref{cL-cO} are tautologically 
isomorphic. 
\end{rem}

\subsection{Algebras over $\Tw \cO$}
Let us assume that  $V$ is an algebra over $\cO$ equipped with 
a complete descending filtration  (\ref{filtr-V}). We also
assume that the $\cO$-algebra structure on $V$ is 
compatible with this filtration. 

Given a Maurer-Cartan element $\al\in F_1 V$ the formula
\begin{equation}
\label{tw-diff}
\pa^{\al} (v) = \pa(v) + \sum_{r=1}^{\infty} \frac{1}{r!} 
\vf(1^{\mc}_{r+1})(\,\underbrace{\al, \dots, \al}_{r\textrm{ times}}, v)
\end{equation}
defines a new (twisted) differential on $V$\,.

We will denote by $V^{\al}$ the cochain complex 
$V$ with this new differential. 

In this setting we have the following theorem
\begin{thm}
\label{thm:twisting}
If  $V^{\al}$ is the cochain complex obtained 
from $V$ via twisting the differential by $\al$
then the formula 
\begin{equation}
\label{twisting}
f(v_1, \dots, v_n) = 
\sum_{r=0}^{\infty} \frac{1}{r!} f(1_r) (\al, \dots, \al, v_1, \dots, v_n)
\end{equation}
$$
f\in \Tw\cO(n)\,,  \qquad v_i \in V\,, \qquad 1_r = \bs^{-2r}\, 1 \in  \bs^{-2r}\, \bbK
$$
defines a $\Tw\cO$-algebra structure on $V^{\al}$\,. 
\end{thm}     
\begin{proof}
Let $f\in \Tw\cO(n)$, $g \in \Tw\cO(k)$, 
$$
f_r := f(1_r) \in \big(\cO(r+n)\big)^{S_r}\,, 
\qquad \textrm{and} \qquad g_r = g(1_r) \in  \big(\cO(r+k)\big)^{S_r}\,. 
$$  

Our first goal is to verify that 
\begin{equation}
\label{operad-ok}
(-1)^{|g| (|v_{i}| + \dots + |v_{i-1}|) }
f(v_1, \dots, v_{i-1}, g(v_i, \dots, v_{i+k-1}), v_{i+k}, \dots, v_{n+k-1}) 
\end{equation}
$$
= f\, \c_i \, g (v_1, \dots, v_{n+k-1})\,. 
$$

The left hand side of \eqref{operad-ok} can be rewritten 
as 
$$
(-1)^{|g| (|v_{i}| + \dots + |v_{i-1}|) }
f(v_1, \dots, v_{i-1}, g(v_i, \dots, v_{i+k-1}), v_{i+k}, \dots, v_{n+k-1})=  
$$
$$
\sum_{p,q \ge 0} \frac{(-1)^{|g| (|v_{i}| + \dots + |v_{i-1}|) } }{p! q!}
f_p(\al, \dots, \al, v_1, \dots, v_{i-1}, g_{q}(\al, \dots, \al, v_i, \dots, v_{i+k-1}), 
 v_{i+k}, \dots, v_{n+k-1})\,.
$$

Using the obvious combinatorial identity
\begin{equation}
\label{number-shuffles}
|\Sh_{p,q}| = \frac{(p+q)!}{p! q!}
\end{equation}
we rewrite the  left hand side of \eqref{operad-ok}  further
$$
\textrm{L.H.S. of \eqref{operad-ok}} = 
$$
$$
\sum_{p, q \ge 0} \frac{ (-1)^{|g| (|v_{i}| + \dots + |v_{i-1}|) }  }{(p+q)!}
|\Sh_{p,q}| f_p(\al, \dots, \al, v_1, \dots, v_{i-1}, g_{q}(\al, \dots, \al, v_i, \dots, v_{i+k-1}), 
 v_{i+k}, \dots, v_{n+k-1}) = 
$$
$$
\sum_{r=0}^{\infty}
\frac{1}{r!} 
\sum_{p= 0}^r
\sum_{\si \in \Sh_{p, r-p}}
\si \circ \vr_{r,p,i} (f_p \circ_{p+i} g_{r-p})  
(\underbrace{\al, \dots, \al}_{r~\textrm{arguments}}, v_1, \dots, v_{n+k-1}),
$$
where $\vr_{r,p,i}$ is the following permutation in $S_r$ 
\begin{equation}
\label{vr-r-p-i}
 \vr_{r,p,i} = 
\left(
\begin{array}{cccccc}
p+1  & \dots  & p+i-1 & p+i & \dots & r+i-1    \\
r+1  &  \dots  & r+i-1 & p+1 & \dots & r  
\end{array}
\right)\,.
 \end{equation}

Thus 
$$
\textrm{L.H.S. of \eqref{operad-ok}} = f\, \c_i \, g (v_1, \dots, v_{n+k-1})
$$
and equation  \eqref{operad-ok} holds.

Next, we need to show that 
\begin{equation}
\label{diff-ok}
\pa^{\Tw}(f)(v_1, \dots, v_n)  = \pa^{\al} f (v_1, \dots, v_n) 
\end{equation}
$$
- (-1)^{|f|} \sum_{i=1}^n (-1)^{|v_{i}| + \dots + |v_{i-1}|}
 f (v_1, \dots, v_{i-1}, \pa^{\al}(v_i),  v_{i+1}, \dots, v_n)
$$

The right hand side of  \eqref{diff-ok} can be rewritten as 
$$
\textrm{R.H.S. of \eqref{diff-ok}} =
$$
$$
\sum_{p\ge 0}\frac{1}{p!} \pa f_p(\al, \dots, \al, v_1, \dots, v_n) + 
\sum_{p\ge 0, q\ge 1}\frac{1}{p! q!} \vf_q(\al, \dots, \al, f_p(\al, \dots, \al, v_1, \dots, v_n))
$$
$$
- (-1)^{|f|} \sum_{i=1}^n \sum_{p\ge 0}
 \frac{(-1)^{|v_{i}| + \dots + |v_{i-1}|}}{p!}
 f_p (\al, \dots, \al, v_1, \dots, v_{i-1}, \pa(v_i),  v_{i+1}, \dots, v_n)
$$
$$
- (-1)^{|f|} \sum_{i=1}^n \sum_{p\ge 0, q\ge 1}
 \frac{(-1)^{|v_{i}| + \dots + |v_{i-1}|}}{p! q!}
 f_p (\al, \dots, \al, v_1, \dots, v_{i-1}, \vf_q(\al, \dots, \al, v_i),  v_{i+1}, \dots, v_n)
$$
where $f_p = f(1_p)$ and $\vf_q = \vf(1^{\mc}_q)$\,.

Let us now add to and subtract from the
right hand side of  \eqref{diff-ok} the sum 
$$
- (-1)^{|f|}\sum_{p\ge 0}\frac{1}{p!} f_{p+1}(\pa \al, \al, \dots, \al, v_1, \dots, v_n)\,.
$$
We get
$$
\textrm{R.H.S. of \eqref{diff-ok}} =
$$
$$
\sum_{p\ge 0}\frac{1}{p!} \pa f_p(\al, \dots, \al, v_1, \dots, v_n) 
- (-1)^{|f|}\sum_{p\ge 0}\frac{1}{p!} f_{p+1}(\pa \al, \al, \dots, \al, v_1, \dots, v_n)
$$
$$
- (-1)^{|f|} \sum_{i=1}^n \sum_{p\ge 0}
 \frac{(-1)^{|v_{i}| + \dots + |v_{i-1}|}}{p!}
 f_p (\al, \dots, \al, v_1, \dots, v_{i-1}, \pa(v_i),  v_{i+1}, \dots, v_n)
$$
$$
+ (-1)^{|f|} \sum_{p\ge 0}\frac{1}{p!} f_{p+1}(\pa \al, \al, \dots, \al, v_1, \dots, v_n)
$$
$$
+\sum_{p\ge 0, q\ge 1}\frac{1}{p! q!} \vf_q(\al, \dots, \al, f_p(\al, \dots, \al, v_1, \dots, v_n))
$$
$$
- (-1)^{|f|} \sum_{i=1}^n \sum_{p\ge 0, q\ge 1}
 \frac{(-1)^{|v_{i}| + \dots + |v_{i-1}|}}{p! q!}
 f_p (\al, \dots, \al, v_1, \dots, v_{i-1}, \vf_q(\al, \dots, \al, v_i),  v_{i+1}, \dots, v_n) = 
$$
$$
(\pa^{\cO} f) (v_1, \dots, v_n) 
$$
$$
+ (-1)^{|f|} \sum_{p\ge 0}\frac{1}{p!} f_{p+1}(\pa \al, \al, \dots, \al, v_1, \dots, v_n)
$$
$$
+\sum_{p\ge 0, q\ge 1}\frac{1}{p! q!} \vf_q(\al, \dots, \al, f_p(\al, \dots, \al, v_1, \dots, v_n))
$$
$$
- (-1)^{|f|} \sum_{i=1}^n \sum_{p\ge 0, q\ge 1}
 \frac{(-1)^{|v_{i}| + \dots + |v_{i-1}|}}{p! q!}
 f_p (\al, \dots, \al, v_1, \dots, v_{i-1}, \vf_q(\al, \dots, \al, v_i),  v_{i+1}, \dots, v_n)\,. 
$$

Due to the Maurer-Cartan equation for $\al$
$$
\pa (\al) + \frac{1}{q!} \vf_q(\al, \al, \dots, \al) = 0
$$
we have 
$$
+ (-1)^{|f|} \sum_{p\ge 0}\frac{1}{p!} f_{p+1}(\pa \al, \al, \dots, \al, v_1, \dots, v_n)=
$$
$$
-(-1)^{|f|} \sum_{p\ge 0, q \ge 2}\frac{1}{p! q!}  
f_{p+1}( \vf_q(\al, \dots, \al), \al, \dots, \al, v_1, \dots, v_n)\,.
$$ 
Hence, using \eqref{number-shuffles} we get 
$$
\textrm{R.H.S. of \eqref{diff-ok}} =
$$
$$
(\pa^{\cO} f) (v_1, \dots, v_n) 
+
(\vf \cdot f )(v_1, \dots, v_n) 
$$
$$
+ \ka(\vf) \,\c_1\, f  (v_1, \dots, v_n)
- (-1)^{|f|} f \, \c_1 \, \ka(\vf) (v_1, \dots, v_n)\,.
$$

Theorem \ref{thm:twisting} is proved.  

\end{proof} 

Let us now observe that the dg operad $\Tw\cO$ is equipped with 
complete descending filtration. Namely,
\begin{equation}
\label{Tw-cO-filtr}
\cF_k \Tw\cO(n) = \{f \in  \Tw\cO(n) ~|~ f(1_r) = 0 \quad \forall ~ r < k \}
\end{equation}
It is clear that the operad structure on $\Tw\cO$ is compatible with 
this filtration.

The endomorphism  operad $\End_V$ also carries 
a complete descending filtration since so does $V$\,.

For this reason it makes sense to give this definition:
\begin{defi}
\label{dfn:}
A \und{filtered  $\Tw\cO$-algebra} is a cochain complex 
$V$ equipped with a complete descending filtration for 
which the operad map 
$$
\Tw\cO \to \End_V
$$
is compatible with the filtrations. 
\end{defi} 
It is easy to see that the $\Tw\cO$-algebra $V^{\al}$ 
from Theorem  \ref{thm:twisting} 
is a filtered $\Tw\cO$-algebra in the sense of this definition. 

Thus Theorem   \ref{thm:twisting} gives us a functor to
the category of filtered  $\Tw\cO$-algebras 
from the category of pairs 
$$
(V, \al)\,,
$$ 
where $V$ is a filtered 
cochain complex equipped with an 
action of the operad $\cO$ and $\al$ is a Maurer-Cartan element in $\cF_1 V$\,.

According to \cite{DeligneTw} this functor establishes an equivalence of 
categories. 

\subsection{A useful modification $\Tw^{\oplus}\cO$}
\label{sec:Tw-oplus}
In practice the morphism \eqref{from-hoLie} often
comes from the map (of dg operads)
$$
\mj : \La\Lie \to \cO\,.
$$

In this case, the above construction of twisting 
is well defined for the suboperad $\Tw^{\oplus}(\cO)\subset \Tw\cO$ with 
\begin{equation}
\label{Tw-oplus}
\Tw^{\oplus}(\cO)(n) = \bigoplus_{r \ge 0} \bs^{2r} \big( \cO(r+n) \big)^{S_r}\,. 
\end{equation}

It is not hard to see that the Maurer-Cartan element 
$$
\vf \in \Conv(\La^2 \coCom, \cO)
$$
corresponding to the composition 
$$
\mj \circ U_{\La\Lie} : \Cobar(\La^2 \coCom) \to \cO
$$
is given by the formula:
\begin{equation}
\label{al-coCom-cO}
\vf  = \mj (\{a_1, a_2\}) \otimes b_1 b_2\,.
\end{equation}

Hence
\begin{equation}
\label{cL-oplus}
\cL^{\oplus}_{\cO} =  \bigoplus_{r \ge 0} \bs^{2r-2} \big( \cO(r) \big)^{S_r}
\end{equation}
is a sub- dg Lie algebra of $\cL_{\cO}$ \eqref{cL-cO}\,.

Specifying general formula \eqref{pa-Tw} to this particular 
case, we see that  the differential $\pa^{\Tw}$ on \eqref{Tw-oplus}
is given by the equation:
\begin{equation}
\label{diff-Tw-cO}
\pa^{\Tw} (v) = - (-1)^{|v|} \sum_{\si \in \Sh_{2, r-1}}
\si \big(v \circ_1 \mj (\{a_1,a_2\}) \big) + 
\sum_{\tau \in \Sh_{1, r}} \tau \big( \mj(\{a_1, a_2\}) \circ_2  v \big)
\end{equation}
$$
- (-1)^{|v|}  \sum_{\tau' \in \Sh_{r, 1}}  \sum_{i=1}^{n}  
\tau'  \circ \vs_{r+1,r+i} \big( v \circ_{r+i} \mj(\{a_1,a_2\}) \big)\,,
$$
where 
$$
v \in  \bs^{2r} \big( \cO(r+n) \big)^{S_r}\,,
$$
and $\vs_{r+1,r+i}$ is the cycle $(r+1, r+2, \dots, r+i)$\,.

\begin{rem}
\label{rem:diff-Tw}
We should remark that, when we apply elementary insertions
in the right hand side of \eqref{diff-Tw-cO}, we view
$v$ and $\mj(\{a_1, a_2\})$ as vectors in $\cO(r+n)$ and 
$\cO(2)$ respectively. The resulting sum in the  right hand 
side of \eqref{diff-Tw-cO} is viewed as a vector in $\Tw\cO(n)$\,.  
\end{rem}

\subsection{Example: The operad $\Tw \Ger$}
\label{sec:TwGer}

Let $\Ger$ be the operad which governs Gerstenhaber algebras 
(see Subsection \ref{sec:Ger}). Since $\La \Lie$ 
receives a quasi-isomorphism \eqref{LaLie-infty-LaLie} from 
$\La\Lie_{\infty}$ and embeds into $\Ger$,  
we have a canonical map 
\begin{equation}
\label{LaLieinfty-Ger}
\La\Lie_{\infty} \to  \Ger\,.
\end{equation}
This section is devoted to the dg operad $\Tw\Ger$ which 
is associated to the operad $\Ger$ and the map  \eqref{LaLieinfty-Ger}.

According to the general procedure of twisting 
\begin{equation}
\label{cL-Ger}
\cL_{\Ger} = \Conv (\La^2 \coCom, \Ger)  = 
\prod_{r =1}^{\infty} \bs^{2r-2} \big( \Ger(r)  \big)^{S_r}
\end{equation}
and the Maurer-Cartan element $\al \in \cL_{\Ger}$  corresponding to the 
map \eqref{LaLieinfty-Ger} equals
\begin{equation}
\label{MC-cL-Ger}
\al =  \{a_1, a_2\}\,.
\end{equation}
The graded vector space $\Tw\Ger(n)$ is the product
\begin{equation}
\label{TwGer-n}
\Tw\Ger(n) = \prod_{r \ge 0} \bs^{2r} \big( \Ger(r+n) \big)^{S_r}\,.
\end{equation}
Furthermore, adapting \eqref{diff-Tw-cO} to this case we get
\begin{equation}
\label{diff-Tw-Ger}
\pa^{\Tw} (v) = - (-1)^{|v|} \sum_{\si \in \Sh_{2, r-1}}
\si \big(v \circ_1 \{a_1,a_2\} \big) + 
\sum_{\tau \in \Sh_{1, r}} \tau \big( \{a_1, a_2\} \circ_2  v \big)
\end{equation}
$$
- (-1)^{|v|}  \sum_{\tau' \in \Sh_{r, 1}}  \sum_{i=1}^{n}  
\tau'  \circ \vs_{r+1,r+i} \big( v \circ_{r+i} \{a_1,a_2\} \big)\,,
$$
where 
$$
v \in  \bs^{2 r} \big( \Ger(r+n) \big)^{S_r}\,,
$$
and $\vs_{r+1,r+i}$ is the cycle $(r+1, r+2, \dots, r+i)$\,.

\begin{exer}
\label{exer:TwGer-sum}
Prove that for every $v \in   \bs^{2r} \Ger(r+n)$
\begin{equation}
\label{ineq:TwGer}
r \le |v| + n  -1\,.
\end{equation}
Similarly, prove that, for every vector  $v \in   \bs^{2r-2} \Ger(r)$
\begin{equation}
\label{ineq:cL-Ger}
r \le |v| + 1\,.
\end{equation}
\end{exer}

Inequalities \eqref{ineq:TwGer} and \eqref{ineq:cL-Ger} imply that 
\begin{equation}
\label{TwGer-n-sum}
\Tw\Ger(n) = \bigoplus_{r = 0}^{\infty} \bs^{2r} \big( \Ger(r+n) \big)^{S_r}
\end{equation}
and 
\begin{equation}
\label{cL-Ger-sum}
\cL_{\Ger} = \bigoplus_{r =1}^{\infty} \bs^{2r-2} \big( \Ger(r)  \big)^{S_r}\,.
\end{equation} 
In other words, $\Tw^{\oplus}\Ger = \Tw\Ger$ and $\cL^{\oplus}_{\Ger} = \cL_{\Ger}$\,.

To give a simpler description of the cochain complexes $\Tw\Ger(n)$
\eqref{TwGer-n-sum} we consider the free Gerstenhaber 
algebra  
$$
\Ger(a, a_1, \dots, a_n)
$$
in $n$ variables $a_1, \dots, a_n$ of degree zero and 
one additional variable $a$ of degree $2$\,. 

We introduce the following (degree $1$) derivation 
\begin{equation}
\label{pa-Ger-a}
\de (a) = \frac{1}{2} \{a,a\}\,, \qquad 
\de (a_i) = 0 \qquad \forall~~ 1 \le i \le n 
\end{equation}
of $\Ger(a, a_1, \dots, a_n)$ and observe that 
$$
\de^2 = 0
$$ 
in virtue of the Jacobi identity. 

Then we denote by $\cG_{n}$ the subspace
\begin{equation}
\label{cG-n}
\cG_{n}  \subset \Ger(a, a_1, \dots, a_n)
\end{equation}
spanned by monomials in which each variable $a_1, a_2, \dots, a_n$ appears 
exactly once. It is obvious that  $\cG_{n}$ is a subcomplex of  $\Ger(a, a_1, \dots, a_n)$\,.

We claim that
\begin{prop}
\label{prop:cG-n}
The cochain complex $\cG_n$ is isomorphic 
to $\Tw\Ger(n)$\,. 
\end{prop}
\begin{proof}
Indeed, given a monomial $v \in \cG_{n}$ of degree $r$
in $a$ we shift the indices of $a_i$ up by $r$ and replace the 
$r$ factors $a$ in $v$ by $a_1, a_2, \dots, a_r$ in an arbitrary order.
This way we get a monomial $v' \in \Ger(r+n)$\,. It is easy to 
see that the formula 
\begin{equation}
\label{cG-n-TwGer}
f(v) = \sum_{\si \in S_r} \bs^{2r} \si (v')
\end{equation}
defines a linear map of vector spaces 
$$
f : \cG_n \to \Tw\Ger(n) = \bigoplus_{r = 0}^{\infty} \bs^{2r} \big( \Ger(r+n) \big)^{S_r}\,.
$$
For example, 
$$
f(\{a,a\} a_1) = \{a_1,a_2\} a_3 +  \{a_2,a_1\} a_3\,, 
\qquad 
 f(\{a,a_1\} a a_2) = \{a_1,a_3\} a_2 a_4 + \{a_2,a_3\} a_1 a_4\,. 
$$

It is not hard to see that $f$ is an isomorphism of graded vector 
spaces. Furthermore,  $f$ is compatible with the differentials due 
to the following exercise.
\begin{exer}
\label{exer:cG-n-TwGer}
Show that the map 
$$
f : \cG_n \to \Tw\Ger(n) = \bigoplus_{r = 0}^{\infty} \bs^{2r} \big( \Ger(r+n) \big)^{S_r}
$$
defined by \eqref{cG-n-TwGer} is compatible with the differentials $\pa^{\Tw}$ and $\de$\,. 
In other words, 
\begin{equation}
\label{f-pa-de}
f(\de v) = \pa^{\Tw} f (v) \qquad \forall ~~ v \in \cG_n\,.
\end{equation}
\end{exer}

Thus the proposition is proved. 

\end{proof}

Proposition \ref{prop:cG-n} implies that every vector 
$v \in \Ger(n)\subset \Tw\Ger(n)$ is $\pa$-closed.
Therefore, the obvious embedding
\begin{equation}
\label{Ger-TwGer}
i : \Ger \to \Tw\Ger
\end{equation}
is a map of dg operads.

We claim that\footnote{This statement is also proved in \cite{DeligneTw}.} 
\begin{thm}
\label{thm:Ger-TwGer}
The map \eqref{Ger-TwGer} is a quasi-isomorphism of dg operads. 
In particular, the dg Lie algebra $\Conv(\La^2 \coCom, \Ger)$ is acyclic. 
\end{thm}
\begin{proof}
Let us observe that $\La\Lie(a,a_1, \dots, a_n)$ is 
a subcomplex of $\Ger(a,a_1, \dots, a_n)$\,. Moreover,  
\begin{equation}
\label{Ger-Sym-a}
\Ger(a,a_1, \dots, a_n) = \und{S}(\La\Lie(a,a_1, \dots, a_n))\,, 
\end{equation}
where $\und{S}$ is the notation for the truncated symmetric 
algebra. 

Let us denote by 
\begin{equation}
\label{LaLie-pr}
\La\Lie'(a,a_1, \dots, a_n)
\end{equation}
the subspace of $\La\Lie(a,a_1, \dots, a_n)$ which 
is spanned by monomials involving each variable 
in the set $\{a_1, a_2, \dots, a_n\}$ at most once. 
It is clear that  $\La\Lie'(a,a_1, \dots, a_n)$ is a subcomplex 
in $\La\Lie(a,a_1, \dots, a_n)$. Hence, the subspace 
\begin{equation}
\label{Ger-pr-a}
\Ger'(a,a_1, \dots, a_n) : = \und{S}(\La\Lie'(a,a_1, \dots, a_n))
\end{equation}
is a subcomplex of $\Ger(a,a_1, \dots, a_n)$\,.

We will prove every cocycle in $ \Ger'(a,a_1, \dots, a_n)$ is 
cohomologous to a (unique) cocycle in the intersection
$$
 \Ger'(a,a_1, \dots, a_n) ~\cap~  \Ger(a_1, \dots, a_n)
$$
and then we will deduce statements of the theorem.

Let us, first, show that every cocycle in $\La\Lie'(a,a_1, \dots, a_n)$
is cohomologous to a cocycle in the intersection 
$$
\La\Lie'(a,a_1, \dots, a_n) ~\cap~ \La\Lie(a_1, \dots, a_n)\,.
$$  

For this purpose we consider a 
non-empty ordered subset  $\{i_1< i_2 < \dots < i_k\}$ of $\{1,2,\dots, n\}$
and denote by 
\begin{equation}
\label{La-Lie-prpr}
\La\Lie''(a,a_{i_1}, \dots, a_{i_k})
\end{equation}
the subcomplex of $ \La\Lie'(a,a_1, \dots, a_n)$ which is spanned by $\La\Lie$-monomials
in $\La\Lie(a,a_{i_1}, \dots, a_{i_k})$ involving each variable in the set 
$\{a_{i_1}, \dots, a_{i_k} \}$ exactly once. 

It is clear that $ \La\Lie'(a,a_1, \dots, a_n)$ splits into the direct
sum of subcomplexes: 
\begin{equation}
\label{LaLiepr-decomp}
\La\Lie'(a,a_1, \dots, a_n) = \bbK\L a, \{a,a\}\R ~\oplus~ \bigoplus_{\{i_1< i_2 < \dots < i_k\}}
\La\Lie''(a,a_{i_1}, \dots, a_{i_k})\,,
\end{equation}
where the summation runs over all non-empty 
ordered subsets $\{i_1< i_2 < \dots < i_k\}$ of $\{1,2,\dots, n\}$\,.

It is not hard to see that the subcomplex $ \bbK\L a, \{a,a\}\R$ is 
acyclic. Thus our goal is to show that every cocycle in $\La\Lie''(a,a_{i_1}, \dots, a_{i_k})$
is cohomologous to cocycle in the intersection 
$$
\La\Lie''(a,a_{i_1}, \dots, a_{i_k}) \cap \La\Lie(a_{i_1}, \dots, a_{i_k})\,. 
$$

To prove this fact we consider the tensor algebra
\begin{equation}
\label{T-1-2-k}
T\big( \bbK \L \bsi a, \bsi a_{i_1}, \bsi a_{i_2}, \dots, \bsi a_{i_{k-1}} \R \big) 
\end{equation}
in the variables $\bsi a, \bsi a_{i_1}, \bsi a_{i_2}, \dots, \bsi a_{i_{k-1}}$ and denote by 
\begin{equation}
\label{Tpr-1-2-k}
T'(\bsi a, \bsi a_{i_1}, \bsi a_{i_2}, \dots, \bsi  a_{i_{k-1}}) 
\end{equation}
the subspace of \eqref{T-1-2-k} which is spanned by monomials 
involving each variable from the set $\{\bsi a_{i_1}, \bsi a_{i_2}, \dots, \bsi a_{i_{k-1}}\}$
exactly once. 

It is not hard to see that the formula 
\begin{equation}
\label{nu}
\nu (x_{j_1} \otimes x_{j_2} \otimes \dots \otimes x_{j_N}) = 
\{\bs\,x_{j_1}, \{\bs\,x_{j_2} ,\{ \dots \{\bs\, x_{j_{N}}, a_{i_k} \br 
\end{equation}
defines an isomorphism of the graded vector spaces
$$
\nu : T'(\bsi a, \bsi a_{i_1}, \bsi a_{i_2}, \dots, \bsi  a_{i_{k-1}})  \stackrel{\cong}{\longrightarrow} 
\La\Lie''(a,a_{i_1}, \dots, a_{i_k})\,.
$$ 

Let us denote by $\de_T$ a degree $1$ derivation of 
the tensor algebra \eqref{T-1-2-k} defined by the equations 
\begin{equation}
\label{de-T}
\de_T (\bsi a_{i_t}) = 0\,, \qquad
\de_T (\bsi a) = \bsi a \otimes \bsi a\,. 
\end{equation}
It is not hard to see that $(\de_T)^2=0$\,. Thus, $\de_T$ is 
a differential on the tensor algebra \eqref{T-1-2-k}\,.

The subspace \eqref{Tpr-1-2-k} is obviously a subcomplex of \eqref{T-1-2-k}. 
Furthermore, using the following consequence of Jacobi identity 
$$
\{a,\{a, X\}\} = - \frac{1}{2} \{\{a,a\},X\}\,, \qquad \forall ~~X \in \La\Lie(a,a_1, \dots, a_n),
$$
it is easy to show that 
$$
\de \circ \nu = \nu \circ \de_T\,.
$$

Thus $\nu$ is an isomorphism from the cochain complex 
$$
\big( T'(\bsi a, \bsi a_{i_1}, \bsi a_{i_2}, \dots, \bsi  a_{i_{k-1}}), \de_T \big)
$$
to the cochain complex 
$$
\big( \La\Lie''(a,a_{i_1}, \dots, a_{i_k}), \de \big)\,.
$$

To compute cohomology of the cochain complex 
\begin{equation}
\label{T-de-T}
\Big( T\big( \bbK \L \bsi a, \bsi a_{i_1}, \bsi a_{i_2}, \dots, \bsi a_{i_{k-1}} \R \big), 
 \de_T
\Big)
\end{equation}
we observe that the truncated tensor algebra 
\begin{equation}
\label{undT-bsi-a}
\und{T}_{\,\bsi a}: = \und{T}\big(\bbK\L \bsi a \R\big) 
\end{equation}
form an acyclic subcomplex of \eqref{T-de-T}.
 
We also observe that the cochain complex 
 \eqref{T-de-T}
splits into the direct sum of subcomplexes
\begin{equation}
\label{T-de-T-sum}
T\big( \bbK \L \bsi a, \bsi a_{i_1}, \bsi a_{i_2}, \dots, \bsi a_{i_{k-1}} \R \big)  = 
T \big( \bbK \L \bsi a_{i_1}, \bsi a_{i_2}, \dots, \bsi a_{i_{k-1}} \R \big)  ~ \oplus 
\end{equation}
$$
\bigoplus_{m \ge 2,\, p_1, \dots, p_m}
V^{\otimes\, p_1}_{a_{\bul}} \otimes \und{T}_{\,\bsi a}
\otimes V^{\otimes\, p_2}_{a_{\bul}}\otimes 
\und{T}_{\,\bsi a} \otimes \dots  \otimes 
V^{\otimes\, p_{m-1}}_{a_{\bul}}\otimes \und{T}_{\,\bsi a} \otimes  V^{\otimes\, p_m}_{a_{\bul}}\,,
$$
where $V_{a_{\bul}}$ is the cochain complex 
$$ 
V_{a_{\bul}}: = \bbK \L \bsi a_{i_1}, \bsi a_{i_2}, \dots, \bsi a_{i_{k-1}}  \R
$$
with the zero differential and the summation runs over all 
combinations $(p_1, \dots, p_m)$ of integers satisfying the conditions
$$
p_1, p_m \ge 0, \qquad p_2, \dots, p_{m-1} \ge 1\,. 
$$

By K\"unneth's theorem all the subcomplexes 
$$
V^{\otimes\, p_1}_{a_{\bul}} \otimes \und{T}_{\,\bsi a}
\otimes V^{\otimes\, p_2}_{a_{\bul}}\otimes 
\und{T}_{\,\bsi a} \otimes \dots  \otimes 
V^{\otimes\, p_{m-1}}_{a_{\bul}}\otimes \und{T}_{\,\bsi a} \otimes  V^{\otimes\, p_m}_{a_{\bul}}
$$
are acyclic. Hence for every cocycle $c$  in \eqref{T-de-T}
there exists a vector $c_1$ in  \eqref{T-de-T}
such that 
$$
c - \de_T (c_1) \in T \big( \bbK \L \bsi a_{i_1}, \bsi a_{i_2}, \dots, \bsi a_{i_{k-1}} \R \big)\,.
$$

Combining this observation with the fact that the 
subcomplex  \eqref{Tpr-1-2-k} is a direct summand in  \eqref{T-de-T}, 
we conclude that, for every cocycle $c$ in   \eqref{Tpr-1-2-k} there 
exists a vector $c_1$  in   \eqref{Tpr-1-2-k}  such that 
$$
c- \de_T(c_1) \in  T'(\bsi a, \bsi a_{i_1}, \bsi a_{i_2}, \dots, \bsi  a_{i_{k-1}})
~\cap~  T \big( \bbK \L \bsi a_{i_1}, \bsi a_{i_2}, \dots, \bsi a_{i_{k-1}} \R \big)\,.
$$

Since the map $\nu$ \eqref{nu} is an isomorphism from the 
cochain complex   \eqref{Tpr-1-2-k} with the differential $\de_T$ to 
the cochain complex \eqref{La-Lie-prpr} with the differential $\de$, 
we deduce that every cocycle in \eqref{La-Lie-prpr} is cohomologous to 
a unique cocycle in the intersection
$$
\La\Lie''(a,a_{i_1}, \dots, a_{i_k}) \cap \La\Lie(a_{i_1}, \dots, a_{i_k})\,. 
$$

Therefore  every cocycle in $\La\Lie'(a,a_1, \dots, a_n)$
is cohomologous to a unique cocycle in the intersection 
$$
\La\Lie'(a,a_1, \dots, a_n) ~\cap~ \La\Lie(a_1, \dots, a_n)\,.
$$  

Combining the latter observation with decomposition \eqref{Ger-pr-a}
we conclude that  every cocycle in $ \Ger'(a,a_1, \dots, a_n)$ is 
cohomologous to a (unique) cocycle in the intersection
$$
\Ger'(a,a_1, \dots, a_n) ~\cap~  \Ger(a_1, \dots, a_n)\,.
$$

Thus, using the isomorphism 
$$
\cG_n \cong \Tw\Ger(n)
$$
together with the fact that the cochain complex 
$\cG_n$ is a direct summand in $\Ger'(a,a_1, \dots, a_n)$\,, 
we deduce the first statement of Theorem  \ref{thm:Ger-TwGer}\,.

On the other hand, since $\Ger(0) = \bfzero$, 
Remark \ref{rem:Tw-cO-0} implies that 
$$
\Conv(\La^2 \coCom, \Ger) \cong \bs^{-2}\, \Tw\Ger(0)\,. 
$$ 
Hence the second statement of Theorem \ref{thm:Ger-TwGer} follows as well. 

The theorem is proved. 
\end{proof}

\subsection{The dg Lie algebra $\Conv^{\oplus}(\Ger^{\vee}, \cO)$. The filtration by Lie words of length $1$}
\label{sec:Ger-cO}

Let $\cO$ be a dg operad and $\io$ be a map (of dg operads)
\begin{equation}
\label{Ger-to-cO}
\io : \Ger \to \cO\,.
\end{equation}
(Here, we assume that $\cO(0) = \bfzero$.)

In this subsection we will describe an auxiliary construction related to 
the pair $(\cO, \io)$\,. In these notes, we will use this construction twice. 
First, we will use it in the case when $\cO = \Ger$. Second, we will use it 
in the case when\footnote{The operad $\Gra$ is introduced 
in Section \ref{sec:Gra} below.} $\cO = \Gra$.

Restricting $\io$ to the suboperad $\La\Lie$ we get a morphism
of dg operads 
\begin{equation}
\label{LaLie-to-cO}
\mj = \io \Big|_{\La \Lie} : \La \Lie \to \cO\,.
\end{equation}
 
Thus, following Section \ref{sec:Tw-oplus}, we may construct the dg operad 
$\Tw\cO$ as well as its suboperad $\Tw^{\oplus}(\cO)\subset \Tw\cO$
\eqref{Tw-oplus}. 

On the other hand composing $\io$ with $U_{\Ger}$ 
\eqref{Ger-infty-Ger} we get a morphism 
\begin{equation}
\label{Ger-infty-to-cO}
\io \circ U_{\Ger} : \Cobar(\Ger^{\vee}) \to \cO\,. 
\end{equation}

It is not hard to see that the Maurer-Cartan element 
$\al \in \Conv(\Ger^{\vee}, \cO)$ corresponding to 
the  morphism \eqref{Ger-infty-to-cO} is given by the 
formula
\begin{equation}
\label{MC-Ger-cO}
\al = \io(\{a_1, a_2\}) \otimes b_1 b_2 + \io(a_1 a_2) \otimes \{b_1, b_2\}\,.
\end{equation}

Since $\al \in  \Conv^{\oplus}(\Ger^{\vee}, \cO)$, it makes sense to 
consider the cochain complex 
\begin{equation}
\label{oplus-Ger-cO}
 \Conv^{\oplus}(\Ger^{\vee}, \cO) = \bigoplus_{n \ge 1} \Big( \cO(n) \otimes \La^{-2}\Ger(n) \Big)^{S_n}
\end{equation}
with the differential 
\begin{equation}
\label{diff-Ger-cO}
\pa =  [\al, ~]\,.
\end{equation}

Let us denote by $\mL_1(w)$ the number of Lie words of length $1$ in 
a monomial  $w \in \La^{-2} \Ger(n)$. For example, $\mL_1(b_1 b_2) = 2$ and
$\mL_1(\{b_1, b_2\}) =0$\,.

Next we consider a vector $v \in \cO(n)$ and observe that 
for every vector $ v_i\otimes w_i$ in the linear combination 
$$
\pa \left( \sum_{\si \in S_n} \si (v) \otimes \si(w)  \right) 
$$
we have $\mL_1 (w_i) = \mL_1 (w)$ or $\mL_1 (w_i) = \mL_1 (w)+1$\,.

This observation allows us to introduce the following ascending filtration
\begin{equation}
\label{filtr-Ger-cO-oplus}
\dots \subset \cF^{m-1}\, \Conv^{\oplus}(\Ger^{\vee}, \cO) \subset 
\cF^m\, \Conv^{\oplus}(\Ger^{\vee}, \cO) \subset \dots\,,
\end{equation}
where $\cF^m\, \Conv^{\oplus}(\Ger^{\vee}, \cO)$ consists of 
sums
$$
\sum_i v_i \otimes w_i \in  \bigoplus_n \big( \cO(n) \otimes \La^{-2} \Ger(n) \big)^{S_n}
$$
which satisfy 
$$
\mL_1 (w_i) - |\, v_i \otimes w_i \,| \le  m\,, \qquad \forall ~~ i\,.
$$

It is clear that the associated graded complex
\begin{equation}
\label{Gr-Ger-cO}
\Gr\, \Conv^{\oplus}(\Ger^{\vee}, \cO)  \cong 
\bigoplus_{n=1}^{\infty} \Big( \cO(n) \otimes \La^{-2} \Ger(n)
\Big)^{S_n}
\end{equation}
as a graded vector space, 
and the differential $\pa^{\Gr}$ on $\Gr\, \Conv^{\oplus}(\Ger^{\vee}, \cO)$
is obtained from the differential $\pa$ \eqref{diff-Ger-cO} on $\Conv^{\oplus}(\Ger^{\vee}, \cO)$ 
by keeping only terms which raise the number of Lie brackets of length $1$
in the second tensor factors. For example, the adjoint action 
$$
[ \,\io(a_1 a_2) \otimes \{b_1, b_2\} , ~ ]
$$
of $ \io(a_1 a_2) \otimes \{b_1, b_2\}$ does not contribute to the differential 
$\pa^{\Gr}$ at all. 

To give a convenient description of the cochain complex \eqref{Gr-Ger-cO} we introduce 
the collection 
\begin{equation}
\label{Ger-hrt}
\{ \La^{-2}\Ger^{\hrt}(n) \}_{n \ge 0} 
\end{equation}
where 
$$
\La^{-2}\Ger^{\hrt}(0) = \bs^{-2} \bbK
$$
and
$$
\La^{-2}\Ger^{\hrt}(n), \qquad n \ge 1
$$
is the $S_n$-submodule of $\La^{-2}\Ger(n)$ spanned by 
monomials $w \in \La^{-2}\Ger(n)$ for which $\mL_1(w) = 0$\,. 

Next, we introduce the cochain complex
\begin{equation}
\label{Ger-hrt-TwO}
\bigoplus_{n \ge 1}  \Big( \Tw^{\oplus}\cO(n)
\otimes \La^{-2}\Ger^{\hrt}(n)  \Big)^{S_n} =
\bigoplus_{r \ge 0}
\bigoplus_{n \ge 1} \Big( \big( \bs^{2r} \cO(r + n) \big)^{S_r} 
\otimes \La^{-2}\Ger^{\hrt}(n)  \Big)^{S_n}
\end{equation}
with the differential $\pa^{\Tw}$ coming from $\Tw\cO$\,.

We observe that the formula
\begin{equation}
\label{Ups-cO-def}
\Ups_{\cO} \left( \sum_i v_i \otimes w_i \right) : =
\sum_{\si \in \Sh_{r,n}} \sum_i
\si (v_i) \otimes
\si (b_1 \dots b_r \, w_i(b_{r+1}, \dots, b_{r+n})) 
\end{equation}
$$
 \sum_i v_i \otimes w_i  \in  \Big(\bs^{2r} \cO(r+n)^{S_r}
\otimes \La^{-2}\Ger^{\hrt}(n)  \Big)^{S_n}
$$ 
defines a morphism of graded vector spaces
\begin{equation}
\label{Ups-cO}
\Ups_{\cO} ~: ~ \bigoplus_{n \ge 0}  \Big( \Tw^{\oplus}\cO(n)
\otimes \La^{-2}\Ger^{\hrt}(n)  \Big)^{S_n} ~~\to~~ 
\Gr\, \Conv^{\oplus}(\Ger^{\vee}, \cO)\,.
\end{equation}

We claim that 
\begin{prop}
\label{prop:Gr-Ger-cO}
The map $\Ups_{\cO}$ \eqref{Ups-cO} is an isomorphism of 
cochain complexes. 
\end{prop}
\begin{proof}
It is clear that \eqref{Gr-Ger-cO} is spanned by vectors of the form 
\begin{equation}
\label{spanning-GrGer-cO}
\sum_{\tau \in S_{r+n}} \tau (v)\otimes \tau(b_1 \dots b_r \, w(b_{r+1}, \dots, b_{r+n}))\,,
\end{equation}
where $v$ is a vector in $\cO(r+n)$, $w$ is a monomial  
$\La^{-2}\Ger^{\hrt}(n)$, and numbers $r, n$ vary within the 
range $r, n \ge 0$, $r+n \ge 1$\,.

Using the obvious identity
$$
\sum_{\tau \in S_{r+n}} \tau (v)\otimes \tau(b_1 \dots b_r \, w(b_{r+1}, \dots, b_{r+n})) = 
$$
$$
\sum_{\si \in \Sh_{r,n}} \si 
\left( \sum_{(\tau', \tau'') \in S_{r} \times S_n \subset S_{r+n} }
 (\tau' , \tau'')(v)\otimes b_1 \dots b_r \, (\tau'' w)(b_{r+1}, \dots, b_{r+n})
\right)
$$
we see that the formula
\begin{equation}
\label{wt-Ups-dfn}
\wt{\Ups}_{\cO} \left( \sum_{\si \in S_{r+n}} \si (v)\otimes \si(b_1 \dots b_r \, 
w(b_{r+1}, \dots, b_{r+n})) \right) =
\end{equation}
$$
\sum_{\tau'' \in S_n} 
\tau'' \left( \sum_{\tau' \in S_r} \tau' (v) \right) ~ \otimes ~ \tau''(w)
$$
gives us a well-defined map 
\begin{equation}
\label{wt-Ups}
\wt{\Ups}_{\cO} ~:~ \Gr\, \Conv^{\oplus}(\Ger^{\vee}, \cO) \to
\bigoplus_{n \ge 0}  \Big( \Tw\cO(n)
\otimes \La^{-2}\Ger^{\hrt}(n)  \Big)^{S_n}\,. 
\end{equation}
Furthermore, it is obvious that $\wt{\Ups}_{\cO}$ is the inverse of $\Ups_{\cO}$\,.

Thus $\Ups_{\cO}$ is an isomorphism of graded vector spaces. 

Before proving that $\Ups$ is compatible with the differentials, 
let us recall that, for $i<j$, $\vs_{i,j}$ denotes the cycle
$(i,i+1, \dots, j) \in S_n$ for any $n \ge j$\,. Furthermore, 
$S_{i, i+1, \dots, n}$ denotes the permutation group of the set 
$\{i, i+1 \dots, n\}$\,.

Let, as above,  $v$ be a vector in $\cO(r+n)$ and 
$w$ be a monomial in $\La^{-2}\Ger^{\hrt}(n)$. Due to the above consideration, 
\begin{equation}
\label{Ups-v-w}
\sum_{\tau \in S_{r+n}} \tau (v)\otimes \tau(b_1 \dots b_r \, w(b_{r+1}, \dots, b_{r+n})) = 
\Ups_{\cO} \left( \,
 \sum_{\la \in S_n} \la\big( \Av_r(v) \big)  \otimes \la(w)
 \,\right)\,,
\end{equation}
where
$$
\Av_r(v) = \sum_{\la_1 \in S_r} \la_1(v)
$$
is viewed as a vector in $\Tw\cO(n)$\,.

Thus our goal is to show that
\begin{equation}
\label{goal-Gr-Ger-cO}
\pa^{\Gr} \left( 
\sum_{\tau \in S_{r+n}} \tau (v)\otimes \tau(b_1 \dots b_r \, w(b_{r+1}, \dots, b_{r+n})) 
\right) =
\end{equation}
$$
\Ups_{\cO} \left( \,
 \sum_{\la \in S_n} \pa^{\Tw} \circ \la \big(\Av_r(v)\big)   \otimes \la(w)
 \,\right)\,.
$$

Collecting terms with $r+1$ Lie words of length $1$ in the 
second tensor factors in 
$$
\Big[~ \io(\{a_1, a_2\})  \otimes b_1 b_2 ~,~ 
  \sum_{\tau \in S_{r+n}} \tau (v)\otimes \tau(b_1 \dots b_r \, w(b_{r+1}, \dots, b_{r+n}))  ~\Big]
$$
and using the obvious identity
$$
\sum_{\tau \in S_{r+n}} \tau (v)\otimes \tau(b_1 \dots b_r \, w(b_{r+1}, \dots, b_{r+n})) =
$$
$$
\sum_{\tau' \in S_{2,3,\dots, r+n}}
\sum_{i=1}^{r+n} \tau'\big( 
\vs_{1,i}(v) \otimes \vs_{1,i} (b_1 \dots b_r \, w(b_{r+1}, \dots, b_{r+n})) \big)
$$
we get 
\begin{equation}
\label{diff-Gr-Ger-cO}
\pa^{\Gr}  \left( \sum_{\tau \in S_{r+n}} \tau (v)\otimes \tau(b_1 \dots b_r \, w(b_{r+1}, \dots, b_{r+n}))
\right) =
\end{equation}
$$
= \sum_{\si \in \Sh_{r+n,1} }\sum_{\tau \in S_{r+n}} 
\si \big(\io(\{a_1, a_2\}) \circ_1 \tau(v) \big) \otimes 
\si \big( \tau(b_1 \dots b_r \, w(b_{r+1}, \dots, b_{r+n})) b_{r+n+1} \big)
$$

$$
-(-1)^{|v|}
~\sum_{\la \in \Sh_{2, r+n-1}}^{\tau' \in S_{3,4, \dots, r+n+1}}~
\sum_{i=1}^{r} 
\la\Big(\tau' \circ \te_{i} \big( v \circ_i \io( \{a_1, a_2 \}) \big)
$$
$$
\otimes\, b_1b_2 b_{\tau'(3)} \dots b_{\tau'(r+1)} ~ w(b_{\tau'(r+2)}, \dots, b_{\tau'(r+1+n)}) \Big)
$$

$$
-(-1)^{|v|}
~\sum_{\la \in \Sh_{2, r+n-1}}^{\tau' \in S_{3,4, \dots, r+n+1}}~
\sum_{i=r+1}^{r+n}  
\la\Big( \tau' \circ \te_{i} \big( v \circ_i \io( \{a_1, a_2 \}) \big)
$$
$$
\otimes\, b_{\tau'(3)} \dots b_{\tau'(r+2)} b_1 ~ w(b_{\tau'(r+3)},\dots, 
b_{\tau'(i+1)}, b_2, b_{\tau'(i+2)}, \dots, b_{\tau'(r+1+n)})
 \Big)
$$

$$
-(-1)^{|v|}
~\sum_{\la \in \Sh_{2, r+n-1}}^{\tau' \in S_{3,4, \dots, r+n+1}}~
\sum_{i=r+1}^{r+n}  
\la\Big( \tau' \circ \te_{i} \big( v \circ_i \io( \{a_1, a_2 \}) \big)
$$
$$
\otimes\, b_{\tau'(3)} \dots b_{\tau'(r+2)} b_2 ~ w(b_{\tau'(r+3)},\dots, 
b_{\tau'(i+1)}, b_1, b_{\tau'(i+2)}, \dots, b_{\tau'(r+1+n)})
 \Big)\,,
$$
where $\te_i$ is the following permutation in $S_{r+1+n}$
\begin{equation}
\label{te-i}
\te_i = 
\left(
\begin{array}{cccccc}
 1 & 2  & \dots & i-1 & i & i+1  \\
 3 & 4  & \dots & i+1 & 1 & 2 
\end{array}
\right)\,.
\end{equation}

The first sum in the right hand side of \eqref{diff-Gr-Ger-cO}
can be simplified as follows.
$$ 
 \sum_{\si \in \Sh_{r+n,1} }\sum_{\tau \in S_{r+n}} 
\si \big(\io(\{a_1, a_2\}) \circ_1 \tau(v) \big) \otimes 
\si \big( \tau(b_1 \dots b_r \, w(b_{r+1}, \dots, b_{r+n})) b_{r+n+1} \big) =
$$
\begin{equation}
\label{sum1-cO}
\sum_{\la \in S_{r+1+n}}
\la   \big(\io(\{a_1, a_2\}) \circ_1  v \big)  \otimes 
\la (b_1 \dots b_r \, w(b_{r+1}, \dots, b_{r+n}) b_{r+n+1} ) =
\end{equation}
$$
\sum_{\la \in S_{r+1+n}}
\la \circ \vs_{1, r+1+n}  \big(\io(\{a_1, a_2\}) \circ_1 v \big) \otimes 
\la \circ \vs_{1, r+1+n} (b_1 \dots b_r \, w(b_{r+1}, \dots, b_{r+n}) b_{r+n+1} ) =
$$
$$
\sum_{\la \in S_{r+1+n}} 
\la   \big(\io(\{a_1, a_2\}) \circ_2  v \big) \otimes 
\la  (b_1 \dots b_{r+1} \, w(b_{r+2}, \dots, b_{r+1+n}) ) =
$$

$$
\sum_{\si \in \Sh_{r+1, n}}^{(\la_1, \la_2) \in S_{r+1} \times S_{n}} \,
\si \big(
(\la_1, \la_2)  \big(\io(\{a_1, a_2\}) \circ_2  v \big) \otimes 
\la_2(b_1 \dots b_{r+1} \, w(b_{r+2}, \dots, b_{r+1+n}) )
\big) =
$$

$$
\sum_{\si \in \Sh_{r+1, n}}^{\tau \in \Sh_{1,r}}
\sum_{\la' \in S_{2, \dots, r+1}}^{\la'' \in S_{r+2, \dots, r+1+n}}
\si \big( \tau \circ \la' \circ \la''  \big(\io(\{a_1, a_2\}) \circ_2 v \big) 
$$
$$ 
\otimes~ b_1 \dots b_{r+1} \, w(b_{\la''(r+2)}, \dots, b_{\la''(r+1+n)})
\big)\,.
$$

Thus 
\begin{equation}
\label{sum1-cO-upshot}
\textrm{The first sum in the R.H.S. of \eqref{diff-Gr-Ger-cO}} =
\end{equation}
$$
\Ups_{\cO} \left(\,
\sum_{\tau\in \Sh_{1,r}} \sum_{\la \in S_n}
\tau \big( \{a_1, a_2\} \circ_2  \la (\Av_r(v)) \big)  \otimes \la (w)
\,\right)\,.
$$

Using the symmetry of the bracket $\{ \,,\, \}$, we rewrite the second 
sum in the right hand side of \eqref{diff-Gr-Ger-cO} as follows

$$
-(-1)^{|v|}
~\sum_{\la \in \Sh_{2, r+n-1}}^{\tau' \in S_{3,4, \dots, r+n+1}}~
\sum_{i=1}^{r} 
\la\Big( \tau' \circ \te_i \big( v \circ_i  \io(\{a_1, a_2 \}) \big)
$$
\begin{equation}
\label{sum2-cO}
\otimes\, b_1b_2 b_{\tau'(3)} \dots b_{\tau'(r+1)} ~ w(b_{\tau'(r+2)}, \dots, b_{\tau'(r+1+n)}) \Big) =
\end{equation}

$$
-\frac{(-1)^{|v|}}{2} 
\sum_{\la \in S_{r+1+n}} 
\sum_{i=1}^{r} 
\la \Big( \te_i \big( v \circ_i  \io(\{a_1, a_2 \}) \big)
~\otimes ~ b_{1} b_{2} b_{3} \dots b_{r+1} ~ w(b_{r+2}, \dots, b_{r+1+n}) 
\Big)=
$$
$$
-\frac{(-1)^{|v|}}{2} 
\sum_{\si \in \Sh_{r+1,n}}
\sum_{\la' \in S_{r+1}}^{\la'' \in S_{r+2, \dots, r+1+n}} 
\sum_{i=1}^{r}   \si \circ \la'' \circ \la' \Big( \te_i \big( v \circ_i  \io(\{a_1, a_2 \}) \big)
$$
$$ 
\otimes ~ b_{1} b_{2} \dots b_{r+1} ~ w(b_{r+2}, \dots, b_{r+1+n})  \Big)=
$$

$$
-(-1)^{|v|} 
\sum_{\si \in \Sh_{r+1,n}}^{\tau \in \Sh_{2, r-1}}~
\sum_{\la_1 \in S_{3, \dots, r+1}}^{\la_2 \in S_{r+2, \dots, r+1+n}} ~
\sum_{i=1}^{r}  \si \circ \tau \Big(
\la_2 \circ \la_1 \circ \te_i \big( v \circ_i  \io(\{a_1, a_2 \}) \big)
$$
$$ 
~ \otimes ~ b_{1} b_{2} \dots b_{r+1} ~ w(b_{\la_2(r+2)}, \dots, b_{\la_2(r+1+n)})  \Big)=
$$

$$
-(-1)^{|v|}  \sum_{\si \in \Sh_{r+1,n}}^{\tau \in \Sh_{2, r-1}}
~
\sum_{\la'_1 \in S_r }^{\la_2 \in S_{r+2, \dots, r+1+n}} ~
\si \Big( \la_2 \circ \tau \big(\la'_1(v) \circ_1 \{a_1, a_2\} \big)
~ \otimes ~ b_{1} b_{2} \dots b_{r+1} ~ w(b_{\la_2(r+2)}, \dots, b_{\la_2(r+1+n)})
\Big)\,,
$$
where $\te_i$ is defined in \eqref{te-i}.

Thus
\begin{equation}
\label{sum2-cO-upshot}
\textrm{The second sum in the R.H.S. of \eqref{diff-Gr-Ger-cO}} =
\end{equation}
$$
-(-1)^{|v|} 
\Ups_{\cO} \left(\,
\sum_{\tau \in \Sh_{2,r-1}} \sum_{\la \in S_n}
\la \Big(\tau \big(  \Av_r(v) \circ_1 \{a_1,a_2\} \big) \Big) \otimes \la (w)
\,\right)\,,
$$
where $ \Av_r(v)$ is viewed as a vector in $\cO(r+n)$ and 
$\tau \big(  \Av_r(v) \circ_1 \{a_1,a_2\} \big) $ is viewed as 
a vector in $\Tw\cO(n)$\,.

Due to Exercise \ref{exer:sum3-4} below, 
\begin{equation}
\label{sum3-4-cO-upshot}
\textrm{The combination of the 
last two sums in the R.H.S. of \eqref{diff-Gr-Ger-cO}} =
\end{equation}
$$
-(-1)^{|v|} 
\Ups_{\cO} \left(\,
\sum_{\tau \in \Sh_{r,1}} \sum_{\la \in S_n}
\sum_{i=1}^{n}
\la \Big(\tau\circ \vs_{r+1,r+i} \big( \Av_r(v) \circ_{r+i} \io(\{a_1,a_2\}) \big) \Big) \otimes \la (w)
\,\right)\,,
$$
where $\Av_r(v)$ is viewed as a vector 
in $\cO(r+n)$ and 
$$
\tau\circ \vs_{r+1,r+i} \big( \Av_r(v) \circ_{r+i} \io(\{a_1,a_2\}) \big) 
$$
is viewed as a vector in $\Tw\cO(n)$\,.

Comparing \eqref{sum1-cO-upshot}, \eqref{sum2-cO-upshot}, and \eqref{sum3-4-cO-upshot}
with the second, the first and the third sums, respectively, in the right hand side of 
equation \eqref{diff-Tw-cO} from Section \ref{sec:Tw-oplus},
we see that the equation \eqref{goal-Gr-Ger-cO} indeed holds.

Proposition \ref{prop:Gr-Ger-cO} is proved. 
\end{proof}
\begin{exer}
\label{exer:sum3-4}
Let  $v$ be a vector in $\cO(r+n)$ and $w$ be a monomial 
in $\La^{-2}\Ger^{\hrt}(n)$\,.
Prove that 

\begin{equation}
\label{sum3-4-exer}
\textrm{The combination of the 
last two sums in the R.H.S. of \eqref{diff-Gr-Ger-cO}} =
\end{equation}
$$
-(-1)^{|v|} 
\Ups_{\cO} \left(\,
\sum_{\tau \in \Sh_{r,1}} \sum_{\la \in S_n}
\sum_{i=1}^{n}
\la \Big(\tau\circ \vs_{r+1,r+i} \big( \Av_r(v) \circ_{r+i} \io(\{a_1,a_2\}) \big) \Big) \otimes \la (w)
\,\right)\,,
$$
where 
$$
\Av_r (v) = \sum_{\la_1 \in S_r} \la_1(v)\,.
$$
\end{exer}

{\it Hint for Exercise \ref{exer:sum3-4}:}
Using the symmetry of the bracket $\{~,~\}$ we can 
rewrite the combination of the last two sums in the right hand side 
of  \eqref{diff-Gr-Ger-cO} as follows:
\begin{equation}
\label{sum3-4-hint}
\textrm{The combination of the 
last two sums in the R.H.S. of \eqref{diff-Gr-Ger-cO}} =
\end{equation}
$$
-(-1)^{|v|} \sum_{i=r+1}^{r+n} 
\sum_{\la\in S_{r+1+n}} 
\la \Big(
\te_i \big( v \circ_i  \io(\{a_1, a_2 \}) \big)
$$
$$
~\otimes~ b_2 b_3 \dots b_{r+2} \, w (b_{r+3}, \dots, b_{i+1}, b_1, b_{i+2}, 
\dots, b_{r+1+n}) 
\Big)=
$$

$$
-(-1)^{|v|} 
\sum_{i=r+1}^{r+n} 
\sum_{\la\in S_{r+1+n}} 
\la \circ \vs^{-1}_{1, i+1} \Big(
\te_i \big( v \circ_i  \io(\{a_1, a_2 \}) \big)
$$
$$
~\otimes~ b_2 b_3 \dots b_{r+2} \, w (b_{r+3}, \dots, b_{i+1}, b_1, b_{i+2}, 
\dots, b_{r+1+n}) 
\Big)=
$$

$$
-(-1)^{|v|} 
\sum_{i=r+1}^{r+n}  
\sum_{\la \in S_{r+1+n}} 
\la  \Big( \vs_{1,i} \big(v \circ_i \io(\{a_1, a_2\})\big)
~\otimes~ b_1 b_2 \dots b_{r+1} \, w (b_{r+2}, \dots, b_{r+1+n}) 
\Big) =
$$

$$
-(-1)^{|v|} 
\sum_{i=r+1}^{r+n}  
\sum_{\la \in S_{r+1+n}} 
\la \circ \vs^{-1}_{1, r+1} \Big( \vs_{1,i} \big(v \circ_i \io(\{a_1, a_2\})\big)
~\otimes~ b_1 b_2 \dots b_{r+1} \, w (b_{r+2}, \dots, b_{r+1+n}) 
\Big)=
$$

$$
-(-1)^{|v|} 
\sum_{i=r+1}^{r+n}  
\sum_{\la \in S_{r+1+n}} 
\la  \Big( \vs_{r+1,i} \big(v \circ_i \io(\{a_1, a_2\})\big)
~\otimes~ b_1 b_2 \dots b_{r+1} \, w (b_{r+2}, \dots, b_{r+1+n}) 
\Big)\,.
$$

\section{The operad $\Gra$ and its link to the operad $\Ger$}
\label{sec:Gra}

Let us recall from \cite{Thomas} the operad (in $\grVect_{\bbK}$) of 
labeled graphs $\Gra$\,. 

To define the space $\Gra(n)$ (for $n \ge 1$) we introduce an auxiliary set $\gra_{n}$.
An element of $\gra_{n}$ is a labelled graph $\G$
with $n$ vertices and with the additional piece 
of data: the set of edges of $\G$ is equipped with a 
total order. An example of an element in $\gra_4$ is 
shown on figure \ref{fig:exam}. 
\begin{figure}[htp] 
\centering 
\begin{tikzpicture}[scale=0.5, >=stealth']
\tikzstyle{w}=[circle, draw, minimum size=4, inner sep=1]
\tikzstyle{b}=[circle, draw, fill, minimum size=4, inner sep=1]
\node [b] (b1) at (0,0) {};
\draw (-0.4,0) node[anchor=center] {{\small $1$}};
\node [b] (b2) at (1,2) {};
\draw (1,2.6) node[anchor=center] {{\small $2$}};
\node [b] (b3) at (1,-2) {};
\draw (1,-2.6) node[anchor=center] {{\small $3$}};
\node [b] (b4) at (2,0) {};
\draw (2,0.6) node[anchor=center] {{\small $4$}};
\draw (b1) edge (b2);
\draw (0.2,1.4) node[anchor=center] {{\small $ii$}};
\draw (b1) edge (b3);
\draw (0.2,-1.4) node[anchor=center] {{\small $iii$}};
\draw (-1,0) circle (1);
\draw (-2.3,0) node[anchor=center] {{\small $i$}};
\end{tikzpicture}
\caption{The Roman numerals 
indicate that we chose the total order on 
the set of edges $(1,1) < (1,2) < (1,3)$} \label{fig:exam}
\end{figure}
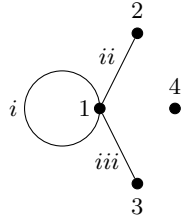
We will often use Roman numerals to specify total orders 
on sets of edges. Thus the Roman numerals on figure \ref{fig:exam} 
indicate that we chose the total order $(1,1) < (1,2) < (1,3)$\,.

The space  $\Gra(n)$ (for $n \ge 1$) is spanned by elements of 
$\gra_n$, modulo the relation $\G^{\si} = (-1)^{|\si|} \G$
where the elements $\G^{\si}$ and $\G$ correspond to the same
labelled graph but differ only by permutation $\si$
of edges. We also declare that 
the degree of a graph $\G$ in $\Gra(n)$ equals 
$-e(\G)$, where $e(\G)$ is the number of edges in $\G$\,.
For example, the graph $\G$ on figure \ref{fig:exam} has 
$3$ edges. Thus its degree is $-3$\,. 

Finally, we set 
\begin{equation}
\label{Gra-0}
\Gra(0) = \bfzero\,.
\end{equation}

\begin{rem}
\label{rem:double-edges}
It clear that, if a graph $\G \in \gra_n$ has
multiple edges, then
$$
\G = - \G
$$
in $\Gra(n)$\,. Thus for every graph  $\G \in \gra_n$
with multiple edges $\G = 0$ in $\Gra(n)$\,.
\end{rem}

We will now define elementary insertions for the collection
$\{\Gra(n)\}_{n\ge 0}$\,.

Let $\G$ and $\wt{\G}$ be graphs representing vectors 
in $\Gra(n)$ and $\Gra(m)$, respectively. 
Let  $1\le i \le m$\,.

The vector $\wt{\G}\, \circ_{i}\, \G \in \Gra(n+m-1)$ is represented 
by the sum of graphs $\G_{\al} \in \gra_{n+m-1}$
 \begin{equation}
\label{circ-i-mc}
\wt{\G} \,\circ_{i}\, \G  = \sum_{\al} \G_{\al}\,,
\end{equation}
where $\G_{\al}$ is obtained by  
 ``plugging in'' the graph $\G$ into the $i$-th vertex of the graph $\tG$ and   
reconnecting the edges incident to the $i$-th vertex of $\tG$ to vertices of $\G$ in 
all possible ways. (The index $\al$ refers to a particular way of connecting the 
edges incident to the $i$-th vertex of $\tG$ to vertices of $\G$. )
After reconnecting edges we label vertices of $\G_{\al}$ as follows: 
\begin{itemize}

\item we leave the same labels on the first $i-1$ vertices 
of $\wt{\G}$; 

\item we shift all labels on vertices of $\G$ up by $i-1$;

\item finally, we shift the labels on the last $m-i$ vertices 
of $\wt{\G}$ up by $n-1$\,.

\end{itemize}
To define the total order on edges of the graph $\G_{\al}$ we declare 
that all edges of $\wt{\G}$ are smaller than all edges of the graph $\G$\,.

\begin{example}
\label{ex:Gra-ins}
Let $\wt{\G}$ (resp. $\G$) be the graph depicted on figure 
\ref{fig:wtG} (resp. figure \ref{fig:G12})\,. The vector 
$\wt{\G} \,\c_2\, \G$ is shown on figure \ref{fig:dGra-ins}.
\begin{figure}[htp]
\begin{minipage}[t]{0.45\linewidth}
\centering 
\begin{tikzpicture}[scale=0.5, >=stealth']
\tikzstyle{w}=[circle, draw, minimum size=4, inner sep=1]
\tikzstyle{b}=[circle, draw, fill, minimum size=4, inner sep=1]
\node [b] (b1) at (0,0) {};
\draw (0,-0.6) node[anchor=center] {{\small $1$}};
\node [b] (b2) at (1,2) {};
\draw (1.4,2.2) node[anchor=center] {{\small $2$}};
\node [b] (b3) at (-1,2) {};
\draw (-1.4,2.2) node[anchor=center] {{\small $3$}};
\draw (b1) edge (b2);
\draw (0.8,0.8) node[anchor=center] {{\small $i$}};
\draw (b1) edge (b3);
\draw (-0.9,0.8) node[anchor=center] {{\small $ii$}};
\draw (b3) edge (b2);
\draw (0,2.5) node[anchor=center] {{\small $iii$}};
\end{tikzpicture}
~\\[0.3cm]
\caption{A graph $\wt{\G} \in \gra_{3}$} \label{fig:wtG}
\end{minipage}
\begin{minipage}[t]{0.45\linewidth}
\centering 
\begin{tikzpicture}[scale=0.5, >=stealth']
\tikzstyle{w}=[circle, draw, minimum size=4, inner sep=1]
\tikzstyle{b}=[circle, draw, fill, minimum size=4, inner sep=1]
\node [b] (b1) at (0,0) {};
\draw (0,-0.6) node[anchor=center] {{\small $1$}};
\node [b] (b2) at (1,2) {};
\draw (1.4,2.2) node[anchor=center] {{\small $2$}};
\draw (b1) edge (b2);
\end{tikzpicture}
~\\[0.3cm]
\caption{A graph $\G \in \gra_{2}$} \label{fig:G12}
\end{minipage}
\end{figure} 
\begin{figure}[htp]
\begin{minipage}[t]{0.3\linewidth}
\centering 
\begin{tikzpicture}[scale=0.5, >=stealth']
\tikzstyle{w}=[circle, draw, minimum size=4, inner sep=1]
\tikzstyle{b}=[circle, draw, fill, minimum size=4, inner sep=1]
\draw (-3.5,1) node[anchor=center] {{$\wt{\G} \,\c_2\, \G \quad = $}};
\node [b] (b1) at (0,0) {};
\draw (0,-0.6) node[anchor=center] {{\small $1$}};
\node [b] (b2) at (1,2) {};
\draw (1.3,1.6) node[anchor=center] {{\small $2$}};
\node [b] (b3) at (2.5,2.5) {};
\draw (2.9,2.6) node[anchor=center] {{\small $3$}};
\node [b] (b4) at (-1,2) {};
\draw (-1.4,2.2) node[anchor=center] {{\small $4$}};
\draw (b1) edge (b2);
\draw (0.8,0.8) node[anchor=center] {{\small $i$}};
\draw (b1) edge (b4);
\draw (-0.9,0.8) node[anchor=center] {{\small $ii$}};
\draw (b4) edge (b2);
\draw (0,2.5) node[anchor=center] {{\small $iii$}};
\draw (b2) edge (b3);
\draw (1.7,2.67) node[anchor=center] {{\small $iv$}};
\end{tikzpicture}
~\\[0.3cm]
\end{minipage}
\begin{minipage}[t]{0.2\linewidth}
\centering 
\begin{tikzpicture}[scale=0.5, >=stealth']
\tikzstyle{w}=[circle, draw, minimum size=4, inner sep=1]
\tikzstyle{b}=[circle, draw, fill, minimum size=4, inner sep=1]
\draw (-2.5,1) node[anchor=center] {{$ + $}};
\node [b] (b1) at (0,0) {};
\draw (0,-0.6) node[anchor=center] {{\small $1$}};
\node [b] (b2) at (1,2) {};
\draw (1.3,1.6) node[anchor=center] {{\small $3$}};
\node [b] (b3) at (2.5,2.5) {};
\draw (2.9,2.6) node[anchor=center] {{\small $2$}};
\node [b] (b4) at (-1,2) {};
\draw (-1.4,2.2) node[anchor=center] {{\small $4$}};
\draw (b1) edge (b2);
\draw (b1) edge (b4);
\draw (b4) edge (b2);
\draw (b2) edge (b3);
\draw (0.8,0.8) node[anchor=center] {{\small $i$}};
\draw (-0.9,0.8) node[anchor=center] {{\small $ii$}};
\draw (0,2.5) node[anchor=center] {{\small $iii$}};
\draw (1.7,2.67) node[anchor=center] {{\small $iv$}};
\end{tikzpicture}
~\\[0.3cm]
\end{minipage}
\begin{minipage}[t]{0.2\linewidth}
\centering 
\begin{tikzpicture}[scale=0.5, >=stealth']
\tikzstyle{w}=[circle, draw, minimum size=4, inner sep=1]
\tikzstyle{b}=[circle, draw, fill, minimum size=4, inner sep=1]
\draw (-3,1) node[anchor=center] {{$ + $}};
\node [b] (b1) at (0,0) {};
\draw (0,-0.6) node[anchor=center] {{\small $1$}};
\node [b] (b2) at (1.5,1.5) {};
\draw (1.6,1) node[anchor=center] {{\small $2$}};
\node [b] (b3) at (0,3) {};
\draw (0,3.5) node[anchor=center] {{\small $3$}};
\node [b] (b4) at (-1.5,1.5) {};
\draw (-1.6,1) node[anchor=center] {{\small $4$}};
\draw  (b1) edge (b2);
\draw (1,0.5) node[anchor=center] {{\small $i$}};
\draw  (b2) edge (b3);
\draw (1.1,2.6) node[anchor=center] {{\small $iv$}};
\draw  (b4) edge (b3);
\draw (-1.1,2.6) node[anchor=center] {{\small $iii$}};
\draw  (b1) edge (b4);
\draw (-1,0.45) node[anchor=center] {{\small $ii$}};
\end{tikzpicture}
~\\[0.3cm]
\end{minipage}
\begin{minipage}[t]{0.2\linewidth}
\centering 
\begin{tikzpicture}[scale=0.5, >=stealth']
\tikzstyle{w}=[circle, draw, minimum size=4, inner sep=1]
\tikzstyle{b}=[circle, draw, fill, minimum size=4, inner sep=1]
\draw (-2.8,1) node[anchor=center] {{$ + $}};
\node [b] (b1) at (0,0) {};
\draw (0,-0.6) node[anchor=center] {{\small $1$}};
\node [b] (b2) at (1.5,1.5) {};
\draw (1.6,1) node[anchor=center] {{\small $3$}};
\node [b] (b3) at (0,3) {};
\draw (0,3.5) node[anchor=center] {{\small $2$}};
\node [b] (b4) at (-1.5,1.5) {};
\draw (-1.6,1) node[anchor=center] {{\small $4$}};
\draw  (b1) edge (b2);
\draw  (b2) edge (b3);
\draw  (b4) edge (b3);
\draw  (b1) edge (b4);
\draw (1,0.5) node[anchor=center] {{\small $i$}};
\draw (1.1,2.6) node[anchor=center] {{\small $iv$}};
\draw (-1.1,2.6) node[anchor=center] {{\small $iii$}};
\draw (-1,0.45) node[anchor=center] {{\small $ii$}};
\end{tikzpicture}
~\\[0.3cm]
\end{minipage}
\caption{The vector $\wt{\G} \,\c_2 \, \G \in \Gra(4)$} \label{fig:dGra-ins}
\end{figure} 
\end{example}

The symmetric group $S_n$ acts on $\Gra(n)$ in the 
obvious way by rearranging the labels on vertices. 
It is not hard to see that insertions \eqref{circ-i-mc}
together with this action of $S_n$ give on $\Gra$ an 
operad structure with the identity element being the unique
graph in $\gra_1$ with no edges.

It is clear that if two graphs $\wt{\G}$ and $\G$ representing vectors 
in $\Gra$ do not have loops (i.e. cycles of length $1$) then each graph in the linear 
combination  $\wt{\G} \circ_i \G$ does not have loops either.  
Thus, by discarding graphs with loops, we arrive at a suboperad 
$\Gra_{\nl}$ of $\Gra$\,.

The graphs depicted below represent vectors in $\Gra_{\nl}(2)$
and in $\Gra(2)$\,.
\begin{equation}
\label{binary}
\G_{\ed} =   \begin{tikzpicture}[scale=0.5, >=stealth']
\tikzstyle{w}=[circle, draw, minimum size=4, inner sep=1]
\tikzstyle{b}=[circle, draw, fill, minimum size=4, inner sep=1]
\node [b] (b1) at (0,0) {};
\draw (0,0.6) node[anchor=center] {{\small $1$}};
\node [b] (b2) at (1.5,0) {};
\draw (1.5,0.6) node[anchor=center] {{\small $2$}};
\draw (b1) edge (b2);
\end{tikzpicture}
\qquad \quad
\G_{\bb} =   \begin{tikzpicture}[scale=0.5, >=stealth']
\tikzstyle{w}=[circle, draw, minimum size=4, inner sep=1]
\tikzstyle{b}=[circle, draw, fill, minimum size=4, inner sep=1]
\node [b] (b1) at (0,0) {};
\draw (0,0.6) node[anchor=center] {{\small $1$}};
\node [b] (b2) at (1.5,0) {};
\draw (1.5,0.6) node[anchor=center] {{\small $2$}};
\end{tikzpicture}
\end{equation}
Later they will play a special role.

\subsection{``Graphical'' interpretation of the operad $\Ger$} 

Since $\Ger$ is generated by the monomials $a_1 a_2, ~ \{a_1,a_2\} \in \Ger(2)$,
any map of operads 
$$
f : \Ger \to \cO
$$
is uniquely determined by its values on $a_1 a_2$ and $\{a_1, a_2\}$\,.

\begin{exer}
\label{exer:Ger-Gra}
Let $\G_{\ed}$ and $\G_{\bb}$ be the vectors in $\Gra(2)$ 
introduced in \eqref{binary}. Prove that the assignment 
\begin{equation}
\label{io}
\io(a_1 a_2) = \G_{\bb}, \qquad \io( \{a_1, a_2\}) = \G_{\ed} 
\end{equation}
defines a map of operads (in $\grVect_{\bbK}$)
\begin{equation}
\label{io-Ger-Gra}
\io : \Ger \to \Gra\,.
\end{equation}
Notice that, one only has to check that 
$$
\io \big( (a_1 a_2) a_3   - a_1 (a_2 a_3)  \big) = 0\,,
$$
$$
\io \big( \{a_1 , a_2 a_3 \}  -  \{a_1, a_2\} a_3 + a_2 \{a_1, a_3\} \big) = 0\,,
$$
$$
\io \big(  \{\{a_1, a_2\} , a_3\} +  \{\{a_2, a_3\} , a_1\} +
 \{\{a_3, a_1\} , a_2\}  \big) = 0\,.
$$
\end{exer}

We claim that 
\begin{prop}
\label{prop:Ger-Gra}
The map of operads $\io:  \Ger \to \Gra$
is injective. 
\end{prop}
\begin{proof}
Recall that due to Exercise \ref{exer:Ger-n-basis}, 
the monomials 
\begin{equation}
\label{Ger-n-basis-here}
\{ a_{i_{11}},  \dots, \{ a_{i_{1 (p_1-1)}}, a_{i_{1 p_1}} \br \dots
\{ a_{i_{t1}},  \dots, \{ a_{i_{t (p_t-1)}}, a_{i_{t p_t}} \br
\end{equation}
corresponding to the ordered partitions \eqref{sp-partition}
form a basis of $\Ger(n)$\,.

Let us observe that for every ordered partition 
 \eqref{sp-partition}
the graph depicted on figure \ref{fig:lines}
enters the linear combination 
$$
\io\big(  \{ a_{i_{11}},  \dots, \{ a_{i_{1 (p_1-1)}}, a_{i_{1 p_1}} \br \dots
\{ a_{i_{t1}},  \dots, \{ a_{i_{t (p_t-1)}}, a_{i_{t p_t}}  \br \big)
$$
with the coefficient $1$.
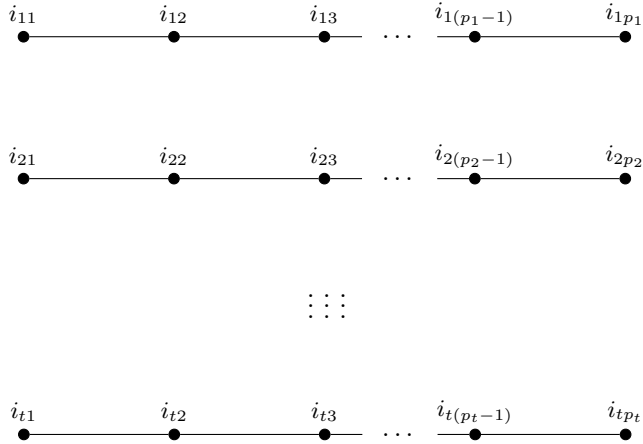
\begin{figure}[htp] 
\begin{minipage}[t]{\linewidth}
\centering 
\begin{tikzpicture}[scale=0.5, >=stealth']
\tikzstyle{w}=[circle, draw, minimum size=4, inner sep=1]
\tikzstyle{b}=[circle, draw, fill, minimum size=4, inner sep=1]
\node [b] (b1) at (0,0) {};
\draw (0,0.6) node[anchor=center] {{\small $i_{11}$}};
\node [b] (b2) at (4,0) {};
\draw (4,0.6) node[anchor=center] {{\small $i_{12}$}};
\node [b] (b3) at (8,0) {};
\draw (8,0.6) node[anchor=center] {{\small $i_{13}$}};
\node [b] (b1t1) at (12,0) {};
\draw (12,0.6) node[anchor=center] {{\small $i_{1 (p_1-1)}$}};
\node [b] (bt1) at (16,0) {};
\draw (16,0.6) node[anchor=center] {{\small $i_{1 p_1}$}};
\draw (b1) edge (b2);
\draw (b2) edge (b3);
\draw (b3) edge (9,0);
\draw (10,0) node[anchor=center] {{$\dots$}};
\draw (11,0) edge (b1t1);
\draw (b1t1) edge (bt1);
\end{tikzpicture}
\end{minipage}
\begin{minipage}[t]{\linewidth}
~\\[0.3cm]
\end{minipage}
\begin{minipage}[t]{\linewidth}
\centering 
\begin{tikzpicture}[scale=0.5, >=stealth']
\tikzstyle{w}=[circle, draw, minimum size=4, inner sep=1]
\tikzstyle{b}=[circle, draw, fill, minimum size=4, inner sep=1]
\node [b] (b1) at (0,0) {};
\draw (0,0.6) node[anchor=center] {{\small $i_{21}$}};
\node [b] (b2) at (4,0) {};
\draw (4,0.6) node[anchor=center] {{\small $i_{22}$}};
\node [b] (b3) at (8,0) {};
\draw (8,0.6) node[anchor=center] {{\small $i_{23}$}};
\node [b] (b1t) at (12,0) {};
\draw (12,0.6) node[anchor=center] {{\small $i_{2 (p_2-1)}$}};
\node [b] (bt) at (16,0) {};
\draw (16,0.6) node[anchor=center] {{\small $i_{2 p_2}$}};
\draw (b1) edge (b2);
\draw (b2) edge (b3);
\draw (b3) edge (9,0);
\draw (10,0) node[anchor=center] {{$\dots$}};
\draw (11,0) edge (b1t);
\draw (b1t) edge (bt);
\end{tikzpicture}
\end{minipage}
\begin{minipage}[t]{\linewidth}
~\\[0.1cm]
\begin{center}
$\vdots$ ~~~$\vdots$~~~  $\vdots$
\end{center}
~\\[0.1cm]
\end{minipage}
\begin{minipage}[t]{\linewidth}
\centering 
\begin{tikzpicture}[scale=0.5, >=stealth']
\tikzstyle{w}=[circle, draw, minimum size=4, inner sep=1]
\tikzstyle{b}=[circle, draw, fill, minimum size=4, inner sep=1]
\node [b] (b1) at (0,0) {};
\draw (0,0.6) node[anchor=center] {{\small $i_{t1}$}};
\node [b] (b2) at (4,0) {};
\draw (4,0.6) node[anchor=center] {{\small $i_{t2}$}};
\node [b] (b3) at (8,0) {};
\draw (8,0.6) node[anchor=center] {{\small $i_{t3}$}};
\node [b] (b1t) at (12,0) {};
\draw (12,0.6) node[anchor=center] {{\small $i_{t (p_t-1)}$}};
\node [b] (bt) at (16,0) {};
\draw (16,0.6) node[anchor=center] {{\small $i_{t p_t}$}};
\draw (b1) edge (b2);
\draw (b2) edge (b3);
\draw (b3) edge (9,0);
\draw (10,0) node[anchor=center] {{$\dots$}};
\draw (11,0) edge (b1t);
\draw (b1t) edge (bt);
\end{tikzpicture}
\end{minipage}
\caption{The edges are ordered ``left to right'', ``top to bottom''} 
\label{fig:lines}
\end{figure} 

Since such graphs are linearly independent in $\Gra(n)$,  we 
conclude that $\io$ is indeed injective.

The proposition is proved. 
\end{proof}

\section{The full graph complex $\fGC$: the first steps}
\label{sec:fGC-first}

Let $\G_{\ed}$ and $\G_{\bb}$ be the vectors in $\Gra(2)$ introduced 
in \eqref{binary}. Following Exercise \ref{exer:Ger-Gra} and 
Proposition \ref{exer:Ger-Gra}
the formulas 
$$
\io(\{a_1, a_2\}) = \G_{\ed}\,, \qquad 
\io (a_1 a_2) = \G_{\bb}
$$ 
define an embedding $\io$ of the operad $\Ger$ into the operad $\Gra$. 

Thus, restricting $\io$ to the suboperad $\La\Lie \subset \Ger$
we get an embedding 
$$
\La\Lie  \hookrightarrow  \Gra\,. 
$$
Hence we have a canonical map of (dg) operads
\begin{equation}
\label{LaLie-infty-Gra}
\vf_{\Gra} : \La\Lie_{\infty} \to \Gra\,.
\end{equation}

Applying the general procedure of twisting (see Section \ref{sec:twist})
to the pair $(\Gra, \vf_{\Gra})$ we get a dg operad $\Tw\Gra$ 
and a dg Lie algebra 
\begin{equation}
\label{cL-Gra}
\cL_{\Gra} = \Conv( \La^2 \coCom, \Gra)
\end{equation}
which acts on the operad $\Tw\Gra$\,.

Following \cite{Thomas} we denote the dg Lie algebra 
$\cL_{\Gra}$ by $\fGC$. In other words, 
\begin{equation}
\label{fGC}
\fGC = \Conv( \La^2 \coCom, \Gra)
\end{equation}

The vector 
\begin{equation}
\label{MC-fGC}
\G_{\ed} \in \fGC
\end{equation}
is a Maurer-Cartan element in $\fGC$ and the differential on
$\fGC$ is given by the formula: 
\begin{equation}
\label{diff-fGC}
\pa =  \ad_{\G_{\ed}}\,.
\end{equation}

\begin{defi}
\label{dfn:fGC}
The cochain complex $\fGC$ \eqref{fGC} with the differential \eqref{diff-fGC}
is called the {\it full graph complex}.
\end{defi}

In this subsection we take a first few steps towards analyzing
the cochain complex $\fGC$.

Unfolding the definition of the convolution Lie algebra 
we get
 \begin{equation}
\label{fGC-unfold}
\fGC = \prod_{n=1}^{\infty} \Hom_{S_n}\big(\La^2 \coCom(n), \Gra(n) \big) = 
 \prod_{n=1}^{\infty} \Hom_{S_n}\big( \bs^{2-2n} \bbK, \Gra(n) \big) = 
\end{equation}
$$
= \prod_{n=1}^{\infty} \bs^{2n-2} \big(\Gra(n) \big)^{S_n}\,.
$$

In other words, vectors in $\fGC$ are infinite sums
\begin{equation}
\label{sum-fGC}
\ga = \sum_{n=1}^{\infty} \ga_n
\end{equation}
of $S_n$-invariant vectors $\ga_n \in \bs^{2n-2}\Gra(n)$\,.

The vector space  
\begin{equation}
\label{Gra-n-inv}
\bs^{2n-2} \big(\Gra(n) \big)^{S_n} 
\end{equation}
is spanned by vectors of the  form
\begin{equation}
\label{ave-G}
\Av(\G) =  \sum_{\si \in S_n}  \si(\G)
\end{equation}
where $\G$ is an element in $\gra_n$\,.
In other words, formula \eqref{ave-G}
defines a surjective $\bbK$-linear map
\begin{equation}
\label{ave-map}
\Av : \bbK \L \gra_n \R \onto  \bs^{2n-2}\big(\Gra(n) \big)^{S_n}\,. 
\end{equation}

To describe the kernel of the map $\Av$, we observe that 
$\Av(\G) = 0$ if and only if the underlying 
unlabeled graph has an 
automorphism which induces an odd permutation on 
the set of edges. In this case we say that the element 
$\G \in \gra_n$ is {\it odd}. Otherwise, we say that 
the element $\G \in \gra_n$ is {\it even}\,.  For example, 
the square depicted on figure \ref{fig:square} is odd 
and the pentagon depicted on figure \ref{fig:penta} 
is even. It is obvious that the property of being even 
or odd depends only on the isomorphism class of 
the underlying unlabeled graph.
\begin{figure}[htp]
\begin{minipage}[t]{0.45\linewidth}
\centering 
\begin{tikzpicture}[scale=0.5, >=stealth']
\tikzstyle{w}=[circle, draw, minimum size=4, inner sep=1]
\tikzstyle{b}=[circle, draw, fill, minimum size=4, inner sep=1]
\node [b] (b1) at (0,0) {};
\draw (0,-0.6) node[anchor=center] {{\small $1$}};
\node [b] (b2) at (2,0) {};
\draw (2,-0.6) node[anchor=center] {{\small $2$}};
\node [b] (b3) at (2,2) {};
\draw (2,2.6) node[anchor=center] {{\small $3$}};
\node [b] (b4) at (0,2) {};
\draw (0,2.6) node[anchor=center] {{\small $4$}};
\draw  (b1) edge (b2);
\draw  (b2) edge (b3);
\draw  (b3) edge (b4);
\draw  (b4) edge (b1);
\end{tikzpicture}
~\\[0.3cm]
\caption{We choose this order on the set of edges: 
$(1,2)< (2,3) < (3,4) < (4,1)$} \label{fig:square}
\end{minipage}
\hspace{0.08\linewidth}
\begin{minipage}[t]{0.45\linewidth}
\centering 
\begin{tikzpicture}[scale=0.5, >=stealth']
\tikzstyle{w}=[circle, draw, minimum size=4, inner sep=1]
\tikzstyle{b}=[circle, draw, fill, minimum size=4, inner sep=1]
\node [b] (b1) at (2,0) {};
\draw (2.5, 0) node[anchor=center] {{\small $1$}};
\node [b] (b2) at (0.62,1.90) {};
\draw (0.77, 2.4) node[anchor=center] {{\small $2$}};
\node [b] (b3) at (-1.62,1.18) {};
\draw (-2.02, 1.47) node[anchor=center] {{\small $3$}};
\node [b] (b4) at (-1.62,-1.18) {};
\draw (-2.02, -1.47) node[anchor=center] {{\small $4$}};
\node [b] (b5) at (0.62,-1.90) {};
\draw (0.77, -2.4) node[anchor=center] {{\small $5$}};
\draw (b1) edge (b2);
\draw (b2) edge (b3);
\draw (b3) edge (b4);
\draw (b4) edge (b5);
\draw (b5) edge (b1);
\end{tikzpicture}
~\\[0.3cm]
\caption{We choose this order on the set of edges: 
$(1,2)< (2,3) < (3,4) < (4,5)< (5,1)$} \label{fig:penta}
\end{minipage}
\end{figure} 

Let us consider a
pair of even elements $\G, \G' \in  \gra_n$
whose underlying unlabeled graphs are 
isomorphic. 
Any isomorphism of the underlying  unlabeled graphs
gives us a bijection from the set $E(\G)$ of edges of $\G$ to 
the set $E(\G')$ of edges of $\G'$\,. Since both sets 
$E(\G)$  and  $E(\G')$ are totally ordered, this bijection 
determines a permutation $\si \in S_m$ where $m = |E(\G)|$\,.  
Furthermore, since $\G$ and $\G'$ are even, such permutations
$\si$ are either all even or all odd. 
In the later case, we say that even elements $\G$ and $\G'$ are {\it opposite} 
 and the former case we say that even elements $\G$ and 
$\G'$ are {\it concordant}.

It is clear that 
\begin{prop}
\label{prop:ker-Av}
The kernel of the map $\Av$ \eqref{ave-map} 
is spanned by vectors of the form 
\begin{equation}
\label{rel-Gra-inv}
\G, \qquad \G_1 - \G_2, \qquad \G'_1 + \G'_2\,,
\end{equation}
where $\G$ is odd, $(\G_1, \G_2)$ is a 
pair of concordant (even) graphs, and  $(\G'_1, \G'_2)$ is a 
pair of opposite (even) graphs. $\Box$
\end{prop} 

In view of Proposition \ref{prop:ker-Av}, we may identify
the vector space \eqref{Gra-n-inv} with the quotient of 
$\bbK\L \gra_n \R$ by the subspace spanned by vectors 
\eqref{rel-Gra-inv}. 

The following proposition gives us a convenient description 
of the differential  \eqref{diff-fGC} on $\fGC$:
\begin{prop}
\label{prop:dfGC}
For every (even) element $\G\in  \gra_n$ we have
\begin{equation}
\label{pa-dfGC-simple}
\pa \big( \Av(\G) \big) =  \Av( \G_{\ed} \circ_1 \G )  - (-1)^{e(\G)}\,
\frac{1}{2}
\sum_{i=1}^n \Av(\G \circ_i \G_{\ed}) 
\end{equation}
where $e(\G)$ is the number of edges of $\G$\,. Moreover, if $\G$ is 
a connected (even) graph in $\gra_n$ with at least one edge, then   
\begin{equation}
\label{pa-dfGC-simpler}
\pa \big( \Av(\G) \big)  =  - \frac{ (-1)^{e(\G)} }{2}
\sum_{i=1}^n \Av(\G'_i)
\end{equation}
where $\G'_i$ is obtained from   $\G \circ_i \G_{\ed}$
by discarding all graphs in which either vertex $i$ or vertex 
$i+1$ has valency $1$\,.
\end{prop}
\begin{proof} It is straightforward to verify 
the first claim by unfolding the definition of 
the Lie bracket on $\Conv(\La^2\coCom, \Gra)$\,.
The second claim follows from the observation
that 
$$
  \Av( \G_{\ed} \circ_1 \G )  = \frac{(-1)^{e(\G)}}{2}
\sum_{i=1}^n \Av( \wt{\G}_i )
$$
where $\wt{\G}_i$ is obtained from the linear combination 
$\G \circ_i \G_{\ed}$ by keeping only the graphs in which either
vertex $i$ or vertex $i+1$ has valency $1$\,.
\end{proof}

\begin{exer}
\label{exer:G-bul-G-lp}
Let $\G_{\bul}$ be the graph in $\gra_1$ which consists of a
single vertex. Show that 
\begin{equation}
\label{G-bul-G-ed}
\pa \G_{\bul} = \G_{\ed}\,. 
\end{equation}
Let $\G_{\lp}$ be the graph in $\gra_1$ with consists of 
a single loop. Show that 
$$
\pa \G_{\lp} = 0\,.
$$
Thus $\G_{\lp}$ represents a degree $-1$ (non-trivial) cocycle in $\fGC$\,. 
\end{exer}

\subsection{The subcomplex of cables} 
\label{sec:cables}
 
Let us denote by $\cK_{\cab}$ the subspace of $\fGC$ which 
is spanned by vectors 
$$
\Av(\G)
$$
where $\G$ is either the single vertex graph $\G_{\bul}$
or a graph  $\G^{-}_{l}$ depicted on figure \ref{fig:G-l} for 
$l \ge 2$\,. For example, $\G^{-}_{2} = \G_{\ed}$\,.
\begin{figure}[htp]
\centering 
\begin{tikzpicture}[scale=0.5, >=stealth']
\tikzstyle{w}=[circle, draw, minimum size=4, inner sep=1]
\tikzstyle{b}=[circle, draw, fill, minimum size=4, inner sep=1]
\node [b] (b1) at (0,0) {};
\draw (0,-0.6) node[anchor=center] {{\small $1$}};
\node [b] (b2) at (2,0) {};
\draw (2,-0.6) node[anchor=center] {{\small $2$}};
\node [b] (b3) at (4,0) {};
\draw (4,-0.6) node[anchor=center] {{\small $3$}};
\node [b] (b4) at (6,0) {};
\draw (6,-0.6) node[anchor=center] {{\small $4$}};
\node [b] (bl) at (10,0) {};
\draw (10,-0.6) node[anchor=center] {{\small $l$}};
\draw  (b1) edge (b2);
\draw (1,0.6) node[anchor=center] {{\small $i$}};
\draw  (b2) edge (b3);
\draw (3,0.6) node[anchor=center] {{\small $ii$}};
\draw  (b3) edge (b4);
\draw (5,0.6) node[anchor=center] {{\small $iii$}};
\draw  (b4) edge (7,0);
\draw (8,0) node[anchor=center] {{$\dots$}};
\draw  (9,0) edge (bl);
\end{tikzpicture}
~\\[0.3cm]
\caption{The graph $\G^{-}_{l}$} \label{fig:G-l}
\end{figure} 

It is easy to see that the vectors $\Av(\G^{-}_l)$ have degrees
$$
|\Av(\G^{-}_l)| = l-1\,,
$$
$$
\Av(\G^{-}_l) = 0\,, \qquad \textrm{if} \qquad l = 0, 3~ \textrm{mod}~ 4,
$$
and 
$$
\Av(\G^{-}_l) \neq 0\,, \qquad  \textrm{if} \qquad l = 1, 2~ \textrm{mod}~ 4\,.
$$
Furthermore, due to Exercise  \ref{exer:cables} below, 
$$
\pa \Av(\G^{-}_{4k+1}) = \Av(\G^{-}_{4k+2})
$$ 
for all $k \ge 1$\,.

Combining these observations with equation \eqref{G-bul-G-ed}
we conclude that 
\begin{prop}
\label{prop:cables}
The subspace $\cK_{\cab}$ is subcomplex of $\fGC$. 
Moreover $\cK_{\cab}$ is acyclic. $\Box$
\end{prop}
We call $\cK_{\cab}$ the subcomplex of {\it cables}.

\begin{exer}
\label{exer:cables}
Let $\G^{-}_{l}$ be the family of graphs for $l \ge 2$ defined 
on figure \ref{fig:G-l}. Prove that for every $k \ge 1$
$$
\pa \Av(\G^{-}_{4k+1}) = \Av(\G^{-}_{4k + 2})\,. 
$$
\end{exer}

\subsection{The subcomplex of polygons} 
\label{sec:dia} 
Let us denote by $\cK_{\dia}$ the subspace of $\fGC$ which 
is spanned by vectors of the form 
$$
\Av(\G^{\dia}_{m}),
$$
where $\G^{\dia}_m$ is the element of $\gra_m$ depicted 
on figure \ref{fig:dia-m}. 
\begin{figure}[htp]
\centering 
\begin{tikzpicture}[scale=0.5, >=stealth']
\tikzstyle{w}=[circle, draw, minimum size=4, inner sep=1]
\tikzstyle{b}=[circle, draw, fill, minimum size=4, inner sep=1]
\node [b] (b1) at (2,0) {};
\draw (2.5, 0) node[anchor=center] {{\small $1$}};
\node [b] (b2) at (0.62,1.90) {};
\draw (0.77, 2.4) node[anchor=center] {{\small $2$}};
\node [b] (b3) at (-1.62,1.18) {};
\draw (-2.02, 1.47) node[anchor=center] {{\small $3$}};

\node [b] (b5) at (0.62,-1.90) {};
\draw (0.77, -2.4) node[anchor=center] {{\small $m$}};
\draw (b1) edge (b2);
\draw (b2) edge (b3);
\draw (b3) edge (-1.82,0.2);
\draw (-1.9, -0.62 ) node[anchor=center] {{\small $\cdot$}};
\draw (-1.62, -1.18) node[anchor=center] {{\small $\cdot$}};
\draw (-1.18, -1.62) node[anchor=center] {{\small $\cdot$}};
\draw (-0.5,-1.90) edge (b5);
\draw (b5) edge (b1);
\end{tikzpicture}
~\\[0.3cm]
\caption{The edges are ordered as follows $(1,2) < (2,3) < \dots < (m-1,m) < (m,1)$} \label{fig:dia-m}
\end{figure}
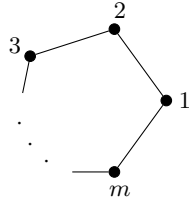
For example, $\G^{\dia}_{1}$ is the graph $\G_{\lp}$ in $\gra_1$
which consists of a single loop. 

Due to Exercise \ref{exer:polygons} below, $\cK_{\dia}$ is 
a subcomplex of $\fGC$ with 
\begin{equation}
\label{H-cK-dia}
H^{\bul} (\cK_{\dia}) \cong \bigoplus_{q \ge 1} \bs^{4q-1}\, \bbK\,. 
\end{equation}
We call $\cK_{\dia}$ the subcomplex of polygons. 

\begin{exer}
\label{exer:polygons}
Show that the graph $\G^{\dia}_m$ is odd if $m \neq 1~\textrm{mod}~ 4$
and even if $m = 1~\textrm{mod}~ 4$\,. Using equation \eqref{pa-dfGC-simpler}, 
prove that for every $q \ge 0$ 
$$
\Av(\G^{\dia}_{4 q+1})
$$
is a non-trivial cocycle of $\fGC$ of degree $4q-1$\,.
\end{exer}

\subsection{The connected part $\fGC_{\conn}$ of $\fGC$}
\label{sec:fGC-conn}

Let us denote by  $\fGC_{\conn}$ the subspace of $\fGC$ 
which consists of infinite sums
$$
\ga = \sum_{n=1}^{\infty} \ga_n\,, \qquad 
\ga_n \in \bs^{2n-2} \big( \Gra(n) \big)^{S_n}
$$
where $\ga_n$ is a linear combination of 
connected graphs in $\Gra(n)$\,.

It is clear that $ \fGC_{\conn}$ is a Lie subalgebra of $\fGC$ and
hence a subcomplex. It is also clear that 
\begin{equation}
\label{fGC-fGC-conn}
\fGC = \bs^{-2}  \wh{S}\big(\bs^2 \, \fGC_{\conn}\big)\,,
\end{equation}
where $\wh{S}$ denotes the completed symmetric algebra. 
Thus the question of computing cohomology of $\fGC$
reduces to the question of computing cohomology 
of its connected part $\fGC_{\conn}$\,.

\section{Analyzing the dg operad $\Tw\Gra$}
\label{sec:TwGra}

According to the general procedure of twisting 
\begin{equation}
\label{TwGra-n}
\Tw\Gra(n) = \prod_{r=0}^{\infty} \bs^{2r} \Big( \Gra(r+n) \Big)^{S_r}\,.
\end{equation}
In other words, vectors in  $\Tw\Gra(n)$ are infinite linear 
combinations
\begin{equation}
\label{sum}
\ga =  \sum_{r=0}^{\infty} \ga_r\,, 
\end{equation}
where $\ga_r$ is an $S_r$ invariant vector in 
$\bs^{2r}\Gra(r+n)$\,.

It is clear that the first $r$ vertices and the 
last $n$ vertices in graphs of $\ga_r$ play different roles. 
We call the first $r$ vertices {\it neutral} and the remaining 
$n$ vertices {\it operational}. It is convenient to represent 
neutral (reps. operational) vertices on figures by small black circles 
(reps. small white circles). In this way, the same element of $\gra_m$ may be treated 
as a vector in different spaces of the operad $\Tw\Gra$.
For example, the graph on figure \ref{fig:G-bb} represents a vector
in $\Tw\Gra(0)$, the graph on figure \ref{fig:G-bw} represents a vector 
in $\Tw\Gra(1)$, and the graph on figure \ref{fig:G-ww} represents a vector 
in $\Tw\Gra(2)$. 
\begin{figure}[htp]
\centering 
\begin{minipage}[t]{0.28\linewidth}
\centering 
\begin{tikzpicture}[scale=0.5, >=stealth']
\tikzstyle{w}=[circle, draw, minimum size=4, inner sep=1]
\tikzstyle{b}=[circle, draw, fill, minimum size=4, inner sep=1]
\node [b] (b1) at (0,0) {};
\draw (0,0.6) node[anchor=center] {{\small $1$}};
\node [b] (b2) at (2,0) {};
\draw (2,0.6) node[anchor=center] {{\small $2$}};
\draw  (b1) edge (b2);
\end{tikzpicture}
~\\[0.3cm]
\caption{A vector in $\Tw\Gra(0)$} \label{fig:G-bb}
\end{minipage}
\hspace{0.03\linewidth}
\begin{minipage}[t]{0.28\linewidth}
\centering 
\begin{tikzpicture}[scale=0.5, >=stealth']
\tikzstyle{w}=[circle, draw, minimum size=4, inner sep=1]
\tikzstyle{b}=[circle, draw, fill, minimum size=4, inner sep=1]
\node [b] (b1) at (0,0) {};
\draw (0,0.6) node[anchor=center] {{\small $1$}};
\node [w] (b2) at (2,0) {};
\draw (2,0.6) node[anchor=center] {{\small $2$}};
\draw  (b1) edge (b2);
\end{tikzpicture}
~\\[0.3cm]
\caption{A vector in $\Tw\Gra(1)$} \label{fig:G-bw}
\end{minipage}
\hspace{0.03\linewidth}
\begin{minipage}[t]{0.28\linewidth}
\centering 
\begin{tikzpicture}[scale=0.5, >=stealth']
\tikzstyle{w}=[circle, draw, minimum size=4, inner sep=1]
\tikzstyle{b}=[circle, draw, fill, minimum size=4, inner sep=1]
\node [w] (b1) at (0,0) {};
\draw (0,0.6) node[anchor=center] {{\small $1$}};
\node [w] (b2) at (2,0) {};
\draw (2,0.6) node[anchor=center] {{\small $2$}};
\draw  (b1) edge (b2);
\end{tikzpicture}
~\\[0.3cm]
\caption{A vector in $\Tw\Gra(2)$} \label{fig:G-ww}
\end{minipage}
\end{figure}

It is obvious that the vector space  
$$
\bs^{2r}\big( \Gra(r+n)\big)^{S_r}  \subset \Tw\Gra(n)
$$
is spanned by vectors of the form 
\begin{equation}
\label{Av-r-G}
\Av_r(\G) = \sum_{\si \in S_r } \si(\G)\,,
\end{equation}
where $\G$ is an element in $\gra_{r+n}$\,.

In other words, equation \eqref{Av-r-G} defines 
a surjective map 
\begin{equation}
\label{Av-r}
\Av_r : \bbK \L \gra_{r+n} \R  \onto  \bs^{2r}\big( \Gra(r+n)\big)^{S_r}\,. 
\end{equation}

For an element $\G\in \gra_{r+n}$ we denote by 
$\G^{\oub}$ the partially labeled graph which is obtained from 
$\G$ by forgetting labels on neutral vertices and shifting 
labels on operational vertices down by $r$\,. Note that, 
since $\G^{\oub}$ has unlabeled vertices, it may have 
non-trivial automorphisms. 

It is obvious that $\Av_r(\G)=0$ if and only if $\G^{\oub}$
has an automorphism which induces an odd permutation on 
the set of edges. In this case, we say that an element 
$\G \in \gra_{r+n}$ is {\it $r$-odd}. Otherwise, we say that $\G$ is 
{\it $r$-even}. 

Let us consider two $r$-even elements $\G, \G'\in \gra_{r+n}$ whose 
underlying partially labeled graphs $\G^{\oub}$ and $(\G')^{\oub}$
are isomorphic. 
Any isomorphism from $\G^{\oub}$ to $(\G')^{\oub}$
gives us a bijection from the set $E(\G)$ of edges of $\G$ to 
the set $E(\G')$ of edges of $\G'$\,. Since both sets 
$E(\G)$  and  $E(\G')$ are totally ordered, this bijection 
determines a permutation $\si \in S_e$ where $e = |E(\G)|$\,.  
Furthermore, since $\G$ and $\G'$ are $r$-even, such permutations
$\si$ are either all even or all odd. 
In the latter case, we say that $r$-even elements $\G$ and $\G'$ are {\it $r$-opposite} 
 and in the former case we say that even elements $\G$ and 
$\G'$ are {\it $r$-concordant}.

It is clear that 
\begin{prop}
\label{prop:ker-Av-r}
The kernel of the map $\Av_r$ \eqref{Av-r} 
is spanned by vectors of the form 
\begin{equation}
\label{rel-Gra-inv-r}
\G, \qquad \G_1 - \G_2, \qquad \G'_1 + \G'_2\,,
\end{equation}
where $\G$ is $r$-odd, $(\G_1, \G_2)$ is a 
pair of $r$-concordant ($r$-even) graphs, and  $(\G'_1, \G'_2)$ is a 
pair of $r$-opposite ($r$-even) graphs in $\gra_{r+n}$\,. $\Box$
\end{prop}

In the following proposition we give a convenient 
formula for the differential on $\Tw\Gra$\,.
\begin{prop}
\label{prop:diff-TwGra}
Let $\G$ be an $r$-even element in $\gra_{r+n}$\,.
Then
\begin{equation}
\label{diff-TwGra}
\pa^{\Tw} \Av_r (\G) = \Av_{r+1} \big( \G_{\ed} \circ_2 \G \big) 
- (-1)^{e(\G)}  \Av_{r+1} \Big( 
\sum_{i = 1}^{n}  \vs_{r+1, r+i} \big( \G \circ_{r+i} \G_{\ed} \big)  \Big)
\end{equation}
$$
- \frac{(-1)^{e(\G)}}{2} \sum_{i=1}^{r} 
 \Av_{r+1} \big( \G \circ_i \G_{\ed} \big)\,,
$$
where $e(\G)$ is the number of edges of $\G$, 
$\G_{\ed}$ is defined in \eqref{binary}, and $\vs_{r+1, r+i}$ 
is the cycle $(r+1, r+2, \dots, r+i)\in S_{r+1+n}$\,.
\end{prop}
\begin{rem}
\label{rem:neutral-in-gra}
We should remark that 
the vector $\pa^{\Tw} \Av_r (\G) $ is a linear combination 
of graphs in $\gra_{r+1+n}$ in which the first $r+1$ vertices 
are treated as neutral. Thus vertices with labels $r+1$ and 
$r+i+1$ in graphs in 
$ \vs_{r+1, r+i} \big( \G \circ_{r+i} \G_{\ed} \big) $ come from 
$\G_{\ed}$. The vertex with label $r+1$ is treated as 
neutral and the vertex with label $r+ i +1$ is treated as 
operational.   
\end{rem}

\begin{proof}
Adapting general formula \eqref{diff-Tw-cO} to the case
when $\cO= \Gra$ we get
$$
\pa^{\Tw} \Av_r (\G) = \sum_{\tau \in \Sh_{1, r}} 
\tau \big(\G_{\ed} \circ_2 \Av_r(\G)  \big)
$$
\begin{equation}
\label{pa-Tw-work} 
- (-1)^{e(\G)}  
\sum_{\tau' \in \Sh_{r,1}}\sum_{i=1}^{n} \tau' \circ \vs_{r+1, r+i} 
\big( \Av_r(\G)  \circ_{r+i} \G_{\ed} \big)
\end{equation}
$$
- (-1)^{e(\G)}  \sum_{\la \in \Sh_{2, r-1}} \big( \Av_r(\G) \circ_{1} \G_{\ed} \big) \,. 
$$

Using the obvious identity 
$$
\Av_r(\G) = \sum_{i=1}^{r}\sum_{\si' \in S_{2, \dots, r}} 
\si' \circ \vs_{1,i} (\G),  
$$ 
axioms of operad, and $S_2$-invariance of $\G_{\ed}$
we rewrite the last sum in \eqref{pa-Tw-work}
as follows
$$
 \sum_{\la \in \Sh_{2, r-1}} \la \big( \Av_r(\G) \circ_{1} \G_{\ed} \big)= 
\sum_{i =1}^r \sum_{\si \in S_{r+1}}^{\si(i) < \si(i+1)} 
\si \big( \G \circ_i \G_{\ed} \big) =
$$
$$
\frac{1}{2}  \sum_{\si \in S_{r+1}} \sum_{i =1}^r 
\si \big( \G \circ_i \G_{\ed} \big) =\frac{1}{2} \Av_{r+1}
\Big( \sum_{i=1}^r \G \circ_i \G_{\ed}  \Big)\,.
$$

The first sum in the right hand side of \eqref{pa-Tw-work}
can be rewritten as 
$$
\sum_{\tau \in \Sh_{1, r}} 
\tau \big(\G_{\ed} \circ_2 \Av_r(\G)  \big)
= \sum_{\tau \in \Sh_{1, r}} 
\sum_{\si' \in S_{2, \dots, r+1}}
 \tau \circ \si' \big(\G_{\ed} \circ_2 \G \big) =
$$
$$
\sum_{\si \in S_{r+1}} 
\big(\G_{\ed} \circ_2 \G \big) = \Av_{r+1} \big(\G_{\ed} \circ_2 \G \big)\,,
$$
where $S_{2, \dots, r+1}$ denotes the permutation group
of the set $\{ 2, \dots, r+1\}$\,.

As for the second sum in the right hand side of \eqref{pa-Tw-work},  
we have 
$$
\sum_{\tau' \in \Sh_{r,1}}\sum_{i=1}^{n} \tau' \circ \vs_{r+1, r+i} 
\big( \Av_r(\G)  \circ_{r+i} \G_{\ed} \big) = 
\sum_{\tau' \in \Sh_{r,1}}\sum_{i=1}^{n} 
\sum_{\si' \in S_r} \tau' \circ \vs_{r+1, r+i} 
\circ \si' \big( \G \circ_{r+i} \G_{\ed}   \big) =
$$
$$
\sum_{i =1}^{n} \sum_{\si \in S_{r+1}} \si \circ \vs_{r+1, r+i} 
\big( \G \circ_{r+i} \G_{\ed}   \big) = \Av_{r+1}
\Big(  \sum_{i=1}^n  \vs_{r+1, r+i} 
\big( \G \circ_{r+i} \G_{\ed}  \big) \Big)\,.
$$

Thus, equation \eqref{diff-TwGra} indeed holds. 
\end{proof}

\subsection{The Euler characteristic trick}
\label{sec:Euler}
Let us consider sums \eqref{sum} satisfying
\begin{pty}
\label{P:Euler}
For every $r \ge 0$, each graph in the linear combination
$\ga_r$ has Euler characteristic $\chi$\,. 
\end{pty} 

Using equation \eqref{diff-TwGra}, it is not hard to see that the 
subspace of such sums is a subcomplex in $\Tw\Gra(n)$\,. 
We denote this subcomplex by 
\begin{equation}
\label{TwGra-chi}
\Tw\Gra_{\chi}(n)\,.
\end{equation}

We claim that 
\begin{prop}
\label{prop:Euler}
For every triple of integers $n \ge 0, m$ and $\chi$ the subspace 
$\Tw\Gra_{\chi}(n)^m$ of degree $m$ vectors in $\Tw\Gra_{\chi}(n)$
is spanned by graphs with 
\begin{equation}
\label{e-m-chi-n}
e = 2(n - \chi) + m
\end{equation}
edges and
\begin{equation}
\label{r-m-chi-n}
r = n + m - \chi
\end{equation}
neutral vertices. In particular, the subspace $\Tw\Gra_{\chi}(n)^m$ 
is finite dimensional.
\end{prop}

\begin{proof}
Recall that for every graph $\G \in \gra_{r+n}$ the vector 
$\Av_r(\G)\in \Tw\Gra(n)$ has degree
$$
\big| \Av_{r}(\G) \big| = 2 r - e\,,
$$
where $e$ is the number of edges of $\G$\,.

Hence, if $\Av_r(\G) \in \Tw\Gra_{\chi}(n)^{m}$ then 
\begin{equation}
\label{m-equals}
m = 2r - e\,,
\end{equation}
and
\begin{equation}
\label{chi-equals}
\chi = n+r - e\,.
\end{equation}

Subtracting \eqref{m-equals} from \eqref{chi-equals}, we get 
$$
\chi - m = n - r\,.
$$
Therefore, 
$$
r = n +m - \chi
$$
and 
$$
e = 2n - 2 \chi + m \,.
$$

Thus the proposition follows from the fact that
the number of graphs with a fixed number of 
vertices and a fixed number of edges is finite.
\end{proof}

Proposition \ref{prop:Euler} has the following useful corollary.
\begin{cor}
\label{cor:Euler}
The cochain complex $\Tw\Gra(n)$ decomposes into 
the product of sub-complexes
\begin{equation}
\label{Euler}
\Tw\Gra(n) = \prod_{\chi \in \bbZ} \Tw\Gra_{\chi}(n)\,.
\end{equation}
\end{cor}
\begin{proof}
Let 
$$
\ga= \sum_{r =1}^{\infty} \ga_r\,, \qquad 
\ga_r \in \bs^{2r}\big( \Gra(r+n) \big)^{S_r}
$$
be a vector of degree $m$\,. 

Equations \eqref{m-equals} and \eqref{chi-equals} imply that
for every $r$
$$
\ga_r \in  \Tw\Gra_{\chi}(n)
$$
where 
$$
\chi = n+m-r\,.
$$

Thus
$$
\Tw\Gra(n) \subset \prod_{\chi \in \bbZ} \Tw\Gra_{\chi}(n)\,.
$$

The inclusion 
$$
\prod_{\chi \in \bbZ} \Tw\Gra_{\chi}(n) \subset \Tw\Gra(n)
$$
is proved in a similar way.  
\end{proof}

\begin{rem}
\label{rem:Euler-trick}
We will often need to prove that any cocycle in $\Tw\Gra(n)$ or 
a similar cochain complex is cohomologous to a cocycle 
satisfying a certain property.  Proposition \ref{prop:Euler} 
and Corollary \ref{cor:Euler}
(or its corresponding versions) will allow us to reduce such 
questions to the corresponding questions for \und{finite} sums of graphs. 
We will refer to this maneuver as {\it the Euler characteristic trick}. 
\end{rem}

\subsection{The suboperads $\sGraphs \subset \sfGraphs \subset \Tw\Gra$}
\label{sec:tele-sharp}

Let us denote by $\sfGraphs(n)$ the subspace of $\Tw\Gra(n)$
which consists of linear combinations \eqref{sum} satisfying 
\begin{pty}
\label{P:no-cab:no-poly}
If a connected component of a graph in $\ga_r$ for some $r > 0$
has no operational vertices then this connected component
has at least one vertex of valency $\ge 3$.
\end{pty}
\begin{rem}
\label{rem:bival-uni}
It is not hard to see that, if all vertices of 
a connected graph $\G$ have valencies $\le 2$
then $\G$ is isomorphic to one of the graphs   
in the list: $\G_{\bul}$, $\G^{-}_l$ (see figure \ref{fig:G-l}), or $\G^{\dia}_m$
(see figure \ref{fig:dia-m}). In other words, $\sfGraphs(n)$ is 
obtained from $\Tw\Gra(n)$ by ``throwing away'' graphs which have 
connected components  $\G_{\bul}$, $\G^{-}_l$ (see figure \ref{fig:G-l}), or $\G^{\dia}_m$
(see figure \ref{fig:dia-m}) with all neutral vertices. 
\end{rem}

Let us also denote by  $\sGraphs(n)$ the subspace of $\Tw\Gra(n)$
which consists of linear combinations \eqref{sum} whose neutral 
vertices all have valencies $\ge 3$\,. 

We claim that 
\begin{prop}
\label{prop:sfGraphs}
Both  $\sGraphs(n)$  and $\sfGraphs(n)$ are subcomplexes 
of $\Tw\Gra(n)$
\begin{equation}
\label{inclusions}
\sGraphs(n) \subset \sfGraphs(n) \subset \Tw\Gra(n)\,.
\end{equation}
Moreover, the collections
$$
\{\sGraphs(n) \}_{n \ge 0}\,, \qquad 
\{\sfGraphs(n) \}_{n \ge 0}
$$
are suboperads of $\Tw\Gra$\,. 
\end{prop}
\begin{proof}
The only non-obvious statement in this proposition
is that the subspace $\sGraphs(n)$
is closed with respect to the differential $\pa^{\Tw}$\,.

So let us denote by $\G$ an $r$-even graph in $\gra_{r+n}$ whose 
neutral vertices all have valencies $\ge 3$ and 
analyze the right hand side of \eqref{diff-TwGra}.  

All graphs in the first linear combination in the right hand side of  \eqref{diff-TwGra}
have a univalent neutral vertex.  However, it is not hard to 
see that they cancel with the corresponding terms in the second and the third linear 
combinations in the right hand side of  \eqref{diff-TwGra}. 

Graphs with bivalent neutral vertices come from both the second 
and third linear combinations of the right hand side of  \eqref{diff-TwGra}. 
Again, it is not hard to see that these contributions cancel each other. 

The proposition is proved. 
\end{proof}

The goal of this subsection is to prove that 
\begin{prop}
\label{prop:sGra-sfGra}
The embedding 
\begin{equation}
\label{emb-1-sharp}
\emb^{\sharp}_1 : \sGraphs(n) \hookrightarrow \sfGraphs(n)
\end{equation}
is a quasi-isomorphism.
\end{prop}
\begin{proof}
Let us denote by $\sgraphs(n)$ (resp.  $\sfgraphs(n)$) the subcomplex 
of $\sGraphs(n)$ (resp. $\sfGraphs(n)$) which consists of {\bf finite} linear 
combinations \eqref{sum} in $\sGraphs(n)$ (resp. $\sfGraphs(n)$).
In other words, 
\begin{equation}
\label{sgraphs-sfgraphs}
\sgraphs(n) : = \sGraphs(n) \cap \Tw^{\oplus} \Gra\,, 
\qquad
\sfgraphs(n) : = \sfGraphs(n) \cap \Tw^{\oplus} \Gra\,.
\end{equation}

Next we observe 
that the cochain complexes  $\sGraphs(n)$ and $\sfGraphs(n)$
admit decompositions with respect to the Euler characteristic 
$$
\sGraphs(n) = \prod_{\chi\in \bbZ} \sGraphs(n) \cap \Tw\Gra_{\chi}(n)\,, 
\qquad 
\sfGraphs(n) = \prod_{\chi\in \bbZ} \sfGraphs(n) \cap \Tw\Gra_{\chi}(n)
$$
and Proposition \ref{prop:Euler} implies that the subspace of 
elements of fixed degree in $ \sGraphs(n) \cap \Tw\Gra_{\chi}(n)$ and in 
$\sfGraphs(n) \cap \Tw\Gra_{\chi}(n)$ is spanned by a finite number of 
graphs. 

Thus, in virtue of  Remark \ref{rem:Euler-trick}, it suffices to prove that 
the embedding 
\begin{equation}
\label{sgra-sfgra}
\sgraphs(n) \hookrightarrow \sfgraphs(n)
\end{equation}
is a quasi-isomorphism. 

Let $\G$ be an element in $\gra_{r+n}$ such that 
$\Av_r(\G)$ represents a vector in $\sfgraphs(n)$\,.  
Let us denote by $\nu_2(\G)$ the number of 
neutral vertices having valency $2$\,.

It is clear that the linear combination  
$$
\pa^{\Tw} \Av(\G)
$$
may involve only graphs $\G'$ with $\nu_2(\G') = \nu_2(\G)$
or  $\nu_2(\G') = \nu_2(\G) + 1$\,.

Thus we may introduce on the complex $\sfgraphs(n)$ an 
ascending filtration 
\begin{equation}
\label{cF-sfgraphs}
\dots \subset \cF^{m-1} \sfgraphs(n)
\subset  \cF^m \sfgraphs(n) \subset  \cF^{m+1} \sfgraphs(n)
\subset \dots
\end{equation}
where $ \cF^m \sfgraphs(n)$ consists of vectors  
$\ga \in \sfgraphs(n)$
which only involve graphs $\G$ satisfying the inequality
$$
\nu_2(\G)  -  |\ga|  \le  m\,. 
$$

It is clear that 
$$
 \cF^{m} \sfgraphs(n)
$$
does not have non-zero vectors in degree $< -m$\,.
Therefore, the filtration \eqref{cF-sfgraphs} is locally bounded 
from the left. Furthermore, since $\sfgraphs(n)$ consists of 
finite sums of graphs, 
$$
 \sfgraphs(n) = \bigcup_{m} \cF^{m} \sfgraphs(n)\,.
$$
In other words,  the filtration \eqref{cF-sfgraphs} is cocomplete.
 
It is also clear that the differential $\pa^{\Gr}$ on the associated graded 
complex 
\begin{equation}
\label{Gr-sfgraphs}
\Gr( \sfgraphs(n) ) =
\bigoplus_m \cF^m \sfgraphs(n)  ~ \Big/~ \cF^{m-1} \sfgraphs(n)\,. 
\end{equation} 
is obtained from $\pa^{\Tw}$ by keeping only the terms 
which raise the number of the bivalent neutral vertices. 

Thus, since $\sgraphs(n)$ is a subcomplex of $\sfgraphs(n)$, 
we conclude that
$$
\sgraphs(n)^k \subset  \cF^{-k} \sfgraphs(n)^k~ \cap ~ \ker \pa^{\Gr}\,, 
$$ 
where $\sgraphs(n)^k$ (resp. $ \cF^{-k} \sfgraphs(n)^k$) denotes the 
subspace of degree $k$ vectors in $\sgraphs(n)$  (resp. in $ \cF^{-k} \sfgraphs(n)$)\,.

To complete the proof of the proposition, we need the 
following technical lemma which is proved in Subsection 
\ref{sec:lem:Gr} below.
\begin{lem}
\label{lem:Gr}
For the filtration  \eqref{cF-sfgraphs}  on $\sfgraphs(n)$ we have 
\begin{equation}
\label{Gr-sfgraphs-k-m}
 H^{k}\Big( \cF^m \sfgraphs(n)  \big/  \cF^{m-1} \sfgraphs(n) \Big) = 0 
\end{equation}
for all $m > - k$\,. Moreover, 
\begin{equation}
\label{sfgraphs-OK}
\sgraphs(n)^k =  \cF^{-k} \sfgraphs(n)^k~ \cap ~ \ker \pa^{\Gr}\,.
\end{equation} 
\end{lem}

It is easy to see that the restriction of  
\eqref{cF-sfgraphs} to the subcomplex 
$\sgraphs(n)$ gives us the ``silly'' filtration: 
\begin{equation}
\label{silly-filtr}
\cF^{m} \sgraphs(n)^k  = 
\begin{cases}
 \sgraphs(n)^k \qquad {\rm if} ~~ m \ge -k\,,  \\
 \bfzero  \qquad {\rm otherwise}\,.
\end{cases} 
\end{equation}
The associated graded complex $\Gr(\sgraphs(n))$ for this 
filtration has the zero differential.

Since 
$$
 \cF^{m} \sfgraphs(n)^k = \bfzero \qquad \forall ~~ m < -k\,,
$$
we have 
$$
\cF^{-k} \sfgraphs(n)^k~ \cap ~ \ker \pa^{\Gr} = 
  H^k \big( \, \cF^{-k} \sfgraphs(n) \big/  \cF^{-k-1} \sfgraphs(n)  \, \big)\,.
$$

Thus, Lemma \ref{lem:Gr} implies that, the embedding  \eqref{sgra-sfgra} 
induces a quasi-isomorphism of cochain complexes 
$$
\Gr (\sgraphs(n))  \stackrel{\sim}{\longrightarrow}  \Gr(\sfgraphs(n))\,.
$$

On the other hand, both filtrations \eqref{cF-sfgraphs} and 
\eqref{silly-filtr} are locally bounded from 
the left and cocomplete. 

Therefore the embedding \eqref{sgra-sfgra}
satisfies all the conditions of Lemma \ref{lem:q-iso} from Appendix \ref{app:q-iso}
and Proposition \ref{prop:sGra-sfGra} follows. 
\end{proof}

\subsubsection{An alternative description of $\Gr( \sfgraphs(n) )$}
\label{sec:Gr-sfgraphs}

In order to prove Lemma \ref{lem:Gr} we need a convenient  
description of the associated graded complex \eqref{Gr-sfgraphs}.

For this purpose we introduce three cochain complexes: 
\begin{itemize}

\item The first cochain complex is the tensor algebra 
\begin{equation}
\label{T-a}
T_a = T(\bbK \L a \R) 
\end{equation}
in a single variable $a$ carrying degree $1$ with the 
differential  $\de$ defined by the formula
\begin{equation}
\label{de-T-a}
\de (a) =  a^2\,.
\end{equation}

\item The second cochain complex is 
the truncation of the above tensor algebra
\begin{equation}
\label{tT-a}
\und{T}_{\,a} = \und{T}(\bbK \L a \R) = \bbK \L a \R \oplus \bbK \L a \otimes a \R 
\oplus \bbK \L a \otimes a \otimes a \R \oplus \dots
\end{equation}
with the same differential \eqref{de-T-a}.

\item Finally, the third cochain complex 
\begin{equation}
\label{L}
L = \bbK\L \{ \ml_{n} \}_{n > 0,~ n =0,1 \textrm{ mod } 4} \R 
\end{equation}
has the basis vector  $\{ \ml_{n} \}_{n > 0,~ n =0,1 \textrm{ mod } 4}$
carrying degrees 
$$
| \ml_n|= n -2\,.
$$
The differential on \eqref{L} is given by the formulas:
\begin{equation}
\label{de-L}
\de (\ml_{4k}) = - \ml_{4k+1}\,, \qquad 
\de (\ml_{4k+1}) = 0\,. 
\end{equation}

\end{itemize}
It is easy to see that the cochain complex 
$(\und{T}_{\,a}, \de)$ is acyclic,
 \begin{equation}
\label{H-T-a}
H^{\bul}(T_a) = 
\begin{cases}
\bbK \qquad {\rm if} ~~ \bul = 0\,,  \\
\bfzero \qquad {\rm otherwise}\,,
\end{cases}
\end{equation}
and
\begin{equation}
\label{H-L}
H^{\bul}(L) = 
\begin{cases}
\bbK \qquad {\rm if} ~~ \bul = -1\,,  \\
\bfzero \qquad {\rm otherwise}\,.
\end{cases}
\end{equation}
Moreover $H^0(T_a, \de)$ is 
spanned by the class of $1$ and  $H^{-1}(L) $ is spanned by the cohomology class 
of $\ml_1$\,.

Next, to every pair of non-negative integers $r, n$ satisfying $r+n > 0$ 
we assign an auxiliary groupoid $\Frame_{r,n}$\,. 
An object of this groupoid is a labeled \und{directed} graph $\Gim$ with $r+n$ vertices
and with an additional piece of data: the set $E(\Gim)$ of edges of $\Gim$ is 
equipped with a total order.  For our purposes, we call
the first $r$ vertices neutral and the last $n$ vertices operational. 
(On figures we use small black circles (resp. small white circles) for
neural (resp. operational) vertices.)  
Each object $\Gim \in \Frame_{r,n}$ obeys the following properties: 
\begin{itemize}

\item $\Gim$ does not have bivalent neutral 
vertices;

\item $\Gim$ does not have a 
connected component which consists of a single neutral 
vertex; 

\item $\Gim$ does not have a connected component which 
consists of a single edge which connects two neutral vertices;

\item each edge adjacent to a univalent neutral vertex (if any)
of $\Gim$ originates at this univalent neutral vertex; 
 
\item the set $E(\Gim)$ is ordered in 
such a way that edges adjacent to  univalent neutral vertices (if any)
are smaller than all the remaining edges; 

\item finally, loops of $\Gim$ (if any) are bigger than all the remaining 
edges.

\end{itemize}
Objects of the groupoid $\Frame_{r,n}$ are called {\it frames}.

A morphism from a frame $\Gim$ to a frame $\Gim'$ is 
an isomorphism of the underlying graphs which respects  
labels only on the operational vertices and respects
neither labels on neutral vertices, nor the total 
order on the set of edges, nor the directions of
edges. 

\begin{example}
\label{exam:gimel}
Let  $\Gim$ be the frame in $\Frame_{3,4}$  depicted on 
figure \ref{fig:Gim-3-4}. Let $g_1$ be the automorphism 
of $\Gim$ which  swaps the first edge 
with the second edge and $g_2$ be  the automorphism 
of $\Gim$ which 
swaps the fifth edge with the sixth edge.
It is obvious that $\Aut(\Gim)$ is generated by $g_1$ and $g_2$. 
Moreover, $\Aut(\Gim) \cong S_2 \times S_2$\,.
\begin{figure}[htp]
\centering 
\begin{tikzpicture}[scale=0.5, >=stealth']
\tikzstyle{w}=[circle, draw, minimum size=4, inner sep=1]
\tikzstyle{b}=[circle, draw, fill, minimum size=4, inner sep=1]
\node [b] (b1) at (0,1) {};
\draw (0,1.6) node[anchor=center] {{\small $1$}};
\node [b] (b2) at (0,-1) {};
\draw (0,-1.6) node[anchor=center] {{\small $2$}};
\node [b] (b3) at (2,0) {};
\draw (2,-0.6) node[anchor=center] {{\small $3$}};
\node [w] (b4) at (4,1) {};
\draw (4,1.6) node[anchor=center] {{\small $4$}};
\node [w] (b5) at (4,-1) {};
\draw (4,-1.6) node[anchor=center] {{\small $5$}};
\node [w] (b6) at (7,1) {};
\draw (7,1.6) node[anchor=center] {{\small $6$}};
\node [w] (b7) at (7,-1) {};
\draw (7,-1.6) node[anchor=center] {{\small $7$}};
\draw [->] (b1) edge (b3);
\draw (1,1) node[anchor=center] {{\small $i$}};
\draw  [->] (b2) edge (b3);
\draw (1,-1) node[anchor=center] {{\small $ii$}};
\draw [<-] (b3) edge (b4);
\draw (3,1) node[anchor=center] {{\small $iii$}};
\draw [->] (b3) edge (b5);
\draw [<-] (b4) .. controls (6,0.5) and (6,-0.5) .. (b5); 
\draw (6,0) node[anchor=center] {{\small $vi$}};
\draw (3,-0.85) node[anchor=center] {{\small $iv$}};
\draw [->] (b4) edge (b5);
\draw (3.7,0) node[anchor=center] {{\small $v$}};
\draw (b6) .. controls (7.5,1.5) and (8,1.5) .. (8,1) ..
controls (8,0.5) and (7.5,0.5) .. (b6);
\draw (8.5,1) node[anchor=center] {{\small $vii$}};
\end{tikzpicture}
~\\[0.3cm]
\caption{As above, we use Roman numerals to 
specify the total order on the set of edges} \label{fig:Gim-3-4}
\end{figure}
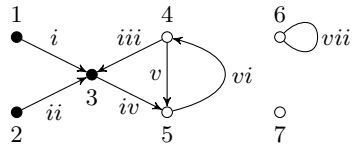 
\end{example}

The total number of edges $e$ of any frame $\Gim$ splits into the sum
$$
e = e_{\bul}  + e_{\c} + e_{-}\,,
$$
where $e_{\bul}$ is the number of 
edges of $\Gim$ adjacent to univalent neutral vertices (if any), 
$e_{\c}$ is the number of loops of $\Gim$ and $e_{-}$ is the number of 
the remaining edges. Thus, for the frame $\Gim$ on figure \ref{fig:Gim-3-4}
we have $e_{\bul}=2$, $e_{\c}=1$, and $e_{-} = 4$.

For every frame $\Gim \in \Frame_{r,n}$ we construct a linear map 
\begin{equation}
\label{F-Gim}
F_{\Gim} : \bs^{2r - 2e_{\bul}} \big( \und{T}_{\,a} \big)^{\otimes\, e_{\bul}}  
~\otimes  ~ \big( \bsi T_a \big)^{\otimes\, e_{-}}
~ \otimes ~  L^{\otimes\, e_{\c}}  
\to  \Gr (\sfgraphs(n))\,.
\end{equation}

Namely, given a collection of monomials 
$a^{k_1}, a^{k_2}, \dots, a^{ k_{e_{\bul} + e_{-}} } $   
with $k_i > 0$ for all $i \le k_{e_{\bul}}$  and 
vectors $\ml_{k_{e_{\bul} + e_{-}+1} }, \dots, \ml_{k_e}$ in $L$ 
we form a graph $\G \in \gra_{(r+r') + n}$ with
$$
r' = \sum_{i=1}^{e_{\bul}} (k_i-1) + \sum_{i={e_{\bul}+1}}^{e_{\bul} + e_{-}} k_i
+ \sum_{i = e_{\bul} + e_{-} +1}^{e}  (k_i -1) 
$$
following these steps: 

\begin{itemize}

\item  first, for each $1 \le i \le e_{\bul}$, we 
divide the $i$-th edge into $k_i$ sub-edges;

\item second, for each $e_{\bul} < i \le e_{\bul} + e_{-} $, we 
divide the $i$-th edge into $k_i + 1$ sub-edges; 

\item third, for each  $e_{\bul} + e_{-} < i \le e$,
we divide the $i$-th edge\footnote{Note that, for 
each  $e_{\bul} + e_{-} < i \le e$ the $i$-th edge is necessarily 
a loop.} into $k_i$ sub-edges;

\item we declare that the additional $r'$ vertices obtained 
in the above steps are neutral, label them 
by numbers $r+1, r+2, \dots, r+r'$  in an 
arbitrary possible way and shift labels on all operational 
vertices up by $r'$\,;

\item we order the set $E(\G)$ of edges 
of $\G$ in the following way\footnote{
The order on the set $E(\G)$ is defined up to an even 
permutation.}:
{if $s_1, s_2 \in E(\G)$ 
are parts of different edges of $\Gim$ then $s_1 < s_2$ provided 
$s_1$ is a part of a smaller edge; if  $s_1, s_2 \in E(\G)$ 
are parts of the same edge of $\Gim$ which is not a loop
then $s_1 < s_2$ provided $s_1$ is closer to the origin of its edge;
finally, we order sub-edges of each loop of  $\Gim$ by 
choosing one of the two possible directions of walking 
around the loop.}

\end{itemize}
We will refer to this graph $\G$ as the graph {\it reconstructed}
from the monomial 
$$
 (a^{k_1}, a^{k_2}, \dots, a^{k_{e_{\bul}  + e_{-}}}, \ml_{k_{e_{\bul} + e_{-} +1}}, \
\dots, \ml_{k_e}) \in  \bs^{2r - 2e_{\bul}} \big( \und{T}_{\,a} \big)^{\otimes\, e_{\bul}}  
~\otimes  ~ \big( \bsi T_a \big)^{\otimes\, e_{-}}
~ \otimes ~  L^{\otimes\, e_{\c}}  
$$
using the frame $\Gim$\,.

It is not hard to see that the equation 
\begin{equation}
\label{F-Gim-def}
F_{\Gim} (a^{k_1}, a^{k_2}, \dots, a^{k_{e_{\bul}  + e_{-}}}, \ml_{k_{e_{\bul} + e_{-} +1}}, \
\dots, \ml_{k_e}) = 
\sum_{\si \in S_{r+r'}} \si(\G)
\end{equation}
defines a (degree zero) map of graded vector spaces \eqref{F-Gim}.

\begin{example}
\label{exam:frame-graph}
Let $\Gim$ be the frame in $\Frame_{3,4}$ depicted on figure 
\ref{fig:Gim-3-4}. Then 
$$
F_{\Gim}(a,a^3, a, 1, 1, a, \ml_4) = \sum_{\si \in S_{10}} \si(\G)\,,
$$
where $\G$ is the element in $\gra_{10+4}$ depicted on figure 
\ref{fig:G-10-4}.
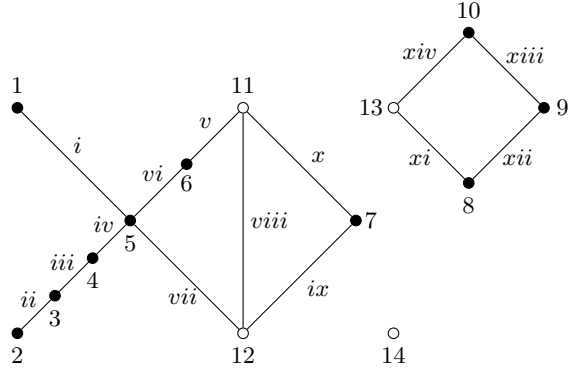
\begin{figure}[htp]
\centering 
\begin{tikzpicture}[scale=0.5, >=stealth']
\tikzstyle{w}=[circle, draw, minimum size=4, inner sep=1]
\tikzstyle{b}=[circle, draw, fill, minimum size=4, inner sep=1]
\node [b] (b1) at (0,3) {};
\draw (0,3.6) node[anchor=center] {{\small $1$}};
\node [b] (b2) at (0,-3) {};
\draw (0.3,-2.1) node[anchor=center] {{\small $ii$}};
\draw (0,-3.6) node[anchor=center] {{\small $2$}};
\node [b] (b3) at (1,-2) {};
\draw (1,-2.6) node[anchor=center] {{\small $3$}};
\draw (1.2,-1.1) node[anchor=center] {{\small $iii$}};

\node [b] (b4) at (2,-1) {};
\draw (2,-1.6) node[anchor=center] {{\small $4$}};
\draw (2.3,-0.1) node[anchor=center] {{\small $iv$}};

\node [b] (b5) at (3,0) {};
\draw (3,-0.6) node[anchor=center] {{\small $5$}};
\node [w] (b11) at (6,3) {};
\draw (6,3.6) node[anchor=center] {{\small $11$}};
\draw (5,2.6) node[anchor=center] {{\small $v$}};
\node [b] (b6) at (4.5,1.5) {};
\draw (4.5,1) node[anchor=center] {{\small $6$}};
\draw (3.6,1.25) node[anchor=center] {{\small $vi$}};

\node [b] (b7) at (9,0) {};
\draw (9.4,0) node[anchor=center] {{\small $7$}};

\draw (4.4,-2) node[anchor=center] {{\small $vii$}};
\node [w] (b12) at (6,-3) {};
\draw (6,-3.6) node[anchor=center] {{\small $12$}};
\draw (6.7,0) node[anchor=center] {{\small $viii$}};

\node [w] (b13) at (10,3) {};
\draw (9.4,3) node[anchor=center] {{\small $13$}};
\node [b] (b8) at (12,1) {};
\draw (12,0.4) node[anchor=center] {{\small $8$}};
\node [b] (b9) at (14,3) {};
\draw (14.5,3) node[anchor=center] {{\small $9$}};
\node [b] (b10) at (12,5) {};
\draw (12,5.6) node[anchor=center] {{\small $10$}};

\node [w] (b14) at (10,-3) {};
\draw (10,-3.6) node[anchor=center] {{\small $14$}};
\draw (b1) edge (b5);
\draw (1.6,2) node[anchor=center] {{\small $i$}};
\draw  (b2) edge (b5);
\draw  (b5) edge (b11);
\draw  (b5) edge (b12);
\draw  (b11) edge (b12);
\draw  (b11) edge (b7);
\draw (8,1.7) node[anchor=center] {{\small $x$}};
\draw  (b7) edge (b12);
\draw (8,-1.8) node[anchor=center] {{\small $ix$}};
\draw  (b13) edge (b8);
\draw (10.7,1.7) node[anchor=center] {{\small $xi$}};
\draw  (b8) edge (b9);
\draw (13.3,1.7) node[anchor=center] {{\small $xii$}};
\draw  (b9) edge (b10);
\draw (13.5,4.4) node[anchor=center] {{\small $xiii$}};
\draw  (b10) edge (b13);
\draw (10.7,4.5) node[anchor=center] {{\small $xiv$}};
\end{tikzpicture}
~\\[0.3cm]
\caption{The graph $\G \in \gra_{10+4}$ corresponding to 
the vector $(a,a^3, a, 1, 1, a, \ml_4)$} \label{fig:G-10-4}
\end{figure} 
\end{example}

Let us denote by $\pa^{\Gr}$ the differential on 
the associated graded complex $\Gr( \sfgraphs(n) )$\,.
It is clear from the definition of the filtration \eqref{cF-sfgraphs}
on $\sfgraphs(n)$ that  $\pa^{\Gr}$
is obtained from $\pa^{\Tw}$ by keeping only the terms 
which raise the number of 
the bivalent neutral vertices. Hence the image of the map 
$F_{\Gim}$ is closed with respect to the action the differential 
$\pa^{\Gr}$\,. Furthermore, going through the steps of the 
definition of $F_{\Gim}$, it is not hard to verify that 
\begin{equation}
\label{F-Gim-diff}
\pa^{\Gr} \circ F_{\Gim} = F_{\Gim} \circ \de \,.
\end{equation}

Our next goal is to describe the  kernel of the map $F_{\Gim}$\,.
For this purpose, we introduce the semi-direct 
product 
\begin{equation}
\label{group-rearrange}
S_e \ltimes \big(S_2\big)^{e}
\end{equation}
of the groups $S_e$ and $\big(S_2\big)^{e}$
with the multiplication rule: 
\begin{equation}
\label{group-law}
(\tau; \si_1, \dots, \si_{e}) \cdot (\la; \si'_1, \dots, \si'_{e}) =  
(\tau \la; \si_{\la(1)} \si'_1, \dots, \si_{\la(e)} \si'_e)\,.
\end{equation}

Next we observe that 
the group $\Aut(\Gim)$ admits an obvious homomorphism 
to the subgroup 
\begin{equation}
\label{really-rearrange}
\Big( S_{e_{\bul}} \times S_{e_{-}} \times S_{e_{\c}} \Big) 
 \ltimes \Big( \{\id\}^{e_{\bul}}  \times \big(S_2\big)^{e_{-}} \times  \{\id\}^{e_{\c}} \Big)
\end{equation}
of \eqref{group-rearrange}, where $\{\id\}$ denotes the trivial group.
Namely, this homomorphism assigns to an element $g\in \Aut (\Gim)$
the string 
$$
(\tau; \si_1, \dots, \si_{e}), \qquad \tau \in S_e, ~ \si_1, \dots, \si_e \in S_2   
$$
in \eqref{group-rearrange} according to this rule: $\tau(i) =j$ if 
the automorphism $g$ sends the 
$i$-th edge to the $j$-th edge; 
$\si_i$ is non-trivial (for $e_{\bul} +1 \le i \le e_{\bul} + e_{-}$) if $g$  sends the 
$i$-th edge to the $j$-th edge and the directions of these edges are opposite.
It is clear that this homomorphism lands in the subgroup \eqref{really-rearrange}. 

The group \eqref{really-rearrange} acts on the 
graded vector space
\begin{equation}
\label{kotovasija}
 \bs^{2r - 2e_{\bul}} \big( \und{T}_{\,a} \big)^{\otimes\, e_{\bul}}  
~\otimes  ~ \big( \bsi T_a \big)^{\otimes\, e_{-}}
~ \otimes ~   L ^{\otimes\, e_{\c}}\,. 
\end{equation}
Namely, if $\si$ is the non-trivial element of $S_2$ and $e_{\bul} < i  \le  e_{\bul} +e_{-}$  
then
$$
(1, \dots, 1,
\underbrace{\si}_{i\textrm{-th spot}},1, \dots 1) (a^{k_1}, a^{k_2}, \dots, a^{ k_{e_{\bul} +e_{-}} }, 
\ml_{ k_{e_{\bul} +e_{-} +1}}, \dots, \ml_{k_{e}}) = 
$$
$$
(-1)^{\frac{k_i(k_i+1)}{2}}  (a^{k_1}, a^{k_2}, \dots, a^{ k_{e_{\bul} +e_{-}} }, 
\ml_{ k_{e_{\bul} +e_{-} +1}}, \dots, \ml_{k_{e} })\,.
$$
Furthermore, for every $\tau \in  S_{e_{\bul}} \times S_{e_{-}} \times S_{e_{\c}}$ we set
$$
\tau (a^{k_1}, a^{k_2}, \dots, a^{ k_{e_{\bul} +e_{-}} }, \ml_{ k_{e_{\bul} +e_{-} +1}}, \dots, \ml_{k_{e}}) 
= 
$$
$$
(-1)^{\ve(\tau, k_1, \dots, k_e)}
(a^{k_{\tau^{-1}(1)}}, a^{k_{\tau^{-1}(2)}}, \dots, a^{k_{\tau^{-1}(e_{\bul} +e_{-})}},
  \ml_{k_{\tau^{-1}(e_{\bul} +e_{-} +1)}}, \dots, \ml_{k_{\tau^{-1}(e)}}  )\,,
$$
where the sign factor $(-1)^{\ve(\tau, k_1, \dots, k_e)}$ is determined by 
the usual Koszul rule. 

Thus the graded vector space  \eqref{kotovasija} is equipped with 
a left action of the group $\Aut(\Gim)$\,.

\begin{example}
\label{exam:rearrange}
Let us consider the frame
$\Gim$ depicted on figure 
\ref{fig:Gim-3-4}. Let $g_1$ be the generator 
of  $\Aut(\Gim)$ which swaps the first edge with the second edge
and let $g_2$ be the generator 
of  $\Aut(\Gim)$ which swaps the fifth edge with the sixth edge.
Then for the vector  $(a,a^3, a, 1, 1, a, \ml_4)$ we have 
\begin{equation}
\label{g-1-acts}
g_1(a,a^3, a, 1, 1, a, \ml_4) = - (a^3,a, a, 1, 1, a, \ml_4)\,,
\end{equation}
and
\begin{equation}
\label{g-2-acts}
g_2(a,a^3, a, 1, 1, a, \ml_4) = - (a, a^3, a, 1, a, 1, \ml_4)\,.
\end{equation}
The sign in \eqref{g-1-acts} comes from the fact that $a$ ``jumps'' over $a^3$
and the sign in \eqref{g-2-acts} appears due to the fact that  the fifth edge and 
the sixth edge carry opposite directions.
\end{example}

We can now describe the kernel of the map $F_{\Gim}$ \eqref{F-Gim}. 
\begin{claim}
\label{cl:ker-F-Gim}
Let $\Gim \in \Frame_{r,n}$ be a frame with $e$ edges 
$$
e = e_{\bul} + e_{-} + e_{\c}\,,
$$
where $e_{\bul}$ is the number of edges of $\Gim$ 
adjacent to univalent neutral vertices, $e_{\c}$ is the 
number of loops and $e_{-} = e - e_{\bul} - e_{\c}$\,.
Then the kernel of $F_{\Gim}$ is spanned 
by vectors of the form 
\begin{equation}
\label{ker-F-Gim}
X -  g (X)\,,
\end{equation}
where $X$ is a vector in \eqref{kotovasija} and 
$g$ is an automorphism of $\Gim$ in $\Frame_{r,n}$\,.
\end{claim}
\begin{proof}
Let $Y$ be a monomial in  \eqref{kotovasija} such that 
\begin{equation}
\label{F-Gim-Y-0}
F_{\Gim}(Y) = 0\,.
\end{equation}
The latter means that the graph $\G \in \gra_{(r+r') +n}$ which
is constructed from the monomial $Y$ using the frame $\Gim$
is $(r+r')$-odd. 

In other words, there exists an automorphism $\wt{g}$ of $\G$
which respects labels only on operational vertices and induces an odd permutation 
on the set of edges of $\G$\,.

It is clear that $\wt{g}$ induces an automorphism $g$ of the frame $\Gim$. 
Furthermore, since $\wt{g}$ induces an odd permutation on 
the set of edges of $\G$ we have 
$$
Y = - g (Y)\,. 
$$ 
Hence, 
\begin{equation}
\label{Y-in-ker}
Y = \frac{1}{2} (Y - g(Y))\,. 
\end{equation}
Thus every monomial $Y$  in  \eqref{kotovasija}
satisfying equation \eqref{F-Gim-Y-0} belongs to the span
of vectors of the form \eqref{ker-F-Gim}. 

Let us now consider a linear combination 
\begin{equation}
\label{combin-Ys}
c_1 Y_1 + c_2 Y_2 + \dots + c_m Y_m\,, \qquad c_i \in \bbK 
\end{equation}
of monomials $Y_1, \dots, Y_m$ in   \eqref{kotovasija} such that 
\begin{equation}
\label{F-Gim-Y-s}
\sum_i c_i F_{\Gim}(Y_i) = 0\,.
\end{equation}

Due to the above observation about monomials 
satisfying \eqref{F-Gim-Y-0} we may assume, without loss 
of generality, that 
$$
F_{\Gim}(Y_i)  \neq 0 \qquad \forall~~ 1 \le i \le m\,. 
$$ 

We may also assume, without loss of generality, that 
the graphs $\{\G_i \}_{1 \le i \le m}$ reconstructed
from the monomial  $\{Y_i \}_{1 \le i \le m}$ have the same number 
of neutral vertices $r+r'$.

Thus, for every $1 \le i \le m$, the graph $\G_i \in \gra_{(r+r') + n}$
is $(r+r')$-even.

Combining this observation with Proposition \ref{prop:ker-Av-r}
we conclude that the number $m$ is even and the set of graphs 
$\{\G_i \}_{1 \le i \le m}$ splits into pairs
$$
(\G_{i_{t}}, \G_{i'_{t}})\,, \qquad t \in  \{1, \dots, m/2 \}
$$
such that for every $t$ the graphs $\G_{i_{t}}$  and $ \G_{i'_{t}}$
are either $(r+r')$-opposite or  $(r+r')$-concordant. For every 
pair $(\G_{i_{t}}, \G_{i'_{t}})$ of $(r+r')$-opposite graphs we have 
\begin{equation}
\label{for-opposite}
c_{i_t} = c_{i'_t}\,.  
\end{equation}
For every pair $(\G_{i_{t}}, \G_{i'_{t}})$ of $(r+r')$-concordant graphs we have 
\begin{equation}
\label{for-concordant}
c_{i_t} = -c_{i'_t}\,.  
\end{equation}

Let $e_t$ denote the number of edges of $\G_{i_t}$ (or  $\G_{i'_t}$) and 
let $\wt{g}_t$ be the isomorphism from $\G_{i_t}$ to $\G_{i'_t}$ which 
induces an odd or even permutation in $S_{e_t}$
depending on whether  $\G_{i_t}$ and $\G_{i'_t}$ are  $(r+r')$-opposite or 
$(r+r')$-concordant. Let $g_t$ be the automorphism of the frame $\Gim$ which 
is induced by the isomorphism $\wt{g}_t$\,.

Equations \eqref{for-opposite} and \eqref{for-concordant} imply that
$$
\sum_{i=1}^m c_i Y_i = 
\sum_{t=1}^{m/2} c_{i_t} (Y_{i_t} - g_t (Y_{i_t}))\,.
$$

In other words, the linear combination \eqref{combin-Ys} belongs 
to the span of vectors of the form \eqref{ker-F-Gim} and the claim 
follows. 
\end{proof}

Now we are ready to give a convenient description of 
the associated graded complex  $\Gr( \sfgraphs(n) )$\,. 
\begin{claim}
\label{cl:Gr-sfgraphs}
Let us choose a representative $\Gim_z$ for every isomorphism class 
$z \in \pi_0(\Frame_{r,n})$\,. Let $e^z_{\bul}$ be the number of edges 
of $\Gim_z$ adjacent to univalent neutral vertices,  $e^z_{\c}$ be the 
number of loops of $\Gim_z$ and 
$$
e^z_{-} = |E(\Gim_z)| - e^z_{\bul} - e^z_{\c}\,.
$$
Then the cochain complex  $\Gr( \sfgraphs(n) )$ splits into the 
direct sum 
$$
\Gr( \sfgraphs(n) ) \cong 
$$
\begin{equation}
\label{Gr-sfgraphs-desc}
\bigoplus_{r \ge 0}~
\bigoplus_{z \in \pi_0(\Frame_{r,n})}~
 \bs^{2r - 2e^z_{\bul}} 
\Big( \big( \und{T}_{\,a} \big)^{\otimes\, e^z_{\bul}}  
~\otimes  ~ \big( \bsi T_a \big)^{\otimes\, e^z_{-}}
~ \otimes ~  L ^{\otimes\, e^z_{\c}} \Big)_{\Aut(\Gim_z)}\,.
\end{equation}\,.
\end{claim} 
\begin{proof}
Let us recall that the map $F_{\Gim_z}$ \eqref{F-Gim} is a morphism 
from the cochain complex 
$$
\bs^{2r - 2e^z_{\bul}}\, \big( \und{T}_{\,a} \big)^{\otimes\, e^z_{\bul}}  
~\otimes  ~ \big( \bsi T_a \big)^{\otimes\, e^z_{-}}
~ \otimes ~  L ^{\otimes\, e^z_{\c}} 
$$
with the differential $\de$ to $\Gr( \sfgraphs(n) )$\,. 

Thus, Claim \ref{cl:ker-F-Gim} implies that $F_{\Gim_z}$ induces 
an isomorphism from the cochain complex of 
coinvariants
$$
\bs^{2r - 2e^z_{\bul}}\,
\Big( \big( \und{T}_{\,a} \big)^{\otimes\, e^z_{\bul}}  
~\otimes  ~ \big( \bsi T_a \big)^{\otimes\, e^z_{-}}
~ \otimes ~  L ^{\otimes\, e^z_{\c}} \Big)_{\Aut(\Gim_z)}
$$
to the subcomplex  
$$
\Im(F_{\Gim_z}) \subset  \Gr( \sfgraphs(n) )\,.
$$

On the other hand, the cochain complex $\Gr( \sfgraphs(n) )$ 
is obviously the direct sum  
\begin{equation}
\label{sum-over-Gim}
\Gr( \sfgraphs(n) ) =
\bigoplus_{r \ge 0}~
\bigoplus_{z \in \pi_0(\Frame_{r,n})}~ \Im(F_{\Gim_z})\,.
\end{equation}

Thus, the desired statement follows. 
\end{proof}

\subsubsection{Proof of Lemma \ref{lem:Gr}}
\label{sec:lem:Gr}

We will now use the above description of  the cochain complex $\Gr( \sfgraphs(n) )$ 
to prove  Lemma \ref{lem:Gr}. 

First, we observe that, since the cochain complex $\und{T}_{\,a}$ is acyclic, 
the direct summand 
\begin{equation}
\label{acyclic-Im-Gim}
\Im (F_{\Gim}) 
\end{equation}
of  $\Gr( \sfgraphs(n) )$ is acyclic for every frame $\Gim$ with at least 
one univalent neutral vertex. 

So let us consider a frame $\Gim$ with $e_{\bul} = 0$\,. 

It is easy to see that the cochain complex
\begin{equation}
\label{kitten}
\bs^{2r} \Big( \big( \bsi T_a \big)^{\otimes\, e_{-}} ~ \otimes ~  L ^{\otimes\, e_{\c}}
\Big)_{\Aut(\Gim)}
\end{equation}
is concentrated is degrees 
$$
 \ge 2r -e_{-} -e_{\c}\,.
$$

Furthermore, using \eqref{H-T-a}, \eqref{H-L}, K\"unneth's theorem,   
 and the fact that the cohomology 
functor commutes with taking coinvariants, we conclude that every cocycle 
$X$ in \eqref{kitten} of degree $ > 2r -e_{-} -e_{\c} $ is trivial and the space
\begin{equation}
\label{H-kitten}
H^{2r - e_{-} -e_{\c}} \Big(\, \bs^{2r} \Big( \big( \bsi T_a \big)^{\otimes\, e_{-}} ~ \otimes ~  L^{\otimes\, e_{\c}}
\Big)_{\Aut(\Gim)} \,\Big) = \bbK
\end{equation}
is spanned by the class of the vector 
\begin{equation}
\label{the-lowest-deg}
\bs^{2r} \, (\bsi 1)^{\otimes\, e_{-}} ~\otimes ~ (\ml_1)^{\otimes\, e_{\c}}\,.
\end{equation}

Since images of cocycles $X$ in \eqref{kitten} of degrees $ > 2r -e_{-} -e_{\c} $ 
lie in 
$$
\Big( \cF^m \sfgraphs(n)  \big/  \cF^{m-1} \sfgraphs(n) \Big)^{k}
$$
for $m > -k $ and images of the vectors \eqref{the-lowest-deg}
belong to $\sgraphs(n)^{2r -e_{-} -e_{\c}}$, Lemma \ref{lem:Gr}
follows from Claim \ref{cl:Gr-sfgraphs}.

\subsection{We are getting rid of loops}
\label{sec:remove-loops}

Let us denote by $\snlGraphs(n)$ the subspace of 
$\sGraphs(n)$ which consists of vectors in $\sGraphs(n)$ 
involving exclusively graphs without loops. 

Since the differential $\pa^{\Tw}$ ``does not create'' loops, 
the subspace $\snlGraphs(n)$ is a subcomplex of 
$\sGraphs(n)$ for every $n$\,. Moreover the collection 
\begin{equation}
\label{snlGraphs}
\snlGraphs =
\{ \snlGraphs(n)\}_{n \ge 0}
\end{equation}
is obviously a suboperad $\sGraphs$\,.

The goal of this section is to prove that 
\begin{prop}
\label{prop:get-rid-of-loops}
The embedding 
\begin{equation}
\label{snl-into-sGraphs}
\emb^{\sharp}_2 :
\snlGraphs \hookrightarrow \sGraphs 
\end{equation}
is a quasi-isomorphism (of dg operads).
\end{prop}
\begin{proof}
Let us introduce the subcomplex 
$\snlgraphs(n) $ of $ \snlGraphs $ which consists of finite sums of 
graphs, i.e.
\begin{equation}
\label{snlgraphs}
\snlgraphs(n) : = \snlGraphs(n) \cap \Tw^{\oplus} \Gra(n)\,. 
\end{equation}

We will prove that the embedding 
\begin{equation}
\label{snl-in-sgraphs}
\snlgraphs(n) \hookrightarrow \sgraphs(n)
\end{equation}
is a quasi-isomorphism of cochain complexes.
Then the desired statement can be easily deduced from this fact
using the Euler characteristic trick (see Remark \ref{rem:Euler-trick}).

Let $\G$ be a $r$-even graph in $\gra_{r+n}$ whose first $r$ vertices have valency $\ge 3$\,.
Let us denote by $\tp_r(\G)$ the number of loops (if any) of $\G$ which 
are based on a trivalent vertex whose label $\le r$\,.  For example, 
the graph $\G\in \gra_{3+3}$ depicted on figure \ref{fig:with-loops}
has $\tp_3(\G) = 1$\,. Indeed, the vertex with label $1$ supports a loop 
but it has valency $4$; the vertex with label $2$ does not support a loop; 
finally, the vertex with label $3$ supports a loop and has valency $3$. 
\begin{figure}[htp]
\centering 
\begin{tikzpicture}[scale=0.5, >=stealth']
\tikzstyle{w}=[circle, draw, minimum size=4, inner sep=1]
\tikzstyle{b}=[circle, draw, fill, minimum size=4, inner sep=1]
\node [b] (b1) at (0,0) {};
\draw (0.3,0.5) node[anchor=center] {{\small $1$}};
\node [b] (b2) at (0,3) {};
\draw (0,3.6) node[anchor=center] {{\small $2$}};
\node [w] (b4) at (3,3) {};
\draw (3, 3.6) node[anchor=center] {{\small $4$}};
\node [w] (b5) at (3,0) {};
\draw (3.5,0) node[anchor=center] {{\small $5$}};
\node [b] (b3) at (6,3) {};
\draw (6,2.4) node[anchor=center] {{\small $3$}};
\node [w] (b6) at (-3,3) {};
\draw (-3,2.4) node[anchor=center] {{\small $6$}};
\draw (b6) .. controls (-3.5, 3.5) and (-3.5,4.5) .. (-3,4.5) ..
controls (-2.5, 4.5) and (-2.5, 3.5) .. (b6);
\draw (-3, 5) node[anchor=center] {{\small $ix$}};

\draw (b3) .. controls (5.25, 3.75) and (5.25,4.5) .. (6,4.5) ..
controls (6.75, 4.5) and (6.75, 3.75) .. (b3);
\draw (6,5) node[anchor=center] {{\small $vii$}};

\draw (b1) edge (b2);
\draw (-0.4,1.5) node[anchor=center] {{\small $ii$}};

\draw (b2) edge (b4);
\draw (1.5,3.5) node[anchor=center] {{\small $iii$}};

\draw (b4) edge (b5);
\draw (3.4,1.5) node[anchor=center] {{\small $iv$}};

\draw (b5) edge (b1);
\draw (1.5,-0.5) node[anchor=center] {{\small $v$}};

\draw (-0.4,-0.4) circle (0.6);
\draw (-1.1,-1.1) node[anchor=center] {{\small $i$}};

\draw (b4) edge (b3);
\draw (4.5,3.5) node[anchor=center] {{\small $vi$}};

\draw (b2) edge (b6);
\draw (-1.5, 3.5) node[anchor=center] {{\small $viii$}};
\end{tikzpicture}
~\\[0.3cm]
\caption{It is the vertex with label $3$ which contributes to $\tp_3(\G)$} \label{fig:with-loops}
\end{figure}
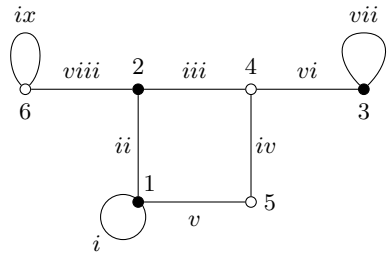 

It is obvious that the expression 
$$
\pa^{\Tw} (\Av_r(\G)) 
$$
involves graphs $\G' \in \gra_{(r+1)+n}$ with $\tp_{r+1}(\G') = \tp_r(\G)$ or 
$\tp_{r+1}(\G') = \tp_r(\G) + 1$\,. 

Thus the cochain complex $\sgraphs(n)$ carries the following 
ascending filtration 
\begin{equation}
\label{filtr-sgraphs}
\dots \subset \cF^{m-1} \sgraphs(n) \subset 
 \cF^{m} \sgraphs(n) \subset  \cF^{m+1} \sgraphs(n) \subset \dots
\end{equation}
where $\cF^{m} \sgraphs(n)$ is spanned by vectors in $\sgraphs(n)$
of the form
$$
\Av_r(\G) \qquad  \G \in \gra_{r+n}
$$  
with 
$$
\tp_r(\G) - |\Av_r(\G)| \le m\,.
$$

It is clear that the differential $\pa^{\tp}$ on the associated 
graded complex 
\begin{equation}
\label{Gr-sgraphs}
\Gr( \sgraphs(n) )  = 
\bigoplus_m ~ \cF^{m} \sgraphs(n) \big/ \cF^{m-1} \sgraphs(n)
\end{equation}
is obtained from $\pa^{\Tw}$ by keeping only terms which 
raise the number of loops based on trivalent neutral vertices. 

It is also clear that the restriction of \eqref{filtr-sgraphs} to 
the subcomplex $\snlgraphs(n)$ gives us the ``silly'' filtration 
\begin{equation}
\label{silly-snl}
\cF^{m} \snlgraphs(n)^k  = 
\begin{cases}
 \snlgraphs(n)^k \qquad {\rm if} ~~ m \ge -k\,,  \\
 \bfzero  \qquad {\rm otherwise}
\end{cases} 
\end{equation}
with the associated graded complex $\Gr(\snlgraphs(n))$
carrying the zero differential. 

It is not hard to see that the cochain complex 
$\Gr( \sgraphs(n) ) $ splits into the direct sum of 
subcomplexes
\begin{equation}
\label{Gr-sgraphs-sum}
\Gr(\sgraphs(n)) \cong \Gr(\snlgraphs(n)) ~\oplus~
\sgraphs_{\lp}(n)\,,
\end{equation}
where $\sgraphs_{\lp}(n)$ is spanned by vectors in  $\sgraphs(n)$ of the form
$$
\Av_r(\G)\,, \qquad \G \in \gra_{r+n}
$$
with $\G$ having at least one loop. 

Let $\G$ be graph in $\gra_{r+n}$ for which $\Av_r(\G)\in \sgraphs_{\lp}(n)$
and let $V^r_{\lp}(\G)$ denote the following subset of vertices of $\G$
$$
V^r_{\lp}(\G) = 
$$
\begin{equation}
\label{V-lp-G}
\big\{ v \in V(\G) ~\big|~  v \textrm{ carries label } \le r, \textrm{ has valency } > 3, 
\textrm{ and supports a loop} \big\} ~\cup
\end{equation}
$$
\big\{ v \in V(\G) ~\big|~  v \textrm{ carries label } > r \textrm{ and supports a loop} \big\}\,.
$$
For example, if $\G$ is the graph depicted on figure \ref{fig:with-loops} then 
$V^3_{\lp}(\G)$ consists of vertices labeled by $1$ and $6$\,.

Collecting terms in \eqref{diff-TwGra} which raise the number of 
loops based on trivalent neutral vertices we see that
\begin{equation}
\label{diff-tp}
\pa^{\tp}(\Av_r(\G)) = 
\begin{cases}
\displaystyle - \sum_{v \in V_{\lp}(\G)} \Av_{r+1} \big( \Tp_v (\G) \big)\,,  \qquad {\rm if} ~~   
V_{\lp}(\G) \textrm{ is non-empty}\\[0.5cm]
\phantom{aaaaaaaaa} 0  \qquad \phantom{aaaaaaaaaaa}  {\rm if} ~~ V_{\lp}(\G)= \emptyset \,,
\end{cases}
\end{equation}
where $ \Tp_v (\G) $ is a graph in $\gra_{(r+1)+n}$ obtained from $\G$ by 
\begin{itemize}

\item shifting labels on all vertices of $\G$ up by $1$;

\item removing the loop based at the vertex $v$;

\item attaching to $v$ the piece 
$$
\begin{tikzpicture}[scale=0.5, >=stealth']
\tikzstyle{w}=[circle, draw, minimum size=4, inner sep=1]
\tikzstyle{b}=[circle, draw, fill, minimum size=4, inner sep=1]
\node [b] (b1) at (0,0) {};
\draw (-0.3,0.4) node[anchor=center] {{\small $1$}};
\draw (0.5,0) circle (0.5);
\draw (b1) edge (-1,0);
\end{tikzpicture}
$$
\item declaring that the loop 
based at first neutral vertex takes the spot 
of the removed loop in $E(\G)$ and 
the edge connecting the first neutral vertex 
to $v$ is the smallest in $E(\Tp_v(\G))$\,.

\end{itemize}

Let $\G$ be graph in $\gra_{r+n}$ for which $\Av_r(\G)\in \sgraphs_{\lp}(n)$
and let $V^r_{\tp}(\G)$ denote the set of trivalent vertices (if any) which 
support loops and carry labels $\le r$\,. We denote by $h$ the 
linear map of degree $-1$
$$
h: \sgraphs_{\lp}(n) \to  \sgraphs_{\lp}(n)
$$
defined by formula
\begin{equation}
\label{chi-tp}
h(\Av_r(\G)) : =
\begin{cases}
\displaystyle - \sum_{v \in V^r_{\lp}(\G)} \Av_{r-1} \big( \Tp^*_v (\G) \big)\,,  \qquad {\rm if} ~~   
V^r_{\tp}(\G) \textrm{ is non-empty}\\[0.5cm]
\phantom{aaaaaaaaa} 0  \qquad \phantom{aaaaaaaaaaa}  {\rm if} ~~ V^r_{\tp}(\G)= \emptyset \,,
\end{cases}
\end{equation}
where $ \Tp^*_v (\G)$ is a vector in $\Gra(r-1+n)$ obtained from $\G$ by 
\begin{itemize}

\item switching the label on $v$ with the label $1$ 
on the first vertex of $\G$ (provided $v$ is not the first vertex); 
 
\item changing the order of the edges of $\G$ such 
that the single edge $\e_v$ connecting $v$ to another vertex 
becomes the smallest one (this step may produce the sign factor
$(-1)$ in front of $\G$); 

\item removing the edge $\e_v$ together with the vertex $v$ 
and attaching the vacated loop to the other end of $\e_v$\,;

\item shifting labels on all the remaining vertices down by $1$\,.

\end{itemize}

For example, if $\G$ is the graph depicted on figure \ref{fig:with-loops} then
\begin{equation}
\label{chi-example}
h (\Av_3 (\G)) = -\Av_2(\G'),
\end{equation}
where $\G'$ is the vector in $\Gra(5)$ depicted on figure \ref{fig:G-pr}.
\begin{figure}[htp]
\centering
\begin{tikzpicture}[scale=0.5, >=stealth']
\tikzstyle{w}=[circle, draw, minimum size=4, inner sep=1]
\tikzstyle{b}=[circle, draw, fill, minimum size=4, inner sep=1]
\draw (-5,2) node[anchor=center] {{$ - $}};
\node [b] (b1) at (0,0) {};
\draw (0.3,0.5) node[anchor=center] {{\small $2$}};
\node [b] (b2) at (0,3) {};
\draw (0,3.6) node[anchor=center] {{\small $1$}};
\node [w] (b4) at (3,3) {};
\draw (3.5, 3) node[anchor=center] {{\small $3$}};
\node [w] (b5) at (3,0) {};
\draw (3.5,0) node[anchor=center] {{\small $4$}};

\node [w] (b6) at (-3,3) {};
\draw (-3,2.4) node[anchor=center] {{\small $5$}};
\draw (b6) .. controls (-3.5, 3.5) and (-3.5,4.5) .. (-3,4.5) ..
controls (-2.5, 4.5) and (-2.5, 3.5) .. (b6);
\draw (-3, 5) node[anchor=center] {{\small $viii$}};

\draw (b4) .. controls (2.25, 3.75) and (2.25,4.5) .. (3,4.5) ..
controls (3.75, 4.5) and (3.75, 3.75) .. (b4);
\draw (3,5) node[anchor=center] {{\small $vi$}};

\draw (b1) edge (b2);
\draw (-0.4,1.5) node[anchor=center] {{\small $ii$}};

\draw (b2) edge (b4);
\draw (1.5,3.5) node[anchor=center] {{\small $iii$}};

\draw (b4) edge (b5);
\draw (3.4,1.5) node[anchor=center] {{\small $iv$}};

\draw (b5) edge (b1);
\draw (1.5,-0.5) node[anchor=center] {{\small $v$}};

\draw (-0.4,-0.4) circle (0.6);
\draw (-1.1,-1.1) node[anchor=center] {{\small $i$}};

\draw (b2) edge (b6);
\draw (-1.5, 3.5) node[anchor=center] {{\small $vii$}};
\end{tikzpicture}
~\\[0.3cm]
\caption{The vector $\G'$ defining $h(\Av_3(\G))$} \label{fig:G-pr}
\end{figure}
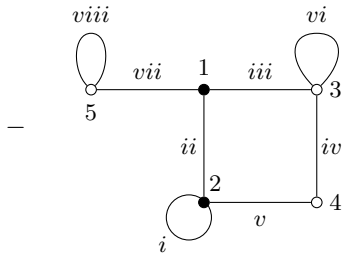 

Figure \ref{fig:getting-Gpr} illustrates intermediate steps in the construction of 
$\G'$\,.
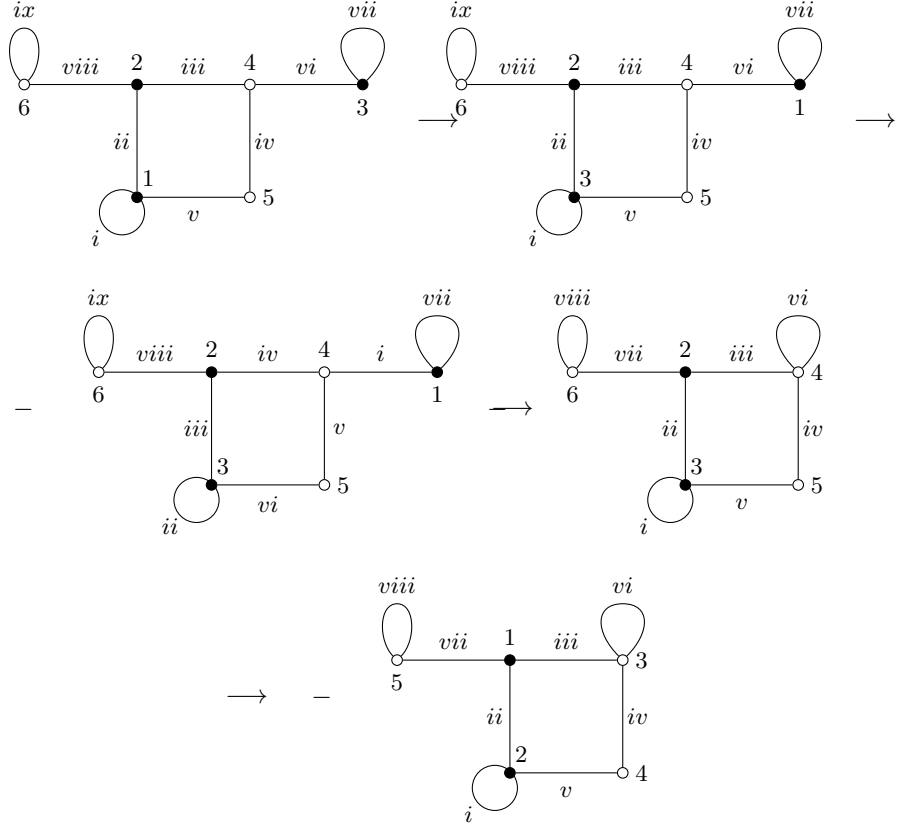
\begin{figure}[htp]
\centering 
\begin{minipage}[t]{0.45\linewidth}
\centering
\begin{tikzpicture}[scale=0.5, >=stealth']
\tikzstyle{w}=[circle, draw, minimum size=4, inner sep=1]
\tikzstyle{b}=[circle, draw, fill, minimum size=4, inner sep=1]
\node [b] (b1) at (0,0) {};
\draw (0.3,0.5) node[anchor=center] {{\small $1$}};
\node [b] (b2) at (0,3) {};
\draw (0,3.6) node[anchor=center] {{\small $2$}};
\node [w] (b4) at (3,3) {};
\draw (3, 3.6) node[anchor=center] {{\small $4$}};
\node [w] (b5) at (3,0) {};
\draw (3.5,0) node[anchor=center] {{\small $5$}};
\node [b] (b3) at (6,3) {};
\draw (6,2.4) node[anchor=center] {{\small $3$}};
\node [w] (b6) at (-3,3) {};
\draw (-3,2.4) node[anchor=center] {{\small $6$}};
\draw (8,2) node[anchor=center] {{$\longrightarrow$}};
\draw (b6) .. controls (-3.5, 3.5) and (-3.5,4.5) .. (-3,4.5) ..
controls (-2.5, 4.5) and (-2.5, 3.5) .. (b6);
\draw (-3, 5) node[anchor=center] {{\small $ix$}};

\draw (b3) .. controls (5.25, 3.75) and (5.25,4.5) .. (6,4.5) ..
controls (6.75, 4.5) and (6.75, 3.75) .. (b3);
\draw (6,5) node[anchor=center] {{\small $vii$}};

\draw (b1) edge (b2);
\draw (-0.4,1.5) node[anchor=center] {{\small $ii$}};

\draw (b2) edge (b4);
\draw (1.5,3.5) node[anchor=center] {{\small $iii$}};

\draw (b4) edge (b5);
\draw (3.4,1.5) node[anchor=center] {{\small $iv$}};

\draw (b5) edge (b1);
\draw (1.5,-0.5) node[anchor=center] {{\small $v$}};

\draw (-0.4,-0.4) circle (0.6);
\draw (-1.1,-1.1) node[anchor=center] {{\small $i$}};

\draw (b4) edge (b3);
\draw (4.5,3.5) node[anchor=center] {{\small $vi$}};

\draw (b2) edge (b6);
\draw (-1.5, 3.5) node[anchor=center] {{\small $viii$}};
\end{tikzpicture}
\end{minipage}
\begin{minipage}[t]{0.45\linewidth}
\centering
\begin{tikzpicture}[scale=0.5, >=stealth']
\tikzstyle{w}=[circle, draw, minimum size=4, inner sep=1]
\tikzstyle{b}=[circle, draw, fill, minimum size=4, inner sep=1]
\node [b] (b1) at (0,0) {};
\draw (0.3,0.5) node[anchor=center] {{\small $3$}};
\node [b] (b2) at (0,3) {};
\draw (0,3.6) node[anchor=center] {{\small $2$}};
\node [w] (b4) at (3,3) {};
\draw (3, 3.6) node[anchor=center] {{\small $4$}};
\node [w] (b5) at (3,0) {};
\draw (3.5,0) node[anchor=center] {{\small $5$}};
\node [b] (b3) at (6,3) {};
\draw (6,2.4) node[anchor=center] {{\small $1$}};
\node [w] (b6) at (-3,3) {};
\draw (-3,2.4) node[anchor=center] {{\small $6$}};
\draw (8,2) node[anchor=center] {{$\longrightarrow$}};
\draw (b6) .. controls (-3.5, 3.5) and (-3.5,4.5) .. (-3,4.5) ..
controls (-2.5, 4.5) and (-2.5, 3.5) .. (b6);
\draw (-3, 5) node[anchor=center] {{\small $ix$}};

\draw (b3) .. controls (5.25, 3.75) and (5.25,4.5) .. (6,4.5) ..
controls (6.75, 4.5) and (6.75, 3.75) .. (b3);
\draw (6,5) node[anchor=center] {{\small $vii$}};

\draw (b1) edge (b2);
\draw (-0.4,1.5) node[anchor=center] {{\small $ii$}};

\draw (b2) edge (b4);
\draw (1.5,3.5) node[anchor=center] {{\small $iii$}};

\draw (b4) edge (b5);
\draw (3.4,1.5) node[anchor=center] {{\small $iv$}};

\draw (b5) edge (b1);
\draw (1.5,-0.5) node[anchor=center] {{\small $v$}};

\draw (-0.4,-0.4) circle (0.6);
\draw (-1.1,-1.1) node[anchor=center] {{\small $i$}};

\draw (b4) edge (b3);
\draw (4.5,3.5) node[anchor=center] {{\small $vi$}};

\draw (b2) edge (b6);
\draw (-1.5, 3.5) node[anchor=center] {{\small $viii$}};
\end{tikzpicture}
\end{minipage} 
~\\[0.3cm] 
\begin{minipage}[t]{0.45\linewidth}
\centering
\begin{tikzpicture}[scale=0.5, >=stealth']
\tikzstyle{w}=[circle, draw, minimum size=4, inner sep=1]
\tikzstyle{b}=[circle, draw, fill, minimum size=4, inner sep=1]
\draw (-5,2) node[anchor=center] {{$ - $}};
\node [b] (b1) at (0,0) {};
\draw (0.3,0.5) node[anchor=center] {{\small $3$}};
\node [b] (b2) at (0,3) {};
\draw (0,3.6) node[anchor=center] {{\small $2$}};
\node [w] (b4) at (3,3) {};
\draw (3, 3.6) node[anchor=center] {{\small $4$}};
\node [w] (b5) at (3,0) {};
\draw (3.5,0) node[anchor=center] {{\small $5$}};
\node [b] (b3) at (6,3) {};
\draw (6,2.4) node[anchor=center] {{\small $1$}};
\node [w] (b6) at (-3,3) {};
\draw (-3,2.4) node[anchor=center] {{\small $6$}};
\draw (8,2) node[anchor=center] {{$\longrightarrow$}};
\draw (b6) .. controls (-3.5, 3.5) and (-3.5,4.5) .. (-3,4.5) ..
controls (-2.5, 4.5) and (-2.5, 3.5) .. (b6);
\draw (-3, 5) node[anchor=center] {{\small $ix$}};

\draw (b3) .. controls (5.25, 3.75) and (5.25,4.5) .. (6,4.5) ..
controls (6.75, 4.5) and (6.75, 3.75) .. (b3);
\draw (6,5) node[anchor=center] {{\small $vii$}};

\draw (b1) edge (b2);
\draw (-0.4,1.5) node[anchor=center] {{\small $iii$}};

\draw (b2) edge (b4);
\draw (1.5,3.5) node[anchor=center] {{\small $iv$}};

\draw (b4) edge (b5);
\draw (3.4,1.5) node[anchor=center] {{\small $v$}};

\draw (b5) edge (b1);
\draw (1.5,-0.5) node[anchor=center] {{\small $vi$}};

\draw (-0.4,-0.4) circle (0.6);
\draw (-1.1,-1.1) node[anchor=center] {{\small $ii$}};

\draw (b4) edge (b3);
\draw (4.5,3.5) node[anchor=center] {{\small $i$}};

\draw (b2) edge (b6);
\draw (-1.5, 3.5) node[anchor=center] {{\small $viii$}};
\end{tikzpicture}
\end{minipage}
\begin{minipage}[t]{0.45\linewidth}
\centering
\begin{tikzpicture}[scale=0.5, >=stealth']
\tikzstyle{w}=[circle, draw, minimum size=4, inner sep=1]
\tikzstyle{b}=[circle, draw, fill, minimum size=4, inner sep=1]
\draw (-5,2) node[anchor=center] {{$ - $}};
\node [b] (b1) at (0,0) {};
\draw (0.3,0.5) node[anchor=center] {{\small $3$}};
\node [b] (b2) at (0,3) {};
\draw (0,3.6) node[anchor=center] {{\small $2$}};
\node [w] (b4) at (3,3) {};
\draw (3.5, 3) node[anchor=center] {{\small $4$}};
\node [w] (b5) at (3,0) {};
\draw (3.5,0) node[anchor=center] {{\small $5$}};

\node [w] (b6) at (-3,3) {};
\draw (-3,2.4) node[anchor=center] {{\small $6$}};
\draw (b6) .. controls (-3.5, 3.5) and (-3.5,4.5) .. (-3,4.5) ..
controls (-2.5, 4.5) and (-2.5, 3.5) .. (b6);
\draw (-3, 5) node[anchor=center] {{\small $viii$}};

\draw (b4) .. controls (2.25, 3.75) and (2.25,4.5) .. (3,4.5) ..
controls (3.75, 4.5) and (3.75, 3.75) .. (b4);
\draw (3,5) node[anchor=center] {{\small $vi$}};

\draw (b1) edge (b2);
\draw (-0.4,1.5) node[anchor=center] {{\small $ii$}};

\draw (b2) edge (b4);
\draw (1.5,3.5) node[anchor=center] {{\small $iii$}};

\draw (b4) edge (b5);
\draw (3.4,1.5) node[anchor=center] {{\small $iv$}};

\draw (b5) edge (b1);
\draw (1.5,-0.5) node[anchor=center] {{\small $v$}};

\draw (-0.4,-0.4) circle (0.6);
\draw (-1.1,-1.1) node[anchor=center] {{\small $i$}};

\draw (b2) edge (b6);
\draw (-1.5, 3.5) node[anchor=center] {{\small $vii$}};
\end{tikzpicture}
\end{minipage}
~\\[0.3cm]
\begin{minipage}[t]{\linewidth}
\centering
\begin{tikzpicture}[scale=0.5, >=stealth']
\tikzstyle{w}=[circle, draw, minimum size=4, inner sep=1]
\tikzstyle{b}=[circle, draw, fill, minimum size=4, inner sep=1]
\draw (-7,2) node[anchor=center] {{$\longrightarrow$}};
\draw (-5,2) node[anchor=center] {{$ - $}};
\node [b] (b1) at (0,0) {};
\draw (0.3,0.5) node[anchor=center] {{\small $2$}};
\node [b] (b2) at (0,3) {};
\draw (0,3.6) node[anchor=center] {{\small $1$}};
\node [w] (b4) at (3,3) {};
\draw (3.5, 3) node[anchor=center] {{\small $3$}};
\node [w] (b5) at (3,0) {};
\draw (3.5,0) node[anchor=center] {{\small $4$}};

\node [w] (b6) at (-3,3) {};
\draw (-3,2.4) node[anchor=center] {{\small $5$}};
\draw (b6) .. controls (-3.5, 3.5) and (-3.5,4.5) .. (-3,4.5) ..
controls (-2.5, 4.5) and (-2.5, 3.5) .. (b6);
\draw (-3, 5) node[anchor=center] {{\small $viii$}};

\draw (b4) .. controls (2.25, 3.75) and (2.25,4.5) .. (3,4.5) ..
controls (3.75, 4.5) and (3.75, 3.75) .. (b4);
\draw (3,5) node[anchor=center] {{\small $vi$}};

\draw (b1) edge (b2);
\draw (-0.4,1.5) node[anchor=center] {{\small $ii$}};

\draw (b2) edge (b4);
\draw (1.5,3.5) node[anchor=center] {{\small $iii$}};

\draw (b4) edge (b5);
\draw (3.4,1.5) node[anchor=center] {{\small $iv$}};

\draw (b5) edge (b1);
\draw (1.5,-0.5) node[anchor=center] {{\small $v$}};

\draw (-0.4,-0.4) circle (0.6);
\draw (-1.1,-1.1) node[anchor=center] {{\small $i$}};

\draw (b2) edge (b6);
\draw (-1.5, 3.5) node[anchor=center] {{\small $vii$}};
\end{tikzpicture}
\end{minipage}
~\\[0.3cm]
\caption{Intermediate steps in the construction of $\G'$} \label{fig:getting-Gpr}
\end{figure}

Let $\G$ be a graph in $\gra_{r+n}$ such that
$\Av_r(\G)\in \sgraphs_{\lp}(n)$\,. Using the fact that 
$\G$ has no bivalent neutral vertices, 
it is not hard to show that the operations $\pa^{\tp}$ and $h$
satisfy the identity
\begin{equation}
\label{diff-tp-chi}
\pa^{\tp} \circ h (\Av_r(\G)) + h \circ \pa^{\tp}(\Av_r(\G)) = 
\la_{\G} \Av_r(\G)\,,
\end{equation}
where $\la_{\G}$ is the number of loops of $\G$\,. 

Therefore, the cochain complex 
$$
(\sgraphs_{\lp}(n), \pa^{\tp})
$$
is acyclic and hence the embedding \eqref{snl-in-sgraphs}
induces a quasi-isomorphism of cochain complexes: 
$$
\Gr (\snlgraphs(n))  ~\stackrel{\sim}{\longrightarrow}~
\Gr (\sgraphs(n))\,.
$$

On the other hand, both filtrations  \eqref{filtr-sgraphs}
and \eqref{silly-snl} are cocomplete and locally bounded from 
the left. Thus Lemma \ref{lem:q-iso} from Appendix \ref{app:q-iso}
implies that  the embedding \eqref{snl-in-sgraphs} is a quasi-isomorphism. 
Hence so is the embedding 
$$
\snlGraphs(n)  \hookrightarrow \sGraphs(n)\,.
$$

Proposition \ref{prop:get-rid-of-loops} is proved. 
\end{proof}

\subsection{The suboperads $\Graphs_{\nl} \subset \Graphs \subset \fGraphs \subset \Tw\Gra$} 
\label{sec:tele-flat}

In this subsection we introduce yet another series of 
suboperads of $\Tw\Gra$
$$
\Graphs_{\nl} \subset \Graphs \subset \fGraphs \subset \Tw\Gra\,.
$$
We will show that the embeddings 
$$
\Graphs_{\nl} \hookrightarrow \Graphs\,,
$$
$$
 \Graphs \hookrightarrow \fGraphs
$$
are quasi-isomorphisms of dg operads.

We denote by $\fGraphs(n)$ the subspace of $\Tw\Gra(n)$
which consists of linear combinations \eqref{sum} satisfying 
\begin{pty}
\label{P:no-black-comp}
For every $r$, each graph in the linear combination $\ga_r$
has no connected components which involve exclusively
neutral vertices.
\end{pty}
For example, it means that 
$$
\fGraphs(0) = \bfzero\,.
$$

We denote by $\Graphs(n)$ the subspace of  $\fGraphs(n)$
which consists of sums of graphs with neutral vertices having 
valencies $\ge 3$\,. 

Finally, $\Graphs_{\nl}(n)$ is the subspace of $\Graphs(n)$
which consists of sums of graphs without loops.

It is easy to see that for every $n$,  $\Graphs_{\nl}(n)$, 
$\Graphs(n)$, and  $\fGraphs(n)$ are subcomplexes
of $\Tw\Gra(n)$\,. Moreover, collections
\begin{equation}
\label{fGraphs}
\fGraphs = \{ \fGraphs(n) \}_{n \ge 0}\,,
\end{equation}
\begin{equation}
\label{Graphs}
\Graphs = \{ \Graphs(n) \}_{n \ge 0}\,,
\end{equation}
and   
\begin{equation}
\label{Graphs-nl}
\Graphs_{\nl} = \{ \Graphs_{\nl}(n) \}_{n \ge 0}
\end{equation}
are suboperads of $\Tw\Gra$\,.

We claim that 
\begin{prop}
\label{prop:Graphs-nl-Graphs-fGraphs}
The embeddings
\begin{equation}
\label{Graphs-fGraphs}
\emb_1 : \Graphs \hookrightarrow \fGraphs
\end{equation}
and 
\begin{equation}
\label{Graphs-nl-Graphs}
\emb_2 : \Graphs_{\nl}  \hookrightarrow \Graphs
\end{equation}
are quasi-isomorphisms of dg operads. 
\end{prop}
\begin{proof}
It clear that the cone
$$
\Cone(\emb_1) =  \Graphs ~ \oplus ~ \bs \, \fGraphs
$$ 
of the embedding \eqref{Graphs-fGraphs} is a direct summand 
in the cone $\Cone(\emb^{\sharp}_1)$ of 
$$
\emb^{\sharp}_1 : \sGraphs \hookrightarrow \sfGraphs\,.
$$ 

Thus the desired statement about the embedding $\emb_1$
follows from Proposition \ref{prop:sGra-sfGra} above and 
Claim \ref{cl:Cone} given in Appendix \ref{app:q-iso}.

Similarly the cone 
$$
\Cone(\emb_2) = \Graphs_{\nl}  ~ \oplus ~ \Graphs
$$
of the embedding \eqref{Graphs-nl-Graphs} is a direct 
summand in the cone $\Cone(\emb^{\sharp}_2)$ of 
$$
\emb^{\sharp}_2 : \snlGraphs \hookrightarrow \sGraphs\,.
$$ 

Thus, using Proposition \ref{prop:get-rid-of-loops} above and 
Claim \ref{cl:Cone} given in Appendix \ref{app:q-iso}, 
it is easy to prove that $\emb_2$ is a quasi-isomorphism.

\end{proof}

\subsection{The master diagram for the dg operad $\Tw\Gra$}
Let $\cO$ be a (dg) operad which receives a morphism 
from $\La\Lie_{\infty}$\,. Let us observe that we have the obvious embedding
\begin{equation}
\label{cO-TwcO}
\emb_{\cO} : \cO \hookrightarrow \Tw\cO
\end{equation}
$$
\emb_{\cO} (v) (\bs^{-2r} \,1) = \begin{cases}
 v \qquad {\rm if} ~~ r = 0\,, \\
 0 \qquad {\rm otherwise}\,,
\end{cases}
$$
which is compatible with the operad structure but may not be 
compatible with the differentials. 

We denote by $\G_{\cc}\in \Tw\Gra(2)$ (resp. $\G_{\ww} \in \Tw\Gra(2)$) the images of 
$\G_{\ed}$ and $\G_{\bb}$ with respect to the embedding
$$
\emb_{\Gra} : \Gra \to \Tw\Gra\,.
$$ 
Namely, 
\begin{equation}
\label{G-cc}
\G_{\cc}  = 
\begin{tikzpicture}[scale=0.5, >=stealth']
\tikzstyle{w}=[circle, draw, minimum size=4, inner sep=1]
\tikzstyle{b}=[circle, draw, fill, minimum size=4, inner sep=1]
\node [w] (b1) at (0,0) {};
\draw (0,0.6) node[anchor=center] {{\small $1$}};
\node [w] (b2) at (1,0) {};
\draw (1,0.6) node[anchor=center] {{\small $2$}};
\draw (b1) edge (b2);
\end{tikzpicture}\,,
\end{equation}
and
\begin{equation}
\label{G-ww}
\G_{\ww}   = 
\begin{tikzpicture}[scale=0.5, >=stealth']
\tikzstyle{w}=[circle, draw, minimum size=4, inner sep=1]
\tikzstyle{b}=[circle, draw, fill, minimum size=4, inner sep=1]
\node [w] (b1) at (0,0) {};
\draw (0,0.6) node[anchor=center] {{\small $1$}};
\node [w] (b2) at (1,0) {};
\draw (1,0.6) node[anchor=center] {{\small $2$}};
\end{tikzpicture}\,.
\end{equation}

Although $\emb_{\Gra}$ is not compatible with the differential $\pa^{\Tw}$,
the vectors $\G_{\cc}, \G_{\ww} \in \Tw\Gra(2)$ are $\pa^{\Tw}$-closed
(see Exercise \ref{exer:Gcc-Gww} below).

Therefore, the composition of embeddings $\io$ \eqref{io-Ger-Gra} and
$\emb_{\Gra}$ 
\begin{equation}
\label{io-pr-Ger-Graphs}
\io' = \emb_{\Gra} \circ \io : \Ger \hookrightarrow \Tw\Gra
\end{equation}
is a morphism of dg operads. 
Furthermore, it is obvious that $\io'$ lands in the suboperad $\Graphs_{\nl}$\,.
 
It turns out that the map $\io'$ satisfies the following 
remarkable property\footnote{For a more detailed proof of 
this fact we refer the reader to paper \cite{Volic} by P. Lambrechts and I. Volic.}:
\begin{thm}[M. Kontsevich, \cite{K-mot}, Section 3.3.4]
\label{thm:Ger-nlGraphs}
The embedding 
\begin{equation}
\label{io-pr}
\io' : \Ger \hookrightarrow \Graphs_{\nl}
\end{equation}
induces an isomorphism
$$
\Ger \cong  H^{\bul}(\Graphs_{\nl})\,.
$$
\end{thm}
\begin{rem}
\label{rem:Graphs-nl-graphs-nl}
It is obvious that the map \eqref{io-pr} lands in the 
suboperad  $\graphs_{\nl} \subset \Graphs_{\nl}$ with
$$
\graphs_{\nl}(n) = \Graphs_{\nl}(n) \cap \Tw^{\oplus}\Gra(n)\,.
$$
The arguments given in \cite[Section 3.3.4]{K-mot} or  \cite{Volic} 
allow us to prove that the embedding 
$$
\io' : \Ger \hookrightarrow \graphs_{\nl}
$$
is a quasi-isomorphism of dg operads. 
The desired statement about \eqref{io-pr} can be 
easily deduced from this fact using the Euler characteristic trick.
\end{rem}

We now assemble all the above results about suboperads of $\Tw\Gra$
into the following theorem:
\begin{thm}
\label{thm:master-TwGra}
The suboperads $\sfGraphs$, $\sGraphs$, $\snlGraphs$, 
$\fGraphs$, $\Graphs$ and $\Graphs_{\nl}$ of $\Tw\Gra$
introduced in Sections \ref{sec:tele-sharp} and \ref{sec:tele-flat} 
fit into the following commutative diagram: 
\begin{equation}
\label{master-TwGra}
\xymatrix@M=0.5pc{
~ & \snlGraphs \ar@{^{(}->}[r]^{\sim} & \sGraphs \ar@{^{(}->}[r]^{\sim} & \sfGraphs \ar@{^{(}->}[r] & \Tw\Gra \\
\Ger  \ar@{^{(}->}[r]^{\sim~~} & \Graphs_{\nl} \ar@{^{(}->}[u] \ar@{^{(}->}[r]^{\sim} & 
\Graphs \ar@{^{(}->}[u]  \ar@{^{(}->}[r]^{\sim} & 
\fGraphs \ar@{^{(}->}[u] & ~~
}
\end{equation}
Here the arrow $\hookrightarrow$ denotes an embedding and 
the arrow $\stackrel{\sim}{\hookrightarrow}$ denotes an embedding 
which induces an isomorphism on the level of cohomology. \phantom{aaa}$\Box$
\end{thm}
We refer to \eqref{master-TwGra} as the master diagram for the dg operad 
$\Tw\Gra$\,.

Theorem \ref{thm:master-TwGra} has the following obvious corollary
\begin{cor}
\label{cor:Ger-fGraphs}
The embedding 
\begin{equation}
\label{io-pr-Ger-fGraphs}
\io' : \Ger \hookrightarrow \fGraphs
\end{equation}
induces an isomorphism on the level of cohomology. $\Box$
\end{cor}

Let us observe that the map $\io'$ lands in the suboperad 
$\fgraphs \subset  \fGraphs$ for which
\begin{equation}
\label{fgraphs-n}
\fgraphs(n) : =  \fGraphs(n) \cap \Tw^{\oplus}\Gra(n)\,.
\end{equation}

Furthermore, using the Euler characteristic trick, it is not hard
to deduce from  Corollary \ref{cor:Ger-fGraphs} that
\begin{cor}
\label{cor:Ger-fgraphs}
The embedding 
\begin{equation}
\label{io-pr-Ger-fgraphs}
\io' : \Ger \hookrightarrow \fgraphs
\end{equation}
induces an isomorphism on the level of cohomology. $\Box$
\end{cor}

\begin{exer}
\label{exer:Gcc-Gww}
Prove that the vectors  $\G_{\cc}, \G_{\ww} \in \Tw\Gra(2)$
defined in \eqref{G-cc} and \eqref{G-ww}, respectively, satisfy the 
conditions
$$
\pa^{\Tw} \G_{\cc} = \pa^{\Tw} \G_{\ww} = 0\,.
$$  
\end{exer}

\section{The full graph complex $\fGC$ revisited} 
\label{sec:fGC-revisited}

Let us recall that $\Gra(0)= \bfzero$\,. Hence, due to Remark \ref{rem:Tw-cO-0}, 
we have a  tautological isomorphism 
\begin{equation}
\label{fGC-TwGra0}
\fGC \cong \bs^{-2} \Tw\Gra(0) 
\end{equation}
between the full graph complex $\fGC$ \eqref{fGC} and the 
cochain complex $\bs^{-2} \Tw\Gra(0)$\,.

Here we will use the isomorphism \eqref{fGC-TwGra0}  
together with the results of Sections \ref{sec:tele-sharp} and
\ref{sec:remove-loops} to deduce various useful facts about the 
full graph complex $\fGC$\,.

Recall that vectors of $\fGC$ are infinite sums 
\begin{equation}
\label{sum-fGC-here}
\ga = \sum_{n=1}^{\infty} \ga_n
\end{equation}
of $S_n$-invariant vectors $\ga_n \in \bs^{2n-2}\Gra(n)$\,.

We denote by 
\begin{equation}
\label{fGC-ge-3}
\fGC_{\ge 3} \subset \fGC 
\end{equation}
the subspace of sums \eqref{sum-fGC-here} satisfying
\begin{pty}
\label{P:ge-3}
For every $n$, each connected component of a graph in
$\ga_n$ has at least one vertex of valency $\ge 3$.
\end{pty}

We also denote by $\GC$ the subspace of sums \eqref{sum}  
which involve exclusively graphs whose vertices all have valencies $\ge 3$\,.
It is obvious that $\GC \subset \fGC_{\ge 3}$\,.

Comparing $\fGC_{\ge 3}$ and  $\GC$ with the suboperads
$\sfGraphs$ and  $\sGraphs$ from  Section \ref{sec:tele-sharp}
we see that  
\begin{equation}
\label{fGC-ge-3-sfGraphs}
\fGC_{\ge 3} \cong \bs^{-2}\sfGraphs(0)\,, \qquad 
\GC \cong  \bs^{-2}\sGraphs(0)\,.
\end{equation}

In particular, $\GC$ and $\fGC_{\ge 3}$ are  subcomplexes 
of $\fGC$\,.

Let us denote by $\GC_{\nl}$ the subspace of vectors in $\GC$ 
involving exclusively graphs without loops. It is clear that 
$\GC_{\nl}$ is a subcomplex of $\GC$. Moreover,
\begin{equation}
\label{GC-nl}
\GC_{\nl} \cong \bs^{-2} \snlGraphs(0)\,,
\end{equation}
where $\snlGraphs$ is the suboperad of $\Tw\Gra$ introduced 
in Section \ref{sec:remove-loops}.

Thus,   Propositions \ref{prop:sGra-sfGra}, \ref{prop:get-rid-of-loops} 
imply that 
\begin{cor}
\label{cor:GC-nl-GC}
The embeddings 
\begin{equation}
\label{emb-GC}
\emb_{\GC} ~:~ \GC \hookrightarrow \fGC_{\ge 3}
\end{equation}
and
\begin{equation}
\label{emb-GC-nl}
\emb_{\nl} ~:~ \GC_{\nl} \hookrightarrow \GC
\end{equation}
are quasi-isomorphisms of cochain complexes. ~~~$\Box$
\end{cor}

Let us denote by $\fGC_{\ge 3, \, \conn}$, $\GC_{\conn}$, and 
$\GC_{\nl,\, \conn}$ the ``connected'' versions of the subcomplexes 
$\fGC_{\ge 3}$, $\GC$ and $\GC_{\nl}$\,, respectively. Namely, 
\begin{equation}
\label{conn-versions}
\begin{array}{c}
\displaystyle \fGC_{\ge 3, \,\conn} :  = \fGC_{\ge 3} \cap \fGC_{\conn}\,,    \\[0.3cm]
\displaystyle \GC_{\conn} :  = \GC \cap \fGC_{\conn}\,,    \\[0.3cm]
\displaystyle \GC_{\nl \,\conn} :  = \GC_{\nl} \cap \fGC_{\conn}\,,
\end{array}
\end{equation}
where $\fGC_{\conn}$ is the subcomplex of $\fGC$ introduced 
in Section \ref{sec:fGC-conn}.

For the subcomplexes \eqref{conn-versions} we have 
\begin{prop}
\label{prop:conn-versions}
The embeddings 
\begin{equation}
\label{emb-GC-conn}
\emb_{\GC, \,\conn} ~:~ \GC_{\conn} \hookrightarrow \fGC_{\ge 3, \conn}
\end{equation}
and
\begin{equation}
\label{emb-GC-nl-conn}
\emb_{\nl,\, \conn} ~:~ \GC_{\nl, \, \conn} \hookrightarrow \GC_{\conn}
\end{equation}
are quasi-isomorphisms of cochain complexes. 
\end{prop}
\begin{proof}
It is easy to see that the cone
of the embedding  $\emb_{\GC, \,\conn}$ (resp. $\emb_{\nl,\, \conn}$) is a direct summand in 
the cone of the embedding $\emb_{\GC}$ (resp. $\emb_{\nl}$).

Thus the desired statements follow from Corollary \ref{cor:GC-nl-GC}
above and Claim \ref{cl:Cone} from Appendix \ref{app:q-iso}.
\end{proof}

Let us observe that, if all vertices of 
a connected graph $\G$ have valencies $\le 2$,
then $\G$ is isomorphic to one of the graphs   
in the list: $\G_{\bul}$, $\G^{-}_l$ (see figure \ref{fig:G-l}), or $\G^{\dia}_m$
(see figure \ref{fig:dia-m}). Hence, $\fGC_{\conn}$ decomposes as 
\begin{equation}
\label{fGC-conn-decomp}
\fGC_{\conn} =   \fGC_{\ge 3, \, \conn}~ \oplus ~ \cK_{\dia} ~ \oplus ~ \cK_{\cab}\,,
\end{equation}
where $\cK_{\cab}$ (resp. $\cK_{\dia}$) is the subcomplex of cables 
(resp. polygons) introduced in Subsection \ref{sec:cables} 
(resp. Subsection \ref{sec:dia}).

Therefore, using Proposition \ref{prop:cables} and isomorphism \eqref{H-cK-dia}
we deduce that 
\begin{equation}
\label{H-fGC-conn}
H^{\bul}(\fGC_{\conn}) \cong  H^{\bul} (\fGC_{\ge 3, \,\conn}) ~ \oplus ~ 
\bigoplus_{q\ge 1} \bs^{4q-1} \, \bbK\,.
\end{equation}

Thus we arrive at the main result of this section.
\begin{thm}[T. Willwacher, \cite{Thomas}]
\label{thm:fGC-decomp}
Let $\fGC_{\conn}$ be the ``connected part'' of the
full graph complex $\fGC$ \eqref{fGC}. Moreover, 
let  $\GC_{\nl,\, \conn}$ 
be the subcomplex of vectors in $\fGC_{\conn}$ involving exclusively 
graphs $\G$ satisfying these two properties: 
\begin{itemize}

\item $\G$ does not have loops;

\item each vertex of $\G$ has valency $\ge 3$\,.

\end{itemize}

Then 
\begin{equation}
\label{HfGC-fGC-conn-here}
H^{\bul}(\fGC) \cong \bs^{-2}  \wh{S}\big(\bs^2 \, H^{\bul}(\fGC_{\conn})\big)\,,
\end{equation}
and
\begin{equation}
\label{H-fGC-conn-better}
H^{\bul}(\fGC_{\conn}) \cong  H^{\bul} (\GC_{\nl, \,\conn}) ~ \oplus ~ 
\bigoplus_{q\ge 1} \bs^{4q-1} \, \bbK \,,
\end{equation}
where $\wh{S}$ is the notation for the completed symmetric algebra.
\end{thm}
\begin{proof}
The first decomposition \eqref{HfGC-fGC-conn-here} is obtained
by applying the K\"unneth theorem to \eqref{fGC-fGC-conn}.
The second decomposition \eqref{H-fGC-conn-better} is
obtained by applying Proposition  \ref{prop:conn-versions}
to the isomorphism \eqref{H-fGC-conn}.
\end{proof}

\begin{exer}
\label{exer:tri-val}
Using equation \eqref{pa-dfGC-simpler}, prove that 
for every even trivalent graph $\G\in \gra_n$ the vector 
\begin{equation}
\label{Av-G}
\Av(\G) = \sum_{\si \in S_n} \si (\G)
\end{equation}
is a cocycle in $\fGC$\,. Show that the tetrahedron 
depicted on figure \ref{fig:tetra} represents a non-trivial 
(degree zero) cocycle in $\fGC$\,. 
\begin{figure}[htp]
\centering 
\begin{tikzpicture}[scale=0.5, >=stealth']
\tikzstyle{w}=[circle, draw, minimum size=4, inner sep=1]
\tikzstyle{b}=[circle, draw, fill, minimum size=4, inner sep=1]
\node [b] (v1) at (2,0) {};
\draw (2,-0.6) node[anchor=center] {{\small $1$}};
\node [b] (v2) at (4,3) {};
\draw (4,3.6) node[anchor=center] {{\small $2$}};
\node [b] (v3) at (0,3) {};
\draw (0,3.6) node[anchor=center] {{\small $3$}};
\node [b] (v4) at (2,1.8) {};
\draw (2, 2.4) node[anchor=center] {{\small $4$}};
\draw (v1) edge (v2);
\draw (v1) edge (v3);
\draw (v1) edge (v4);
\draw (v2) edge (v3);
\draw (v2) edge (v4);
\draw (v3) edge (v4);
\end{tikzpicture}
~\\[0.3cm]
\caption{We may choose this order on the 
set of edges: $ (1,2) < (1,3) < (1,4) < (2,3)< (2,4) < (3,4)$ } \label{fig:tetra}
\end{figure} 
\end{exer}

\section{Deformation complex of $\Ger$}
\label{sec:Def-Ger}
Let us consider the following graded Lie algebra 
\begin{equation}
\label{Def-Ger}
\Conv(\Ger^{\vee}_{\c}, \Ger)\,.
\end{equation}

Due to \eqref{Ger-vee} we have 
\begin{equation}
\label{Def-Ger-new}
\Conv(\Ger^{\vee}_{\c}, \Ger) = \prod_{n\ge 2} 
\big( \Ger(n) \otimes \La^{-2} \Ger(n) \big)^{S_n}\,.
\end{equation}

The operad $ \La^{-2} \Ger$ is generated by the vectors
$b_1 b_2$ and $\{b_1, b_2\}$ in $ \La^{-2} \Ger(2) $\,. 
Moreover, the vectors $b_1 b_2$ and $\{b_1, b_2\}$ 
carry the degrees $2$ and $1$, respectively:
\begin{equation}
\label{deg-b1b2}
|b_1 b_2 | = 2\,, \qquad 
|\{b_1, b_2\}| = 1\,.
\end{equation}
 
Following Section \ref{sec:HA}, the canonical map 
$\Cobar(\Ger^{\vee}) \to \Ger$ \eqref{Ger-infty-Ger} corresponds to 
the Maurer-Cartan element\footnote{From now on we omit the subscript $\Ger$ in
the notation for the Maurer-Cartan element $\al_{\Ger}$\,.}
\begin{equation}
\label{MC-Def-Ger}
\al = a_1 a_2 \otimes \{b_1, b_2\} + 
\{a_1, a_2\} \otimes b_1 b_2  \in \Conv(\Ger^{\vee}_{\c}, \Ger)\,.
\end{equation}
Thus, using this Maurer-Cartan element, we can equip  the graded Lie algebra \eqref{Def-Ger-new}
with the differential 
\begin{equation}
\label{diff-Def-Ger}
\pa = [\al, ~]\,.
\end{equation}
According to \cite{MV-nado}, the cochain complex  \eqref{Def-Ger-new}
with the differential  \eqref{diff-Def-Ger}
``governs'' deformations of the operad structure on $\Ger$\,. So we 
refer to \eqref{Def-Ger-new}  as the {\it deformation complex} 
of the operad $\Ger$\,.

\begin{exer}
\label{ex:al-al-zero}
Verify the identity 
$$
[\al, \al] = 0
$$
by a direct computation.
\end{exer}

For our purposes it is convenient to extend the 
deformation complex of $\Ger$ to 
\begin{equation}
\label{Def-Ger-ext}
\Conv(\Ger^{\vee}, \Ger) = \prod_{n\ge 1} 
\big( \Ger(n) \otimes \La^{-2} \Ger(n) \big)^{S_n}\,.
\end{equation}

Vectors in the cochain complex 
\eqref{Def-Ger-ext} are formal infinite sum
\begin{equation}
\label{sum-Ger-Ger}
\sum_{n=1}^{\infty} \ga_n\,, 
\end{equation}
where each $\ga_n$ is an $S_n$-invariant vector in 
$\Ger(n) \otimes \La^{-2} \Ger(n)$\,. For example,
$$
a_1 a_2 \otimes \{b_1, b_2\} 
$$
is a degree $1$ vector in \eqref{Def-Ger-ext}.

It is obvious that 
$$
\Conv(\Ger^{\vee}, \Ger) = \bbK \L a_1 \otimes b_1  \R ~\oplus ~ 
\Conv(\Ger^{\vee}_{\c}, \Ger) 
$$
and 
$$
\pa (a_1 \otimes b_1) = \al\,.
$$
Thus 
\begin{equation}
\label{H-Def-Ger}
H^{\bul}  \Big( \Conv(\Ger^{\vee}_{\c}, \Ger) \Big) = 
H^{\bul}  \Big( \Conv(\Ger^{\vee}, \Ger) \Big) ~ \oplus
~ \bs\, \bbK\,,
\end{equation}
where the additional degree $1$ class is represented by 
the Maurer-Cartan element \eqref{MC-Def-Ger}.

Using the map $\io$ \eqref{io-Ger-Gra}, we 
embed  $\Conv(\Ger^{\vee}, \Ger) $ into
the vector space
\begin{equation}
\label{Gra-Gra}
\prod_{n\ge 1} 
 \Gra(n) \otimes \La^{-2} \Gra(n) 
\end{equation}
and represent vectors in (\ref{Gra-Gra}) by formal 
linear combinations of labeled graphs with two types 
of edges:  solid edges for left tensor factors and
dashed edges for right tensor factors.

For example, the Maurer-Cartan element \eqref{MC-Def-Ger}
corresponds to the linear combination of graphs depicted on 
figure \ref{fig:MC-Def-Ger}  and the vector 
$$
\{a_1, a_2\} \otimes \{b_1, b_2\} 
$$
corresponds to the graph depicted on 
figure \ref{fig:brack-brack}
\begin{figure}[htp]
\centering
\begin{tikzpicture}[scale=0.5, >=stealth']
\tikzstyle{w}=[circle, draw, minimum size=4, inner sep=1]
\tikzstyle{b}=[circle, draw, fill, minimum size=4, inner sep=1]
\node [b] (v1) at (1,1) {};
\draw (1,1.6) node[anchor=center] {{\small $1$}};
\node [b] (v2) at (3,1) {};
\draw (3,1.6) node[anchor=center] {{\small $2$}};
\draw  (v1) edge (v2);
\node at (4,1) {\large $+$};
\node [b] (v3) at (5,1) {};
\draw (5,1.6) node[anchor=center] {{\small $1$}};
\node [b] (v4) at (7,1) {};
\draw (7,1.6) node[anchor=center] {{\small $2$}};
\draw  (v3) edge[dashed] (v4);
\end{tikzpicture}
\caption{\label{fig:MC-Def-Ger} The Maurer-Cartan element in the deformation 
complex of $\Ger$}
\end{figure}
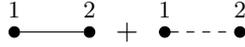 
\begin{figure}[htp]
\centering
\begin{tikzpicture}[scale=0.5, >=stealth']
\tikzstyle{w}=[circle, draw, minimum size=4, inner sep=1]
\tikzstyle{b}=[circle, draw, fill, minimum size=4, inner sep=1]
\node [b] (v1) at (1,1) {};
\draw (1,1.6) node[anchor=center] {{\small $1$}};
\node [b] (v2) at (3,1) {};
\draw (3,1.6) node[anchor=center] {{\small $2$}};
\draw  (v1) .. controls (1.5, 1.5) and (2.5,1.5) .. (v2);
\draw [dashed] (v1) .. controls (1.5, 0.5) and (2.5,0.5) .. (v2);
\end{tikzpicture}
\caption{\label{fig:brack-brack} The graph corresponding to 
the vector $\{a_1, a_2\} \otimes \{b_1, b_2\}$ }
\end{figure}
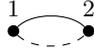 

\begin{defi}
\label{dfn:GerGer-conn}
We say that a monomial $X \in \Ger(n) \otimes \La^{-2} \Ger(n) $
is {\rm connected} if its image in \eqref{Gra-Gra} is a linear combination 
of connected graphs.  
We denote by $\Conv(\Ger^{\vee}, \Ger)_{\conn}$ the subspace of 
$\Conv(\Ger^{\vee}, \Ger)$ which consists of sums \eqref{sum-Ger-Ger}
involving exclusively connected monomials. 
\end{defi}
\begin{example}
\label{exam:connected}
According to the above definition the monomials
$$
\{a_1, a_2\} \otimes \{b_1, b_2\}, \qquad  a_1 a_2 \otimes \{b_1, b_2\}, 
\qquad \{a_1, a_2\} \otimes  b_1 b_2, \qquad 
a_2 \{a_1, a_3\} \otimes b_1 \{b_2, b_3\}
$$
are connected while the monomials 
$$
a_1 a_2 \otimes  b_1 b_2, \qquad 
a_2 \{a_1, a_3\} \otimes b_2 \{b_1, b_3\}
$$
are disconnected. 
\end{example}

It is not hard to see that $\Conv(\Ger^{\vee}, \Ger)_{\conn}$ is 
a subcomplex of  $\Conv(\Ger^{\vee}, \Ger)$\,.
Furthermore, we have 
\begin{equation}
\label{Def-Ger-as-Sym}
\Conv(\Ger^{\vee}, \Ger) = 
\bs^{-2} \wh{S}\big(\bs^2 \Conv(\Ger^{\vee}, \Ger)_{\conn}\big)\,,
\end{equation}
where $\wh{S}$ stands for the completed symmetric algebra.

\begin{rem}
\label{rem:degree-bound}
A simple degree bookkeeping shows that for every 
monomial $X \in  \Ger(n) \otimes \La^{-2} \Ger(n)$ 
$$
|X| \ge 0\,.
$$
Thus, 
\begin{equation}
\label{degree-bound}
H^{< 0}( \Conv(\Ger^{\vee}, \Ger)_{\conn}) =
H^{< 0}( \Conv(\Ger^{\vee}, \Ger) ) = \bfzero.
\end{equation}
\end{rem}

\subsection{Decomposition of $\Conv(\Ger^{\vee}, \Ger)$ with respect to the Euler characteristic}

Let us denote by $\mb(v)$ the total number of Lie brackets in 
the Gerstenhaber  monomial $v\in \Ger(n)$ or $v \in \La^{-2}\Ger(n)$\,.
Using the embedding of $\Ger(n) \otimes \La^{-2} \Ger(n)$ into 
$\Gra(n)\otimes \La^{-2} \Gra(n)$ we introduce the notion of 
Euler characteristic for monomials in $\Ger(n) \otimes \La^{-2} \Ger(n)$: 
\begin{defi}
\label{dfn:Euler-GerGer}
Let $v$ and $w$ be monomials in 
$\Ger(n)$ and $\La^{-2}\Ger(n)$, respectively.
We call the number 
$$
\chi (v \otimes w) : = n -\mb(v) - \mb(w) 
$$
the \emph{Euler characteristic} of the monomial 
$v \otimes w \in \Ger(n) \otimes \La^{-2}\Ger(n)$\,.
\end{defi}    

We observe that for every sum 
$$
\sum_i v_i \otimes w_i \in \Big( \Ger(n) \otimes \La^{-2}\Ger(n) \Big)^{S_n}
$$
of monomials with the same Euler characteristic $\chi$, each 
monomial in
the linear combination
$$
\pa   \left( \sum_i v_i \otimes w_i \right) 
$$
also has Euler characteristic $\chi$\,.
Thus sums \eqref{sum-Ger-Ger} in which each $\ga_n$ is a linear 
combination of monomials of Euler characteristic $\chi$
form a subcomplex of $\Conv(\Ger^{\vee}, \Ger)$\,.
We denote this subcomplex by 
\begin{equation}
\label{ConvGerGer-chi}
\Conv(\Ger^{\vee}, \Ger)_{\chi}\,.
\end{equation}

We claim that 
\begin{prop}
\label{prop:Euler-GerGer}
For every pair of integers $m$, $\chi$ the 
subspace $\Conv(\Ger^{\vee}, \Ger)^m_{\chi}$ of degree $m$ 
vectors in $\Conv(\Ger^{\vee}, \Ger)_{\chi}$ is a subspace in 
$$
\Ger(n) \otimes \La^{-2}\Ger(n)  
$$
where 
\begin{equation}
\label{n-deg-chi}
n = m -\chi + 2\,.
\end{equation}
In particular,  $\Conv(\Ger^{\vee}, \Ger)^m_{\chi}$ is finite dimensional. 
\end{prop}
\begin{proof}
Let $v \otimes w$ be a monomial in $\Ger(n) \otimes \La^{-2}\Ger(n)$
of degree $m$ and Euler characteristic $\chi$\,. 

Let $t_v$ (resp. $t_w$) be the number of Lie words in the monomial $v$ 
(resp. in the monomial $w$). For example, if $v = \{a_2, a_4 \} a_1 a_7 \{\{ a_3, a_5\}, a_6\}$ 
then $t_v = 4$\,.

It is not hard to see that 
\begin{equation}
\label{deg-v-w}
|v| = t_v - n, \qquad  |w| = n+ t_w - 2\,,
\end{equation}
and
\begin{equation}
\label{b-vw-t-vw}
\mb(v) = n - t_v\,, \qquad 
\mb(w) = n - t_w\,.
\end{equation}

Hence 
\begin{equation}
\label{deg-v-otimes-w}
m = t_v + t_w - 2\,.
\end{equation}
and 
\begin{equation}
\label{chi-tv-tw}
\chi = t_v + t_w - n\,.
\end{equation}

Using equations \eqref{deg-v-otimes-w} 
and \eqref{chi-tv-tw} we deduce that 
$$
n = m - \chi + 2\,. 
$$

Thus a combination ``degree and Euler characteristic'' determines 
the arity $n$ uniquely via equation \eqref{n-deg-chi}. Furthermore, 
since $\Ger(n) \otimes \La^{-2}\Ger(n) $ is finite dimensional,
so is  $\Conv(\Ger^{\vee}, \Ger)^m_{\chi}$\,. 

The proposition is proved. 
\end{proof}

Proposition \ref{prop:Euler-GerGer} has the following useful corollary.
\begin{cor}
\label{cor:Euler-GerGer}
The cochain complex $\Conv(\Ger^{\vee}, \Ger)$ splits into the
following product of its subcomplexes: 
\begin{equation}
\label{Euler-GerGer}
\Conv(\Ger^{\vee}, \Ger) = \prod_{\chi \in \bbZ} \Conv(\Ger^{\vee}, \Ger)_{\chi}\,.
\end{equation}
\end{cor}
\begin{proof}
Let 
$$
\ga = \sum_{n=1}^{\infty} \ga_n, \qquad \ga_n \in \Big( \Ger(n)\otimes \La^{-2}\Ger(n) \Big)^{S_n}
$$
be a homogeneous vector of degree $m$ in  $\Conv(\Ger^{\vee}, \Ger)$\,. 

Equation \eqref{n-deg-chi} implies that for every $n$
$$
\ga_n \in \Conv(\Ger^{\vee}, \Ger)_{\chi}
$$
with $\chi = m + 2 - n$\,. Thus, 
$$
\Conv(\Ger^{\vee}, \Ger) 
\subset \prod_{\chi \in \bbZ} \Conv(\Ger^{\vee}, \Ger)_{\chi}\,.
$$

The other inclusion 
$$
\prod_{\chi \in \bbZ} \Conv(\Ger^{\vee}, \Ger)_{\chi}
\subset 
\Conv(\Ger^{\vee}, \Ger) 
$$
is proved similarly. 
\end{proof}

The combination of Proposition \ref{prop:Euler-GerGer}
and Corollary \ref{cor:Euler-GerGer} will allow us to reduce questions 
about cocycles in $\Conv(\Ger^{\vee}, \Ger) $ to the corresponding 
questions about cocycles in its subcomplex\footnote{The functor 
$\Conv^{\oplus}$ was introduced in Section \ref{sec:Conv-oplus}.} 
$\Conv^{\oplus}(\Ger^{\vee}, \Ger) $\,.

\subsection{We are getting rid of Lie words of length $1$}
\label{sec:thm-Xi}
Let us recall that, for a monomial $w \in \La^{-2} \Ger(n)$, 
the notation $\mL_1(w)$ is reserved for
the number of Lie words in $w$ of 
length $=1$\,. For example, $\mL_1(b_1 b_2) = 2$ and
$\mL_1(\{b_1, b_2\}) =0$\,. 

Let us also recall that for the collection $\{ \La^{-2}\Ger^{\hrt}(n)\}_{n \ge 0}$
\eqref{Ger-hrt}
$$
\La^{-2}\Ger^{\hrt}(0) = \bs^{-2} \bbK
$$
and
$$
\La^{-2}\Ger^{\hrt}(n), \qquad n \ge 1
$$
is the $S_n$-submodule of $\La^{-2}\Ger(n)$ spanned by 
monomials $w \in \La^{-2}\Ger(n)$ for which $\mL_1(w) = 0$\,. 

Using this collection, we introduce the subspace of $\Conv(\Ger^{\vee}, \Ger)$
\begin{equation}
\label{Xi-dfn}
\Xi : = \prod_{n \ge 2} \Big( \Ger(n) \otimes \La^{-2}\Ger^{\hrt}(n) \Big)^{S_n}\,, 
\end{equation}
which will play an important role in establishing a link between 
the deformation complex \eqref{Def-Ger} of $\Ger$ and the 
full graph complex $\fGC$ \eqref{fGC}.

We reserve the notation $\Xi_{\conn}$ for the ``connected part'' of $\Xi$:
\begin{equation}
\label{Xi-conn}
\Xi_{\conn} : = \Xi \cap \Conv(\Ger^{\vee}, \Ger)_{\conn}\,.
\end{equation}
Furthermore, 
\begin{equation}
\label{Xi-oplus}
\Xi^{\oplus} : = \Xi \cap \Conv^{\oplus}(\Ger^{\vee}, \Ger)
\end{equation}
and
\begin{equation}
\label{Xi-oplus-conn}
\Xi^{\oplus}_{\conn} : = \Xi_{\conn} \cap \Conv^{\oplus}(\Ger^{\vee}, \Ger)
 \cap \Conv(\Ger^{\vee}, \Ger)_{\conn}
\end{equation}

We claim that 
\begin{prop}
\label{prop:Xi-is-sub}
The subspaces $\Xi$, $\Xi_{\conn}$, $\Xi^{\oplus}$, and $\Xi^{\oplus}_{\conn}$ are subcomplexes 
of 
$$
\Conv(\Ger^{\vee}, \Ger)\,.
$$
\end{prop}
\begin{proof}
Let 
\begin{equation}
\label{a-vector}
X = \sum_{n=2}^{\infty} v_n \otimes w_n 
\end{equation}
be a vector in $\Xi$\,. 

The bracket 
$$
\big[ \, a_1 a_2 \otimes \{b_1, b_2\}\,,\, X \,\big] 
$$
is obviously a vector in $\Xi$\,. So we need to prove that 
the vector 
\begin{equation}
\label{X-brack-prod}
\big[\, \{ a_1,  a_2\} \otimes b_1 b_2 \,,\, X\, \big] 
\end{equation}
belongs $\Xi$\,.

We have 
$$
\big[\, \{ a_1,  a_2\} \otimes b_1 b_2 \,,\, v_n \otimes w_n\, \big] = 
$$
\begin{equation}
\label{no-unwanted}
\sum_{i = 1}^{n+1} \vs_{i, n+1}  \big( \{v_n, a_{n+1}\} \big) \otimes 
\vs_{i, n+1}  \big( w_n\, b_{n+1} \big) - 
\end{equation}
$$
(-1)^{|v_n|} \sum_{\si \in \Sh_{2, n-1}} \si \big( v_n \circ_1  \{ a_1,  a_2\} \big)  \otimes 
\si \big( w_n \circ_1  b_1 b_2 \big)\,,
$$
where $ \vs_{i, n+1}$ is the cycle $(i, i+1, \dots, n+1)$ in $S_{n+1}$\,.

Using the defining identities of the Gerstenhaber algebra, it is not hard to prove
that unwanted terms in \eqref{no-unwanted} cancel each other. 
\end{proof}

We will need the following Theorem.
\begin{thm}
\label{thm:Xi}
The embeddings
\begin{equation}
\label{Xi-in-Conv}
\Xi \hookrightarrow  \Conv(\Ger^{\vee}, \Ger)\,,
\end{equation}
\begin{equation}
\label{Xi-in-Conv-conn}
\Xi_{\conn} \hookrightarrow  \Conv(\Ger^{\vee}, \Ger)_{\conn}\,,
\end{equation}
and
\begin{equation}
\label{Xi-in-Conv-oplus}
\Xi^{\oplus} \hookrightarrow  \Conv^{\oplus}(\Ger^{\vee}, \Ger)
\end{equation}
are quasi-isomorphisms of cochain complexes. 
\end{thm}
\begin{proof}
We will prove that the embedding \eqref{Xi-in-Conv-oplus}
is a quasi-isomorphism of cochain complexes. 
Then we will deduce that the embeddings \eqref{Xi-in-Conv}
and \eqref{Xi-in-Conv-conn} are also quasi-isomorphisms.

Let us recall from Section \ref{sec:Ger-cO} that 
$ \Conv^{\oplus}(\Ger^{\vee}, \Ger)$ has 
the following ascending filtration
\begin{equation}
\label{filtr-Conv-Ger-oplus}
\dots \subset \cF^{m-1}\, \Conv^{\oplus}(\Ger^{\vee}, \Ger) \subset 
\cF^m\, \Conv^{\oplus}(\Ger^{\vee}, \Ger) \subset \dots\,,
\end{equation}
where $\cF^m\, \Conv^{\oplus}(\Ger^{\vee}, \Ger)$ consists of 
sums
$$
\sum_i v_i \otimes w_i \in  \bigoplus_n \big( \Ger(n) \otimes \La^{-2} \Ger(n) \big)^{S_n}
$$
which satisfy 
$$
\mL_1 (w_i) - |\, v_i \otimes w_i \,| \le  m\,, \qquad \forall ~~ i\,.
$$

The restriction of \eqref{filtr-Conv-Ger-oplus} to the subcomplex $\Xi^{\oplus}$
gives us the ``silly'' filtration 
\begin{equation}
\label{silly-Xi}
\cF^m \big( \Xi^{\oplus} \big)^k =
\begin{cases}
 \big( \Xi^{\oplus} \big)^k \qquad {\rm if} ~~ m \ge -k \,, \\
 \bfzero  \qquad {\rm if}~~ m < -k
\end{cases}
\end{equation}
with the zero differential on the associated graded complex 
\begin{equation}
\label{Gr-Xi-oplus}
\Gr\, \Xi^{\oplus} ~~\cong~~  \bigoplus_{n \ge 2} \Big( 
\Ger(n) \otimes \La^{-2}\Ger^{\hrt}(n)  \Big)^{S_n}\,.
\end{equation}

Due to Proposition \ref{prop:Gr-Ger-cO} in 
Section \ref{sec:Ger-cO}, the 
formula\footnote{Recall that, due to Exercise \ref{exer:TwGer-sum},  
$\Tw\Ger= \Tw^{\oplus}\Ger$\,.} 
\begin{equation}
\label{Ups-Ger-def}
\Ups_{\Ger} \left( \sum_i v_i \otimes w_i \right) : =
\sum_{\si \in \Sh_{r,n}} \sum_i
\si (v_i(a_1, \dots, a_{r+n})) \otimes
\si (b_1 \dots b_r \, w_i(b_{r+1}, \dots, b_{r+n})) 
\end{equation}
$$
 \sum_i v_i \otimes w_i  \in  \Big(\bs^{2r} \Ger(r+n)^{S_r}
\otimes \La^{-2}\Ger^{\hrt}(n)  \Big)^{S_n}
$$ 
defines an isomorphism of cochain complexes
\begin{equation}
\label{Ups-Ger}
\Ups_{\Ger} ~: ~ \bigoplus_{n \ge 1}  \Big( \Tw\Ger(n)
\otimes \La^{-2}\Ger^{\hrt}(n)  \Big)^{S_n} ~~\to~~ 
\Gr\, \Conv^{\oplus}(\Ger^{\vee}, \Ger)\,,
\end{equation}
where the differential on the source comes from the differential 
$\pa^{\Tw}$ on $\Tw\Ger$\,. 

It is easy to see that the natural map 
\begin{equation}
\label{GrXi-to-GrConv}
\Gr\, \Xi^{\oplus} \hookrightarrow  \bigoplus_{n \ge 1}  \Big( \Tw\Ger(n)
\otimes \La^{-2}\Ger^{\hrt}(n)  \Big)^{S_n} 
\end{equation}
induced by the embedding \eqref{Ger-TwGer} fits into the 
commutative diagram 
$$
\xymatrix@M=0.7pc{
\Gr\, \Xi^{\oplus}  \ar@{^{(}->}[d] \ar@{^{(}->}[rd] & ~ \\
 \displaystyle\bigoplus_{n \ge 1}  \Big( \Tw\Ger(n)
\otimes \La^{-2}\Ger^{\hrt}(n)  \Big)^{S_n} \ar[r]^{\phantom{aaaaaaa}\Ups_{\Ger}}  & \Gr\, \Conv^{\oplus}(\Ger^{\vee}, \Ger)\,,
}
$$
where the slanted arrow is the canonical embedding of $\Gr\, \Xi^{\oplus} $
into $\Gr\, \Conv^{\oplus}(\Ger^{\vee}, \Ger)$\,.

On the other hand, using K\"unneth's theorem and  
Theorem \ref{thm:Ger-TwGer} together with the fact that, 
in characteristic zero,  the cohomology commutes with taking
invariants we deduce that the embedding \eqref{GrXi-to-GrConv}
is a quasi-isomorphism of cochain complexes. 

Therefore the embedding \eqref{Xi-in-Conv-oplus} induces a 
quasi-isomorphism of the associated graded complexes. 

Thus, since the filtrations \eqref{filtr-Conv-Ger-oplus} and \eqref{silly-Xi} are
locally bounded and cocomplete, we deduce from Lemma \ref{lem:q-iso}
that  the embedding \eqref{Xi-in-Conv-oplus} is also a quasi-isomorphism 
of cochain complexes. 

Combining this fact with Proposition \ref{prop:Euler-GerGer}
and Corollary \ref{cor:Euler-GerGer} we conclude that 
the embedding \eqref{Xi-in-Conv} is a quasi-isomorphism of 
cochain complexes. 

Since the cone of the map \eqref{Xi-in-Conv-conn} is the 
direct summand in the cone of the map \eqref{Xi-in-Conv}, 
the embedding  \eqref{Xi-in-Conv-conn} is also a quasi-isomorphism 
by Claim  \ref{cl:Cone}.

Theorem \ref{thm:Xi} is proved. 
\end{proof}

\section{Tamarkin's rigidity in the stable setting}
\label{sec:rigidity} 
Let us consider the Lie algebra 
\begin{equation}
\label{Ger-Gra}
\Conv(\Ger^{\vee}, \Gra) = \prod_{n \ge 1}
\big( \Gra(n) \otimes \La^{-2}\Ger(n) \big)^{S_n}\,.
\end{equation}

The map of operads \eqref{io-Ger-Gra} induces a homomorphism 
of Lie algebras 
\begin{equation}
\label{io-star}
\io_* : \Conv(\Ger^{\vee}, \Ger) \to \Conv(\Ger^{\vee}, \Gra)\,.   
\end{equation}
In particular, the vector\footnote{The vectors $\G_{\ed}, \G_{\bb} \in \Gra(2)$  
are defined in \eqref{binary}.} 
\begin{equation}
\label{MC-Ger-Gra}
\io_* (\al) = \G_{\ed} \otimes b_1 b_2  + 
\G_{\bb} \otimes \{b_1, b_2\}
\end{equation}
is a Maurer-Cartan element in \eqref{Ger-Gra} and the formula 
\begin{equation}
\label{diff-Ger-Gra}
\pa = [\io_* (\al),~]
\end{equation}
defines a differential on the Lie algebra \eqref{Ger-Gra}. 

Using map \eqref{io-Ger-Gra} once again
we can embed (\ref{Ger-Gra}) into \eqref{Gra-Gra}. 
Thus, by analogy with \eqref{Def-Ger-as-Sym}, we have 
\begin{equation}
\label{Ger-Gra-as-Sym}
\Conv(\Ger^{\vee}, \Gra) = 
\bs^{-2} \wh{S}(\bs^2 \Conv(\Ger^{\vee}, \Gra)_{\conn})
\end{equation}
where $ \Conv(\Ger^{\vee}, \Gra)_{\conn}$ is the subcomplex 
of $\Conv(\Ger^{\vee}, \Gra)$ which consists of formal  
linear combinations of connected monomials in  
$\Conv(\Ger^{\vee}, \Gra)$. 

The goal of this section is to prove the following 
theorem\footnote{Another version of this theorem is proved in \cite{stable11}.}
\begin{thm}
\label{thm:rigidity} 
For the cooperad $\Ger^{\vee}$ and the operad $\Gra$ we have 
\begin{equation}
\label{H-Conv-Ger-Gra}
H^{m} \big( \Conv(\Ger^{\vee}, \Gra) \big) = 
\begin{cases}
 \bbK  \qquad {\rm if} \quad  m = 1   \\
 \bfzero \qquad \qquad {\rm otherwise}\,.
\end{cases}
\end{equation}
Furthermore,  $H^1 \big( \Conv(\Ger^{\vee}, \Gra) \big)$ is spanned 
by the cohomology class of the vector $\G_{\ed} \otimes b_1 b_2$\,. 
\end{thm} 
The proof of this theorem is given below in Subsection \ref{sec:thm:rigidity}. 
It is based on auxiliary constructions which 
are described in Subsections \ref{sec:Euler-Ger-Gra},  \ref{sec:Gra-n-SVe}, 
\ref{sec:Ger-SVe}, and \ref{sec:Gr-Ger-Gra}. 

Before proceeding to these constructions, we will give 
a couple of useful corollaries of Theorem \ref{thm:rigidity} and
discuss its relation to Tamarkin's rigidity from 
\cite[Subsection 5.4.5.]{Hinich}.

First, we claim that
\begin{cor}
\label{cor:Conv-Ger-Gra-conn}
For the cochain complex $\Conv(\Ger^{\vee}, \Gra)_{\conn}$
we have
\begin{equation}
\label{H-Conv-Ger-Gra-conn}
H^{m} \big( \Conv(\Ger^{\vee}, \Gra)_{\conn} \big) = 
\begin{cases}
 \bbK  \qquad {\rm if} \quad  m = 1   \\
 \bfzero \qquad \qquad {\rm otherwise}\,.
\end{cases}
\end{equation}
Furthermore,  $H^1 \big( \Conv(\Ger^{\vee}, \Gra)_{\conn} \big)$ is spanned 
by the cohomology class of the vector $\G_{\ed} \otimes b_1 b_2$\,. 
\end{cor}
\begin{proof}
Due to Theorem \ref{thm:rigidity}, every cocycle $c \in \Conv(\Ger^{\vee}, \Gra)$
is cohomologous to a cocycle of the form
\begin{equation}
\label{la-Ged-b1b2}
\la\, \G_{\ed} \otimes b_1 b_2\,, \qquad \la \in \bbK
\end{equation}

On the other hand, the subcomplex $\Conv(\Ger^{\vee}, \Gra)_{\conn}$ 
is a direct summand in $\Conv(\Ger^{\vee}, \Gra)$. Therefore every 
cocycle  $c \in \Conv(\Ger^{\vee}, \Gra)_{\conn}$
is cohomologous to a cocycle of the form \eqref{la-Ged-b1b2}. 

Thus, since the cocycle 
$$ 
\G_{\ed} \otimes b_1 b_2 \in \Conv(\Ger^{\vee}, \Gra)_{\conn}
$$
is non-trivial the desired statement about cohomology of 
$\Conv(\Ger^{\vee}, \Gra)_{\conn}$ follows. 
\end{proof}

Following the terminology of Section \ref{sec:Conv-exact} the
Lie algebra  $\Conv(\Ger^{\vee}, \Gra)$ is equipped with 
the descending filtration ``by arity'': 
\begin{equation}
\label{arity-Ger-Gra}
\cF_m \Conv(\Ger^{\vee}, \Gra) : = 
\{
f\in \Conv(\Ger^{\vee}, \Gra) ~|~ f(w) = 0 \quad \forall ~ w \in \Ger(n), ~~ n \le m  
\}\,.
\end{equation}
In other words,
\begin{equation}
\label{arity-Ger-Gra-new}
\cF_m \Conv(\Ger^{\vee}, \Gra) = 
\prod_{n \ge m+1} \Big( \Gra(n) \otimes \La^{-2} \Ger(n)
\Big)^{S_n} \,.
\end{equation}
Furthermore, since the Maurer-Cartan element \eqref{MC-Ger-Gra}
belongs to $\cF_1 \Conv(\Ger^{\vee}, \Gra)$, the differential 
\eqref{diff-Ger-Gra} is compatible with the filtration \eqref{arity-Ger-Gra}.

Theorem \ref{thm:rigidity} implies that 
\begin{cor}
\label{cor:F2-Ger-Gra}
The cochain complex 
$$
\cF_2 \Conv(\Ger^{\vee}, \Gra) = 
\prod_{n \ge 3} \Big( \Gra(n) \otimes \La^{-2} \Ger(n)
\Big)^{S_n} 
$$
with the differential \eqref{diff-Ger-Gra} is acyclic. 
\end{cor}
\begin{proof}
Let $c$ be a cocycle in $\cF_2 \Conv(\Ger^{\vee}, \Gra)$\,. 

Due to Theorem \ref{thm:rigidity} there exists a vector 
$c_1 \in \Conv(\Ger^{\vee}, \Gra)$ and a scalar $\la\in \bbK$
such that 
\begin{equation}
\label{c-la-c1}
c = \la\, \G_{\ed} \otimes b_1 b_2  +  \pa(c_1)\,.
\end{equation}

On the other hand, it is easy to see that $\G_{\ed} \otimes b_1 b_2 $
represents a non-trivial cocycle in the quotient 
$$ 
\Conv(\Ger^{\vee}, \Gra) ~\big/~ \cF_2 \Conv(\Ger^{\vee}, \Gra)\,.
$$
Hence $\la=0$ and $c$ is exact.
\end{proof}

\subsection{Why rigidity?}
Let $\PV_d$ be the  graded vector space of polyvector fields on the affine space 
$\bbK^d$\,. This graded vector space carries the canonical structure of 
a Gerstenhaber algebra. The multiplication is the exterior multiplication of 
polyvector fields and the Lie bracket is the well-known Schouten bracket.   

Let recall from \cite{Thomas} or \cite[Section 3.5]{stable1} that the operad 
$\Gra$ acts on $\PV_d$\,. Moreover, the vector $\G_{\ed}$ (resp. $\G_{\bb}$) gives us 
the Schouten bracket (resp. the exterior multiplication) on $\PV_d$\,. 

Let us suppose that we are interested in $\Ger_{\infty}$-structures $\cQ$ on
$\PV_d$ which satisfy these two properties:
\begin{itemize}

\item  $\cQ$ factors through the canonical map 
$\Gra \to \End_{\PV_d}$;

\item the binary operations of $\cQ$ on $\PV_d$ coincide with the  
Schouten bracket and the exterior multiplication.

\end{itemize}

Using Corollary \ref{cor:F2-Ger-Gra}, it is not hard to 
prove that any $\Ger_{\infty}$-structure $\cQ$ on
$\PV_d$ satisfying the above properties is homotopy 
equivalent to the canonical Gerstenhaber algebra structure on 
$\PV_d$\,. 

This property is an analog of the rigidity\footnote{This rigidity property  is one 
of the corner stones of Tamarkin's proof \cite{Hinich}, \cite{Dima-Proof} of 
Kontsevich's formality theorem \cite{K}.}
of the Gerstenhaber algebra $\PV_d$ of polyvector fields in 
the homotopy category.  We refer the reader to \cite{stable11} for 
more details.

\subsection{Decomposition of $\Conv(\Ger^{\vee}, \Gra)$  with respect to the Euler characteristic}
\label{sec:Euler-Ger-Gra}

Let $\chi$ be an integer and let $c$ be a vector 
$$
c \in  \Conv(\Ger^{\vee}, \Gra)
$$
for which the image
$$
1 \otimes \io (c) \in  \prod_{n\ge 1} 
 \Gra(n) \otimes \La^{-2} \Gra(n) 
$$ 
is a (possibly infinite) sum of graphs whose 
Euler characteristic\footnote{As above, both solid and dashed edges 
enter with the same weight.} equals $\chi$. 
We denote by 
$$
\Conv(\Ger^{\vee}, \Gra)_{\chi}
$$ 
the subspace of such vectors. 
For example, both summands in  $\io \otimes \io (\al)$
have Euler characteristic $1$\,. Hence
$$
\io_*(\al) \in \Conv(\Ger^{\vee}, \Gra)_{1}\,.
$$
  
It is not hard to see that, for every integer $\chi$, 
the subspace $\Conv(\Ger^{\vee}, \Gra)_{\chi}$ 
is a subcomplex of $\Conv(\Ger^{\vee}, \Gra)$\,. 

Let us recall that we represent vectors in  the space
\begin{equation}
\label{Gra-Gra-n}
 \Gra(n) \otimes \La^{-2} \Gra(n) 
\end{equation}
by linear combinations of labeled graphs with two types 
of edges:  solid edges for left tensor factors and dashed 
edges for right tensor factors. 

Let us denote by 
\begin{equation}
\label{Gra-Gra-n-e}
\big( \Gra(n) \otimes \La^{-2} \Gra(n)  \big)_e
\end{equation}
the subspace of \eqref{Gra-Gra-n} which is spanned 
by graphs whose total number of edges (solid and dashed)
equals $e$\,.  It is obvious that the subspace \eqref{Gra-Gra-n-e}
is finite dimensional.

We have the following proposition.
\begin{prop}
\label{prop:Euler-GerGra}
For every pair of integers $m$, $\chi$ the 
subspace $\Conv(\Ger^{\vee}, \Gra)^m_{\chi}$ of degree $m$ 
vectors in $\Conv(\Ger^{\vee}, \Gra)_{\chi}$ is isomorphic 
to the subspace of  \eqref{Gra-Gra-n-e} with 
\begin{equation}
\label{n-deg-chi-Gra}
n = m -\chi + 2
\end{equation}
and
\begin{equation}
\label{e-deg-chi}
e = m - 2 \chi + 2\,.
\end{equation}
In particular,  $\Conv(\Ger^{\vee}, \Gra)^m_{\chi}$ is finite dimensional. 
\end{prop}
\begin{proof}
Let $\G$ be an graph in $\gra_n$ representing a vector in $\Gra(n)$ 
and $w$ be a monomial in $\La^{-2}\Ger(n)$\,. As above, 
$t_w$ denotes the total number of Lie monomials  and 
$\mb(w)$ denotes the the total number of brackets in $w$\,.

Let us suppose that $\G \otimes w$ carries degree $m$
and $\G \otimes \io(w)$ has Euler characteristic $\chi$. 
In other words, 
$$
m = - e_{\G} + |w|
$$
and
\begin{equation}
\label{chi-G-w}
\chi = n - e_{\G} - \mb(w)\,,
\end{equation}
where $e_{\G}$ is the number of edges of $\G$\,.

Due to equations \eqref{deg-v-w} and \eqref{b-vw-t-vw}
\begin{equation}
\label{degw-tw-bw-n}
|w| = n+ t_w - 2\,, \qquad t_w = n - \mb(w)\,. 
\end{equation}
 
Therefore, 
$$
m = n+ t_w - 2 - e_{\G} = 2n -\mb(w) -e_{\G} -2 = 
n + \chi - 2\,.
$$

Thus 
\begin{equation}
\label{n-deg-chi-Gra-new}
n = m -\chi + 2\,.
\end{equation}

Using \eqref{chi-G-w} and \eqref{n-deg-chi-Gra-new} we deduce 
that 
$$
e_{\G} + \mb(w) =  m - 2\chi + 2\,.
$$

Hence  $\G \otimes \io(w)$ is a vector in  \eqref{Gra-Gra-n-e} 
with numbers $n$ and $e$ given by equations 
\eqref{n-deg-chi-Gra} and \eqref{e-deg-chi}.

Thus, the proposition follows from the fact that
the map
$$
\io :  \La^{-2}\Ger(n) \to \La^{-2}\Gra(n)
$$
is injective.
\end{proof}

Proposition \ref{prop:Euler-GerGra} has the following useful corollary.
\begin{cor}
\label{cor:Euler-GerGra}
The cochain complex $\Conv(\Ger^{\vee}, \Gra)$ splits into the
following product of its subcomplexes:
\begin{equation}
\label{Euler-Ger-Gra}
\Conv(\Ger^{\vee}, \Gra) = \prod_{\chi \in \bbZ}
\Conv(\Ger^{\vee}, \Gra)_{\chi}\,.
\end{equation} 
\end{cor}
\begin{proof}
The proof of this statement is very similar to 
that of Corollary \ref{cor:Euler-GerGer}. So we leave it 
as an exercise for the reader.
\end{proof}
\begin{exer}
\label{exer:Euler-GerGra}
Prove Corollary \ref{cor:Euler-GerGra}. 
\end{exer}

Just as for $\Tw\Gra$ and  $\Conv(\Ger^{\vee}, \Ger)$, 
the combination of Proposition \ref{prop:Euler-GerGra}
and Corollary \ref{cor:Euler-GerGra} will allow us to reduce questions 
about cocycles in $\Conv(\Ger^{\vee}, \Gra) $ to the corresponding 
questions about cocycles in its subcomplex
\begin{equation}
\label{Ger-Gra-oplus}
\Conv^{\oplus}(\Ger^{\vee}, \Gra) : =  \bigoplus_{n \ge 1}
\big( \Gra(n) \otimes \La^{-2}\Ger(n) \big)^{S_n}\,.
\end{equation}

\subsection{An alternative description of $\Gra(n)$} 
\label{sec:Gra-n-SVe}

Let $e$ be a positive integer and 
\begin{equation}
\label{rhos}
\{ \rho_1, \rho'_1, \rho_2, \rho'_2, \dots, \rho_e, \rho'_e  \}
\end{equation}
be a set of auxiliary variables with degrees $|\rho_i |=-1$ and $|\rho'_i| = 0$\,.

We will need the symmetric algebra 
\begin{equation}
\label{S-V-e}
S(V_e) = \bbK \oplus V_e \oplus S^2 (V_e) \oplus S^3 (V_e) \oplus \dots  
\end{equation}
of the vector space 
\begin{equation}
\label{V-e}
V_e = \bbK\L \rho_1, \rho'_1, \rho_2, \rho'_2, \dots, \rho_e, \rho'_e  \R\,,
\end{equation}
spanned by elements on the set \eqref{rhos}.

We view $S(V_e)$ as the cocommutative coalgebra 
with the standard comultiplication.  

Let us denote by $T_n \big(S(V_e) \big) $ the subspace 
\begin{equation}
\label{T-n-SVe}
T_n \big(S(V_e) \big) \subset \big( S(V_e) \big)^{\otimes n}
\end{equation}
of $ \big(S(V_e) \big)^{\otimes n}$ which is spanned by 
monomials 
\begin{equation}
\label{XXX}
X_1 \otimes X_2 \otimes \dots \otimes X_n
\end{equation}
in which each variable from the set \eqref{rhos} appears 
exactly once. 

For example, if $e=3$ then
$$
\rho'_1 \rho_2 \otimes \rho'_2 \rho_1 \otimes 1 \otimes \rho_3 \rho'_3  \in 
T_4 \big(S(V_e) \big)\,, 
\qquad 
\rho'_1 \rho_2 \rho'_3 \otimes \rho'_2 \rho_1 \rho_3 
\in T_2 \big(S(V_e) \big)\,,
$$
and 
$$
\rho_1 \rho_2 \rho_3 \otimes 1 \otimes \rho'_2 \rho_1 \otimes \rho_3 \rho'_3  \notin
T_4 \big(S(V_e) \big)\,, \qquad 
\rho_1 \rho_2 \rho'_3 \otimes  \rho'_1 \rho_3 
\notin T_2 \big(S(V_e) \big)\,.
$$

It makes sense to include the degenerate case $e=0$ in our 
consideration. If $e=0$ then the set \eqref{rhos} is empty,
$$
S(V_e) = \bbK\,,
$$
and 
$$
T_n \big(S(V_e) \big) = \bbK^{\otimes n} \cong \bbK\,.
$$

Given a monomial \eqref{XXX} we form 
a labeled graph $\G'$ with $n$ vertices and $e$ \und{directed} edges following 
these two steps: 
\begin{itemize}
\item  we declare that edge $i$ originates at the $j$-th vertex if
the factor $X_j$ involves the variable $\rho_i$; 

\item we declare that edge $i$ terminates at the $k$-th vertex if
the factor $X_k$ involves the variable $\rho'_i$.

\end{itemize}

Since each variable in the set \eqref{rhos} appears in the monomial 
\eqref{XXX} exactly once, these two steps give us a 
labeled graph with $n$ vertices and with $e$ directed edges. 

Notice that we use indices of the variables \eqref{rhos} 
to keep track of edges of $\G'$. 
This bijection between the set of edges of $\G'$ and 
natural numbers $1,2, \dots, e$ plays a purely 
auxiliary role and we do not keep it 
for $\G'$ as a piece of additional data.  

It is more important to observe that 
the set $E(\G')$ of edges of $\G'$ is equipped with an order
up to an even permutation. This order is 
defined by the following rule: 
\begin{itemize}

\item if the initial vertex of edge $i_1$ carries a 
smaller label than the initial vertex of edge $i_2$
then we set $i_1 < i_2$; 

\item if edges $i_1$ and $i_2$ originate from the same 
vertex (say, vertex $j$) and $\rho_{i_1}$ stands to 
the left from $\rho_{i_2}$ in the factor $X_j$ then we also 
set $i_1 < i_2$\,.
  
\end{itemize}

For example, the monomial $\rho'_1 \rho_2 \otimes 
\rho'_2 \rho_1 \otimes 1 \otimes \rho_3 \rho'_3$
corresponds to the labeled directed graph $\G'$ depicted 
on figure \ref{fig:Gpr-exam}.  
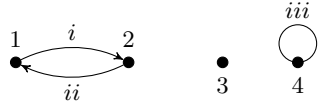
\begin{figure}[htp]
\centering
\begin{tikzpicture}[scale=0.5, >=stealth']
\tikzstyle{w}=[circle, draw, minimum size=4, inner sep=1]
\tikzstyle{b}=[circle, draw, fill, minimum size=4, inner sep=1]
\node [b] (v1) at (0.5,0) {};
\draw (0.5,0.6) node[anchor=center] {{\small $1$}};
\node [b] (v2) at (3.5,0) {};
\draw (3.5,0.6) node[anchor=center] {{\small $2$}};
\node [b] (v3) at (6,0) {};
\draw (6,-0.6) node[anchor=center] {{\small $3$}};
\node [b] (v4) at (8,0) {};
\draw (8,-0.6) node[anchor=center] {{\small $4$}};
\draw [->] (v1) .. controls (1.5, 0.5) and (2.5,0.5) .. (v2);
\draw (2, 0.8) node[anchor=center] {{\small $i$}};
\draw [<-] (v1) .. controls (1.5, -0.5) and (2.5,-0.5) .. (v2);
\draw (2, -0.8) node[anchor=center] {{\small $ii$}};
\draw (8,0.5) circle (0.5);
\draw (8, 1.4) node[anchor=center] {{\small $iii$}};
\end{tikzpicture}
\caption{\label{fig:Gpr-exam} The directed labeled graph corresponding 
to the monomial $\rho'_1 \rho_2 \otimes \rho'_2 \rho_1 \otimes 1 \otimes \rho_3 \rho'_3$}
\end{figure}

Let us denote by $\G$ the undirected graph (with an order
on the set of edges up to an even permutation) which is obtained 
from $\G'$ by forgetting the directions. It is clear that 
the formula 
\begin{equation}
\label{Te-dfn}
\Te (X_1\otimes X_2 \otimes \dots \otimes X_n) = \G
\end{equation}
defines a surjective map of graded vector spaces 
\begin{equation}
\label{Te}
\Te :  \bigoplus_{e \ge 0} T_n \big(S(V_e) \big)  \onto \Gra(n)\,.
\end{equation}
For example, $\Te(\rho'_1 \rho_2 \otimes \rho'_2 \rho_1 \otimes 1 \otimes \rho_3 \rho'_3) = 0$
because the graph $\G$ corresponding to the monomial $\rho'_1 \rho_2 \otimes \rho'_2 \rho_1 \otimes 1 \otimes \rho_3 \rho'_3$ has double edges.

To describe the kernel of \eqref{Te} we recall, 
from  Subsection \ref{sec:Gr-sfgraphs},  the group
\eqref{group-rearrange}
\begin{equation}
\label{group-Tn-SVe}
S_e  \ltimes (S_2)^{e}
\end{equation}
with the multiplication law defined by equation \eqref{group-law}.

Let us equip the graded vector space \eqref{T-n-SVe} with 
a left action of the group \eqref{group-Tn-SVe}.

For this purpose, we declare that elements $\tau \in S_e$ and  
$$
\si_j = (\id, \dots, \id, \underbrace{(12)}_{j\textrm{-th spot}}, \id, \dots \id) \in 
 (S_2)^{e}
$$
act on generators \eqref{rhos} as
$$
\tau(\rho_i) = \rho_{\tau(i)}\,, \qquad 
\tau(\rho'_i) = \rho'_{\tau(i)}\,,
$$
and
$$
\si_j (\rho_i) = \begin{cases}
\rho_i  \qquad {\rm if} ~~ i \neq j \\
 \rho'_i \qquad {\rm if}~~ i = j\,, 
\end{cases}
\qquad 
\si_j (\rho'_i) = \begin{cases}
\rho'_i  \qquad {\rm if} ~~ i \neq j \\
 \rho_i \qquad {\rm if}~~ i = j\,,
\end{cases}
$$
respectively.
 
Next, we extend the action of elements $\{\si_j\}_{1 \le j \le e}$
to the space  \eqref{T-n-SVe} by incorporating appropriate 
sign factor which appear if odd variables $\rho_1, \dots, \rho_e$
``move around''. Finally, we declare that elements of $S_e$
act by automorphisms (of the commutative algebra) $S(V_e)$ 
and then extend the action of $S_e$ to  the space  \eqref{T-n-SVe} 
by the formula: 
$$
\tau \big( X_1 \otimes X_2 \otimes \dots \otimes X_n \big) = 
\tau(X_1)  \otimes \tau(X_2) \otimes  \dots \otimes \tau(X_n)\,. 
$$

For example, the transposition $(23)\in S_3$ sends the 
vector 
 $\rho'_1 \rho_2 \otimes \rho'_2 \rho_1 \otimes 1\otimes  \rho_3 \rho'_3$
to the vector  
$$
\rho'_1 \rho_3 \otimes \rho'_3 \rho_1 \otimes 1 \otimes \rho_2 \rho'_2
$$ 
and the element $\si_1$ sends the vector
 $\rho'_1 \rho_2 \otimes \rho'_2 \rho_1 \otimes 1\otimes  \rho_3 \rho'_3$
to the vector 
$$
-\rho_1 \rho_2 \otimes \rho'_2 \rho'_1 \otimes 1\otimes  \rho_3 \rho'_3\,.
$$
The sign factor in the action of $\si_1$ appeared because the variables 
$\rho_1$ and $\rho_2$ changed their order.

Due to Exercise \ref{exer:Ker-Te} below the kernel of $\Te$
\eqref{Te} is spanned by vectors of the form
\begin{equation}
\label{ker-Te}
(X_1\otimes X_2 \otimes \dots \otimes X_n) - 
g (X_1\otimes X_2 \otimes \dots \otimes X_n)\,,
\qquad g \in S_e  \ltimes (S_2)^{e}\,.
\end{equation}

Hence, we conclude that 
\begin{prop}
\label{prop:T-n-Gra-n}
The map $\Te$ \eqref{Te} induces an isomorphism 
of graded vector spaces 
$$
\Gra(n) \cong   \bigoplus_{e \ge 0} \Big(T_n \big(S(V_e) \big) \Big)_{ S_e  \ltimes (S_2)^{e}}\,,
$$
where $ \Big(T_n \big(S(V_e) \big) \Big)_{ S_e  \ltimes (S_2)^{e}}$ 
denotes the space of coinvariants in \eqref{T-n-SVe}\,. $\Box$ 
\end{prop}
\begin{exer}
\label{exer:Ker-Te}
Prove that the kernel of the map $\Te$
\eqref{Te} is spanned by vectors of the form
\eqref{ker-Te}. {\it Hint: First, prove that, if a monomial \eqref{XXX} 
corresponds to a graph with multiple edges, then this monomial 
belongs to the span of vectors of the form \eqref{ker-Te}. 
Second, consider linear combinations of monomials 
\eqref{XXX} each of which does not belong to the kernel of $\Te$\,.}
\end{exer}

\begin{example}
\label{exam:ker-Te}
We mentioned above that 
$$
\Te(\rho'_1 \rho_2 \otimes \rho'_2 \rho_1 \otimes 1 \otimes \rho_3 \rho'_3) = 0\,.
$$
For the monomial $\rho'_1 \rho_2 \otimes \rho'_2 \rho_1 \otimes 1 \otimes \rho_3 \rho'_3$
we have
$$
\si_1(\rho'_1 \rho_2 \otimes \rho'_2 \rho_1 \otimes 1 \otimes \rho_3 \rho'_3) = 
- \rho_1 \rho_2 \otimes \rho'_2 \rho'_1 \otimes 1 \otimes \rho_3 \rho'_3
$$
$$
\si_2 ( \rho_1 \rho_2 \otimes \rho'_2 \rho'_1 \otimes 1 \otimes \rho_3 \rho'_3 ) = 
\rho_1 \rho'_2 \otimes \rho_2 \rho'_1 \otimes 1 \otimes \rho_3 \rho'_3
$$
and 
$$
\vs_{12} (\rho_1 \rho'_2 \otimes \rho_2 \rho'_1 \otimes 1 \otimes \rho_3 \rho'_3) = 
\rho_2 \rho'_1 \otimes \rho_1 \rho'_2 \otimes 1 \otimes \rho_3 \rho'_3\,,
$$
where $\vs_{12}$ is the transposition $(12)$ in $S_3$\,.

Hence, 
$$
\vs_{12} \si_2 \si_1
(\rho'_1 \rho_2 \otimes \rho'_2 \rho_1 \otimes 1 \otimes \rho_3 \rho'_3) = 
- (\rho'_1 \rho_2 \otimes \rho'_2 \rho_1 \otimes 1 \otimes \rho_3 \rho'_3)\,.
$$

Thus 
$$
\rho'_1 \rho_2 \otimes \rho'_2 \rho_1 \otimes 1 \otimes \rho_3 \rho'_3
= \frac{1}{2} 
\big(\,
\rho'_1 \rho_2 \otimes \rho'_2 \rho_1 \otimes 1 \otimes \rho_3 \rho'_3 
- \vs_{12} \si_2 \si_1
(\rho'_1 \rho_2 \otimes \rho'_2 \rho_1 \otimes 1 \otimes \rho_3 \rho'_3) 
\, \big)\,.
$$

In other words, the monomial $\rho'_1 \rho_2 \otimes \rho'_2 \rho_1 
\otimes 1 \otimes \rho_3 \rho'_3$
belongs to the subspace spanned by vectors of the form
\eqref{ker-Te}.
\end{example}

\subsection{An auxiliary cochain complex $\La^{-2}\Ger \big(S(V_e) \big)$}
\label{sec:Ger-SVe}

Let us consider the free $\La^{-2}\Ger$-algebra generated by 
$S(V_e)$
\begin{equation}
\label{Ger-S-Ve}
\La^{-2}\Ger(S(V_e))\,.
\end{equation}

Using the reduced  comultiplication: 
\begin{equation}
\label{wt-D}
\wt{\D} (X) = \D(X) - 1 \otimes X - X \otimes 1 
\end{equation}
on $S(V_e)$ we introduce on \eqref{Ger-S-Ve} the 
degree $1$ derivation $\de$ defined by the formula
\begin{equation}
\label{de-Ger}
\de (X) = - \sum_{i} \{ X'_i , X''_i \}
\end{equation}
where $X \in S(V_e)$ and $X'_i$, $X''_i$ are 
tensor factors in 
$$
\wt{\D} (X)  = \sum_i  X'_i \otimes X''_i\,.
$$

For example, since
$$
\wt{\D} (v) = 0 \qquad \forall~~ v \in V_e \subset S(V_e)
$$
we have 
$$
\de(v) = 0 \qquad \forall~~ v \in V_e \subset S(V_e)\,.
$$

The Jacobi identity implies that 
$$
\de^2 = 0\,.
$$

Thus $\de$ is a differential on \eqref{Ger-S-Ve}.

It is clear that the free $\La^{-1}\Lie$-algebra 
$ \La^{-1}\Lie(S(V_e))$ is a subcomplex 
of  \eqref{Ger-S-Ve}. Furthermore,
\begin{equation}
\label{Ger-S-Ve-new}
\La^{-2}\Ger(S(V_e)) \cong 
\bs^{-2} S\big( \bs^2  \La^{-1}\Lie(S(V_e))\big)
\end{equation}
as cochain complexes. 

On the other hand, Theorem \ref{thm:Harr}
from Appendix\footnote{Since formulas 
\eqref{de-Ger} and \eqref{de-Harr} for the differentials 
differ only by the overall sign factor, Theorem \ref{thm:Harr} can be
applied in this case.} 
\ref{app:Harr} implies that for 
every cocycle $c \in  \La^{-1}\Lie(S(V_e))$
there exists a vector $c_1 \in  \La^{-1}\Lie(S(V_e))$
and a vector $v \in V_e$ such that 
$$
c = v + \de (c_1)\,.
$$
Furthermore, each non-zero vector $v \in V_e$ is 
a non-trivial cocycle in $ \La^{-1}\Lie(S(V_e))$\,.

Therefore, due to K\"unneth's theorem,  
\begin{equation}
\label{H-Ger-S-Ve}
H^{\bul} \big(\, \La^{-2}\Ger(S(V_e)), \de  \,\big) \cong \bs^{-2} S\big(\bs^2 V_e \big)
\end{equation}
and the space 
$$
H^{\bul} \big(\, \La^{-2}\Ger(S(V_e)), \de  \,\big)
$$
is spanned by the cohomology classes of the vectors
\begin{equation}
\label{they-span}
b_1 \dots b_n \otimes (v_1\otimes \dots \otimes v_n)\,,
\end{equation}
where $v_i\in V_e$ and $b_1 \dots b_n$ is the 
generator of $\La^{-2} \Com(n) \subset \La^{-2} \Ger(n)$\,.

Thus we arrive at the following statement.
\begin{prop}
\label{prop:Ger-SVe}
For any cocycle 
$$
c \in \La^{-2}\Ger(S(V_e))
$$
there exists a vector $c_1 \in \La^{-2}\Ger(S(V_e))$
such that the difference 
$$
c - \de (c_1) 
$$
belongs to the linear span of \eqref{they-span}. 
Furthermore, a vector 
$$
Y \in \bs^{-2} S\big(\bs^2 V_e \big)
$$
is $\de$-exact if and only if $Y = 0$\,. ~~~~~~ $\Box$
\end{prop}

\subsubsection{A equivalent description of \eqref{Ger-S-Ve} in terms of invariants}
\label{sec:cG-Ve}

Let $\cG'(V_e)$ denote the following graded vector space
\begin{equation}
\label{cGpr-Ve}
\cG'(V_e) : =    \bigoplus_n
\Big( \La^{-2}\Ger(n) \otimes  \big(S(V_e) \big)^{\otimes\, n} \Big)^{S_n}\,.
\end{equation}

Since our base field has characteristic zero, this graded vector space is 
isomorphic to 
\begin{equation}
\label{Ger-SVe-coinv}
\La^{-2}\Ger(S(V_e)) =   \bigoplus_n
\Big( \La^{-2}\Ger(n) \otimes  \big(S(V_e) \big)^{\otimes\, n} \Big)_{S_n}\,.
\end{equation}

For example, one may define an isomorphism $\mI$ from 
\eqref{Ger-SVe-coinv} to \eqref{cGpr-Ve}  by the formula: 
\begin{equation}
\label{mI-dfn}
\mI (w; X_1 \otimes \dots \otimes X_n) = 
\sum_{\si \in S_n}  (-1)^{\ve(\si)}(\si(w); X_{\si^{-1}(1)} \otimes \dots \otimes X_{\si^{-1}(n)})\,,
\end{equation}
where $w \in \La^{-2}\Ger(n)$, $X_i \in S(V_e)$, the sign 
factor $(-1)^{\ve(\si)}$ comes from the usual Koszul rule,
 and 
$(w; X_1 \otimes \dots \otimes X_n)$ represents a vector in
$$
\Big( \La^{-2}\Ger(n) \otimes  \big(S(V_e) \big)^{\otimes\, n} \Big)_{S_n}\,.
$$

Let 
$$
\sum_t (w_t; X^t_1\otimes \dots \otimes X^t_n) 
$$
be a vector in 
$$
\Big( \La^{-2}\Ger(n) \otimes  \big(S(V_e) \big)^{\otimes\, n} \Big)^{S_n}
$$
and let $\de'$ be a degree $1$ operation on 
$\cG'(V_e)$ given by the equation: 
\begin{equation}
\label{de-pr}
\de'  \left( \sum_t (w_t; X^t_1\otimes \dots \otimes X^t_n) \right)= 
\end{equation}
$$
\sum_t \sum_{\si \in \Sh_{n,1}}
\si (\{w_t, b_{n+1}\};  X^t_1\otimes \dots \otimes X^t_n \otimes 1)  
$$
$$
- \sum_t 
\sum_{\tau \in \Sh_{2,n-1}} (-1)^{|w_t|}
\tau (w_t \circ_1 \{ b_1, b_2 \} ; \D  X^t_1\otimes  X^t_2 
\otimes \dots \otimes X^t_n)\,.
$$

A direct but tedious computation shows that
\begin{equation}
\label{mI-diff}
\mI \circ \de = 2 \de' \circ \mI\,. 
\end{equation}

In other words, $\de'$ \eqref{de-pr} is a differential on 
\eqref{cGpr-Ve} and the cohomology of the cochain complex 
\begin{equation}
\label{cG-Ve-comp}
\big( \cG'(V_e),  \de' \big)
\end{equation}
is isomorphic to the cohomology of  \eqref{Ger-S-Ve} with 
the differential \eqref{de-Ger}\,.

For our purpose, we need to switch to yet another 
cochain complex $\cG(V_e)$ isomorphic to \eqref{cG-Ve-comp}. 
This new cochain complex is obtained from $ \cG'(V_e)$ by 
exchanging the order of the tensor factors. Namely, 
\begin{equation}
\label{cG-Ve}
\cG(V_e) : =    \bigoplus_n
\Big( \big(S(V_e) \big)^{\otimes\, n} \otimes 
\La^{-2}\Ger(n) \Big)^{S_n}\,.
\end{equation}

The differential $\ti{\de}$ induced  on \eqref{cG-Ve} by 
the natural isomorphism between \eqref{cGpr-Ve} and \eqref{cG-Ve} 
is given by the formula:
\begin{equation}
\label{ti-de}
\ti{\de}  \left( \sum_t (X^t_1\otimes \dots \otimes X^t_n; w_t) \right)= 
\end{equation}
$$
\sum_t \sum_{\si \in \Sh_{n,1}} (-1)^{ |X^t_1| + \dots + |X^t_n| }
\si ( X^t_1\otimes \dots \otimes X^t_n \otimes 1; 
\{w_t, b_{n+1}\} )  
$$
$$
- \sum_t 
\sum_{\tau \in \Sh_{2,n-1}} (-1)^{|w_t|+ |X^t_1| + \dots + |X^t_n| }
\tau (\D  X^t_1\otimes  X^t_2 
\otimes \dots \otimes X^t_n ;
w_t \circ_1 \{ b_1, b_2 \} )\,.
$$

Thus Proposition \ref{prop:Ger-SVe} implies the following 
statement.
\begin{cor}
\label{cor:cG-Ve}
For any cocycle 
$$
c \in \cG(V_e) 
$$
there exists a vector $c_1 \in \cG(V_e)$
such that the difference 
$$
c - \ti{\de} (c_1) 
$$
belongs to the linear span of vectors of the form
\begin{equation}
\label{H-cG-Ve}
\sum_{\si \in S_n}
 (v_1\otimes \dots \otimes v_n; b_1 \dots b_n),
\end{equation}
where $v_1, v_2, \dots, v_n \in V_e$ and $b_1 \dots b_n$ is the 
generator of $\La^{-2} \Com(n) \subset \La^{-2} \Ger(n)$\,.
Furthermore, a linear combination $Y$ of vectors of the form
\eqref{H-cG-Ve}
is $\ti{\de}$-exact  if and only if $Y = 0$\,. ~~~~~~ $\Box$
\end{cor}

\subsection{The associated graded complex $\Gr\, \Conv^{\oplus}(\Ger^{\vee}, \Gra)$}
\label{sec:Gr-Ger-Gra}

Let us recall that $\mb(w)$ denotes the total number of Lie brackets 
in a monomial $w\in \La^{-2} \Ger(n)$\,.

Let 
$$
\sum_{i} v_i \otimes w_i 
$$
be a vector in 
$$
\Big( \Gra(n) \otimes  \La^{-2} \Ger(n) \Big)^{S_n}
$$
such that the number $k_{\mb} = \mb(w_i)$ is the same for every 
monomial $w_i$\,.

It is obvious that for every monomial $w'_j$ in 
$$
\pa (v \otimes w) = \sum_{j} v'_j \otimes w'_j
$$
we have $\mb(w'_j) = k_{\mb}$ or  $\mb(w'_j) = k_{\mb} + 1$\,. 

This observation allows us to introduce an ascending filtration 
\begin{equation}
\label{Ger-Gra-filtr}
\dots \subset 
\cF^{m-1}_{\mb} \Conv^{\oplus}(\Ger^{\vee}, \Gra)  \subset \cF^{m}_{\mb} \Conv^{\oplus}(\Ger^{\vee}, \Gra)  
\subset \dots 
\end{equation}
where $\cF^{m}_{\mb} \Conv^{\oplus}(\Ger^{\vee}, \Gra)$ is spanned by 
homogeneous vectors
$$
\ga = \sum_{i} v_i \otimes w_i  \in  \Conv(\Ger^{\vee}, \Gra)
$$
in which each monomial $w_i$ satisfies the inequality 
$$
\mb(w_i) - |\ga| \le m\,.
$$

It is clear that the differential $\pa^{\Gr_{\mb}}$ on the associated graded 
complex
\begin{equation}
\label{Gr-Ger-Gra}
\Gr_{\mb}\, \Conv^{\oplus}(\Ger^{\vee}, \Gra)
\end{equation}
is obtained from the differential $\pa$ \eqref{diff-Ger-Gra}
by keeping only the terms which raise the number of Lie brackets 
in the second tensor factors. Namely, 
\begin{equation}
\label{diffGr-Ger-Gra}
\pa^{\Gr_{\mb}} = \big[\, \G_{\bb}\otimes \{b_1, b_2\}, ~ \big]\,.
\end{equation}

Our goal is  to give a convenient 
description of the associated graded 
complex \eqref{Gr-Ger-Gra} using the map $\Te$ \eqref{Te} introduced in 
Subsection \ref{sec:Gra-n-SVe} and the cochain 
complex \eqref{cG-Ve} introduced in Subsection \ref{sec:cG-Ve}.

First, we observe that, as a graded vector space, 
\begin{equation}
\label{Gr-Ger-Gra-new}
\Gr_{\mb}\, \Conv^{\oplus}(\Ger^{\vee}, \Gra)
\cong 
\bigoplus_{n \ge 1} \Big( \Gra(n) \otimes \La^{-2}\Ger(n) \Big)^{S_n}\,.
\end{equation}

Thus, due to Proposition \ref{prop:T-n-Gra-n}, the 
map $\Te$ \eqref{Te} induces an isomorphism of graded vector 
spaces
\begin{equation}
\label{Gr-Ger-Gra-desc}
\bigoplus_{n \ge 1} 
\bigoplus_{e \ge 0} 
 \left( \, \Big(  T_n \big(S(V_e) \big)  \otimes \La^{-2}\Ger(n) 
 \Big)_{ S_e  \ltimes (S_2)^{e}} \,\right)^{S_n}  \cong
\Gr_{\mb}\, \Conv^{\oplus}(\Ger^{\vee}, \Gra)
\end{equation}

Since the action of the group $ S_e  \ltimes (S_2)^{e}$ commutes 
with the action of $S_n$ we conclude that $\Te$ induces an isomorphism 
of graded vector spaces: 
\begin{equation}
\label{Te-pr}
\Te' ~:~ \bigoplus_{e \ge 0} 
\left(
\bigoplus_{n \ge 1}  \Big( T_n \big(S(V_e) \big)  \otimes \La^{-2}\Ger(n) 
\Big)^{S_n}
\right)_{S_e  \ltimes (S_2)^{e}} ~\to~
\Gr_{\mb}\, \Conv^{\oplus}(\Ger^{\vee}, \Gra)\,.
\end{equation}

On the other hand,
\begin{equation}
\label{sub-cG-Ve}
\bigoplus_{n \ge 1}  \Big( T_n \big(S(V_e) \big)  \otimes \La^{-2}\Ger(n) 
\Big)^{S_n}
\end{equation}
is a subspace in the cochain complex $\cG(V_e)$ \eqref{cG-Ve}.

We claim that 
\begin{prop}
\label{prop:Gr-Ger-Gra}
The subspace \eqref{sub-cG-Ve} is a direct summand in
the cochain complex  $\cG(V_e)$  \eqref{cG-Ve} with the differential $\ti{\de}$
\eqref{ti-de}. Furthermore, the isomorphism \eqref{Te-pr} 
is compatible with the differentials. 
\end{prop}
\begin{proof}
Let us recall, from Subsection \ref{sec:Gra-n-SVe}, that $V_e$ is the graded 
vector space of finite linear combinations of variables from the set  \eqref{rhos}.

The subspace \eqref{sub-cG-Ve} is spanned by vectors
of the form
\begin{equation}
\label{span-sub-cG-Ve}
\sum_{\si \in S_n} \si ( X_1 \otimes X_2 \otimes \dots \otimes X_n \,;\, w)
\end{equation}
where $w\in \La^{-2}\Ger(n)$ and 
\begin{equation}
\label{XXX-here}
X_1 \otimes X_2 \otimes \dots \otimes X_n
\end{equation}
is a monomial in $ \big(S(V_e) \big)^{\otimes\, n}$ satisfying 
\begin{pty}
\label{P:rhos}
Each variable from the set \eqref{rhos} appears in  
\eqref{XXX-here} exactly once. 
\end{pty}

It is clear that this subspace is closed with respect to $\ti{\de}$ 
\eqref{ti-de}. Moreover, the cochain complex $\cG(V_e)$
splits into the direct sum of \eqref{sub-cG-Ve}  and
the subcomplex spanned 
by vectors of the form \eqref{span-sub-cG-Ve} for which 
\eqref{XXX-here} does \und{not} satisfy Property \ref{P:rhos}.

To prove equation
\begin{equation}
\label{Te-pr-diff}
\pa^{\Gr_{\mb}} \circ \Te' =  \Te' \circ \ti{\de}
\end{equation}
we consider a monomial \eqref{XXX-here} satisfying 
Property \ref{P:rhos} and a vector $w\in \La^{-2}\Ger(n)$\,.
We denote by $\G$ the graph in $\gra_n$ which corresponds 
to  the  monomial \eqref{XXX-here}.
 
Going through the construction of the map $\Te$ it 
is easy to verify that 
\begin{equation}
\label{master-eq}
(\Te \otimes 1)  \circ \ti{\de} \left( \sum_{\si \in S_n} (-1)^{\ve(\si)}
X_{\si^{-1}(1)} \otimes X_{\si^{-1}(2)} \otimes \dots \otimes X_{\si^{-1}(n)} 
\otimes \si(w) \right)= 
\end{equation}
$$
\sum_{\la \in \Sh_{n, 1}} (-1)^{|\G|}
\sum_{\si \in S_n} \la(\G_{\bb} \circ_1 \si(\G)) \otimes \la \big(\{ \si(w), b_{n+1} \}\big) -
$$
$$
- \sum_{\tau \in \Sh_{2, n-1}}  (-1)^{|\G| + |w|}
\sum_{\si \in S_n} \tau\big( \si(\G) \circ_1 \G_{\bb}\big)  \otimes  
\tau \big( \si(w) \circ_1 \{b_1, b_2\} \big)\,.
$$

Since the right hand side of \eqref{master-eq} equals
$$
\left[ ~ \G_{\bb} \otimes \{b_1, b_2 \} ~, ~\sum_{\si \in S_n} \si(\G) \otimes \si(w) ~\right]
$$
equation \eqref{Te-pr-diff} follows.
\end{proof}

Combining Corollary \ref{cor:cG-Ve} with Proposition \ref{prop:Gr-Ger-Gra}
we deduce 
\begin{cor}
\label{cor:H-Gr-GerGra}
For the associated graded complex \eqref{Gr-Ger-Gra} we 
have 
\begin{equation}
\label{H-Gr-GerGra}
H^{\bul} \big( 
\Gr_{\mb}\, \Conv^{\oplus}(\Ger^{\vee}, \Gra)
\big) = \begin{cases}
\bbK  \qquad {\rm if} ~~ \bul = 1  \\
 \bfzero \qquad {\rm otherwise}\,.
\end{cases}
\end{equation}
Furthermore, 
$$
H^{1} \big( 
\Gr_{\mb}\, \Conv^{\oplus}(\Ger^{\vee}, \Gra)
\big) 
$$
is spanned by the cohomology class of the vector 
represented by $\G_{\ed} \otimes b_1 b_2$\,.
\end{cor}
\begin{proof}
Since the cochain complex  \eqref{sub-cG-Ve} is a direct summand in 
$\cG(V_e)$ and the cohomology commutes with taking coinvariants, Corollary 
\ref{cor:cG-Ve} implies that the cohomology of
the cochain complex 
\begin{equation}
\label{sub-fixed-e}
\left(
\bigoplus_{n \ge 1}  \Big( T_n \big(S(V_e) \big)  \otimes \La^{-2}\Ger(n) 
\Big)^{S_n}
\right)_{S_e  \ltimes (S_2)^{e}}
\end{equation}
is spanned by the classes of vectors of the form 
\begin{equation}
\label{H-sub-fixed-e}
\sum_{\si \in S_{2e}}
 \si ( \rho_1\otimes \rho'_1\otimes \rho_2 \otimes 
 \rho'_2\otimes \dots \otimes \rho_e \otimes \rho'_e ~ ; ~ b_1 \dots b_{2e}),
\end{equation}
$b_1 \dots b_{2e}$ is the generator of $\La^{-2} \Com(2e) \subset \La^{-2} \Ger(2e)$\,. 

Since variables $\rho_1, \dots, \rho_e$ are odd, 
it is not hard to see that  \eqref{H-sub-fixed-e} represents the zero 
vector in the coinvariants \eqref{sub-fixed-e} whenever $e > 1$\,.

On the other hand, the map $\Te'$ \eqref{Te-pr} sends the vector 
\begin{equation}
\label{non-zero}
( \rho_1\otimes \rho'_1 ; b_1 b_2) + 
( \rho'_1\otimes \rho_1 ; b_1 b_2)
\end{equation}
to the non-trivial cocycle 
$$
2 \G_{\ed} \otimes b_1 b_2\,.
$$

Hence, the corollary follows from Proposition \ref{prop:Gr-Ger-Gra}.
\end{proof}

\subsection{Proof of Theorem  \ref{thm:rigidity}}
\label{sec:thm:rigidity}  

Let us denote by $\cH$ the subcomplex 
of  $ \Conv^{\oplus}(\Ger^{\vee}, \Gra)$ 
\begin{equation}
\label{cH}
\cH = \bbK \L \, \G_{\ed} \otimes b_1b_2 \, \R 
\end{equation}
spanned by the single cocycle $ \G_{\ed} \otimes b_1b_2 $\,.

By construction, the cochain complex $\cH$ carries 
the zero differential. Moreover, restricting  
\eqref{Ger-Gra-filtr} on $\cH$ we get the ``silly'' filtration 
\begin{equation}
\label{silly-filtr-cH}
\cF^{m} \cH^k  = 
\begin{cases}
 \cH^k \qquad {\rm if} ~~ m \ge -k\,,  \\
 \bfzero  \qquad {\rm otherwise}
\end{cases} 
\end{equation}
with 
\begin{equation}
\label{Gr-cH-cH}
\Gr \,\cH \cong \cH\,.
\end{equation}

Corollary \ref{cor:H-Gr-GerGra} implies that the embedding 
\begin{equation}
\label{cH-GerGra}
\cH \hookrightarrow \Conv^{\oplus}(\Ger^{\vee}, \Gra)
\end{equation}
induces a quasi-isomorphism on the level of 
associated graded complexes. 

Since the filtrations on $ \Conv^{\oplus}(\Ger^{\vee}, \Gra)$ and
$\cH$ are bounded from the left and cocomplete,  Lemma \ref{lem:q-iso}
from Appendix \ref{app:q-iso} implies that the embedding \eqref{cH-GerGra}
is quasi-isomorphism of cochains complexes. 

Combining this fact with Proposition \ref{prop:Euler-GerGra} and 
Corollary \ref{cor:Euler-GerGra}  we conclude that 
the embedding 
\begin{equation}
\label{cH-GerGra-big}
\cH \hookrightarrow \Conv(\Ger^{\vee}, \Gra)
\end{equation}
is also a quasi-isomorphism of cochain complexes. 

Since $\cH$ is spanned by the cocycle  $ \G_{\ed} \otimes b_1b_2 $, 
Theorem  \ref{thm:rigidity} is proved. ~~~~~ $\Box$

\section{Deformation complex of $\Ger$ versus Kontsevich's graph complex}
\label{sec:DefGer-fGC}

This section is the culmination of our text. Using the results proved 
above, we establish here 
a link between the (extended) deformation complex $\Conv(\Ger^{\vee}, \Ger)$ \eqref{Def-Ger-ext} 
of the operad $\Ger$ and full graph complex $\fGC$ (See Definition \ref{dfn:fGC}.) 

First, recall that, due to decompositions  \eqref{fGC-fGC-conn}
and \eqref{Def-Ger-as-Sym}, the cohomology of the cochain complex 
$\fGC$ (resp.  $\Conv(\Ger^{\vee}, \Ger)$) can be expressed in terms 
of the cohomology of its ``connected part'' $\fGC_{\conn}$  
(resp.  $\Conv(\Ger^{\vee}, \Ger)_{\conn}$). Namely, 
\begin{equation}
\label{HfGC-fGC-conn}
H^{\bul}(\fGC) \cong \bs^{-2}  \wh{S}\big(\bs^2 \, H^{\bul}(\fGC_{\conn})\big)\,,
\end{equation}
and
\begin{equation}
\label{HGerGer-conn}
H^{\bul}\big(  \Conv(\Ger^{\vee}, \Ger) \big) \cong \bs^{-2}  \wh{S}
\Big(\bs^2 \, H^{\bul}\big(    \Conv(\Ger^{\vee}, \Ger)_{\conn}  \big)  \Big)\,.
\end{equation}

Let us denote by $\mR$ the natural map of graded vector spaces
\begin{equation}
\label{mR}
\mR ~: ~ \Conv(\Ger^{\vee}, \Gra)_{\conn} ~ \to ~  \fGC_{\conn} = \Conv(\La^2\coCom, \Gra)_{\conn} 
\end{equation}
given by the formula
$$
\mR(f) = f \Big|_{\La^2\coCom}\,.  
$$
It is not hard to see that $\mR$ is a map of cochain complexes. 
We observe that the map of dg Lie algebras $\io_*$ \eqref{io-star}
$$
\io_* : \Conv(\Ger^{\vee}, \Ger) \to \Conv(\Ger^{\vee}, \Gra)\,. 
$$
satisfies the following property
$$
\mR(\io_*(X)) = 0\,, \qquad \forall~~ X \in \Xi_{\conn}\,,
$$
where $\Xi_{\conn}$ is defined in \eqref{Xi-conn}.

Therefore, restricting $\io_*$ to the subcomplex $\Xi_{\conn}$
we get a map of cochain complexes 
\begin{equation}
\label{psi}
\psi := \io_* \Big|_{\Xi_{\conn}} ~:~ \Xi_{\conn} ~\to~ \ker \mR\,.
\end{equation}

We claim that 
\begin{prop}
\label{prop:Xi-ker-mR}
The map $\psi$ \eqref{psi} is a quasi-isomorphism of cochain complexes.
\end{prop}
Let us postpone the proof of this proposition to Subsection \ref{sec:ker-mR}
and deduce a link between the cohomology of $\Conv(\Ger^{\vee}, \Ger)_{\conn}$
and the cohomology of $\fGC_{\conn}$\,.

Recall that, due to Corollary \ref{cor:Conv-Ger-Gra-conn}, 
$H^{\bul}\big( \Conv(\Ger^{\vee}, \Gra)_{\conn} \big)$
is spanned by the cohomology class of the vector $\G_{\ed} \otimes b_1 b_2$\,.

Therefore, if we set 
\begin{equation}
\label{Conv-plus}
\Conv(\Ger^{\vee}, \Gra)^+_{\conn} = \bbK \oplus 
\Conv(\Ger^{\vee}, \Gra)_{\conn}
\end{equation}
and extend the differential $\pa$ to $\Conv^+(\Ger^{\vee}, \Gra)_{\conn} $
by declaring that for $1\in \bbK$
\begin{equation}
\label{Conv-plus-pa}
\pa(1) = \G_{\ed} \otimes b_1 b_2\,,
\end{equation}
then we get an acyclic cochain complex 
$$
\left( \Conv(\Ger^{\vee}, \Gra)^+_{\conn}, \pa \right)\,.
$$

Similarly, we ``add'' to the graph complex $\fGC_{\conn}$ 
a one-dimensional vector space 
\begin{equation}
\label{fGC-plus}
\fGC^{+}_{\conn} = \bbK \oplus \fGC_{\conn}
\end{equation}
and extend the differential by declaring that for 
$1\in \bbK$
$$
\pa (1) =  \G_{\ed} \,. 
$$

Due to Exercise \ref{exer:G-bul-G-lp} from Section \ref{sec:fGC-first}
we have
$$
\pa \G_{\bul} = \G_{\ed}\,,
$$
where  $\G_{\bul}$ is the graph with the 
single vertex and no edges. Therefore, 
\begin{equation}
\label{fGC-conn-fGC-plus}
H^{\bul}\big( \fGC^{+}_{\conn} \big) \cong 
H^{\bul}\big( \fGC_{\conn} \big) \oplus  \bbK\L \phi \R\,,
\end{equation}
where $\phi$ is the cohomology class represented by the 
cocycle
$$
\G_{\bul} ~ - ~1 \in  \fGC^{+}_{\conn} \,.
$$

The map $\mR$ \eqref{mR} extends in the obvious way to 
the morphism of cochain complexes: 
\begin{equation}
\label{mR-plus}
\mR^+ ~: ~ \Conv(\Ger^{\vee}, \Gra)^+_{\conn} ~ \to ~  \fGC^+_{\conn}\,.
\end{equation}
Furthermore, 
$$
\ker(\mR^+) = \ker(\mR)\,.
$$

Thus we arrive at the diagram 
\begin{equation}
\label{SES}
\xymatrix@M=0.5pc{
~ &  \Conv(\Ger^{\vee}, \Ger)_{\conn}  & ~ & ~ & ~ \\
~ & \Xi_{\conn} \ar[d]^{\psi}  \ar@{^{(}->}[u]^{\emb_{\Xi}}   & ~ & ~ & ~ \\
\bfzero \ar[r] &
 \ker(\mR) \ar[r] &
\Conv(\Ger^{\vee}, \Gra)^+_{\conn} \ar[r] & 
 \fGC^+_{\conn} \ar[r] & \bfzero
}
\end{equation}

The bottom row of this diagram is an exact sequence of cochain complexes. 
The top vertical arrow $\emb_{\Xi}$ is a quasi-isomorphism due to  Theorem \ref{thm:Xi}. 
The vertical arrow $\psi$ is also a quasi-isomorphism due to Proposition 
\ref{prop:Xi-ker-mR}. Finally the cochain complex $\Conv(\Ger^{\vee}, \Gra)^+_{\conn}$
in the middle of the exact sequence is acyclic.

Using diagram \eqref{SES}, we can now prove the main theorem 
of these notes.
\begin{thm}[T. Willwacher, \cite{Thomas}]
\label{thm:main} 
If $\fGC_{\conn}$ is the ``connected part'' of the full graph complex 
$\fGC$ \eqref{fGC} and $\Conv(\Ger^{\vee}, \Ger)_{\conn}$ is
the ``connected part'' of the extended deformation 
complex $\Conv(\Ger^{\vee}, \Ger)$ \eqref{Def-Ger-ext} of the operad 
$\Ger$ then
\begin{equation}
\label{fGC-Def-Ger}
H^{\bul + 1} \big( \Conv(\Ger^{\vee}, \Ger)_{\conn} \big) \cong
H^{\bul}  \big( \fGC_{\conn} \big) ~ \oplus ~ \bbK\,. 
\end{equation}
\end{thm}
\begin{proof}
Since the cochain complex $\Conv(\Ger^{\vee}, \Gra)^+_{\conn}$ in \eqref{SES} is acyclic,
the connecting homomorphism induces an isomorphism 
$$
H^{\bul}  \big( \fGC^+_{\conn} \big) \cong  H^{\bul+1}\big(\ker(\mR)\big)\,.
$$

On the other hand, 
$$
H^{\bul}\big(\ker(\mR)\big) \cong H^{\bul}\big( \Xi_{\conn} \big) 
\cong H^{\bul}\big(  \Conv(\Ger^{\vee}, \Ger)_{\conn}  \big)
$$
because both $\psi$ and $\emb_{\Xi_{\conn}}$ are quasi-isomorphisms.

Therefore, 
\begin{equation}
\label{fGC-plus-Def-Ger}
H^{\bul + 1} \big( \Conv(\Ger^{\vee}, \Ger)_{\conn} \big) \cong
H^{\bul}  \big( \fGC^+_{\conn} \big) \,. 
\end{equation}

Thus, using the isomorphism \eqref{fGC-conn-fGC-plus},  we arrive at the desired 
result \eqref{fGC-Def-Ger}.
\end{proof}
\begin{rem}
\label{rem:1-aa-brack}
The above proof gives us a concrete isomorphism from 
\begin{equation}
\label{H-fGC-plus}
H^{\bul}  \big( \fGC_{\conn} \big) ~ \oplus ~ \bbK
\end{equation}
to 
\begin{equation}
\label{H-Def-Ger-conn}
H^{\bul + 1} \big( \Conv(\Ger^{\vee}, \Ger)_{\conn} \big)\,. 
\end{equation}

Chasing through diagram  \eqref{SES}, it is not hard to see that the 
vector $1\in \bbK$ in the second summand of \eqref{H-fGC-plus} is sent,
via this isomorphism, to the class represented by the cocycle 
$$
a_1 a_2 \otimes \{b_1, b_2\}
$$
or the cocycle 
$$
- \{ a_1,  a_2 \} \otimes b_1 b_2
$$
in  $\Conv(\Ger^{\vee}, \Ger)_{\conn}$\,. 
\end{rem}
\begin{rem}
\label{rem:grt-H0}
According to \cite{Dima-GT}, the Lie algebra $\grt$ of the Grothendieck-Teichmueller 
group $\GRT$ embeds into $H^0 \big( \Conv(\Ger^{\vee}, \Ger) \big)$\,. Since 
$\grt$ is infinite dimensional \cite{Drinfeld}, the spaces  $H^0 \big( \Conv(\Ger^{\vee}, \Ger) \big)$
and $H^0(\fGC)$ are also infinite dimensional.
\end{rem}

\subsection{Proof of Proposition \ref{prop:Xi-ker-mR}}   
\label{sec:ker-mR}

Let us prove that the map 
\begin{equation}
\label{psi-oplus}
\psi \Big|_{ \Xi^{\oplus}_{\conn} } ~:~ \Xi^{\oplus}_{\conn} \to 
\ker(\mR) \cap  \Conv(\Ger^{\vee}, \Gra)^{\oplus}_{\conn}
\end{equation}
is a quasi-isomorphism of cochain complexes.

For this purpose we apply the general construction of 
Section \ref{sec:Ger-cO} to the case when $\cO = \Gra$\,.

Following Section \ref{sec:Ger-cO}, the cochain complex 
$\Conv(\Ger^{\vee}, \Gra)^{\oplus}$ carries the ascending filtration 
\begin{equation}
\label{filtr-Ger-Gra-oplus}
\dots \subset \cF^{m-1}\, \Conv^{\oplus}(\Ger^{\vee}, \Gra) \subset 
\cF^m\, \Conv^{\oplus}(\Ger^{\vee}, \Gra) \subset \dots\,,
\end{equation}
where $\cF^m\, \Conv^{\oplus}(\Ger^{\vee}, \Gra)$ consists of 
sums
$$
\sum_i v_i \otimes w_i \in  \bigoplus_n \big( \Gra(n) \otimes \La^{-2} \Ger(n) \big)^{S_n}
$$
which satisfy 
$$
\mL_1 (w_i) - |\, v_i \otimes w_i \,| \le  m\,, \qquad \forall ~~ i\,.
$$

Furthermore, due to Proposition \ref{prop:Gr-Ger-cO}, the formula
\begin{equation}
\label{Ups-Gra-def}
\Ups_{\Gra} \left( \sum_i v_i \otimes w_i \right) : =
\sum_{\si \in \Sh_{r,n}} \sum_i
\si (v_i) \otimes
\si (b_1 \dots b_r \, w_i(b_{r+1}, \dots, b_{r+n})) 
\end{equation}
$$
 \sum_i v_i \otimes w_i  \in  \Big(\bs^{2r} \Gra(r+n)^{S_r}
\otimes \La^{-2}\Ger^{\hrt}(n)  \Big)^{S_n}
$$ 
defines an isomorphism of cochain complexes
\begin{equation}
\label{Ups-Gra}
\Ups_{\Gra} ~: ~ \bigoplus_{n \ge 0}  \Big( \Tw^{\oplus}\Gra(n)
\otimes \La^{-2}\Ger^{\hrt}(n)  \Big)^{S_n} ~~\to~~ 
\Gr\, \Conv^{\oplus}(\Ger^{\vee}, \Gra)\,,
\end{equation}
where the differential on 
$$
 \bigoplus_{n \ge 0}  \Big( \Tw^{\oplus}\Gra(n)
\otimes \La^{-2}\Ger^{\hrt}(n)  \Big)^{S_n}
$$
comes from the differential $\pa^{\Tw}$ on $\Tw^{\oplus}\Gra(n)$\,.

Let us restrict the filtration \eqref{filtr-Ger-Gra-oplus} to the 
subcomplex 
$$
\ker(\mR) \cap  \Conv(\Ger^{\vee}, \Gra)^{\oplus}_{\conn}
$$
and recall that the $n$-th space
\begin{equation}
\label{fgraphs-n-here}
\fgraphs(n) : =  \fGraphs(n) \cap \Tw^{\oplus}\Gra(n)
\end{equation}
of the dg operad $\fgraphs$ is spanned  by vectors of the 
form
$$
\sum_{\si \in S_r}  \si (\G)\,,
$$ 
where the graph $\G \in \gra_{r+n}$ has no connected components 
which involve exclusively neutral vertices (i.e. vertices with labels $\le r$)\,. 

It is not hard to see that the restriction of $\Ups_{\Gra}$ to 
$$
 \bigoplus_{n \ge 2}  \Big( \fgraphs(n)
\otimes \La^{-2}\Ger^{\hrt}(n)  \Big)^{S_n}_{\conn}
$$
gives us an isomorphism 
\begin{equation}
\label{Ups-ker-mR-}
\Ups' :  \bigoplus_{n \ge 2}  \Big( \fgraphs(n)
\otimes \La^{-2}\Ger^{\hrt}(n)  \Big)^{S_n}_{\conn}
 ~ \to ~ 
 \Gr \left(
\ker(\mR) \cap  \Conv(\Ger^{\vee}, \Gra)^{\oplus}_{\conn}
\right)
\end{equation}
of cochain complexes. 

On the other hand, Corollary \ref{cor:Ger-fgraphs} implies that 
the natural embedding 
\begin{equation}
\label{Xi-to-fgraphs-Gerhrt}
\Xi^{\oplus} \hookrightarrow  \bigoplus_{n \ge 2}  \Big( \fgraphs(n)
\otimes \La^{-2}\Ger^{\hrt}(n)  \Big)^{S_n}
\end{equation}
is a quasi-isomorphism of cochain complexes.

Therefore, since the cone of the embedding 
\begin{equation}
\label{Xi-to-fgraphs-Gerhrt-conn}
\Xi^{\oplus}_{\conn} \hookrightarrow  \bigoplus_{n \ge 2}  \Big( \fgraphs(n)
\otimes \La^{-2}\Ger^{\hrt}(n)  \Big)^{S_n}_{\conn}
\end{equation}
is a direct summand in the cone of the embedding \eqref{Xi-to-fgraphs-Gerhrt}, 
the map \eqref{Xi-to-fgraphs-Gerhrt-conn} is also a quasi-isomorphism.

This observation allows us to conclude that 
the map \eqref{psi-oplus} induces a quasi-isomorphism 
on the level of associated graded complexes. 

Since the filtration   \eqref{filtr-Ger-Gra-oplus} is locally bounded 
and cocomplete, Lemma \ref{lem:q-iso} implies that  \eqref{psi-oplus} 
is indeed a quasi-isomorphism of cochain complexes. 

Thus, using the Euler characteristic trick, 
we conclude that the map  $\psi$ \eqref{psi} is also a quasi-isomorphism 
of cochain complexes. 

Proposition \ref{prop:Xi-ker-mR} is proved.  ~~~ $\Box$

\appendix

\section{Lemma on a quasi-isomorphism of filtered complexes}
\label{app:q-iso}

Let us recall that a cone $\Cone(f)$ of a morphism of 
cochain complexes $f : C \to K$ is the cochain complex 
$$
C \oplus \bs K
$$
with the differential 
$$
\pa^{\Cone}(v_1 + \bs v_2) = \pa(v_1) + \bs f(v_1) - \bs \pa (v_2)\,, 
$$
where we denote by $\pa$ the differentials on both 
complexes $C$ and $K$\,.

Let us also recall a claim which follows easily from 
Lemma 3 in \cite[Section III.3.2]{GelManin}:
\begin{claim}
\label{cl:Cone}
A morphism $f: C \to K$ of cochain complexes is 
a quasi-isomorphism  if and only if the cochain complex $\Cone(f)$
is acyclic. $\Box$
\end{claim}

Let $C$ be a cochain complex equipped with an ascending 
filtration: 
$$
\dots \subset \cF^{m-1} C \subset  \cF^{m} C \subset 
 \cF^{m+1} C  \subset  \dots\,.
$$
We say that the filtration on $C$ is {\it cocomplete} if 
\begin{equation}
\label{cocomplete}
C = \bigcup_{m}  \cF^{m} C\,.
\end{equation}
Furthermore, we say that the filtration on $C$ is {\it locally bounded from 
the left} if for every 
degree $d$ there exists an integers $m_d$ such that 
\begin{equation}
\label{loc-finite-left}
\cF^{m_d} C^d = \bfzero\,.
\end{equation}

Let us denote by $\Gr(C)$ the associated graded cochain complex 
\begin{equation}
\label{Gr-C}
\Gr(C) : = \bigoplus_m \cF^m C  \big/  \cF^{m-1} C\,.
\end{equation}

We will need the following claim.
\begin{claim}
\label{cl:Gr-acyclic}
Let $C$ be a cochain complex equipped with 
a cocomplete ascending filtration which is locally 
bounded from the left. 
If $\Gr(C)$ is acyclic 
then so is $C$\,.
\end{claim}
\begin{proof}
Let $v$ be cocycle in $C$ of degree $d$\,. 
Our goal is to show that there exists a vector 
$w \in C^{d-1}$ such that 
$$
v = \pa w\,.
$$
 
Since the filtration on $C$ is cocomplete there 
exists an integer $m$ such that 
$$
v \in \cF^m C^d\,.
$$ 

Therefore $v$ represents a cocycle in the quotient
$$
\cF^m C^d \big/  \cF^{m-1} C^d\,.
$$

On the other hand, $\Gr(C)$ is acyclic. Hence there exists 
a vector $w_m \in \cF^m C^{d-1} $ such that 
\begin{equation}
\label{in-1m}
v - \pa (w_m) \in   \cF^{m-1} C^d\,.
\end{equation}

The latter implies that the vector $v - \pa (w_m)$ represents 
a cocycle in  the quotient
$$
\cF^{m-1} C^d \big/  \cF^{m-2} C^d\,.
$$

Hence, there exists a vector $w_{m-1} \in \cF^{m-1} C^{d-1} $ such that 
\begin{equation}
\label{in-2m}
v - \pa (w_m) - \pa (w_{m-1})\in   \cF^{m-2} C^d\,.
\end{equation}

Continuing this process, we conclude that there 
exists a sequence of vectors 
$$
w_k \in  \cF^{k} C^{d-1}, \qquad k \le m
$$
such that for every $k < m $ we have 
\begin{equation}
\label{in-1k}
v - \pa (w_m + w_{m-1} + \dots + w_k) \in   \cF^{k-1} C^d\,.
\end{equation}

Since the filtration on $C$ is locally bounded from the left there exists
an integer $k_d < m $ such that $  \cF^{k_d-1} C^d =\bfzero$ 
and we get 
$$
v - \pa (w_m + w_{m-1} + \dots + w_{k_d}) = 0\,. 
$$

The desired statement is proved. 
\end{proof}

We are now ready to prove the following generalization of 
Claim \ref{cl:Gr-acyclic}.
\begin{lem}
\label{lem:q-iso}
Let $C$ and $K$ be cochain complexes 
equipped with cocomplete ascending filtrations
which are locally bounded from the left. 
Let $f: C \to K$ be a morphism of cochain complexes 
compatible with the filtrations. 
If the induced map of cochain complexes
$$
\Gr(f) : \Gr (C)  \to \Gr(K)
$$
is a quasi-isomorphism then so is $f$\,.
\end{lem}
\begin{proof}
Let us introduce the obvious ascending filtration on the 
cone of $f$
$$
\dots \subset \cF^{m-1} \Cone(f) \subset
\cF^{m} \Cone(f)  \subset  \cF^{m+1} \Cone(f) \subset \dots\,,
$$
\begin{equation}
\label{filtr-Cone}
\cF^m \Cone(f) = \cF^m C \, \oplus \, \bs \cF^m K \,.
\end{equation}
The differential $\pa^{\Cone}$ is compatible with the filtration 
\eqref{filtr-Cone} because $f$ is compatible with the filtrations 
on $C$ and $K$\,.

It is obvious that the filtration \eqref{filtr-Cone} is cocomplete 
and locally bounded from the left. Furthermore, 
it is not hard to see that 
$$
\Gr (\Cone(f)) = \Cone(\Gr(f))\,.
$$

Therefore, Claim \ref{cl:Cone} implies that $\Gr(\Cone(f))$ is 
acyclic. 

Combining this observation with Claim \ref{cl:Gr-acyclic} we 
conclude that $\Cone(f)$ is also acyclic. Therefore, applying Claim 
 \ref{cl:Cone} once again, we deduce the statement of the lemma.
\end{proof}

\begin{rem}
\label{rem:lem:q-iso}
Lemma \ref{lem:q-iso} is often used in the literature under the folklore 
name ``standard spectral sequence argument''. Unfortunately, a clean 
proof of this fact based on the use of a spectral sequence is 
very cumbersome. 
\end{rem}

\section{Harrison complex of the cocommutative coalgebra $S(V)$}
\label{app:Harr}

Let $V$ be a finite dimensional graded vector space. 
We consider the symmetric algebra 
\begin{equation}
\label{S-V}
S(V)
\end{equation}
as the cocommutative coalgebra with the
standard comultiplication: 
$$
\D (v_1 \dots v_n) = 1 \otimes (v_1 \dots v_n) + 
$$
\begin{equation}
\label{Delta}
\sum_{p=1}^{n-1}
\sum_{\si \in \Sh_{p, n-p}} 
(-1)^{\ve(\si, v_1,\dots, v_n)} \, 
v_{\si(1)} \dots v_{\si(p)} \otimes v_{\si(p+1)} \dots v_{\si(n)}
+  (v_1 \dots v_n) \otimes 1\,,
\end{equation}
where $v_1,\dots, v_n$ are homogeneous vectors in $V$
and the sign factor $(-1)^{\ve(\si, v_1,\dots, v_n)}$ is determined 
by the standard Koszul rule. 

We denote by $\wt{\D}$ the reduced comultiplication
which is define by the formula
\begin{equation}
\label{wt-Delta}
\wt{\D} (X) = \D (X) - X \otimes 1 - 1 \otimes X
\end{equation}

For example, $\wt{\D}(1) = - 1\otimes 1$ and 
$\wt{\D}(v) = 0$ for all $v \in V$\,.

Let us consider the free $\La^{-1}\Lie$-algebra
\begin{equation}
\label{Lainv-Lie-SV}
\La^{-1}\Lie(S(V))
\end{equation}
generated by $S(V)$\,. 

Let us denote by $X'_{i}$ and $X''_i$ the tensor factors of
$$
\wt{\D} (X) = \sum_i X'_i \otimes X''_i
$$
for a vector $X \in S(V)$ and introduce the degree $1$ derivation $\de$
of the free $\La^{-1}\Lie$-algebra \eqref{Lainv-Lie-SV} by setting
\begin{equation}
\label{de-Harr}
\de(X) = \sum_i \{X'_i, X''_i\}\,.
\end{equation}

Due to the Jacobi identity  
$$
\de^2 = 0\,.
$$
Hence $\de$ is a differential on  \eqref{Lainv-Lie-SV}
and we call 
\begin{equation}
\label{Harr-comp}
\big( \La^{-1}\Lie(S(V)), \de \big)
\end{equation}
the {\it Harrison complex} of $S(V)$\,.

It is easy to see that each non-zero vector $v \in V \subset  \La^{-1}\Lie(S(V))$
is a non-trivial cocycle in \eqref{Harr-comp}. 

The following theorem and its various 
versions\footnote{For a version of Theorem \ref{thm:Harr} 
we refer the reader to  \cite[Section 3.5]{Loday}. Another 
version of this theorem can also be deduced from statements
in \cite[Appendix B]{Quillen}.} are often referred to as ``well-known''. 
\begin{thm}
\label{thm:Harr}
For the Harrison complex \eqref{Harr-comp} we have 
$$
H^{\bul}\big( \La^{-1}\Lie(S(V)), \de \big) \cong V\,.
$$
More precisely, for every cocycle $c$ in  \eqref{Harr-comp}
there exists a vector $v \in V \subset  \La^{-1}\Lie(S(V))$
and a vector $c_1$ in  \eqref{Harr-comp} such that 
$$
c = v  + \de (c_1)
$$
Furthermore, a vector  $v \in V \subset  \La^{-1}\Lie(S(V))$
is an exact cocycle in  \eqref{Harr-comp}  if and only if $v = 0$\,.
\end{thm}
\begin{proof}
To prove this theorem we embed 
the suspension 
\begin{equation}
\label{sHarr-comp}
\bs \,\La^{-1}\Lie(S(V)) = \Lie (\bs\, S(V))
\end{equation}
of  \eqref{Harr-comp} 
into the tensor algebra 
\begin{equation}
\label{T-bsS-V}
T (\bs\, S(V)) 
\end{equation}
generated by $\bs\, S(V)$\,. 

The differential $\de$ on \eqref{sHarr-comp} can be 
extended to \eqref{T-bsS-V} in the obvious way: 
\begin{equation}
\label{de-on-T}
\de (\bs X) = 2\, \bs \otimes \bs \big(\wt{\D}(X) \big)\,. 
\end{equation}

To compute the cohomology of $\big(T (\bs\, S(V)) , \de \big)$
we consider the restricted dual complex
\begin{equation}
\label{T-dual}
(T (\bsi S(V')), \de' )\,,
\end{equation}
where $V'$ is the linear dual of $V$\,.

Since $T (\bs\, S(V))$ is a free associative algebra, it is convenient 
to view \eqref{T-dual} as the cofree coassociative coalgebra 
with the comultiplication given by deconcatenation. Furthermore, since 
$\de$ is a derivation of \eqref{T-bsS-V}, $\de'$ is coderivation. 
Therefore, $\de'$ is uniquely determined by its composition 
$p \circ \de'$ with the projection 
$$
p : T (\bsi S(V')) \to \bsi S(V')\,.
$$

It is easy to see that 
\begin{equation}
\label{p-de-pr}
p \circ \de' ( \bsi X_1 \otimes \dots \otimes \bsi X_n) = 
\begin{cases}
 (-1)^{|X_1|-1} \,2 \, \bsi \mu (X_1, X_2) \qquad {\rm if} ~~ n=2  \\
 0  \qquad {\rm otherwise}\,.
\end{cases}
\end{equation}
Here $X_1, \dots, X_n$ are homogeneous vectors in $S(V')$ 
and the map 
$$
\mu : S(V') \otimes S(V') \to S(V')
$$
is defined by the formula
\begin{equation}
\label{mu-S-Vpr}
\mu(X_1, X_2) = X_1 X_2 - \ve(X_1) X_2 - X_1 \ve(X_2)\,,
\end{equation}
where $\ve$ is the augmentation $\ve : S(V') \to \bbK$ of 
$S(V')$\,.

Using \eqref{p-de-pr}, it is not hard to see that 
\eqref{T-dual} is the Hochschild chain complex 
with the reversed grading and with rescaled differential
$$
C_{-\bul}(S(V'), \bbK)\,.
$$ 

Hence, due to the Hochschild-Kostant-Rosenberg theorem 
\cite{HKR}, we have 
\begin{equation}
\label{H-T-dual}
H^{\bul} (T (\bsi S(V')), \de' ) \cong S(\bsi V')\,. 
\end{equation} 

If we view $S(\bsi V')$ as the subspace of 
$T(\bsi V')$ which is, in turn, a subspace of \eqref{T-dual}, then 
the  Hochschild-Kostant-Rosenberg theorem can be restated 
as follows. For every cocycle 
$c$ in \eqref{T-dual} there exists a vector 
$X \in S(\bsi V') $  and a vector $c_1$ in \eqref{T-dual} such that 
$$
c = X + \de' (c_1)\,.
$$
Every vector $X \in S(\bsi V') $ is a cocycle in \eqref{T-dual}
and $X \in S(\bsi V')$ is an exact cocycle  if and only if $X = 0$\,.

Let us now go back to the cochain complex  \eqref{T-bsS-V}
with the differential \eqref{de-on-T}\,.  Let us consider 
$S(\bs \,V)$ as the subspace of 
$$
T(\bs\, V) \subset T(\bs \,S(V))\,.
$$ 
It is clear that every vector in $S(\bs\, V)$ is a cocycle 
in  \eqref{T-bsS-V}\,.

Dualizing the above statement about cocycles in 
\eqref{T-dual} we deduce the following.
\begin{claim}
\label{cl:T-bsS-V}
For every cocycle $c \in T(\bs\, S(V))$ there exists a vector 
$X \in S(\bs\, V)$ and a vector $c_1 \in  T(\bs\, S(V))$ such that 
$$
c = X + \de (c_1)\,.
$$
Furthermore, a vector  $X \in S(\bs\, V)$ is a trivial 
cocycle in  \eqref{T-bsS-V}  if and only if $X = 0$\,. $\Box$
\end{claim}  
  
Let us now observe that, due to the PBW theorem, 
we have the isomorphism of graded vector spaces
\begin{equation}
\label{PBW}
T (\bs\, S(V)) \cong S \big( \Lie (\bs\, S(V)) \big)  
\end{equation}
Moreover, the differential $\de$ is compatible with 
this isomorphism. In other words, the cochain 
complex  \eqref{T-bsS-V} is isomorphic to the 
symmetric algebra of the cochain complex 
\eqref{sHarr-comp}. 

Since  the cochain complex $S \big( \Lie (\bs\, S(V)) \big) $
splits into the direct sum 
$$
S \big( \Lie (\bs\, S(V)) \big) = \bbK  ~\oplus~  \Lie (\bs\, S(V)) 
~\oplus ~ \bigoplus_{m\ge 2} S^m \big( \Lie (\bs\, S(V)) \big)
$$
the statement of the theorem follows easily from Claim 
\ref{cl:T-bsS-V}.
\end{proof}

\section{Filtered dg Lie algebras. The Goldman-Millson theorem}
\label{app:GM}
In this section we prove a version of the  Goldman-Millson theorem
\cite{GM} which is often used in applications. 

We consider a Lie algebra $\cL$ in the category $\Ch_{\bbK}$
equipped with a descending filtration
\begin{equation}
\label{cL-filtr}
\cL  = \cF_1 \cL \supset   \cF_2 \cL \supset \cF_3 \cL \supset \dots
\end{equation}
which is compatible with the Lie bracket (and the differential). 

We assume that $\cL$ is complete with respect to this filtration. 
Namely,
\begin{equation}
\label{cL-complete}
\cL = \lim_{k} \cL\, \big/ \, \cF_k \cL\,. 
\end{equation}
We call such Lie algebras {\it filtered}.

Condition \eqref{cL-complete} and equality $\cL = \cF_1 \cL$
guarantee that the subalgebra 
$\cL^0$ of degree zero elements in $\cL$ is
a pro-nilpotent Lie algebra (in the category of $\bbK$-vector spaces). 
Hence,  $\cL^0$ can exponentiated to a pro-unipotent 
group which we denote by 
\begin{equation}
\label{the-group}
\exp(\cL^0)\,.
\end{equation}

We recall that a {\it Maurer-Cartan element}  of $\cL$ is 
a degree $1$ vector $\al \in \cL$ satisfying the 
equation
\begin{equation}
\label{eq:MC}
\pa \al + \frac{1}{2} [\al,\al] =0\,,
\end{equation}
where $\pa$ denotes the differential on $\cL$\,.

For a vector $\xi \in \cL^0$ and a Maurer-Cartan element $\al$ 
we consider the new degree $1$ vector $\wt{\al} \in \cL$
which is given by the formula 
\begin{equation}
\label{xi-acts}
\wt{\al} = \exp(\ad_{\xi})\, \al -
\frac{\exp(\ad_{\xi}) - 1}{\ad_{\xi}} \pa \xi\,,
\end{equation}
where the expressions 
$$
 \exp(\ad_{\xi}) \qquad \textrm{and} \qquad
\frac{\exp(\ad_{\xi}) - 1}{\ad_{\xi}}
$$
are defined in the obvious way using the Taylor 
expansions of the functions 
$$
e^x \qquad \textrm{and} \qquad   \frac{e^x -1}{x}
$$
around the point $x = 0$\,, respectively. 

Conditions \eqref{cL-complete} and  $\cL = \cF_1 \cL$ guarantee that the
right hand side of equation \eqref{xi-acts} is defined. 
 
It is known (see, e.g. \cite[Appendix B]{BDW} or \cite{GM}) that, 
for every Maurer-Cartan element $\al$ and for every degree zero 
vector $\xi \in \cL$, the vector $\wt{\al}$ in \eqref{xi-acts}
is also a Maurer-Cartan element. Furthermore, formula \eqref{xi-acts} 
defines an action of the group \eqref{the-group} on 
the set of Maurer-Cartan elements of $\cL$\,. 

The transformation groupoid $\MC(\cL)$ corresponding to this 
action is called the {\it Deligne groupoid}  of the Lie 
algebra $\cL$\,. This groupoid and its higher versions 
were studied extensively by E. Getzler in \cite{Ezra} 
and \cite{Ezra-infty}.   

\begin{rem}
\label{rem:well}
The transformation groupoid $\MC(\cL)$ may be 
defined without imposing the assumption $\cL = \cF_1 \cL$. 
In this more general case, the group \eqref{the-group}
should be replaced by 
$$
\exp(\cF_1 \cL^0)\,.
$$
\end{rem}

Let 
$$
\vf : \cL \to \wt{\cL} 
$$ 
be a homomorphism of two filtered dg Lie algebras. 

It is obvious that for every Maurer-Cartan element $\al \in \cL$ the 
vector $\vf(\al)$ is a Maurer-Cartan element of $\wt{\cL}$\,. Moreover
the assignment
$$
\al \to \vf(\al) 
$$
extends to the functor 
\begin{equation}
\label{vf-star}
\vf_* : \MC(\cL) \to \MC(\wt{\cL})
\end{equation}
between the corresponding Deligne groupoids. 

The following statement is a version of the famous 
Goldman-Millson theorem \cite{GM}.
\begin{thm}
\label{thm:GM}
Let $\vf : \cL \to \wt{\cL}$ be a quasi-isomorphism of 
filtered dg Lie algebras. If the restriction 
$$
\vf \Big|_{\cF_m \cL} ~ :~ \cF_m \cL \to  \cF_m \wt{\cL} 
$$ 
is a quasi-isomorphism for all $m$ then the functor 
\eqref{vf-star} induces a bijection 
\begin{equation}
\label{vf-pi0}
\vf_* : \pi_0 \big(\MC(\cL) \big) \to 
 \pi_0 \big(\MC(\wt{\cL}) \big)
\end{equation}
from the isomorphism classes of Maurer-Cartan elements in $\cL$
to the isomorphism classes of Maurer-Cartan elements in $\wt{\cL}$\,.
\end{thm}
\begin{proof}
Using the conditions of the theorem and Exercise \ref{exer:5-lem} given below, 
it is not hard to see that $\vf$ induces a quasi-isomorphism 
$$
\Gr(\vf) : \cF_m \cL  \big/ \cF_{m+1} \cL \to 
\cF_m \wt{\cL}  \big/ \cF_{m+1} \wt{\cL}
$$ 
for all $m$\,.

In order to prove that the map \eqref{vf-pi0} is surjective we need 
to show that for every Maurer-Cartan element $\beta\in \wt{\cL}$ there exists 
a vector $\xi \in \tcL^0$ and a Maurer-Cartan element $\al \in \cL$ such that 
\begin{equation}
\label{al-xi-beta}
 \exp(\xi) (\beta) = \vf(\al)\,.
\end{equation}

The Maurer-Cartan equation $\pa \beta + [\beta, \beta]/2 = 0$ implies that 
$\beta$ represents a cocycle in 
$$
 \cF_1 \wt{\cL}  \big/ \cF_{2} \wt{\cL} \,.
$$
Hence there exists $\al_1 \in \cF_1 \cL^1$ 
and $\xi_1 \in \cF_1 \tcL^0$ such that 
\begin{equation}
\label{pa-al-1}
\pa \al_1 \in \cF_2 \cL
\end{equation}
and 
\begin{equation}
\label{beta-al-1}
\beta - \pa \xi_1 - \vf(\al_1) \in \cF_2 \wt{\cL}\,.
\end{equation}

Let us denote by $\beta_1$ the Maurer-Cartan element 
$$
\beta_1 = \exp(\xi_1) (\beta)\,.
$$

Inclusion \eqref{beta-al-1} implies that 
\begin{equation}
\label{beta-1-al-1}
\beta_1 - \vf(\al_1) \in \cF_2 \wt{\cL}\,.
\end{equation} 

We showed that there exists a vector  $\xi_1 \in \cF_1 \tcL^0$
and a vector $\al_1 \in  \cF_1 \cL^1$ such that
for 
$$
\beta_1 = \exp(\xi_1) (\beta)
$$
we have inclusion \eqref{beta-1-al-1} and 
the inclusion 
\begin{equation}
\label{induc-1}
\pa \al_1  + \frac{1}{2} [\al_1, \al_1] \in \cF_2 \cL\,,
\end{equation}
which follows from \eqref{pa-al-1}.
Inclusions  \eqref{beta-1-al-1} and \eqref{induc-1}
form the base of our induction.

Now we assume that there exist vectors 
$$
\xi_k \in \cF_k \tcL^0\,, \qquad 1 \le k \le m
$$
and $\al_m\in \cF_1\cL$ such that 
\begin{equation}
\label{induc-m}
\pa \al_m  + \frac{1}{2} [\al_m, \al_m] \in \cF_{m+1} \cL\,,
\end{equation}
and 
\begin{equation}
\label{induction-m}
\beta_m - \vf(\al_m) \in \cF_{m+1} \tcL\,,
\end{equation}
where 
\begin{equation}
\label{beta-m}
\beta_m =  \exp(\xi_m)  \dots \exp(\xi_1) (\beta)\,.
\end{equation}

Let us consider the vector
\begin{equation}
\label{vector}
\big( \pa \vf(\al_m)  + \frac{1}{2}
 [\vf(\al_m), \vf(\al_m)] \big) - 
 \pa (\vf(\al_m) - \beta_m)
\end{equation}
in  $\cF_{m+1} \tcL^2$\,.

Using the Maurer-Cartan equation for $\beta_m$
we can rewrite \eqref{vector} as
$$
\big( \pa \vf(\al_m)  + \frac{1}{2}
 [\vf(\al_m), \vf(\al_m)] \big) - 
 \pa (\vf(\al_m) - \beta_m) =
\frac{1}{2} \big(  [\vf(\al_m), \vf(\al_m)] - [\beta_m, \beta_m]  \big)=
$$
$$
\frac{1}{2} \big(  [\vf(\al_m), \vf(\al_m)] - [\vf(\al_m), \beta_m] + [\vf(\al_m), \beta_m] 
- [\beta_m, \beta_m]  \big) =
$$
$$
\frac{1}{2} \big(  [\vf(\al_m), \vf(\al_m)- \beta_m]  + [\vf(\al_m)- \beta_m, \beta_m]  \big)\,.
$$
Thus \eqref{induction-m} implies that  vector \eqref{vector} belongs to 
$\cF_{m+2} \tcL^2$\,. 

On the other hand, applying the differential $\pa$ to the vector 
$$
\big( \pa \vf(\al_m)  + \frac{1}{2}
 [\vf(\al_m), \vf(\al_m)] \big)
$$
and using \eqref{induc-m} together with the Jacobi identity we conclude that 
$$
\pa \, \big( \pa \vf(\al_m)  + \frac{1}{2}
 [\vf(\al_m), \vf(\al_m)] \big) \in \cF_{m+2} \wt{\cL}^3\,.
$$

Combining this observation with the fact that  
vector \eqref{vector} belongs to $\cF_{m+2} \tcL^2$ we
deduce that
$$
\vf \big( \pa \al_m + \frac{1}{2} [\al_m, \al_m] \big)
$$
represents an exact cocycle in 
$$
 \cF_{m+1} \wt{\cL}  \big/ \cF_{m+2} \wt{\cL} \,.
$$

Therefore, there exists a vector $\ga_{m+1} \in \cF_{m+1} \cL^1$
such that 
\begin{equation}
\label{alm-MC}
\pa \ga_{m+1} + \pa \al_m + \frac{1}{2} [\al_m, \al_m]  \in \cF_{m+2}\cL\,.
\end{equation}

Let us denote by $\al'_{m+1}$ the vector 
$$
\al'_{m+1} = \al_m + \ga_{m+1}\,.
$$

Combining \eqref{alm-MC} with the fact that vector \eqref{vector} belongs to 
$\cF_{m+2} \tcL^2$ we conclude that 
$$
 \pa (\beta_m - \vf(\al'_{m+1})) \in  \cF_{m+2} \tcL\,.
$$
In other words, $\beta_m - \vf(\al'_{m+1})$ represents a 
cocycle in 
$$
 \cF_{m+1} \wt{\cL}  \big/ \cF_{m+2} \wt{\cL} \,.
$$

Therefore, there exists a vector $\xi_{m+1} \in \cF_{m+1} \tcL^0$
and a vector $\ga'_{m+1} \in \cF_{m+1}\cL^1$ such that
\begin{equation}
\label{pa-al-prpr}
\pa \ga'_{m+1}  \in  \cF_{m+2}\cL^2
\end{equation}
and  
\begin{equation}
\label{alm-m1}
\beta_m - \pa \xi_{m+1} - \vf(\al_{m+1}) - \vf(\ga'_{m+1}) \in \cF_{m+2} \tcL\,.
\end{equation}

We set 
$$
\al_{m+1}= \al'_{m+1} + \ga'_{m+1}
$$
and
$$
\beta_{m+1} = \exp(\xi_{m+1})(\beta_m)\,.
$$

Combining \eqref{alm-MC} together with \eqref{pa-al-prpr} and  
\eqref{alm-m1} we see that $\al_{m+1}$, $\beta_{m+1}$ and 
$\xi_{m+1}$ satisfy the inductive assumption for $m$ replaced by $m+1$. 

Thus, we conclude that, there exist sequences of vectors
$$
\al_m \in \cF_1 \cL^1\,, \qquad  \al_{m+1} - \al_m  \in \cF_{m+1} \cL^1\,,  \qquad 
m \ge 1
$$
and
$$
\xi_m  \in  \cF_m \tcL^0\,, \qquad m \ge 1
$$
such that inclusions \eqref{induc-m} and \eqref{induction-m}
hold for all $m$. 

Since the filtrations on $\cL$ and $\tcL$ are complete the sequence $\{\al_m\}_{m \ge 1}$
converges to a vector $\al\in \cL^1$ and the sequence
$$
\Big\{\, \CH\Big( \xi_m, \dots, \CH\big(\xi_3, \CH (\xi_2, \xi_1)\big) \dots \Big) 
\, \Big\}_{m \ge 1}
$$
converges to a vector $\xi \in \tcL^0$ such that 
$$
\pa \al + \frac{1}{2} [\al, \al] = 0
$$
and  
$$
\exp(\xi) (\beta) = \vf(\al)\,.
$$

We proved that the map   \eqref{vf-pi0} is surjective.

Due to Exercise \ref{exer:GM1/2} below the map  \eqref{vf-pi0}
is also injective. Thus the theorem is proved. 
\end{proof}

\begin{exer}
\label{exer:5-lem}
If the rows in the commutative diagram of cochain complexes 
\begin{center}
\begin{tikzpicture}[scale=.75]
\node at (-4,0){$0$};
\node at (-2,0){$A$};
\node at (0,0){$B$};
\node at (2,0){$C$};
\node at (4,0){$0$};
\draw[->] (-3.5,0) -- (-2.5,0);
\draw[->] (-1.5,0) -- (-0.5,0);
\draw[->] (0.5,0) -- (1.5,0);
\draw[->] (2.5,0) -- (3.5,0);

\node at (-4,-2){$0$};
\node at (-2,-2){$A'$};
\node at (0,-2){$B'$};
\node at (2,-2){$C'$};
\node at (4,-2){$0$};
\draw[->] (-3.5,-2) -- (-2.5,-2);
\draw[->] (-1.5,-2) -- (-0.5,-2);
\draw[->] (0.5,-2) -- (1.5,-2);
\draw[->] (2.5,-2) -- (3.5,-2);

\draw[->] (-2,-0.5) -- (-2,-1.5);
\draw[->] (0,-0.5) -- (0,-1.5);
\draw[->] (2,-0.5) -- (2,-1.5);
\end{tikzpicture}
\end{center}
are exact and any 2 vertical maps are quasi-isomorphisms, then show
that the third vertical map is also a quasi-isomorphism.
{\it Hint: Consider the 5-lemma (Sec.\ II.5 in \cite{GelManin}).}
\end{exer}

\begin{exer}
\label{exer:GM1/2}
Prove that the map \eqref{vf-pi0} is injective. 
\end{exer}

\section{Solutions to selected exercises}
\label{app:solutions}

~\\
{\bf Solution of Exercise \ref{exer:cobar-conv}.}
We need only to consider generators of $\Op(\bs\, \cC_{\c})$
i.e.\ $(\bq_n,\bs X)$, where
$\bq_n$ is the standard $n$-corolla, and $X \in \cC_{\c}(n)$. 

By definition, 
\begin{equation}
\label{soleq0}
F (\pa^{\Cobar} (\bq_n,\bs X)) = \pa^{\cO}F((\bq_n,\bs X))
\end{equation}
if and only if
\begin{equation}
\label{soleq1}
\alpha_{F}(\pa^{\cC}X) + \pa^{\cO}\alpha_{F}(X) - F(\pa'' (\bq_n,\bs X)) =0,
\end{equation}
where $\alpha_{F} \in \Conv(\cC,\cO)$ is the degree
1 map $\alpha_{F}(X) = F ((\bq_n,\bs X))$, and
\begin{equation}
\label{soleq2}
\pa'' (\bq_n,\bs X) = - \sum_{z \in  \pi_0(\Tree_2(n))} ( \bs \otimes \bs) (\bt_{z};\Delta_{\bt_{z}}(X)).
\end{equation}
By definition of the differential on $\Conv(\cC,\cO)$, Eq. \eqref{soleq1} holds if and only if
\begin{equation}\label{soleq3}
\bigl(\pa \alpha_{F} \bigr)(X) - F(\pa'' (\bq_n,\bs X)) =0,
\end{equation}
Next, expanding the right-hand side of Eq. \eqref{soleq2} gives:
\[
\pa'' (\bq_n,\bs X) = - \sum_{z \in  \pi_0(\Tree_2(n))} \sum_{\al} (-1)^{| X^{1}_{\al} |}(\bt_{z} ; \bs X^{1}_{\al}
\otimes \bs X^{2}_{\al}),
\]
where $X^{1}_{\al}$ and $X^{2}_{\al}$ are tensor factors in 
$$
\D_{\bt_{z}}(X)  = \sum_{\al} X^{1}_{\al} \otimes X^{2}_{\al}\,.
$$

Let $p_z$ be the number of edges terminating at the 
second nodal vertex of $\bt_z$ and let
\[
\tilde{\mu}_{\bt_{z}} \colon \Op(\bs\, \cC_{\c})(n-p_z+1) 
\otimes \Op(\bs\, \cC_{\c})(p_z) \to \Op(\bs\, \cC_{\c})(n)
\]
be the multiplication map for the tree $\bt_{z}$.  
By definition of multiplication for the free operad, we have
\[
(\bt_{z} ; \bs X^{1}_{\al} \otimes \bs X^{2}_{\al} ) = 
\tilde{\mu}_{\bt_{z}} \big( (\bq_{n-p_z+1}, \bs X^{1}_{\al}) \otimes (\bq_{p_z},
\bs X^{2}_{\al}) \big)
\]
Since $F$ is a map of operads, we have the following equalities:
\begin{align*}
F(\pa'' (\bq_n,\bs X)) &=  - \sum_{z \in  \pi_0(\Tree_2(n))} \sum_{\al}
(-1)^{| X^{1}_{\al} |}
F \bigl ( 
\tilde{\mu}_{\bt_{z}} \bigl ((\bq_{n-p_z+1}, \bs X^{1}_{\al}) \otimes (\bq_{p_z},
\bs X^{2}_{\al}) \bigr)
\bigr)\\
&=
- \sum_{z \in  \pi_0(\Tree_2(n))} \sum_{\al}
(-1)^{\vert X^{1}_{\al} \vert}
\mu_{\bt_{z}} \bigl ( F (\bq_{n-p_z+1}, \bs X^{1}_{\al}) \otimes F(\bq_{p_z},
\bs X^{2}_{\al}) \bigr)\\
&= -
\sum_{z \in  \pi_0(\Tree_2(n))} \sum_{\al}
(-1)^{\vert X^{1}_{\al} \vert}
\mu_{\bt_{z}} \bigl ( \alpha_{F} (X^{1}_{\al}) \otimes \alpha_{F} (X^{2}_{\al}) \bigr)\\
&= -
\sum_{z \in  \pi_0(\Tree_2(n))}
\mu_{\bt_{z}} \bigl ( \alpha_{F} \otimes \alpha_{F} \circ \Delta_{\bt_{z}}(X) \bigr)\\
&= -\alpha_{F} \bullet \alpha_{F} (X)\\
&=-\frac{1}{2}[\alpha_{F},\alpha_{F}](X).
\end{align*}
By substituting this last equality into Eq.\ \eqref{soleq3}, we see
Eq.\ \eqref{soleq0} holds  if and only if the Maurer-Cartan equation
\[
\pa \alpha_{F}  + \frac{1}{2}[\alpha_{F},\alpha_{F}]=0
\]
holds for $\alpha_{F}$. $\tri$


~\\
~\\
{\bf Solution of Exercise \ref{exer:if}.}
Assume the Maurer-Cartan elements $\alpha_{F}$ and $\alpha_{\wt{F}}$
corresponding to the maps $F,  \wt{F} \colon \Cobar(\cC) \to \cO$ are
isomorphic as objects of the Deligne groupoid. By definition (see Eq. \eqref{xi-acts}) this 
implies that there exists a degree 0 element $\xi \in
\Conv(\cC_{\c}, \cO)$ such that
\[
\alpha_{\wt{F}} = \exp(\ad_{\xi})\, \alpha_{F} -
\frac{\exp(\ad_{\xi}) - 1}{\ad_{\xi}} \pa  \xi\ .
\]
Define $\alpha(t) \in \Conv(\cC_{\c}, \cO)[[t]]$  to be:
\[
\alpha(t) = \exp(-t \ad_{\xi})\, \alpha_{F} -
\frac{\exp(-t \ad_{\xi}) - 1}{\ad_{\xi}} \pa  \xi\ .
\]
Since $\alpha_{F}$ and $\xi$ are elements of
$\cF_{1}\Conv(\cC_{\c}, \cO)$, and the bracket and differential are compatible with the
filtration, we conclude that
\[
\alpha(t) \in \Conv(\cC_{\c}, \cO)\{t\}.
\]
Note $\alpha(0)=\alpha_{F}$ and $\alpha(1)=\alpha_{\wt{F}}$.
Differentiation of $\alpha(t)$ gives:
\begin{align*}
\frac{d  \alpha(t)}{d t} & = -\ad_{\xi} \bigl (\exp(-t \ad_{\xi})\,
\alpha_{F} \bigr) + \exp(-t \ad_{\xi}) \pa  \xi\ \\
& = -\ad_{\xi} \bigl (\exp(-t \ad_{\xi})\,
\alpha_{F} \bigr) + \exp(-t \ad_{\xi}) \pa  \xi - \pa 
\xi + \pa  \xi\\
&=-\ad_{\xi} \bigl (\exp(-t \ad_{\xi})\,
\alpha_{F} \bigr) - \frac{-\ad_{\xi}}{\ad_{\xi}} \biggl( \exp(-t \ad_{\xi}) \pa  \xi - \pa 
\xi \biggr) + \pa  \xi\\
&=-\ad_{\xi} \biggl (\exp(-t \ad_{\xi})\,
\alpha_{F}  - \frac{\exp(-t \ad_{\xi}) - 1}{\ad_{\xi}} \pa  \xi \biggr) + \pa  \xi\\
& = \pa  \xi - [\xi,\alpha(t)].
\end{align*}
Thus, applying Prop.\ C.1 of \cite{stable1}, we conclude that 
\[
\pa  \alpha(t) + \frac{1}{2}[\alpha(t),\alpha(t)]=0
\]
for all $t$.

Hence, equations \eqref{al-H}, \eqref{MC-al-H1}, and \eqref{MC-al-H0},
which are described in the ``only if'' part of the proof, imply that
\[
\al_{H} = \alpha(t) + \xi d t \in \Conv(\cC_{\c}, \cO^{I})
\]
is a Maurer-Cartan element that corresponds to a homotopy $H \colon \Cobar(\cC) \to \cO^I$
between $F$ and $\wt{F}$.  $\tri$

~\\
~\\
{\bf Solution of Exercise \ref{exer:cG-n-TwGer}.}
The space 
$$
\bs^{2r} \big( \Ger(r+n) \big)^{S_r}
$$
is spanned by vectors of the form 
\begin{equation}
\label{Av-w}
\Av(w)  = \sum_{\si \in S_r} \si (w)
\end{equation}
where $w$ is a monomial in $\bs^{2r}\Ger(r+n)$\,.

It is clear that 
$$
f^{-1} \big( \Av(w)\big)= w(\underbrace{a, a, \dots, a}_{r~ \textrm{times}}, a_{1}, \dots, a_{n})\,.  
$$ 
So our goal is to show that 
\begin{equation}
\label{goal-pa-de}
\pa^{\Tw} \big( \Av(w) \big) = 
\end{equation}
$$
\sum_{\si \in S_{r+1}} \sum_{i=1}^{r}
\frac{(-1)^{e_i}}{2}\, w (a_{\si(1)}, \dots, a_{\si(i-1)}, \{a_{\si(i)}, a_{\si(i+1)}\}, a_{\si(i+2)}, 
\dots, a_{\si(r+1)},  a_{r+2}, \dots, a_{r+1+n})\,,
$$
where the sign factor $(-1)^{e_i}$ comes from swapping the odd operator 
$\{a_{\si(i)}, ~\}$ with the corresponding number of brackets.

Following the definition of $\pa^{\Tw}$ \eqref{diff-Tw-Ger} we get 
$$
\pa^{\Tw} \big( \Av(w) \big)  =
\sum_{\tau \in \Sh_{1,r}}
\sum_{\si \in S_{2, \dots, r+1}} 
\tau\big( 
\{a_1, w (a_{\si(2)}, \dots, a_{\si(r+1)}, a_{r+2}, \dots, a_{r+1+n})\}
\big)  
$$
$$
- \sum_{i=1}^{n} \sum_{\si \in S_r}
\sum_{\tau' \in \Sh_{r,1}} (-1)^{e_{r+i}}  
\tau' \big( 
w (a_{\si(1)}, \dots, a_{\si(r)}, a_{r+2},  
\dots, a_{r+i}, \{a_{r+1}, a_{r+i+1}\}, a_{r+i+2},
\dots, a_{r+1+n})
\big)  
$$
$$
-(-1)^{|w|} \sum_{\tau \in \Sh_{2, r-1}}
\tau \big(  w \circ_1 \{a_1,a_2\}\big) =
$$
\begin{equation}
\label{pa-Av-w}
\sum_{\si \in S_{r+1}}
\{a_{\si(1)}, w (a_{\si(2)}, \dots, a_{\si(r+1)}, a_{r+2}, \dots, a_{r+1+n})\}
\end{equation}
$$
- \sum_{i=1}^{n} \sum_{\si \in S_{r+1}} (-1)^{e_{r+i}}  
w (a_{\si(1)}, \dots, a_{\si(r)}, a_{r+2},  
\dots, a_{r+i}, \{a_{\si(r+1)}, a_{r+i+1}\}, a_{r+i+2},
\dots, a_{r+1+n})
$$
$$
- \sum_{\tau \in \Sh_{2, r-1}}
 \sum_{\si \in S_{3, \dots, r+1}}  
\sum_{i=1}^{r}  (-1)^{e_i}
\tau \circ \si \big( 
w(a_{3}, \dots, a_{i+1}, \{a_1, a_2 \},  a_{i+2}, \dots, a_{r +1 +n})  
\big)\,,
$$
where we used the obvious identity 
\begin{equation}
\label{Av-w-new}
\Av(w) = \sum_{\si \in S_{2, \dots, r}}  
\sum_{i=1}^{r} 
w(a_{\si(2)}, \dots, a_{\si(i)}, a_1,  a_{\si(i+1)}, \dots,  a_{\si(r)}, a_{r+1}, \dots, a_{r+n})\,. 
\end{equation}

Using the defining identities of Gerstenhaber algebra
we simplify \eqref{pa-Av-w} further
\begin{equation}
\label{pa-Av-w-new}
\pa^{\Tw} \big( \Av(w) \big)  =
\end{equation}
$$
\sum_{\si \in S_{r+1}} \sum_{i=2}^{r+1}
(-1)^{e_i}
w (a_{\si(2)}, \dots, a_{\si(i-1)}, \{a_{\si(1)}, a_{\si(i)}\}, 
a_{\si(i+1)}, \dots, a_{\si(r+1)}, a_{r+2}, \dots, a_{r+1+n})\}
$$
$$
- \sum_{\si \in S_{r+1}}^{\si(1) < \si(2)}
\sum_{i=1}^{r}  (-1)^{e_i}
w(a_{\si(3)}, \dots, a_{\si(i+1)}, \{a_{\si(1)}, a_{\si(2)} \},  a_{\si(i+2)}, 
\dots, a_{\si(r+1)}, a_{r+2}, 
\dots, a_{r +1 +n}) = 
$$
$$
\sum_{\si \in S_{r+1}} \sum_{i=2}^{r+1}
(-1)^{e_i}
w (a_{\si(2)}, \dots, a_{\si(i-1)}, \{a_{\si(1)}, a_{\si(i)}\}, 
a_{\si(i+1)}, \dots, a_{\si(r+1)}, a_{r+2}, \dots, a_{r+1+n})\}
$$
$$
- \sum_{\si \in S_{r+1}}
\sum_{i=1}^{r}  \frac{(-1)^{e_i}}{2}
w(a_{\si(3)}, \dots, a_{\si(i+1)}, \{a_{\si(1)}, a_{\si(2)} \},  a_{\si(i+2)}, 
\dots, a_{\si(r+1)}, a_{r+2}, 
\dots, a_{r +1 +n}) = 
$$
$$
 \sum_{\si \in S_{r+1}}
\sum_{i=1}^{r}  \frac{(-1)^{e_i}}{2}
w(a_{\si(3)}, \dots, a_{\si(i+1)}, \{a_{\si(1)}, a_{\si(2)} \},  a_{\si(i+2)}, 
\dots, a_{\si(r+1)}, a_{r+2}, 
\dots, a_{r +1 +n})=
$$
$$
\sum_{\si \in S_{r+1}} \sum_{i=1}^{r}
\frac{(-1)^{e_i}}{2}\, w (a_{\si(1)}, \dots, a_{\si(i-1)}, \{a_{\si(i)}, a_{\si(i+1)}\}, a_{\si(i+2)}, 
\dots, a_{\si(r+1)},  a_{r+2}, \dots, a_{r+1+n})\,.
$$

Thus equation \eqref{goal-pa-de} indeed holds and the desired 
statement follows.  $\tri$


~\\
~\\
{\bf Solution of Exercise \ref{exer:Gcc-Gww}.}
According to the formula for $\pa^{\Tw}$ given in Eq.\ \eqref{diff-TwGra} we have 
\begin{equation} 
\label{diff_Gww_1}
\pa^{\Tw} \G_{\ww} = \Av_1 \big( \G_{\ed} \circ_{2} \G_{\ww} \big) - \Av_1 \bigl ( \G_{\ww}
\circ_{1} \G_{\ed} + \vs_{1,2} ( \G_{\ww} \circ_{2} \G_{\ed}) \bigr),
\end{equation}
where $\vs_{1,2}$ is the cycle $(12) \in S_{3}$. Recall that, in the 
right hand side of \eqref{diff_Gww_1}, both graphs $\G_{\ed}$ and 
$\G_{\ww}$ are viewed as vectors in $\Gra(2)$, while the final result of 
the computation is treated as a vector in $\Tw\Gra(2)$\,. In particular, 
the colors of vertices play a role only for the final result of the computation.
(See also Remark \ref{rem:neutral-in-gra}.)

Expanding the terms on the right hand side gives the following equalities:
\begin{center}
\begin{tikzpicture}[scale=.5]
\tikzstyle{w}=[circle, draw, minimum size=4, inner sep=1]
\tikzstyle{b}=[circle, draw, fill, minimum size=4, inner sep=1]
\node at (-4,0){$\G_{\ed} \circ_{2} \G_{\ww}=$};

\node[b] (a1) at (0,0) {};
\draw (0,-0.6) node[anchor=center] {{\small $1$}};
\node[w] (a2) at (2,0) {};
\draw (2,-0.6) node[anchor=center] {{\small $2$}};
\node[w] (a3) at (4,0) {};
\draw (4,-0.6) node[anchor=center] {{\small $3$}};

\draw (a1) edge (a2);

\node at (6,0){$+$};

\node[b] (a4) at (8,0) {};
\draw (8,-0.6) node[anchor=center] {{\small $1$}};
\node[w] (a5) at (10,0) {};
\draw (10,-0.6) node[anchor=center] {{\small $3$}};
\node[w] (a6) at (12,0) {};
\draw (12,-0.6) node[anchor=center] {{\small $2$}};

\draw (a4) edge (a5);

\end{tikzpicture}

\begin{tikzpicture}[scale=.5]
\tikzstyle{w}=[circle, draw, minimum size=4, inner sep=1]
\tikzstyle{b}=[circle, draw, fill, minimum size=4, inner sep=1]
\node at (-4,0){$\G_{\ww} \circ_{1} \G_{\ed}=$};

\node[b] (a1) at (0,0) {};
\draw (0,-0.6) node[anchor=center] {{\small $1$}};
\node[w] (a2) at (2,0) {};
\draw (2,-0.6) node[anchor=center] {{\small $2$}};
\node[w] (a3) at (4,0) {};
\draw (4,-0.6) node[anchor=center] {{\small $3$}};

\draw (a1) edge (a2);

\end{tikzpicture}

\begin{tikzpicture}[scale=.5]
\tikzstyle{w}=[circle, draw, minimum size=4, inner sep=1]
\tikzstyle{b}=[circle, draw, fill, minimum size=4, inner sep=1]
\node at (-4,0){$\vs_{1,2} (\G_{\ww} \circ_{2} \G_{\ed})=$};

\node[b] (a1) at (0,0) {};
\draw (0,-0.6) node[anchor=center] {{\small $1$}};
\node[w] (a2) at (2,0) {};
\draw (2,-0.6) node[anchor=center] {{\small $3$}};
\node[w] (a3) at (4,0) {};
\draw (4,-0.6) node[anchor=center] {{\small $2$}};

\draw (a1) edge (a2);
\end{tikzpicture}
\end{center}

Hence, all terms cancel on the right hand side of Eq.\
\eqref{diff_Gww_1}, and therefore $\pa^{\Tw} \G_{\ww}=0$.

Next, applying the differential $\pa^{\Tw}$ to $\G_{\cc}\in \Tw\Gra(2)$, 
we get
\begin{equation}
\label{diff_Gww_2}
\pa^{\Tw} \G_{\cc} =\Av_1 \bigl( \G_{\ed} \circ_{2} \G_{\cc} \bigr) + \Av_1 \bigl ( \G_{\cc}
\circ_{1} \G_{\ed} + \vs_{1,2} ( \G_{\cc} \circ_{2} \G_{\ed}) \bigr).
\end{equation}
We expand the terms on the right hand side, being mindful of
the ordering on edges, and Remark \ref{rem:neutral-in-gra}.
\begin{center}
\begin{tikzpicture}[scale=.5]
\tikzstyle{w}=[circle, draw, minimum size=4, inner sep=1]
\tikzstyle{b}=[circle, draw, fill, minimum size=4, inner sep=1]
\node at (-4,0){$\G_{\ed} \circ_{2} \G_{\cc}=$};

\node[b] (a1) at (0,0) {};
\draw (0,-0.6) node[anchor=center] {{\small $1$}};
\node[b] (a2) at (2,0) {};
\draw (2,-0.6) node[anchor=center] {{\small $2$}};

\draw (1,+0.6) node[anchor=center] {{\small $i$}};
\draw (a1) edge (a2);

\node at (4,0) {$\circ_{2}$};

\node[w] (a3) at (6,0) {};
\draw (6,-0.6) node[anchor=center] {{\small $1$}};
\node[w] (a4) at (8,0) {};
\draw (8,-0.6) node[anchor=center] {{\small $2$}};

\draw (7,+0.6) node[anchor=center] {{\small $ii$}};
\draw (a3) edge (a4);

\node at (10,0){$=$};

\node[b] (b1) at (12,0) {};
\draw (12,-0.6) node[anchor=center] {{\small $1$}};
\node[w] (b2) at (14,0) {};
\draw (14,-0.6) node[anchor=center] {{\small $2$}};
\node[w] (b3) at (16,0) {};
\draw (16,-0.6) node[anchor=center] {{\small $3$}};

\draw (13,+0.6) node[anchor=center] {{\small $i$}};
\draw (15,+0.6) node[anchor=center] {{\small $ii$}};
\draw (b1) edge (b2);
\draw (b2) edge (b3);

\node at (18,0){$+$};

\node[b] (b4) at (20,0) {};
\draw (20,-0.6) node[anchor=center] {{\small $1$}};
\node[w] (b5) at (22,0) {};
\draw (22,-0.6) node[anchor=center] {{\small $3$}};
\node[w] (b6) at (24,0) {};
\draw (24,-0.6) node[anchor=center] {{\small $2$}};

\draw (21,+0.6) node[anchor=center] {{\small $i$}};
\draw (23,+0.6) node[anchor=center] {{\small $ii$}};
\draw (b4) edge (b5);
\draw (b5) edge (b6);

\end{tikzpicture}

\begin{tikzpicture}[scale=.5]
\tikzstyle{w}=[circle, draw, minimum size=4, inner sep=1]
\tikzstyle{b}=[circle, draw, fill, minimum size=4, inner sep=1]
\node at (-4,0){$\G_{\cc} \circ_{1} \G_{\ed}=$};

\node[w] (a1) at (0,0) {};
\draw (0,-0.6) node[anchor=center] {{\small $1$}};
\node[w] (a2) at (2,0) {};
\draw (2,-0.6) node[anchor=center] {{\small $2$}};

\draw (1,+0.6) node[anchor=center] {{\small $i$}};
\draw (a1) edge (a2);

\node at (4,0) {$\circ_{1}$};

\node[b] (a3) at (6,0) {};
\draw (6,-0.6) node[anchor=center] {{\small $1$}};
\node[b] (a4) at (8,0) {};
\draw (8,-0.6) node[anchor=center] {{\small $2$}};

\draw (7,+0.6) node[anchor=center] {{\small $ii$}};
\draw (a3) edge (a4);

\node at (10,0){$=$};

\node[b] (b1) at (12,0) {};
\draw (12,-0.6) node[anchor=center] {{\small $1$}};
\node[w] (b2) at (14,0) {};
\draw (14,-0.6) node[anchor=center] {{\small $2$}};
\node[w] (b3) at (16,0) {};
\draw (16,-0.6) node[anchor=center] {{\small $3$}};

\draw (13,+0.6) node[anchor=center] {{\small $ii$}};
\draw (15,+0.6) node[anchor=center] {{\small $i$}};
\draw (b1) edge (b2);
\draw (b2) edge (b3);

\node at (18,0){$+$};

\node[w] (b4) at (20,0) {};
\draw (20,-0.6) node[anchor=center] {{\small $2$}};
\node[b] (b5) at (22,0) {};
\draw (22,-0.6) node[anchor=center] {{\small $1$}};
\node[w] (b6) at (24,0) {};
\draw (24,-0.6) node[anchor=center] {{\small $3$}};

\draw (21,+0.6) node[anchor=center] {{\small $ii$}};
\draw (23,+0.6) node[anchor=center] {{\small $i$}};
\draw (b4) edge (b5);
\draw (b5) edge (b6);

\end{tikzpicture}
\end{center}
\begin{tikzpicture}[scale=.5]
\tikzstyle{w}=[circle, draw, minimum size=4, inner sep=1]
\tikzstyle{b}=[circle, draw, fill, minimum size=4, inner sep=1]
\node at (-4,0){$\vs_{1,2}(\G_{\cc} \circ_{2} \G_{\ed})=$};

\node at (1,0){$\vs_{1,2}($};
\node[w] (a1) at (2,0) {};
\draw (2,-0.6) node[anchor=center] {{\small $1$}};
\node[w] (a2) at (4,0) {};
\draw (4,-0.6) node[anchor=center] {{\small $2$}};

\draw (3,+0.6) node[anchor=center] {{\small $i$}};
\draw (a1) edge (a2);

\node at (6,0) {$\circ_{2}$};

\node[b] (a3) at (8,0) {};
\draw (8,-0.6) node[anchor=center] {{\small $1$}};
\node[b] (a4) at (10,0) {};
\draw (10,-0.6) node[anchor=center] {{\small $2$}};

\draw (9,+0.6) node[anchor=center] {{\small $ii$}};
\draw (a3) edge (a4);

\node at (11,0){$)=$};

\node[w] (b1) at (12,0) {};
\draw (12,-0.6) node[anchor=center] {{\small $2$}};
\node[w] (b2) at (14,0) {};
\draw (14,-0.6) node[anchor=center] {{\small $3$}};
\node[b] (b3) at (16,0) {};
\draw (16,-0.6) node[anchor=center] {{\small $1$}};

\draw (13,+0.6) node[anchor=center] {{\small $i$}};
\draw (15,+0.6) node[anchor=center] {{\small $ii$}};
\draw (b1) edge (b2);
\draw (b2) edge (b3);

\node at (18,0){$+$};

\node[w] (b4) at (20,0) {};
\draw (20,-0.6) node[anchor=center] {{\small $2$}};
\node[b] (b5) at (22,0) {};
\draw (22,-0.6) node[anchor=center] {{\small $1$}};
\node[w] (b6) at (24,0) {};
\draw (24,-0.6) node[anchor=center] {{\small $3$}};

\draw (21,+0.6) node[anchor=center] {{\small $i$}};
\draw (23,+0.6) node[anchor=center] {{\small $ii$}};
\draw (b4) edge (b5);
\draw (b5) edge (b6);

\end{tikzpicture}

By definition of the operad $\Gra$, we have
the following equalities in $\Gra(1 +2)$: 

\begin{center}
\begin{tikzpicture}[scale=.5]
\tikzstyle{w}=[circle, draw, fill, minimum size=4, inner sep=1]
\tikzstyle{b}=[circle, draw, fill, minimum size=4, inner sep=1]

\node[b] (b1) at (0,0) {};
\draw (0,-0.6) node[anchor=center] {{\small $1$}};
\node[w] (b2) at (2,0) {};
\draw (2,-0.6) node[anchor=center] {{\small $2$}};
\node[w] (b3) at (4,0) {};
\draw (4,-0.6) node[anchor=center] {{\small $3$}};

\draw (1,+0.6) node[anchor=center] {{\small $i$}};
\draw (3,+0.6) node[anchor=center] {{\small $ii$}};
\draw (b1) edge (b2);
\draw (b2) edge (b3);

\node at (6,0){$=$};

\node at (7,0){$-$};
\node[b] (b1) at (8,0) {};
\draw (8,-0.6) node[anchor=center] {{\small $1$}};
\node[w] (b2) at (10,0) {};
\draw (10,-0.6) node[anchor=center] {{\small $2$}};
\node[w] (b3) at (12,0) {};
\draw (12,-0.6) node[anchor=center] {{\small $3$}};

\draw (9,+0.6) node[anchor=center] {{\small $ii$}};
\draw (11,+0.6) node[anchor=center] {{\small $i$}};
\draw (b1) edge (b2);
\draw (b2) edge (b3);

\end{tikzpicture}

\begin{tikzpicture}[scale=.5]
\tikzstyle{w}=[circle, draw, fill, minimum size=4, inner sep=1]
\tikzstyle{b}=[circle, draw, fill, minimum size=4, inner sep=1]

\node[b] (b1) at (0,0) {};
\draw (0,-0.6) node[anchor=center] {{\small $1$}};
\node[w] (b2) at (2,0) {};
\draw (2,-0.6) node[anchor=center] {{\small $3$}};
\node[w] (b3) at (4,0) {};
\draw (4,-0.6) node[anchor=center] {{\small $2$}};

\draw (1,+0.6) node[anchor=center] {{\small $i$}};
\draw (3,+0.6) node[anchor=center] {{\small $ii$}};
\draw (b1) edge (b2);
\draw (b2) edge (b3);

\node at (6,0){$=$};

\node at (7,0){$-$};
\node[w] (b1) at (8,0) {};
\draw (8,-0.6) node[anchor=center] {{\small $2$}};
\node[w] (b2) at (10,0) {};
\draw (10,-0.6) node[anchor=center] {{\small $3$}};
\node[b] (b3) at (12,0) {};
\draw (12,-0.6) node[anchor=center] {{\small $1$}};

\draw (9,+0.6) node[anchor=center] {{\small $i$}};
\draw (11,+0.6) node[anchor=center] {{\small $ii$}};
\draw (b1) edge (b2);
\draw (b2) edge (b3);

\end{tikzpicture}

\begin{tikzpicture}[scale=.5]
\tikzstyle{w}=[circle, draw, fill, minimum size=4, inner sep=1]
\tikzstyle{b}=[circle, draw, fill, minimum size=4, inner sep=1]

\node[w] (b1) at (0,0) {};
\draw (0,-0.6) node[anchor=center] {{\small $2$}};
\node[b] (b2) at (2,0) {};
\draw (2,-0.6) node[anchor=center] {{\small $1$}};
\node[w] (b3) at (4,0) {};
\draw (4,-0.6) node[anchor=center] {{\small $3$}};

\draw (1,+0.6) node[anchor=center] {{\small $ii$}};
\draw (3,+0.6) node[anchor=center] {{\small $i$}};
\draw (b1) edge (b2);
\draw (b2) edge (b3);

\node at (6,0){$=$};

\node at (7,0){$-$};
\node[w] (b1) at (8,0) {};
\draw (8,-0.6) node[anchor=center] {{\small $2$}};
\node[b] (b2) at (10,0) {};
\draw (10,-0.6) node[anchor=center] {{\small $1$}};
\node[w] (b3) at (12,0) {};
\draw (12,-0.6) node[anchor=center] {{\small $3$}};

\draw (9,+0.6) node[anchor=center] {{\small $i$}};
\draw (11,+0.6) node[anchor=center] {{\small $ii$}};
\draw (b1) edge (b2);
\draw (b2) edge (b3);

\end{tikzpicture}
\end{center}

Thus, in $\Tw\Gra(2)$, we have:

\begin{center}
\begin{tikzpicture}[scale=.5]
\tikzstyle{w}=[circle, draw, minimum size=4, inner sep=1]
\tikzstyle{b}=[circle, draw, fill, minimum size=4, inner sep=1]

\node[b] (b1) at (0,0) {};
\draw (0,-0.6) node[anchor=center] {{\small $1$}};
\node[w] (b2) at (2,0) {};
\draw (2,-0.6) node[anchor=center] {{\small $2$}};
\node[w] (b3) at (4,0) {};
\draw (4,-0.6) node[anchor=center] {{\small $3$}};

\draw (1,+0.6) node[anchor=center] {{\small $i$}};
\draw (3,+0.6) node[anchor=center] {{\small $ii$}};
\draw (b1) edge (b2);
\draw (b2) edge (b3);

\node at (6,0){$=$};

\node at (7,0){$-$};
\node[b] (b1) at (8,0) {};
\draw (8,-0.6) node[anchor=center] {{\small $1$}};
\node[w] (b2) at (10,0) {};
\draw (10,-0.6) node[anchor=center] {{\small $2$}};
\node[w] (b3) at (12,0) {};
\draw (12,-0.6) node[anchor=center] {{\small $3$}};

\draw (9,+0.6) node[anchor=center] {{\small $ii$}};
\draw (11,+0.6) node[anchor=center] {{\small $i$}};
\draw (b1) edge (b2);
\draw (b2) edge (b3);

\end{tikzpicture}

\begin{tikzpicture}[scale=.5]
\tikzstyle{w}=[circle, draw, minimum size=4, inner sep=1]
\tikzstyle{b}=[circle, draw, fill, minimum size=4, inner sep=1]

\node[b] (b1) at (0,0) {};
\draw (0,-0.6) node[anchor=center] {{\small $1$}};
\node[w] (b2) at (2,0) {};
\draw (2,-0.6) node[anchor=center] {{\small $3$}};
\node[w] (b3) at (4,0) {};
\draw (4,-0.6) node[anchor=center] {{\small $2$}};

\draw (1,+0.6) node[anchor=center] {{\small $i$}};
\draw (3,+0.6) node[anchor=center] {{\small $ii$}};
\draw (b1) edge (b2);
\draw (b2) edge (b3);

\node at (6,0){$=$};

\node at (7,0){$-$};
\node[w] (b1) at (8,0) {};
\draw (8,-0.6) node[anchor=center] {{\small $2$}};
\node[w] (b2) at (10,0) {};
\draw (10,-0.6) node[anchor=center] {{\small $3$}};
\node[b] (b3) at (12,0) {};
\draw (12,-0.6) node[anchor=center] {{\small $1$}};

\draw (9,+0.6) node[anchor=center] {{\small $i$}};
\draw (11,+0.6) node[anchor=center] {{\small $ii$}};
\draw (b1) edge (b2);
\draw (b2) edge (b3);

\end{tikzpicture}

\begin{tikzpicture}[scale=.5]
\tikzstyle{w}=[circle, draw, minimum size=4, inner sep=1]
\tikzstyle{b}=[circle, draw, fill, minimum size=4, inner sep=1]

\node[w] (b1) at (0,0) {};
\draw (0,-0.6) node[anchor=center] {{\small $2$}};
\node[b] (b2) at (2,0) {};
\draw (2,-0.6) node[anchor=center] {{\small $1$}};
\node[w] (b3) at (4,0) {};
\draw (4,-0.6) node[anchor=center] {{\small $3$}};

\draw (1,+0.6) node[anchor=center] {{\small $ii$}};
\draw (3,+0.6) node[anchor=center] {{\small $i$}};
\draw (b1) edge (b2);
\draw (b2) edge (b3);

\node at (6,0){$=$};

\node at (7,0){$-$};
\node[w] (b1) at (8,0) {};
\draw (8,-0.6) node[anchor=center] {{\small $2$}};
\node[b] (b2) at (10,0) {};
\draw (10,-0.6) node[anchor=center] {{\small $1$}};
\node[w] (b3) at (12,0) {};
\draw (12,-0.6) node[anchor=center] {{\small $3$}};

\draw (9,+0.6) node[anchor=center] {{\small $i$}};
\draw (11,+0.6) node[anchor=center] {{\small $ii$}};
\draw (b1) edge (b2);
\draw (b2) edge (b3);

\end{tikzpicture}
\end{center}

Hence, all terms on the right hand side of Eq.\ \eqref{diff_Gww_2}
cancel, and therefore $\pa^{\Tw} \G_{\cc} =0$. $\tri$
~\\
~\\

\bibliographystyle{amsplain}





\end{document}